\documentclass[11pt,reqno]{amsproc}
\usepackage[T1]{fontenc}
\usepackage{amsmath,amscd,amssymb}
\usepackage{cite}
\usepackage{color}
\usepackage{graphicx}
\usepackage{psfrag}
\usepackage{multirow}
\usepackage{rotating}
\usepackage{fullpage}
\usepackage{hyperref}
\usepackage{subfigure}
\usepackage{enumerate}
\usepackage[small]{caption}
\newtheorem{theorem}{Theorem}[section]

\newtheorem{proposition}[theorem]{Proposition}
\newtheorem{definition}[theorem]{Definition}

\newtheorem{remark}[theorem]{Remark}

\newtheorem{corollary}[theorem]{Corollary}
\numberwithin{equation}{section}

\usepackage{morefloats}

\newlength{\drop}
\definecolor{amethyst}{rgb}{0.6, 0.4, 0.8}
\definecolor{burgundy}{rgb}{0.5, 0.0, 0.13}

\title{On enforcing maximum principles and achieving 
element-wise species balance for advection-diffusion-reaction 
equations under the finite element method}

\author{\textbf{M.~K.~Mudunuru} and \textbf{K.~B.~Nakshatrala} \\
{\small Department of Civil and Environmental Engineering, University of Houston.}}
\date{\today}
\begin{document}

\begin{titlepage}
  \drop=0.1\textheight
  \centering
  \vspace*{\baselineskip}
  \rule{\textwidth}{1.6pt}\vspace*{-\baselineskip}\vspace*{2pt}
  \rule{\textwidth}{0.4pt}\\[\baselineskip]
  {\LARGE \textbf{\color{burgundy}
    On enforcing maximum principles and \\ [0.3\baselineskip]
    achieving element-wise species balance for \\ [0.3\baselineskip] 
    advection-diffusion-reaction equations under \\ [0.3\baselineskip] 
    the finite element method}}\\[0.3\baselineskip]
    \rule{\textwidth}{0.4pt}\vspace*{-\baselineskip}\vspace{3.2pt}
    \rule{\textwidth}{1.6pt}\\[\baselineskip]
    \scshape
    \vspace*{1\baselineskip}
    Authored by \\[\baselineskip]
    {\Large M.~K.~Mudunuru\par}
    {\itshape Graduate Student, University of Houston.}\\[\baselineskip]
    {\Large K.~B.~Nakshatrala\par}
    {\itshape Department of Civil \& Environmental Engineering \\
    University of Houston, Houston, Texas 77204--4003. \\ 
    \textbf{phone:} +1-713-743-4418, \textbf{e-mail:} knakshatrala@uh.edu \\
    \textbf{website:} http://www.cive.uh.edu/faculty/nakshatrala\par}
    \vspace*{1\baselineskip}
    \begin{figure*}[h]
      \centering
      \subfigure{\includegraphics[scale = 0.30,clip]
      {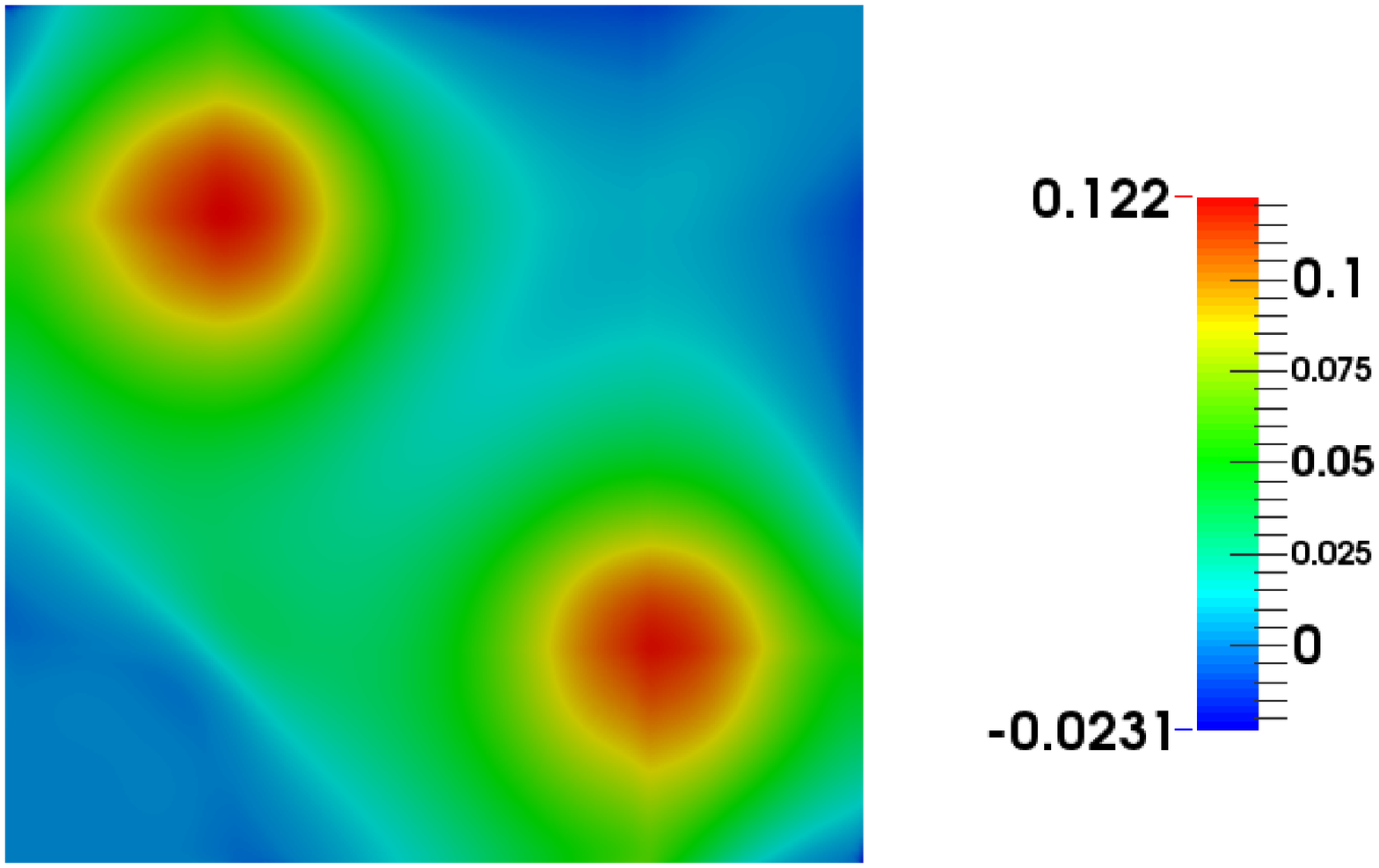}}
      \hspace{0.15in}
      \subfigure{\includegraphics[scale=0.30,clip]
      {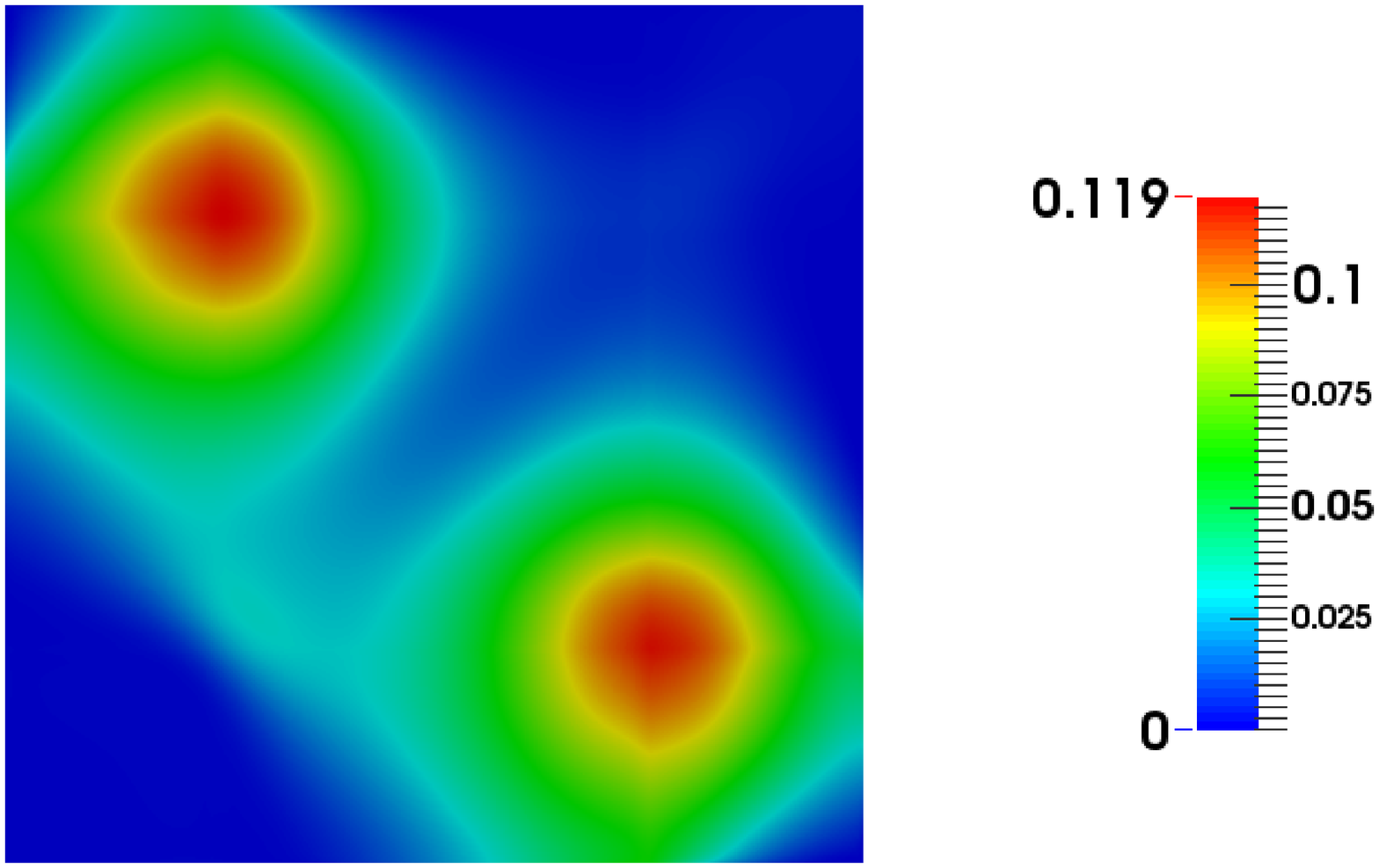}}
      \caption*{These figures show the fate of the product 
      in a transient transport-controlled bimolecular reaction 
      under vortex-stirred mixing. The left figure is obtained 
      using a popular numerical formulation, which violates the 
      non-negative constraint. The right figure is based on the 
      proposed computational framework. These figures clearly 
      illustrate the main contribution of this paper: \emph{The 
        proposed computational framework produces physically 
        meaningful results for advective-diffusive-reactive 
        systems, which is not the case with many popular 
        formulations.}}
    \end{figure*}
    \vfill
    {\scshape 2015} \\
    {\small Computational \& Applied Mechanics Laboratory} \par
\end{titlepage}

\begin{abstract}
  We present a robust computational framework for 
  advective-diffusive-reactive systems that satisfies 
  maximum principles, the non-negative constraint, and 
  element-wise species balance property. The proposed 
  methodology is valid on general computational grids, 
  can handle heterogeneous anisotropic media, and provides 
  accurate numerical solutions even for very high P\'eclet 
  numbers. 
  The significant contribution of this paper is to incorporate 
  advection (which makes the spatial part of the differential 
  operator non-self-adjoint) into the non-negative computational 
  framework, and overcome numerical challenges associated 
  with advection.
  We employ low-order mixed finite element formulations 
  based on least-squares formalism, and enforce explicit 
  constraints on the discrete problem to meet the desired 
  properties. 
  The resulting constrained discrete problem belongs 
  to convex quadratic programming for which a unique 
  solution exists. Maximum principles and the non-negative 
  constraint give rise to bound constraints while element-wise 
  species balance gives rise to equality constraints. 
  The resulting convex quadratic programming problems 
  are solved using an interior-point algorithm. Several 
  numerical results pertaining to advection-dominated 
  problems are presented to illustrate the robustness, 
  convergence, and the overall performance of the 
  proposed computational framework. 
\end{abstract}
\maketitle

\section{INTRODUCTION AND MOTIVATION}
\label{Sec:S1_NN_AD_Intro}
Advection-diffusion-reaction (ADR) equations are 
pervasive in the mathematical modeling of several 
important phenomena in mathematical physics and 
engineering sciences. Some examples include 
degradation/healing of materials under extreme 
environmental conditions \cite{Chatterjee_etal}, 
coupled chemo-thermo-mechano-diffusion problems 
in composites \cite{Sih_etal}, contaminant transport 
\cite{Bear_PorousMedia}, turbulent mixing in atmospheric 
sciences \cite{Cant_Mastorakos}, diffusion-controlled 
biochemical reactions \cite{Saltzman}, tracer modeling 
in hydrogeology \cite{Leibundgut_etal}, and ionic mobility 
in biological systems \cite{Keener_Sneyd_Vol1}. Additionally, 
ADR equation serves as a good mathematical model in numerical 
analysis, as it offers various unique challenges in obtaining 
stable and accurate numerical solutions \cite{Morton}.

The primary variables in these mathematical models are 
typically the concentration and/or the (absolute) temperature. 
These quantities naturally attain non-negative values. Under 
some popular constitutive models (such as Fourier model and 
Fickian model, and their generalizations), these physical 
quantities satisfy diffusion-type equations, which are 
elliptic/parabolic partial differential equations (PDEs) 
and can be non-self-adjoint. These PDEs are known to satisfy 
important mathematical properties like maximum principles and 
the non-negative constraint (e.g., see \cite{Gilbarg_Trudinger}). 
A predictive numerical formulation needs to satisfy these 
mathematical properties and physical laws like the (local 
and global) species balance. It is well-documented in the 
literature that traditional numerical methods perform poorly 
for advection-dominated ADR equations (e.g., see \cite{Donea_Huerta,
Morton}). In the past few decades, considerable progress 
has been made to obtain sufficiently accurate numerical 
solutions for ADR equations on coarse computational grids 
\cite{Codina_CMAME_2000_v188_p61}. It is then natural to 
ask: ``\emph{why there is a need for yet another numerical 
formulation for ADR equation}?''. We now outline several 
reasons for such a need.

\begin{enumerate}[(a)]
\item \emph{Localized phenomena and node-to-node spurious oscillations:}
  For advection-dominated problems, it is well-known 
  that the standard single-field Galerkin finite element 
  formulation gives node-to-node spurious oscillations 
  on coarse computational grids \cite{Morton}. Moreover, 
  it cannot capture steep gradients such as interior and 
  boundary layers. Various alternate numerical techniques 
  have been proposed to avoid these spurious oscillations 
  \cite{Gresho_Sani_v1}. Some methods seem to capture steep 
  gradients in interior layers while others capture boundary 
  layers. However, most of them do not seem to capture both 
  interior and boundary layers, and at the same time avoid 
  node-to-node spurious oscillations \cite{2011_Augustin_etal_CMAME_v200_p3395_p3409}. 
  A notable work towards this direction is by Hsieh and Yang 
  \cite{2009_Hsieh_Yang_CMAME_v199_p183_p196}, which can 
  capture both interior and boundary layers under adequate 
  mesh refinement. However, this formulation has several 
  other deficiencies some of which are discussed below and 
  illustrated using numerical examples in subsequent sections. 
\item \emph{Violation of the non-negative constraint and 
  maximum principles:}~As mentioned earlier, physical 
  quantities such as concentration and temperature 
  naturally attain non-negative values. It is highly 
  desirable for a numerical formulation to respect 
  these physical constraints. This is particularly 
  important in a numerical simulation of reactive 
  transport, as a negative value for concentration 
  will result in an algorithmic failure. However, 
  it is clearly documented in the literature that 
  many existing formulations based on finite element 
  \cite{Liska_Shashkov_CiCP_2008_v3_p852,
  Nakshatrala_Valocchi_JCP_2009_v228_p6726,
  Nagarajan_Nakshatrala_IJNMF_2011_v67_p820}, 
  finite volume \cite{LePotier_CRM_2005_v341_p787}, 
  and finite difference \cite{Brezzi_Lipnikov_Shashkov_SIAMJNA_2005_v43_p1872} 
  do not satisfy the non-negative constraint and 
  maximum principles in the discrete setting. They 
  also discuss various methodologies to satisfy 
  such properties. But most of these methodologies 
  are for pure diffusion equations, which are 
  self-adjoint.
  For example, in Reference \cite{Nakshatrala_Valocchi_JCP_2009_v228_p6726},
  two mixed formulations based on RT0 spaces and variational 
  multiscale formalism have been modified to meet the non-negative 
  constraint for diffusion equations. This approach is not directly 
  applicable to ADR equations for two reasons. First, these formulations 
  do not cure the node-to-node spurious oscillations. Second, they do 
  not possess a variational structure for ADR equations.
  Some numerical formulations are constructed to satisfy 
  the non-negative constraint and maximum principles by taking 
  advection into account (e.g., \cite{Burman_Hansbo_CMAME_2004_v193_p1437,
  Burman_Ern_MathComput_2005_v74_p1637}). However, they do not 
  satisfy local and global species balance, and are restricted 
  to \emph{isotropic} diffusion. Conservative post-processing 
  methods exist in the literature to recover certain desired 
  properties such as species balance. But such formulations 
  are not variationally consistent 
  \cite{2012_Burdakov_etal_JCP_v231_p3126_p3142}.
\item \emph{Satisfying local and global species balance:} In 
  transport problems, the balance of species is an important 
  physical law that needs to be met. It is therefore desirable 
  for a numerical formulation to satisfy local and global species 
  balance, say, up to machine precision (which is approximately 
  $10^{-16}$ on a 64-bit machine). 
  However, many finite element formulations do not 
  satisfy local and global species balance (see 
  \cite{Codina_CMAME_1998_v156_p185,Codina_CMAME_2000_v188_p61,
    2009_Hsieh_Yang_CMAME_v199_p183_p196}). The main focus of the 
  methods outlined in these papers is to capture the localized 
  phenomena such as boundary and interior layers. 
  Moreover, these works did not quantify the errors incurred 
  in satisfying local species balance and global species balance. 
  It needs to be emphasized that many finite element methods do 
  exist that inherently satisfy local and global species balance, 
  for example, Raviart-Thomas spaces \cite{Raviart_Thomas_MAFEM_1977_p292} 
  and BDM spaces \cite{Brezzi_Douglas_Durran_Marini_NumerMath_1987_v51_p237}. 
  But these approaches do not fix other issues discussed herein 
  such as the node-to-node spurious oscillations or meeting 
  maximum principles for ADR equations.
\item \emph{Other influential factors:} Some other important 
  factors that can affect the performance of a numerical formulation 
  are the advection velocity field and its divergence, anisotropy 
  of the medium, reaction coefficients, topology of the medium, 
  computational mesh, multiple spatial-scales arising due to the 
  heterogeneity of the medium, and multiple time-scales involved 
  in various physical processes. Another aspect that brings 
  tremendous numerical challenges is chemical reactions 
  involving multiple species. 
\end{enumerate}

It is a herculean task to address all the 
aforementioned deficiencies, and we strongly 
believe that it may take a series of papers to 
overcome all the deficiencies. A similar sentiment 
is shared in the review article by Stynes entitled 
``\emph{Numerical methods for convection-diffusion 
problems or the 30 years war}'' \cite{2013_Stynes_arXiv}.
We therefore 
take motivation from George P\'olya's quote 
\cite{Polya}: ``\emph{If you can't solve a problem, 
then there is an easier problem you can solve: find it.}'' 
In this paper, we shall pose a problem that is 
simpler than the grand challenge of overcoming 
all the aforementioned numerical deficiencies 
but still make a significant advancement with 
respect to the current state-of-the-art. We then 
provide a solution to this simpler problem. 
To state it more precisely, the main contribution of this 
paper is developing a least-squares-based finite element 
framework for ADR equations that possesses the following 
properties on general computational grids:
\begin{enumerate}[(P1)]
\item No spurious node-to-node oscillations in the entire domain.
\item Captures interior and boundary layers for advection-dominated 
  problems.
\item Satisfies discrete maximum principles and the non-negative 
  constraint.
\item Satisfies local and global species balance.
\item Gives sufficiently accurate solutions even on coarse computational 
  grids\footnote{One may expect some subjectivity in calling a mesh to be 
    coarse. But in this paper, we will define precisely what is meant by a 
    ``coarse mesh'' for advection-diffusion-reaction equations in terms of 
    $M$-matrices.}.
\end{enumerate}
To the best of authors' knowledge, there 
exists no finite element methodology for 
advective-diffusive-reactive systems that 
possesses the desirable properties (P1)--(P5).

The rest of the paper is organized as follows. Section 
\ref{Sec:S2_NN_AD_GE} presents the governing equations 
for an advective-diffusive-reactive system, and discusses 
the associated mathematical properties. 
Section \ref{Sec:S3_NN_AD_PLCP} outlines several 
plausible approaches, and discusses their drawbacks 
in meeting the mentioned mathematical properties. 
In Section \ref{Sec:NN_AD_Proposed}, we propose a 
constrained optimization-based low-order mixed finite 
element method to satisfy maximum principles, the 
non-negative constraint, local species balance, and 
global species balance. 
In Section \ref{Sec:NN_AD_NumericalConvergence}, we perform 
a numerical $h$-convergence study using a benchmark problem. 
In Section \ref{Sec:NN_AD_Fast_Reactions}, we specialize to 
transport-limited bimolecular reactions to solve problems 
related to plume development and mixing in isotropic/anisotropic 
heterogeneous media. Finally, conclusions are drawn in Section 
\ref{Sec:NN_AD_Conclusions}. 
If one is interested only in the implementation 
of the proposed method, the reader can directly 
go to Section \ref{Sec:NN_AD_Proposed}, and 
appendices A and C. 

We shall denote scalars by lowercase English alphabet or 
lowercase Greek alphabet (e.g., concentration $c$ and 
stabilization parameter $\tau$). The continuum vectors 
are denoted by lowercase boldface normal letters, and 
the second-order tensors will be denoted using uppercase 
boldface normal letters (e.g., vector $\mathbf{x}$ and 
second-order tensor $\mathbf{D}$). In the finite element 
context, the vectors are denoted using lowercase boldface 
italic letters, and the matrices are denoted using uppercase 
boldface italic letters (e.g., vector $\boldsymbol{v}$ and 
matrix $\boldsymbol{K}$). 
We shall use NN to denote non-negative, DMP denotes discrete 
maximum principle, LSB to denote local species balance, and 
GSB to denote global species balance. We shall use XSeed to 
denote the number of (finite element) nodes in a mesh along 
x-direction. Likewise for YSeed. 
Other notational conventions adopted in 
this paper are introduced as needed.

\section{GOVERNING EQUATIONS:~ADVECTIVE-DIFFUSIVE-REACTIVE SYSTEMS}
\label{Sec:S2_NN_AD_GE}
Let $\Omega \subset \mathbb{R}^{nd}$ be a bounded open domain, 
where ``$nd$'' denotes the number of spatial dimensions. The 
boundary of the domain is denoted by $\partial \Omega$, which 
is assumed to be piecewise smooth. Mathematically, $\partial 
\Omega := \overline{\Omega} - \Omega$, where a superposed bar 
denotes the set closure. A spatial point is denoted by $\mathbf{x} 
\in \overline{\Omega}$. The gradient and divergence operators 
with respect to $\mathbf{x}$ are, respectively, denoted 
by $\mathrm{grad}[\bullet]$ and $\mathrm{div}[\bullet]$. 
The unit outward normal to the boundary is denoted 
by $\widehat{\mathbf{n}}(\mathbf{x})$.
Let $c(\mathbf{x})$ denote the concentration field. 
The boundary is divided into two parts: $\Gamma^{c}$ 
and $\Gamma^{q}$ such that $\mathrm{meas} (\Gamma^{c}) 
> 0$, $\overline{\Gamma}^{c} \cup \overline{\Gamma}^{q} 
= \partial \Omega$, and $\Gamma^{c} \cap \Gamma^{q} = 
\emptyset$. $\Gamma^{c}$ is the part of the boundary 
on which the concentration is prescribed and $\Gamma^{q}$ 
is the other part of the boundary on which the diffusive/total 
flux is prescribed.

The governing equations for steady-state response 
of an ADR system take the following form:
\begin{subequations}
  \begin{alignat}{2}
    \label{Eqn:NN_AD_GE_BE}
    &\alpha(\mathbf{x}) c(\mathbf{x}) + \mathrm{div} 
    \left[c(\mathbf{x}) \mathbf{v} (\mathbf{x}) - 
    \mathbf{D}(\mathbf{x}) \mathrm{grad}[c(\mathbf{x})] 
    \right] = f(\mathbf{x}) \quad &&\mathrm{in} \; 
    \Omega \\
    \label{Eqn:NN_AD_GE_DBC}
    &c(\mathbf{x}) = c^{\mathrm{p}}(\mathbf{x}) \quad 
    &&\mathrm{on} \; 
    \Gamma^{c} \\
    \label{Eqn:NN_AD_GE_NBC}
    &\left(\left( \frac{1 - \mathrm{Sign}[\mathbf{v} 
    \bullet \widehat{\mathbf{n}}]}{2} \right) \mathbf{v}
    (\mathbf{x}) c(\mathbf{x}) - \mathbf{D}(\mathbf{x}) 
    \mathrm{grad}[c(\mathbf{x})] \right) \bullet 
    \widehat{\mathbf{n}}(\mathbf{x}) = q^{\mathrm{p}} 
    (\mathbf{x}) \quad &&\mathrm{on} \; \Gamma^{q} 
  \end{alignat}
\end{subequations}
$\mathbf{v}(\mathbf{x})$ is the known advection velocity 
field, $f(\mathbf{x})$ is the prescribed volumetric source, 
$\mathbf{D}(\mathbf{x})$ is the anisotropic diffusivity 
tensor, $\alpha(\mathbf{x})$ is the linear reaction 
coefficient, $c^{\mathrm{p}}(\mathbf{x})$ is the prescribed 
concentration, $q^{\mathrm{p}}(\mathbf{x})$ is the prescribed 
diffusive/total flux, and the sign function is defined as 
follows:
\begin{align} 
  \label{Eqn:Sign_Function}
  \mathrm{Sign}[\varphi] := 
  \begin{cases}
    -1 &\quad \mathrm{if} \; \varphi < 0 \\
    0 &\quad \mathrm{if} \; \varphi = 0 \\
    +1 &\quad \mathrm{if} \; \varphi > 0
  \end{cases}
\end{align}
The advection velocity need not be solenoidal in our 
treatment (i.e., $\mathrm{div}[\mathbf{v}(\mathbf{x})]$ 
need not be zero). The Neumann boundary condition given 
in equation \eqref{Eqn:NN_AD_GE_NBC} can be interpreted 
as follows:
\begin{subequations}
  \begin{alignat}{2}
    \label{Eqn:Total_Flux_NeuBCs}
    &\widehat{\mathbf{n}}(\mathbf{x}) \bullet 
    \left(\mathbf{v}(\mathbf{x}) c(\mathbf{x}) 
    - \mathbf{D}(\mathbf{x}) \mathrm{grad}
    [c(\mathbf{x})]\right) = q^{\mathrm{p}} 
    (\mathbf{x}) \; &&\mathrm{on} \; 
    \Gamma^{q}_{-} \; (\mbox{total flux on 
    inflow boundary}) \\
    \label{Eqn:Diffusive_Flux_NeuBCs}
    -&\widehat{\mathbf{n}}(\mathbf{x}) \bullet 
    \mathbf{D}(\mathbf{x}) \mathrm{grad}
    [c(\mathbf{x})] = q^{\mathrm{p}}(\mathbf{x}) 
    \; &&\mathrm{on} \; \Gamma^{q}_{+} \;
    (\mbox{diffusive flux on outflow boundary})
  \end{alignat}
\end{subequations}
where $\Gamma^{q}_{+}$ and $\Gamma^{q}_{-}$ are, 
respectively, defined as follows (see Figure 
\ref{Fig:NN_AD_inflow_outflow_BC}):
\begin{subequations}
  \begin{alignat}{2}
    \label{Eqn:Inflow_Boundary}
    \Gamma^{q}_{-} &:= \left \{ \mathbf{x} \in 
    \Gamma^{q} \; \big| \; \mathbf{v}(\mathbf{x}) 
    \bullet \widehat{\mathbf{n}}(\mathbf{x}) 
    < 0 \right \} \quad && (\mathrm{inflow \; 
    boundary}) \\
    \label{Eqn:Outflow_Boundary}
    \Gamma^{q}_{+} &:= \left \{ \mathbf{x} \in 
    \Gamma^{q} \; \big| \; \mathbf{v}(\mathbf{x}) 
    \bullet \widehat{\mathbf{n}}(\mathbf{x}) 
    \geq 0 \right \} \quad && (\mathrm{outflow \; 
    boundary})
  \end{alignat}
\end{subequations}
\begin{figure}[tbp]
  \centering
  \psfrag{A1}{$\Gamma^{q}_{-}$}
  \psfrag{A2}{$\Gamma^{c}$}
  \psfrag{A3}{$\Gamma^{q}_{+}$}
  \psfrag{v}{$\mathbf{v}(\mathbf{x})$}
  \psfrag{B1}{P}
  \psfrag{B2}{Q}
  \psfrag{B3}{R}
  \includegraphics[scale=0.75,clip]{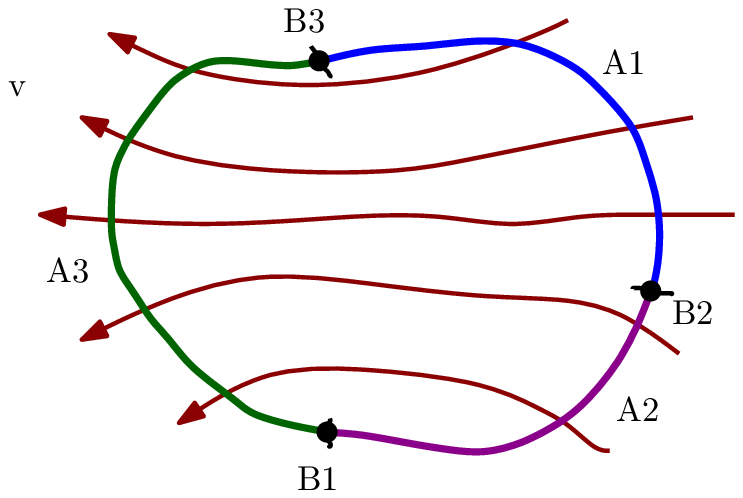}
  \caption{This figure illustrates concentration and flux 
    boundary conditions. $\Gamma^{q}_{-}$ corresponds to 
    the inflow boundary while $\Gamma^{q}_{+}$ corresponds 
    to the outflow boundary. Total flux is prescribed on 
    $\Gamma^{q}_{-}$, diffusive flux is prescribed on 
    $\Gamma^{q}_{+}$, and concentration is prescribed on 
    $\Gamma^{c}$. P = $\overline{\Gamma}^{c} \cap 
    \overline{\Gamma}^{q}_{+}$, Q = $\overline{\Gamma}
    ^{c} \cap \overline{\Gamma}^{q}_{-}$, and R = 
    $\overline{\Gamma}^{q}_{+} \cap \overline{\Gamma}^{q}_{-}$.
    For well-posedness, we have $\overline{\Gamma}^{c} \cup 
    \overline{\Gamma}^{q}_{+} \cup \overline{\Gamma}^{q}_{-} 
    = \partial \Omega$, $\Gamma^{c} \cap \Gamma^{q}_{+} = 
    \emptyset$, $\Gamma^{c} \cap \Gamma^{q}_{-} = \emptyset$, 
    and $\Gamma^{q}_{+} \cap \Gamma^{q}_{-} = \emptyset$.
    \label{Fig:NN_AD_inflow_outflow_BC}}
\end{figure}

\begin{remark}
  In the literature, more predominantly in the numerical 
  literature, the term \emph{advection} is often used 
  synonymously with \emph{convection}. It should, however, 
  be noted that these two terms describe different physical 
  phenomena, and there is a need to clarify the terminology 
  here. An ADR equation arises from the balance of mass of 
  a given species. In 1D, an ADR equation takes the following 
  form:
  \begin{align}
    \label{Eqn:Remark_ADR_1D}
    \alpha(x) c(x) + \frac{d (vc)}{dx} - 
    \frac{d}{dx} \left(D(x) \frac{d c}{dx}
    \right) = f(x)
  \end{align}
  which is mathematically equivalent to the following equation: 
  \begin{align}
    \label{Eqn:1D_ADR}
    \left(\alpha(x) + \frac{dv}{dx}\right)c(x) 
    + v(x)\frac{dc}{dx} - \frac{d}{dx} \left(D(x) 
    \frac{d c}{dx}\right) = f(x)
  \end{align}
  One can obtain a ``similar'' mathematical 
  equation by linearizing the incompressible 
  Navier-Stokes equation, and an appropriate 
  name for this linearized equation is the 
  convection-dissipation-reaction (CDR) 
  equation. The CDR equation in 1D has 
  the following mathematical form:
  \begin{align}
    \label{Eqn:1D_CDR}
    \frac{d v_0}{dx} \tilde{v}(x) 
    + v_0(x) \frac{d \tilde{v}}{dx} 
    - \frac{d}{dx} \left(\mu(x) \frac{d \tilde{v}}{dx}\right) 
    = b(x,p_0(x)) + 2 v_0(x) \frac{d v_0}{dx}
  \end{align}
  where $\tilde{v}(x)$ is the velocity of the fluid, 
  and $p_0(x)$ and $v_0(x)$ are known pressure and 
  velocity fields about which the Navier-Stokes 
  equation is linearized. From equations 
  \eqref{Eqn:1D_ADR} and \eqref{Eqn:1D_CDR}, it is 
  evident that 1D ADR equation and 1D CDR equation 
  have similar mathematical forms. However, their 
  physical underpinnings are completely different, 
  as the Navier-Stokes equation is obtained by 
  substituting a specific constitutive model 
  into the balance of linear momentum.
\end{remark}

\subsection{Weak formulations}
\label{Subsec:Weak_Formulations}
The following function spaces will be 
used in the rest of this paper: 
\begin{subequations}
  \begin{align}
    \label{Eqn:Function_Space_C}
    \mathcal{C} &:= \left\{c(\mathbf{x}) 
    \in H^{1}(\Omega) \; \big| \; c(\mathbf{x}) 
    = c^{\mathrm{p}}(\mathbf{x}) \; \mathrm{on} 
    \; \Gamma^{c} \right\} \\
    \label{Eqn:Function_Space_W}
    \mathcal{W} &:= \left\{w(\mathbf{x}) 
    \in H^{1}(\Omega) \; \big| \; w(\mathbf{x}) 
    = 0 \; \mathrm{on} \; \Gamma^{c} \right\} \\
    \label{Eqn:Function_Space_Q}
    \mathcal{Q} &:= \left\{\mathbf{q}(\mathbf{x}) 
    \in \left(L_{2}(\Omega)\right)^{nd} \; \big| 
    \; \mathrm{div}[\mathbf{q}(\mathbf{x})] \in 
    L_{2}(\Omega) \right\}
  \end{align}
\end{subequations}
where $\mathbf{q}(\mathbf{x}) = c(\mathbf{x}) \mathbf{v}
(\mathbf{x}) - \mathbf{D}(\mathbf{x}) \mathrm{grad}[c(\mathbf{x})]$ 
and $H^{1}(\Omega)$ is a standard Sobolev space \cite{Evans_PDE}. 
Given two vector fields $a(\mathbf{x})$ and $b(\mathbf{x})$ 
on a set $K$, the standard $L_2$ inner product over $K$ is 
denoted as follows:
\begin{align}
  \label{Eqn:L2_Inner_Product}
  \left(a;b\right)_{K} = \displaystyle \int 
  \limits_{K} a(\mathbf{x}) \bullet b(\mathbf{x}) 
  \; \mathrm{d} K
\end{align}
The subscript will be dropped if $K = \Omega$. The 
most popular way to construct a weak formulation 
is to employ the Galerkin formalism. Based on the 
manner in which one applies the divergence theorem, 
the single-field Galerkin formulation for equations 
\eqref{Eqn:NN_AD_GE_BE}--\eqref{Eqn:NN_AD_GE_NBC} 
can be posed in two different ways. 

\textbf{Single-field Galerkin formulation \#1 
  ($\mathrm{SG}_1$):} Find $c(\mathbf{x}) \in 
\mathcal{C}$ such that we have
\begin{align}
  \label{Eqn:SG1}
  (w;\alpha c) &- (\mathrm{grad}[w] \bullet \mathbf{v};c) 
  + (\mathrm{grad}[w];\mathbf{D}(\mathbf{x}) \mathrm{grad}[c]) 
  + \left(w; \left( \frac{1 + \mathrm{Sign}[\mathbf{v} \bullet 
  \widehat{\mathbf{n}}]}{2}\right) (\mathbf{v} \bullet \widehat{
  \mathbf{n}}) \; c \right)_{\Gamma^{q}} \nonumber \\
  &= (w;f) - \left(w;q^{\mathrm{p}}\right)_{\Gamma^{q}} 
  \quad \forall w(\mathbf{x}) \in \mathcal{W}
\end{align}

\textbf{Single-field Galerkin formulation \#2 
  ($\mathrm{SG}_2$):} Find $c(\mathbf{x}) \in 
\mathcal{C}$ such that we have
\begin{align}
  \label{Eqn:SG2}
  \left(w;\left(\alpha + \mathrm{div}[\mathbf{v}] 
  \right) c \right) &+ (w;\mathrm{grad}[c] \bullet 
  \mathbf{v}) + (\mathrm{grad}[w];\mathbf{D}(\mathbf{x}) 
  \mathrm{grad}[c]) - \left(w; \left( \frac{1 - 
  \mathrm{Sign}[\mathbf{v} \bullet \widehat{\mathbf{n}}]}{2} 
  \right) (\mathbf{v} \bullet \widehat{\mathbf{n}}) 
  \; c \right)_{\Gamma^{q}} \nonumber \\
  &= (w;f) - \left(w;q^{\mathrm{p}}\right)_{\Gamma^{q}} 
  \quad \forall w(\mathbf{x}) \in \mathcal{W}
\end{align}

Note that the solution obtained will be the same regardless 
whether we use either $\mathrm{SG}_1$ or $\mathrm{SG}_2$. 
However, this is not true if one uses total/diffusive flux 
on Neumann boundary without giving due consideration to 
inflow and/or outflow Neumann boundary conditions. For 
more details, see subsection \ref{SubSec:Neumann_BCs_Issues}. 

\subsection{Maximum principles and the non-negative constraint}
\label{SubSec:MaxPrinciples_NN_Constraints}
A basic qualitative property of elliptic boundary value problems 
is the maximum principle. This property gives a priori estimate 
for $c(\mathbf{x})$ in $\Omega$ through its values on $\Gamma^{c}$. 
The following assumptions will be made to present a continuous 
maximum principle for ADR equations with both Dirichlet and 
Neumann boundary conditions: 
\begin{enumerate}[({A}1)]
  \item $\Omega$ is piecewise smooth domain with Lipschitz 
    continuous boundary $\partial \Omega$.
  \item The scalar functions $\alpha: \overline{\Omega} 
    \rightarrow \mathbb{R}$, $(\mathbf{v})_{i}: \overline{\Omega} 
    \rightarrow \mathbb{R}$, and $(\mathbf{D})_{ij}: \overline{\Omega} 
    \rightarrow \mathbb{R}$ are continuously differentiable 
    in their respective domains for $i = 1, \cdots, nd$. 
    Furthermore, $f \in L_{2}(\Omega)$, $q^{\mathrm{p}} \in 
    L_{2}(\Gamma^{q})$, and $c^{\mathrm{p}} = g^{*} |_{\Gamma^{c}}$ 
    with $g^{*} \in H^{1}(\Omega)$.
  \item The diffusivity tensor is assumed to be symmetric, 
    uniformly elliptic, and bounded above. That is, there 
    exists two constants (i.e., independent of $\mathbf{x}$), 
    $0 < \gamma_{\mathrm{lb}} \leq \gamma_{\mathrm{ub}} < + 
    \infty$, such that we have
    \begin{align}
      \label{Eqn:Uniformly_Elliptic}
      0 < \gamma_{\mathrm{lb}} \mathbf{y} \bullet \mathbf{y} 
      \leq \mathbf{y} \bullet \mathbf{D}(\mathbf{x}) \mathbf{y} 
      \leq \gamma_{\mathrm{ub}} \mathbf{y} \bullet \mathbf{y} 
      \quad \forall \mathbf{y} \in \mathbb{R}^{nd} \backslash 
      \{\mathbf{0}\}
    \end{align}
  \item The advection velocity field $\mathbf{v}(\mathbf{x})$ 
    and the reaction coefficient $\alpha(\mathbf{x})$ are 
    restricted as follows:
    \begin{subequations}
      \begin{alignat}{2}
        \label{Eqn:CMP_Restrict_1}
        &0 \leq \alpha(\mathbf{x}) + \mathrm{div} 
        \left[\mathbf{v}(\mathbf{x})\right] \leq 
        \beta_{1}(\mathbf{x}) &&\quad \forall 
        \mathbf{x} \in \Omega \\
        \label{Eqn:CMP_Restrict_2}
        &0 \leq \alpha(\mathbf{x}) + \frac{1}{2} 
        \mathrm{div}\left[\mathbf{v}(\mathbf{x}) 
        \right] \leq \beta_{2}(\mathbf{x}) && 
        \quad \forall \; \mathbf{x} \in \Omega \\
        \label{Eqn:CMP_Restrict_3}
        &0 \leq |\mathbf{v}(\mathbf{x}) \bullet 
        \widehat{\mathbf{n}}(\mathbf{x})| \leq 
        \beta_3(\mathbf{x}) && \quad \forall 
        \mathbf{x} \in \Gamma^{q}
      \end{alignat}
    \end{subequations}
    where $\beta_1(\mathbf{x}) \in L_{nd/2}(\Omega)$, 
    $\beta_2(\mathbf{x}) \in L_{nd/2}(\Omega)$, and 
    $\beta_3(\mathbf{x}) \in L_{nd-1}(\Gamma^{q})$. 
    It is assumed that functions $\beta_1(\mathbf{x})$, 
    $\beta_2(\mathbf{x})$, and $\beta_3(\mathbf{x})$ 
    are bounded above for a unique weak solution to 
    exist based on the Lax-Milgram theorem.
  \item The part of the boundary on which Dirichlet 
    boundary condition is enforced is not empty (i.e., 
    $\Gamma^{c} \neq \emptyset$). 
\end{enumerate}
We shall use the standard abbreviation of a.e. 
for almost everywhere \cite{Evans_PDE}. 

\begin{theorem}[\texttt{A continuous maximum principle}]
  \label{Thm:CMP}
  Let assumptions (A1)--(A5) hold and let the unique 
  weak solution $c(\mathbf{x})$ of equations 
  \eqref{Eqn:NN_AD_GE_BE}--\eqref{Eqn:NN_AD_GE_NBC} 
  belong to $C^{1}(\Omega) \cap C^{0}(\overline{\Omega})$. 
  If $f(\mathbf{x}) \in L_2(\Omega)$ and $q^{\mathrm{p}}
  (\mathbf{x}) \in L_2(\Gamma^{q})$ satisfy:  
  \begin{subequations}
    \label{Eqn:Restriction_CMP_VolSource_NormalFluxBCs}
    \begin{alignat}{2}
      &f(\mathbf{x}) \leq 0 && \quad \mathrm{a.e.} \; 
      \mathrm{in} \; \Omega \\
      &q^{\mathrm{p}}(\mathbf{x}) \geq 0 && \quad 
      \mathrm{a.e.} \; \mathrm{on} \; \Gamma^{q}_{+} 
      \cup \Gamma^{q}_{-}
    \end{alignat}
  \end{subequations}
  then $c(\mathbf{x})$ satisfies a continuous 
  maximum principle of the following form: 
  \begin{align}
    \label{Eqn:General_CMP}
    \mathop{\mathrm{max}}_{\mathbf{x} \in \overline{\Omega}} 
    \left [ c(\mathbf{x}) \right ] \leq \mathop{\mathrm{max}} 
    \left [ 0, \mathop{\mathrm{max}}_{\mathbf{x} \in \Gamma^{c}} 
    \left [ c^{\mathrm{p}}(\mathbf{x}) \right ] \right ]
  \end{align}
  In particular, if $c^{\mathrm{p}}(\mathbf{x}) \geq 0$ then
  \begin{align}
    \label{Eqn:Particular_CMP}
    \mathop{\mathrm{max}}_{\mathbf{x} \in \overline{\Omega}} 
    \left [ c(\mathbf{x}) \right ] = \mathop{\mathrm{max}}_
    {\mathbf{x} \in \Gamma^{c}} \left [ c^{\mathrm{p}}(\mathbf{x}) 
    \right ]
  \end{align}
  If $c^{\mathrm{p}}(\mathbf{x}) \leq 0$ then we 
  have the following non-positive property: 
  \begin{align}
    \label{Eqn:Non_Positivity_CMP}
    \mathop{\mathrm{max}}_{\mathbf{x} \in \overline{\Omega}} 
    \left [ c(\mathbf{x}) \right ] \leq 0
  \end{align} 
\end{theorem}
\begin{proof}
  Let $\Phi_{\mathrm{max}}$ and $m(\mathbf{x})$ are, 
  respectively, defined as follows:
  \begin{align}
    \label{Eqn:Phi_Function}
    &\Phi_{\mathrm{max}} := \displaystyle 
    \mathop{\mathrm{max}} \left [ 0, 
    \mathop{\mathrm{max}}_{\mathbf{x} \in 
    \Gamma^{c}} \left [ c^{\mathrm{p}}(
    \mathbf{x}) \right ] \right] \\
    \label{Eqn:m_C1_Function}
    &m(\mathbf{x}) := \mathop{\mathrm{max}} 
    \left [ 0, c(\mathbf{x}) - \Phi_{\mathrm{max}} 
    \right ]
  \end{align}
  It is easy to check that $m(\mathbf{x})$ is 
  a non-negative, continuous, and piecewise 
  $C^{1}(\Omega)$ function. From equation 
  \eqref{Eqn:m_C1_Function}, it is evident 
  that $m(\mathbf{x})|_{\Gamma^{c}} = 0$ 
  and $c(\mathbf{x}) = m(\mathbf{x}) + 
  \Phi_{\mathrm{max}}$ for any $\mathbf{x} \in 
  \overline{\Omega}$ unless $m(\mathbf{x}) = 0$. 
  Moreover, $m(\mathbf{x}) \in \mathcal{W}$. 
  By choosing $w(\mathbf{x}) = m(\mathbf{x})$, 
  the weak formulation given by equation 
  \eqref{Eqn:SG1} becomes:
  \begin{align}
    \label{Eqn:SG1_CMP_1}
    \left(m;\alpha (m + \Phi_{\mathrm{max}}) 
    \right) &- (\mathrm{grad}[m] \bullet 
    \mathbf{v};(m + \Phi_{\mathrm{max}})) + 
    (\mathrm{grad}[m] ;\mathbf{D} \, \mathrm{grad}
    [m]) \nonumber \\
    &+ \left(m; \left( \frac{1 + \mathrm{Sign}[\mathbf{v} 
    \bullet \widehat{\mathbf{n}}]}{2}\right) \mathbf{v}
    \bullet \widehat{\mathbf{n}} \; (m + \Phi_{\mathrm{max}}) 
    \right)_{\Gamma^{q}} \nonumber \\
    &= (m;f) - \left(m;q^{\mathrm{p}}\right)_{\Gamma^{q}}
  \end{align}
  It is easy to establish the following identities: 
  \begin{subequations}
    \begin{align}
      \label{Eqn:Identity_1}
      (m; \mathbf{v} \bullet \widehat{\mathbf{n}} \, 
      (m + \Phi_{\mathrm{max}}))_{\Gamma^{q}} 
      &= (\mathrm{grad}[m] \bullet \mathbf{v};
      (m + \Phi_{\mathrm{max}})) + (m;\mathrm{div}
      [\mathbf{v}] \, (m + \Phi_{\mathrm{max}})) 
      \nonumber \\ &+ (m; \mathrm{grad}[m] \bullet 
      \mathbf{v}) \\
      \label{Eqn:Identity_2}
      2(\mathrm{grad}[m] \bullet \mathbf{v};(m + \Phi_
      {\mathrm{max}})) &= (m; \mathbf{v} \bullet 
      \widehat{\mathbf{n}} \, (m + \Phi_{\mathrm{max}}))_
      {\Gamma^{q}} - (m; \mathrm{div}[\mathbf{v}] 
      \, (m + \Phi_{\mathrm{max}})) \nonumber \\
      &- (\Phi_{\mathrm{max}}; \mathrm{grad}[m] \bullet 
      \mathbf{v}) \\
      \label{Eqn:Identity_3}
      (\Phi_{\mathrm{max}}; \mathrm{grad}[m] \bullet 
      \mathbf{v}) &= (\Phi_{\mathrm{max}}; \mathbf{v} 
      \bullet \widehat{\mathbf{n}} \, m)_{\Gamma^{q}}
      - (\Phi_{\mathrm{max}}; \mathrm{div}[\mathbf{v}] 
      \, m) \\
      \label{Eqn:Identity_4}
      (\mathrm{grad}[m] \bullet \mathbf{v}; (m + \Phi_
      {\mathrm{max}})) &= \left(m; \mathbf{v} \bullet 
      \widehat{\mathbf{n}} \, \left(\Phi_{\mathrm{max}} 
      + \frac{1}{2} m \right) \right)_{\Gamma^{q}} \nonumber \\
      &- \left(m; \mathrm{div}[\mathbf{v}] \, \left(\Phi_
      {\mathrm{max}} + \frac{1}{2} m \right) \right)
    \end{align}
  \end{subequations}
  Using the above identities, equation \eqref{Eqn:SG1_CMP_1} 
  can be written as follows:
  \begin{align}
    \label{Eqn:SG1_CMP_2}
    &\left(m; \left(\alpha + \frac{1}{2} \mathrm{div}
    [\mathbf{v}] \right) m \right) + \left(m; 
    \left(\alpha + \mathrm{div}[\mathbf{v}] \right) 
    \Phi_{\mathrm{max}} \right) + (\mathrm{grad}[m]; 
    \mathbf{D} \, \mathrm{grad}[m]) \nonumber \\
    &+ \left(m; \frac{|\mathbf{v} \bullet \widehat{\mathbf{n}}|}{2} \; 
    m \right)_{\Gamma^{q}} - \left(m; \left( \frac{1 
    - \mathrm{Sign}[\mathbf{v} \bullet \widehat{\mathbf{n}}]}{2} 
    \right) (\mathbf{v} \bullet \widehat{\mathbf{n}}) \; \Phi_
    {\mathrm{max}} \right)_{\Gamma^{q}} = (m;f) 
    - \left(m;q^{\mathrm{p}} \right)_{\Gamma^{q}}
  \end{align}
  From equations \eqref{Eqn:Uniformly_Elliptic} and 
  \eqref{Eqn:CMP_Restrict_1}--\eqref{Eqn:CMP_Restrict_3},
  it is evident that
  \begin{align}
    \label{Eqn:SG1_CMP_3}
    &\left(m; \left(\alpha + \frac{1}{2} \mathrm{div}
    [\mathbf{v}] \right) m \right) + \left(m; \left(
    \alpha + \mathrm{div}[\mathbf{v}] \right) \Phi_
    {\mathrm{max}} \right) + (\mathrm{grad}[m]; \mathbf{D}
    \, \mathrm{grad}[m]) \nonumber \\
    &+ \left(m; \frac{|\mathbf{v} \bullet \widehat{\mathbf{n}}|}{2} 
    \; m \right)_{\Gamma^{q}} - \left(m; \left( \frac{1 - \mathrm{Sign}
    [\mathbf{v} \bullet \widehat{\mathbf{n}}]}{2} \right) \mathbf{v}
    \bullet \widehat{\mathbf{n}} \; \Phi_{\mathrm{max}} \right)_{\Gamma^{q}} 
    \geq 0
  \end{align} 
  From equation \eqref{Eqn:Restriction_CMP_VolSource_NormalFluxBCs} 
  we have: 
  \begin{align}
    \label{Eqn:SG1_CMP_4}
    (m;f) - \left(m;q^{\mathrm{p}}\right)_{\Gamma^{q}} \leq 0
  \end{align}
  Therefore, one can conclude that 
  \begin{align}
    \label{Eqn:SG1_CMP_5}
    &\left(m; \left(\alpha + \frac{1}{2} \mathrm{div}[\mathbf{v}] 
    \right) m \right) + \left(m; \left(\alpha + \mathrm{div}[\mathbf{v}] 
    \right) \Phi_{\mathrm{max}} \right) + (\mathrm{grad}[m];\mathbf{D}
    \, \mathrm{grad}[m]) \nonumber \\
    &+ \left(m; \frac{|\mathbf{v} \bullet \widehat{\mathbf{n}}|}{2} \; 
    m \right)_{\Gamma^{q}} - \left(m; \left( \frac{1 
    - \mathrm{Sign}[\mathbf{v} \bullet \widehat{\mathbf{n}}]}{2}\right) 
    (\mathbf{v} \bullet \widehat{\mathbf{n}}) \; \Phi_{\mathrm{max}} 
    \right)_{\Gamma^{q}} 
    = 0
  \end{align}
  In the light of assumption (A3) and equation \eqref{Eqn:SG1_CMP_5}, 
  we need to have $\mathrm{grad}[m] = 0$ (as $\mathbf{D}(\mathbf{x})$ 
  is bounded below by a constant $\gamma_{\mathrm{lb}}$). This further 
  implies the following:
  \begin{align}
    \label{Eqn:SG1_CMP_6}
    m(\mathbf{x}) \equiv \phi_0 \geq 0 
    \quad \forall \mathbf{x} \in \overline{\Omega} 
  \end{align}
  where $\phi_0$ is a non-negative constant. Since 
  $m(\mathbf{x})|_{\Gamma^{c}} = 0$ and $\mathrm{meas}
  (\Gamma^{c}) > 0$, we have $\phi_0 = 0$. This implies 
  that $c(\mathbf{x}) \leq \Phi_{\mathrm{max}}$, which 
  further implies the validity of the inequality 
  given by equation \eqref{Eqn:General_CMP}. Finally, 
  equations \eqref{Eqn:Particular_CMP} and 
  \eqref{Eqn:Non_Positivity_CMP} are trivial 
  consequences of equation \eqref{Eqn:General_CMP}.
\end{proof}

We have employed the $\mathrm{SG}_1$ formulation in 
the proof of Theorem \ref{Thm:CMP}. One will come to 
the same conclusion even under the $\mathrm{SG}_2$ 
formulation. By reversing the signs in equation 
\eqref{Eqn:Restriction_CMP_VolSource_NormalFluxBCs}, 
one can easily obtain the following continuous 
minimum principle.  

\begin{corollary}[\texttt{A continuous minimum principle}]
  \label{Thm:CMinP}
  Let assumptions (A1)--(A5) hold and let the unique 
  weak solution $c(\mathbf{x})$ of equations 
  \eqref{Eqn:NN_AD_GE_BE}--\eqref{Eqn:NN_AD_GE_NBC} 
  belong to $C^{1}(\Omega) \cap C^{0}(\overline{\Omega})$. 
  If $f(\mathbf{x}) \in L_2(\Omega)$ and $q^{\mathrm{p}}
  (\mathbf{x}) \in L_2(\Gamma^{q})$ satisfy  
  \begin{subequations} 
    \label{Eqn:Restriction_CMinP_VolSource_NormalFluxBCs}
    \begin{alignat}{2}
      &f(\mathbf{x}) \geq 0 && \quad \mathrm{a.e.} \; 
      \mathrm{in} \; \Omega \\
      &q^{\mathrm{p}}(\mathbf{x}) \leq 0 && \quad 
      \mathrm{a.e.} \; \mathrm{on} \; \Gamma^{q}_{+} 
      \cup \Gamma^{q}_{-}
    \end{alignat}
  \end{subequations}
  then $c(\mathbf{x})$ satisfies a continuous 
  minimum principle of the following form: 
  \begin{align}
    \label{Eqn:General_CMinP}
    \mathop{\mathrm{min}}_{\mathbf{x} \in \overline{\Omega}} 
    \left [ c(\mathbf{x}) \right ] \geq \mathop{\mathrm{min}} 
    \left [ 0, \mathop{\mathrm{min}}_{\mathbf{x} \in \Gamma^{c}} 
    \left [ c^{\mathrm{p}}(\mathbf{x}) \right ] \right ]
  \end{align}
  In particular, if $c^{\mathrm{p}}(\mathbf{x}) \leq 0$ then
  \begin{align}
    \label{Eqn:Particular_CMinP}
    \mathop{\mathrm{min}}_{\mathbf{x} \in \overline{\Omega}} 
    \left [ c(\mathbf{x}) \right ] = \mathop{\mathrm{min}}_
    {\mathbf{x} \in \Gamma^{c}} \left [ c^{\mathrm{p}}(\mathbf{x}) 
    \right ]
  \end{align}
  If $c^{\mathrm{p}}(\mathbf{x}) \geq 0$ then we 
  have the following non-negative property: 
  \begin{align}
    \label{Eqn:Non_Negativity_CMinP}
    \mathop{\mathrm{min}}_{\mathbf{x} \in \overline{\Omega}} 
    \left [c(\mathbf{x}) \right ] \geq 0
  \end{align} 
\end{corollary}

This paper also deals with transient analysis, and the 
details are provided in Sections \ref{Sec:NN_AD_Proposed} 
and \ref{Sec:NN_AD_Fast_Reactions}. 

\subsection{On appropriate Neumann BCs}
\label{SubSec:Neumann_BCs_Issues}
Many existing numerical formulations 
\cite{Ayub_Masud_NumerHeatTransB_2003_v44_p1} and packages 
such as \textsf{ABAQUS} \cite{Abaqus_6.14-1}, \textsf{ANSYS} 
\cite{Ansys_16.0}, \textsf{COMSOL} \cite{Comsol_5.0-1}, and 
\textsf{MATLAB's PDE Toolbox} \cite{MATLAB_2015a} do not 
pose the Neumann BCs in correct form for advection-diffusion 
equations. These formulations and packages either use the 
diffusive flux or the total flux on the entire Neumann 
boundary without discerning the following situations:
\begin{itemize}
\item Do we have inflow (i.e., $\mathbf{v} \bullet 
  \widehat{\mathbf{n}} \leq 0$) on the entire 
  Neumann boundary? 
  \item Do we have outflow (i.e., $\mathbf{v} \bullet 
    \widehat{\mathbf{n}} \geq 0$) on the entire 
    Neumann boundary? 
  \item Or do we have both inflow and outflow 
    on the Neumann boundary?
\end{itemize}
These conditions will dictate whether the resulting 
boundary value problem is well-posed or not. If a 
numerical formulation does not take into account 
these conditions, the numerical solution can 
exhibit instabilities, which will have dire 
consequences in mixing problems. To illustrate, 
consider the following 1D boundary value problem: 
\begin{subequations} 
  \begin{align}
    \label{Eqn:1D_ADR_NeuBC_Check1}
    &\frac{d}{dx} \left( vc -  D 
    \frac{dc}{dx} \right) = 0 \quad 
    \forall x \in (0, L) \\
    \label{Eqn:1D_ADR_NeuBC_Check2}
    &c(x = 0) = c_0
  \end{align}
\end{subequations}
where $v$, $D$ and $c_0$ are constants, and $L$ is the 
length of the domain. We now consider two different cases 
for the Neumann BC:
\begin{subequations} 
  \begin{align}
    \label{Eqn:1D_ADR_NeuBC_TotalFlux}
    \left( vc -  D \frac{dc}{dx} \right) 
    &= q_0 \quad \mathrm{at} \; x = L \\
    \label{Eqn:1D_ADR_NeuBC_DiffusiveFlux}
    - D \frac{dc}{dx} &= q_0 \quad \mathrm{at} 
    \; x = L
\end{align}
\end{subequations}
where $q_0$ is a constant. Equation 
\eqref{Eqn:1D_ADR_NeuBC_TotalFlux} 
corresponds to the total flux BC while 
equation \eqref{Eqn:1D_ADR_NeuBC_DiffusiveFlux} 
is the diffusive flux BC. The analytical solutions 
for these two different Neumann BCs are, respectively, 
given as follows: 
\begin{subequations} 
  \begin{align}
    \label{Eqn:1D_ADR_Soln_TotalFlux}
    c_1(x) &= \frac{1}{v} \left(q_0 + \left( 
    v c_0 - q_0 \right) e^{\frac{vx}{D}} \right) \\
    \label{Eqn:1D_ADR_Soln_DiffusiveFlux}
    c_2(x) &= \frac{1}{v} \left(v c_0 + q_0 
    e^{\frac{-vL}{D}} - q_0 e^{\frac{v(x - L)}{D}} 
    \right)
\end{align}
\end{subequations}
The solution $c_1(x)$ blows if $v > 0$, and $c_2(x)$ blows 
if $v < 0$. On the other hand, the exact solution based on 
the Neumann BC given in equation \eqref{Eqn:NN_AD_GE_NBC} 
is well-posed for both inflow and outflow cases. 

To summarize, the boundary value problem is 
well-posed under the prescribed diffusive 
flux on the entire Neumann boundary if the 
flow is outflow on the entire $\Gamma^{q}$. 
The boundary value problem is well-posed under the 
prescribed total flux on the entire Neumann boundary 
if the flow is inflow on the entire $\Gamma^{q}$. The 
Neumann BC given by equation \eqref{Eqn:NN_AD_GE_NBC} 
is more general, and the boundary value problem under 
this BC is well-posed even if the Neumann boundary is 
composed of both inflow and outflow.

\section{PLAUSIBLE APPROACHES AND THEIR SHORTCOMINGS}
\label{Sec:S3_NN_AD_PLCP}
There are numerous numerical formulations available 
in the literature for advective-diffusive-reactive 
systems. A cavalier look at these formulations can 
be deceptive, as one may expect more than what these 
formulations can actually provide. We now discuss 
some approaches that seem plausible to satisfy the 
maximum principle and the non-negative constraint 
for an advective-diffusive-reactive system, and 
illustrate their shortcomings. This discussion is 
helpful in two ways.
\emph{First}, it sheds light on the complexity of 
the problem, and can bring out the main contributions
made in this paper.  
\emph{Second}, the discussion can provide a rationale 
behind the approach taken in this paper in order to 
develop the proposed computational framework.   
To start with, it is well-known that the single-field 
Galerkin formulation does not perform well, as it 
produces spurious node-to-node oscillations on coarse 
grids \cite{Donea_Huerta}. The formulation also violates 
the non-negative constraint and maximum principles for 
anisotropic medium, and does not possess element-wise 
species balance property 
\cite{Nakshatrala_Valocchi_JCP_2009_v228_p6726,
Nagarajan_Nakshatrala_IJNMF_2011_v67_p820}. 

\subsection{Approach \#1:~Clipping/cut-off methods}
\label{Subsec:Approach1_Clipping}
There are various post-processing procedures such as 
clipping/cut-off methods \cite{2012_Burdakov_etal_JCP_v231_p3126_p3142,
2014_Kreuzer_NMPDE_v30_p994_p1002} to ensure that a certain 
numerical formulation satisfies the non-negative constraint. 
  The key idea of these methods is to simply chop-off 
  the negative values in a numerical solution. 
  Although a clipping method is a variational crime, 
  this approach appeals the practitioners because of 
  its simplicity. 
  However, there are many reasons, which are often 
  overlooked by the practitioners, why a clipping 
  method is not appropriate for ADR equations with 
  anisotropic diffusivity. The reasons, which are 
  documented below, go well beyond the philosophical 
  issue of ``variational crime.'' The reasons should 
  sufficiently justify and motivate to employ a rather 
  sophisticated computational framework just like the 
  one proposed in this paper. 
  \begin{enumerate}[(i)]
  \item The violation of the non-negative constraint is small only for 
    pure diffusion equations with isotropic diffusivity. The violations 
    can be large in the case of anisotropic diffusion. 
    If the maximum eigenvalue is not much smaller than unity, 
    then naive $h/p$-refinement will not always reduce the 
    negative values and clipping procedure can give erroneous 
    results. Figure 19 and problem 6.2 in the paper illustrate this 
    point. This has been illustrated even for diffusion equations 
    in Reference \cite{2013_Nakshatrala_Mudunuru_Valocchi_JCP_v253_p278_p307}.
  \item Although tensorial dispersion frequently arises in the 
    modeling of subsurface systems, many practitioners employ 
    isotropic diffusion in their numerical simulations just to 
    avoid large non-negative violations in their reactive-transport 
    modeling. As mentioned earlier, in the case of isotropic diffusion, 
    one can go away with the clipping procedure. But there is a need for 
    predictive simulations for realistic scenarios (e.g., anisotropic 
    diffusivity), and one needs carefully designed computational frameworks. 
    Simple approaches like the clipping procedure will not suffice.
  \item A clipping procedure, by itself, does 
    not ensure local species balance. 
  \item The clipping procedure cannot eliminate 
    the spurious node-to-node oscillations. 
  \item The ramifications of clipping the negative values 
    on the species balance and on the overall accuracy of 
    solutions have not been carefully studies or documented. 
  \item Finally, both $h$- and $p$-refinements may decrease the 
    negative values and reduce spurious node-to-node oscillations 
    for advection-dominated and reaction-dominated ADR problems. 
    However, our objective is to satisfy maximum principles, non-negative 
    constraint, species balance, reduce spurious node-to-node oscillations, 
    and obtain sufficiently accurate numerical solutions on \emph{coarse 
      computational grids}. Extensive mesh and polynomial refinements 
    defeats the main purpose, as these approaches will incur excessive 
    computational cost.
  \end{enumerate}

\subsection{Approach \#2:~Mesh restrictions}
\label{Subsec:Approach2_MeshRestrictions}
Recently, there has been a surge on the study of constructing meshes 
to satisfy various discrete maximum principles both within the context 
of single-field and mixed finite element formulations \cite{2014_Huang_JCP_v274_p230_p244,
2014_Huang_Wang_arXiv_1401_6232,2015_Mudunuru_Nakshatrala_arXiv}. The 
primary objective of these methods is to develop restrictions on the 
computational meshes to meet the underlying principles. However, it 
should be noted that there are various drawbacks for these methods. 
The important ones are described as follows:
\begin{enumerate}[(i)]
  \item Most of these mesh restriction methods are for 
    simplicial meshes (such as three-node triangular 
    element and four-node tetrahedral element). Extending 
    these results to non-simplicial elements is not 
    trivial or may not be possible. 
  \item The boundary conditions are restricted to only Dirichlet 
    on the entire boundary of the domain. Incorporating mixed 
    boundary conditions or a general Neumann BC given by equation 
    \eqref{Eqn:NN_AD_GE_NBC} has not been addressed.
  \item Generating a DMP-based mesh for complex domains is 
    extremely difficult and sometimes impossible. 
  \item For highly advection-dominated and reaction-dominated problems, 
    we need a highly refined DMP-based meshes. Constructing such 
    meshes is computationally intensive.
  \item Even though the mesh restriction conditions put 
    forth for the weak Galerkin method by Huang and Wang 
    \cite{2014_Huang_Wang_arXiv_1401_6232} is locally 
    conservative, it is restricted to pure anisotropic 
    diffusion equations. Generalizing it to obtain locally 
    conservative DMP-based meshes for anisotropic ADR 
    equations is not apparent. Moreover, it still suffers 
    from the above set of drawbacks.
\end{enumerate}

\subsection{Approach \#3:~Using non-negative 
  methodologies for diffusion equations}
\label{Subsec:Approach3_NN_PureDiffusion}
Recently, optimization-based finite element methods 
\cite{Liska_Shashkov_CiCP_2008_v3_p852,
Nakshatrala_Valocchi_JCP_2009_v228_p6726,
Nagarajan_Nakshatrala_IJNMF_2011_v67_p820,
2013_Nakshatrala_Mudunuru_Valocchi_JCP_v253_p278_p307} 
are proposed to satisfy the non-negative constraint 
and maximum principles for diffusion-type equations.
These non-negative methodologies are for self-adjoint 
operators and are constructed by invoking Vainberg's 
theorem \cite{Vainberg}. That is, they utilize the 
fact that there exists a scalar functional such that 
the G\^{a}teaux variation of this functional provides 
the weak formulation and the Euler-Lagrange equations 
provide the corresponding governing equations for the 
diffusion problem. Corresponding to this continuous 
variational/minimization functional, a discrete 
non-negative constrained optimization-based finite 
element method is developed. Unfortunately, such a 
variational principle based on Vainberg's theorem 
does not exist for the Galerkin weak formulation 
for an ADR equation, as the spatial operator is 
non-self-adjoint 
\cite{2010_Nakshatrala_Valocchi_IJCM_v7_p559_p572}.

\subsection{Approach \#4:~Posing the discrete equations as a $P$-LCP}
\label{Subsec:Approach4_PLCP}
Let $h$ be the maximum element size, $\| \mathbf{v} \|_{\infty,
\Omega}$ be the maximum value for advection velocity field, 
$\alpha_{\infty,\Omega}$ be the maximum value for linear 
reaction coefficient, and $\lambda_{\mathrm{min}}$ be the 
minimum eigenvalue of $\mathbf{D}(\mathbf{x})$ in the 
entire domain. Mathematically, these quantities are 
defined as follows:
\begin{subequations}
  \begin{align}
    \label{Eqn:MaxEleSize}
    h &:= \displaystyle \mathop{\mbox{max}}
    _{\Omega_{e} \in \Omega_{h}} \left[h_{\Omega
    _{e}} \right] \\
    \label{Eqn:MaxNormVelOmega}
    \| \mathbf{v} \|_{\infty,\Omega} &:= 
    \displaystyle \mathop{\mbox{max}}_{1 \leq i 
    \leq nd} \left[|(\mathbf{v}(\mathbf{x}))_{i}| 
    \right] \quad \forall \mathbf{x} \in \Omega \\
    \label{Eqn:MaxNormAlphaOmega}
    \alpha_{\infty,\Omega} &:= \displaystyle 
    \mathop{\mbox{max}}_{\mathbf{x} \in \Omega} 
    \left[\alpha(\mathbf{x})\right] \\
    \label{Eqn:MinEigenvalueDiffusivity}
    \lambda_{\mathrm{min}} &:= \displaystyle \mathop{
    \mbox{min}}_{\mathbf{x} \in \Omega} \left[
    \lambda_{\mathrm{min},\mathbf{D}(\mathbf{x})} \right] \\
    \label{Eqn:MaxEigenvalueDiffusivity}
    \lambda_{\mathrm{max}} &:= \displaystyle \mathop{
    \mbox{max}}_{\mathbf{x} \in \Omega} \left[
    \lambda_{\mathrm{max},\mathbf{D}(\mathbf{x})} \right]
  \end{align}
\end{subequations}   
where $\Omega_{h}$ is a regular linear finite element partition 
of the domain $\Omega$ such that $\overline{\Omega}_{h} = \bigcup_
{e = 1}^{Nele} \overline{\Omega}_{\mathrm{e}}$. ``\emph{Nele}'' 
is the total number of discrete non-overlapping open sub-domains.
The boundary of $\Omega_{\mathrm{e}}$ is denoted as $\partial 
\Omega_{\mathrm{e}} := \overline{\Omega}_{\mathrm{e}} - \Omega_
{\mathrm{e}}$. $h_{\Omega_{e}}$ is the diameter of element
$\Omega_{e}$. $\lambda_{\mathrm{min},\mathbf{D}(\mathbf{x})}$ and 
$\lambda_{\mathrm{max},\mathbf{D}(\mathbf{x})}$ are, respectively, 
the minimum and maximum eigenvalue of $\mathbf{D}(\mathbf{x})$ 
at a given point $\mathbf{x} \in \Omega$. Correspondingly, the 
element P\'{e}clet number $\mathbb{P}\mathrm{e}_{h}$ and the 
element Damk\"{o}hler number $\mathbb{D}\mathrm{a}_{h}$ are 
defined as follows:
\begin{subequations}
  \begin{align}
    \label{Eqn:ElementPecletNumber}
    \mathbb{P}\mathrm{e}_{h} &:= \frac{\| \mathbf{v} \|_{\infty,
    \Omega} h}{2 \lambda_{\mathrm{min}}} \\
    \label{Eqn:ElementDamkohlerNumber} 
    \mathbb{D}\mathrm{a}_{h} &:= \frac{\alpha_{\infty,\Omega} 
    h^2}{\lambda_{\mathrm{min}}}
  \end{align}
\end{subequations}
Herein, $\mathbb{D}\mathrm{a}_{h}$ is defined based on linear 
reaction coefficient and diffusivity. However, it should be 
noted that there are various ways to construct different types
of element Damk\"{o}hler numbers (for instance, see Reference 
\cite{Chung} for isotropic diffusivity).

After low-order finite element discretization of either $\mathrm{SG}_1$
or $\mathrm{SG}_2$, the discrete equations for the ADR boundary value 
problem take the following form:
\begin{align}
  \label{Eqn:Discrete_SG1_SG2}
  \boldsymbol{K} \boldsymbol{c} = \boldsymbol{f}
\end{align}
where $\boldsymbol{K}$ is the stiffness matrix (which is neither 
symmetric nor positive definite), $\boldsymbol{c}$ is the vector 
containing nodal concentrations, and $\boldsymbol{f}$ is the 
volumetric source vector. The matrix $\boldsymbol{K}$ is of 
size $ncdofs \times ncdofs$, where ``$ncdofs$'' denotes the number 
of free degrees-of-freedom for the concentration. The vectors 
$\boldsymbol{c}$ and $\boldsymbol{f}$ are of size $ncdofs \times 1$.

In the rest of this paper, the symbols $\succeq$ and 
$\preceq$ will be used to denote the component-wise 
comparison of vectors and matrices. That is, given 
two vectors $\boldsymbol{a}$ and $\boldsymbol{b}$, 
$\boldsymbol{a} \preceq \boldsymbol{b}$ means that 
$(\boldsymbol{a})_i \leq (\boldsymbol{b})_i$ for all $i$. 
Likewise, given two matrices $\boldsymbol{A}$ and 
$\boldsymbol{B}$, $\boldsymbol{A} \preceq \boldsymbol{B}$ 
means that $(\boldsymbol{A})_{ij} \leq (\boldsymbol{B})_{ij}$ for all 
$i$ and $j$. The mathematical means of the symbols $\succeq$, 
$\prec$ and $\succ$ should now be obvious. We shall use 
$\boldsymbol{0}$ and $\boldsymbol{O}$ to denote zero 
vector and zero matrix, respectively.

\begin{definition}[\texttt{$P$-matrix, $Z$-matrix, and $M$-matrix}]
  A matrix $\boldsymbol{A} \in \mathbb{R}^{nd \times nd}$ is a 
  $P$-matrix if $\frac{1}{2}\left(\boldsymbol{A} + \boldsymbol{A}
  ^{\mathrm{T}} \right)$ is positive-definite. The matrix is a 
  $Z$-matrix if $(\boldsymbol{A})_{ij} \leq 0$, where $i \neq j$ 
  and $i, j = 1, \cdots, nd$. The matrix is an $M$-matrix if 
  $\boldsymbol{A}$ is a $P$-matrix and a $Z$-matrix. 
\end{definition}

\begin{definition}[\texttt{Coarse mesh demarcation for 
  anisotropic ADR equations}]
  \label{Definition:CoarseMesh_ADR}
   A regular low-order finite element computational mesh $\Omega_{h}$ 
   is said to be coarse with respect to
   \begin{enumerate}[(a)]
     \item spurious oscillations if $\mathbb{P}\mathrm{e}_{h}
      > 1$
     \item spurious oscillations and large linear reaction 
       coefficient if $\mathbb{P}\mathrm{e}_{h} > 1$ and
       $\mathbb{D}\mathrm{a}_{h} > 1$
     \item spurious oscillations, large linear reaction coefficient, 
       and a discrete maximum principle if the stiffness matrix 
       $\boldsymbol{K}$ associated with either $\mathrm{SG}_1$ or 
       $\mathrm{SG}_2$ is not an $M$-matrix
   \end{enumerate}
\end{definition}

It can be easily shown through counterexamples that the 
stiffness matrix $\boldsymbol{K}$ for ADR equation will 
not always be a $Z$-matrix. We shall now provide two such 
counterexamples. The first counterexample is the low-order 
finite element discretization based on two-node linear element
for the following 1D ADR equation (with constant velocity, 
diffusivity, and linear reaction coefficients): 
\begin{subequations}
  \begin{align}
    \label{Eqn:1D_ADR_FEM_GovEqn_L2_Element}
    & \alpha c +  v \frac{dc}{dx} - D \frac{d^2c}{dx^2} = 
    f(x) \quad \forall x \in \Omega:= (0,1) \\
    \label{Eqn:1D_ADR_FEM_BCs_L2_Element}
    & c(x) = c^{\mathrm{p}}(x) \quad \forall x \in \partial 
    \Omega := \{0, 1\}
  \end{align}
\end{subequations}
with $\alpha \geq 0$, $D > 0$, and $v \in \mathbb{R}$. 
The entries of stiffness matrix $\boldsymbol{K}$ for 
an $i^{\mathrm{th}}$ intermediate node (using equal-sized 
two-node linear finite element) is given as follows:
\begin{align}
  \label{Eqn:FEM_1D_ADR_L2_Element}
  \frac{\alpha h}{6} \left[ 
  \begin{array}{ccc} 
    1 & 4 & 1  
  \end{array} 
  \right] \left  \{ 
  \begin{array}{c} 
    c_{i-1} \\ c_{i} \\ c_{i+1}  
  \end{array} 
  \right \} + \frac{v}{2} \left[ 
  \begin{array}{ccc} 
    -1 & 0 & 1  
  \end{array} 
  \right] \left  \{ 
  \begin{array}{c} 
    c_{i-1} \\ c_{i} \\ c_{i+1}  
  \end{array} 
  \right \} + \frac{D}{h} \left[ 
  \begin{array}{ccc} 
    -1 & 2 & -1  
  \end{array} 
  \right] \left  \{ 
  \begin{array}{c} 
    c_{i-1} \\ c_{i} \\ c_{i+1}  
  \end{array} 
  \right \}
\end{align}
On trivial manipulations on equation \eqref{Eqn:FEM_1D_ADR_L2_Element}, 
it is evident that the stiffness matrix is a $Z$-matrix if and only 
if the following condition is satisfied:
\begin{align}
  \label{Eqn:1D_ADR_Zmatrix_Condition} 
  h \leq h_{\mathrm{max}} := \frac{12 D}{3|v| 
  + \sqrt{9v^2 + 24 \alpha D}}
\end{align}
which is not always the case. The second counterexample 
is based on a simplicial finite element discretization 
(e.g., three-node triangular/four-node tetrahedron 
element) of ADR equation with Dirichlet BCs on the 
entire boundary. If any $nd$-simplicial mesh does not 
satisfy the following condition then $\boldsymbol{K}$ 
is not a $Z$-matrix \cite[Theorem 4.3]{2015_Mudunuru_Nakshatrala_arXiv}:
\begin{align}
  \label{Eqn:2D_3D_ADR_Zmatrix_Condition}
  0 < \frac{h_{p} \, \|\mathbf{v}\|_{\infty,\overline{\Omega}_e}}{(nd+1) 
  \, \Lambda_{\mathrm{min},\widetilde{\boldsymbol{D}}_{\Omega_e}}} &+ \frac{h_{p} \, 
  h_{q} \, \alpha_{\infty,\overline{\Omega}_e}}{(nd+1) \, (nd+2) \, 
  \Lambda_{\mathrm{min},\widetilde{\boldsymbol{D}}_{\Omega_e}}} \leq \cos(\beta_{pq,
  \widetilde{\boldsymbol{D}}^{-1}_{\Omega_e}}) \nonumber \\
  &\forall p,q = 1,2,\cdots,nd+1, 
  \; p \neq q, \; \Omega_e \in \Omega_h
\end{align}
where $p$ and $q$ are the any two given arbitrary vertices of $\Omega_e$.
$\widetilde{\boldsymbol{D}}_{\Omega_e}$ is the integral element average 
anisotropic diffusivity. $\Lambda_{\mathrm{min},\widetilde{\boldsymbol{D}}_{\Omega_e}}$ 
denotes the minimum eigenvalue of $\widetilde{\boldsymbol{D}}_{\Omega_e}$. 
$h_p$ and $h_q$ are the perpendicular distance (or the height) from vertices 
$p$ and $q$ to their respective opposite faces. $\beta_{pq,\widetilde{\boldsymbol{D}}
^{-1}_{\Omega_e}}$ is the dihedral angle measured in $\widetilde{\boldsymbol{D}}^{-1}
_{\Omega_e}$ metric between two faces opposite to vertices $p$ and $q$ of 
an element $\Omega_e$.

\begin{proposition}[\texttt{$P$-matrix linear complementarity 
  problem for ADR equations}]
  \label{Theorem:PLCP_ADR_SG1_SG2}
  Given that assumptions (A1)--(A5) hold, then the stiffness 
  matrix $\boldsymbol{K}$ associated with low-order finite 
  element discretization of either $\mathrm{SG}_1$ or $\mathrm{SG}_2$ 
  is a $P$-matrix. Furthermore, if $\boldsymbol{c}$ has to be 
  DMP-preserving on any arbitrary coarse mesh, then the resulting 
  constrained discrete equations of single-field Galerkin formulation 
  can be posed as a $P$-$\mathrm{LCP}$.
\end{proposition}
\begin{proof}
  We prove only for $\mathrm{SG}_1$ formulation and extending 
  it to $\mathrm{SG}_2$ is a trivial manipulation. The symmetric 
  part of the element stiffness matrix $\boldsymbol{K}_e$ is 
  given as follows:
  \begin{align}
    \label{Eqn:LocalStiffnessMatrix_Ke}
    \frac{1}{2} \left(\boldsymbol{K}_e + \boldsymbol{K}
    _e^{\mathrm{T}} \right) = \int \limits_{\Omega_e} 
    \left(\alpha(\mathbf{x}) + \frac{1}{2} \mathrm{div}[
    \mathbf{v}(\mathbf{x})] \right) \boldsymbol{N}^{\mathrm{T}} 
    \boldsymbol{N} \; \mathrm{d} \Omega_e \, &+ \, \int 
    \limits_{\Omega_e} \boldsymbol{B} \mathbf{D}(\mathbf{x}) 
    \boldsymbol{B}^{\mathrm{T}} \; \mathrm{d} \Omega_e 
    \nonumber \\ 
    &+ \int \limits_{\Gamma^q_e} \frac{1}{2} |\mathbf{v} 
    \bullet \widehat{\mathbf{n}}| \boldsymbol{N}^{\mathrm{T}} 
    \boldsymbol{N} \; \mathrm{d} \Gamma^q_e
  \end{align}
  where $\boldsymbol{N}$ is row vector containing shape functions
  and $\boldsymbol{B} = (\boldsymbol{DN}) \boldsymbol{J}^{-1}$ (see 
  Appendix A for details on $\boldsymbol{DN}$ 
  and $\boldsymbol{J}$). From equation \eqref{Eqn:LocalStiffnessMatrix_Ke} 
  and assumptions (A1)--(A5), it is evident $\frac{1}{2} \left(\boldsymbol{K}_e 
  + \boldsymbol{K}_e^{\mathrm{T}} \right)$ is positive semi-definite. 
  Furthermore, the assembly procedure ensures that the global stiffness 
  matrix $\boldsymbol{K}$ is positive definite \cite[Section 2 and Section 3]
  {1989_Wathen_CMAME_v74_p271_p287}. As the mesh is coarse, $\boldsymbol{K}$ 
  is not an $M$-matrix. But we want $\boldsymbol{c}$ to satisfy the DMP 
  constraints. Hence, this results in the following constrained discrete 
  system of equations: 
  \begin{subequations}
    \begin{align}
      \label{Eqn:PLCP_Eqn1}
      \boldsymbol{K} \boldsymbol{c} &= \boldsymbol{f} + 
      \boldsymbol{\lambda} \\
      \label{Eqn:PLCP_Eqn2}
      \boldsymbol{\lambda} &\succeq \boldsymbol{0} \\
      \label{Eqn:PLCP_Eqn3}
      \boldsymbol{c} &\succeq \boldsymbol{0} \\
      \label{Eqn:PLCP_Eqn4}
      \boldsymbol{\lambda} \bullet 
      \boldsymbol{c} &= 0
    \end{align}
  \end{subequations}
  As $\boldsymbol{K}$ is a $P$-matrix, the above system is a 
  $P$-matrix linear complementarity problem. This completes 
  the proof.
\end{proof}

It needs to be emphasized that solving a LCP problem with 
P-matrix is, in general, NP-hard \cite{2007_Rust_PhD_Thesis_ETH_2007}. 
Therefore, posing the discrete equations as a LCP problem and 
numerically solving it is not a viable approach, especially for 
large-scale ADR problems with high $\mathbb{P}\mathrm{e}_{h}$. 
Moreover, it is not always feasible to find a DMP-based 
$h$-refined mesh that will produce accurate results for 
ADR equation for sufficiently high element P\'{e}clet 
number and element Damk\"{o}hler number.

\subsection{Approach \#5:~Posing the discrete problem 
  as constrained normal equations}
\label{Subsec:Approach5_NormalEquations}
One way of constructing an optimization-based approach 
to meet the non-negative constraint is to rewrite the 
discrete problem as the following constrained normal 
equations: 
\begin{subequations}
  \begin{align}
    \label{Eqn:Constrained_Normal_Equations}
    \mathop{\mathrm{minimize}}_{\boldsymbol{c} \, 
      \in \, \mathbb{R}^{ncdofs}} & \quad\frac{1}{2} 
    \langle \boldsymbol{c}; \boldsymbol{K}^{\mathrm{T}} 
    \boldsymbol{K} \boldsymbol{c} \rangle - \langle 
    \boldsymbol{c}; \boldsymbol{K}^{\mathrm{T}} 
    \boldsymbol{f} \rangle \\
    \mbox{subject to} & \quad \boldsymbol{c} 
    \succeq \boldsymbol{0}
  \end{align}
\end{subequations}
where $\langle \bullet; \bullet \rangle$ denotes the 
standard inner-product in Euclidean spaces. The corresponding 
first-order optimality conditions can be written as:
\begin{subequations}
  \begin{align}
    \label{Eqn:KKT_Constrained_Normal_Equations}
    &\boldsymbol{K}^{\mathrm{T}} \boldsymbol{K} 
    \boldsymbol{c} = \boldsymbol{K}^{\mathrm{T}} 
    \boldsymbol{f} + \boldsymbol{\lambda} \\
    &\boldsymbol{c} \succeq \boldsymbol{0} \\
    &\boldsymbol{\lambda} \succeq \boldsymbol{0} \\
    &\lambda_i c_i  = 0 \quad \forall i = 1, \cdots, 
    ncdofs
  \end{align}
\end{subequations}
If there are no constraints, the optimization problem becomes:
\begin{align}
  \label{Eqn:Unconstrained_Normal_Equations}
  \mathop{\mathrm{minimize}}_{\boldsymbol{c} \, 
  \in \, \mathbb{R}^{ncdofs}} & \quad\frac{1}{2} 
  \langle \boldsymbol{c}; \boldsymbol{K}^{\mathrm{T}} 
  \boldsymbol{K} \boldsymbol{c} \rangle - \langle 
  \boldsymbol{c}; \boldsymbol{K}^{\mathrm{T}} 
  \boldsymbol{f} \rangle 
\end{align}
The first-order optimality condition for the unconstrained 
discrete optimization problem is:
\begin{align}
  \label{Eqn:KKT_Unconstrained_Normal_Equations}
  \boldsymbol{K}^{\mathrm{T}} \boldsymbol{K} \boldsymbol{c} 
  = \boldsymbol{K}^{\mathrm{T}} \boldsymbol{f} 
\end{align}
In the numerical mathematics literature (e.g., see 
\cite{Demmel}), the above system of equations 
\eqref{Eqn:KKT_Unconstrained_Normal_Equations} is 
referred to as normal equations. The three 
main deficiencies of this approach are:
\begin{enumerate}[(i)]
\item The constrained optimization-based normal equations 
  method does not avoid node-to-node spurious 
  oscillations. In addition, there is no obvious way 
  of fixing the method to avoid this type of 
  unphysical solutions.
\item It is well-known that the condition number of 
  $\boldsymbol{K}^{\mathrm{T}} \boldsymbol{K}$ will 
  be much worse than $\boldsymbol{K}$. So the numerical 
  solution will be less reliable, less accurate, and 
  numerically not stable \cite{Demmel}.
\item The discrete optimization problem given by equation 
  \eqref{Eqn:Unconstrained_Normal_Equations} on which 
  non-negative constraints are enforced does not have 
  a corresponding continuous variational/minimization 
  problem.
\end{enumerate}

We shall use the following academic problem to 
illustrate the aforementioned deficiencies:
\begin{subequations}
  \begin{align}
    \label{Eqn:1D_Academic_Problem_GE}
    &v \frac{dc}{dx} - D \frac{d^2 c}{d x^2} 
    = f \quad \forall x \in (0,1) \\
    \label{Eqn:1D_Academic_Problem_BCs}
    &c(x = 0) = c(x = 1) = 0
  \end{align}
\end{subequations}
where $v$, $D$, and $f$ are assumed to be constants. 
In our numerical experiment, we have taken the number 
of mesh elements to be 11, $v/D = 150$, and $f = 1$. 
The element P\'eclet number ($\mathbb{P}\mathrm{e}_{h} 
= \frac{v h}{2 D}$) is approximately 6.82. Since it is 
greater than unity, there be will spurious node-to-node 
oscillations under the standard single-field Galerkin 
formulation. From Figure \ref{Fig:Normal_spurious}, it 
is evident that the normal equations approach does not 
eliminate the spurious node-to-node oscillations.
The condition number of the stiffness matrix under the 
standard single-field Galerkin formulation is 8.41, 
whereas the condition number of the stiffness matrix 
under the normal equations approach is 70.69. For small 
element P\'eclet numbers, deficiency (i) can be avoided. 
But deficiencies (ii) and (iii) will still be present and 
cannot be circumvented. \emph{Hence, posing the discrete 
problem as constrained normal equations is not a viable 
approach to meet maximum principles and the non-negative 
constraint.}

\section{PROPOSED COMPUTATIONAL FRAMEWORK}
\label{Sec:NN_AD_Proposed}
We employ least-squares formalism to develop a class 
of structure-preserving numerical formulations whose 
solutions satisfy DMP, LSB, and GSB. The formulations 
are built based on minimization of unconstrained/constrained 
quadratic least-squares functionals. In a least-squares-based 
finite element formulation, a non-physical least-squares 
functional is constructed in terms of the sum of the squares 
of the residuals in an appropriate norm. These residuals are 
based on the underlying governing equations. However, it should 
be noted that LSFEMs are different from the Galerkin least-squares 
or stabilized mixed methods, where least-squares terms are added 
locally or globally to variational problems.

The success of LSFEM is due to the rich mathematical
foundations that influence both the analysis and the 
algorithmic development. LSFEM offers several attractive
features. The resulting weak formulations are coercive. 
Hence, a unique global minimizer exists for the least-squares 
functional and this minimizer coincides with the exact solution.
Conforming finite element discretizations of least-squares 
functionals leads to stable and (eventually) optimally 
accurate numerical solutions.
For mixed LSFEM-based formulations, equal order 
interpolation can be used for all the unknowns, 
which is computationally the most convenient.
The resulting algebraic system is symmetric and positive 
definite. Thus, the discrete system can be solved using 
standard and robust iterative numerical methods.
For more details on LSFEM for various applications, see 
Bochev and Gunzberger \cite{Bochev_Gunzburger} and Jiang 
\cite{Jiang_LSFEM}.

\subsection{Design synopsis of the proposed numerical methodology}
\label{SubSec:FirstOrder_Mixed_Form}
The central idea of the proposed computational framework 
is to constrain a least-squares functional with LSB 
and non-negative constraints. The main steps involved 
in the design of the proposed computational framework 
are: 
\begin{enumerate}[(i)]
\item The governing equations of the ADR problem are 
  written in first-order mixed form. 
\item We construct a stabilized least-squares functional 
  for these first-order governing equations.
\item We construct algebraic equality constraints 
  to enforce element-wise/local species balance (LSB). 
\item We enforce bound constraints to the constructed 
  LSFEM to meet maximum principles and the non-negative 
  constraint in the discrete setting. In order to achieve 
  this, we shall use low-order finite element interpolation 
  for $c(\mathbf{x})$. 
\end{enumerate}

The first-order mixed form of the governing 
equations can be written as: 
\begin{subequations}
  \begin{alignat}{2}
    \label{Eqn:NN_AD_First_GE_CM}
    &\mathbf{q}(\mathbf{x}) - \mathbf{v}(\mathbf{x}) c(\mathbf{x}) 
    + \mathbf{D}(\mathbf{x}) \mathrm{grad}[c(\mathbf{x})] = \mathbf{0} 
    \quad &&\mathrm{in} \; \Omega \\
    \label{Eqn:NN_AD_First_GE_BE}
    &\mathrm{div}[\mathbf{q}(\mathbf{x})] = f(\mathbf{x}) - \alpha 
    (\mathbf{x}) c (\mathbf{x}) \quad &&\mathrm{in} \; \Omega \\
    \label{Eqn:NN_AD_First_GE_DBC}
    &c(\mathbf{x}) = c^{\mathrm{p}}(\mathbf{x}) \quad &&\mathrm{on} \; 
    \Gamma^{c} \\
    \label{Eqn:NN_AD_First_GE_NBC}
    &\left(\mathbf{q}(\mathbf{x}) - \left( \frac{1 + \mathrm{Sign}
    [\mathbf{v} \bullet \widehat{\mathbf{n}}]}{2}\right) \mathbf{v}
    (\mathbf{x}) c(\mathbf{x}) \right) \bullet \widehat{\mathbf{n}}
    (\mathbf{x}) = q^{\mathrm{p}} (\mathbf{x}) \quad &&\mathrm{on} 
    \; \Gamma^{q}
  \end{alignat}
\end{subequations}
The bound constraints to meet discrete maximum 
principles take the following form:
\begin{align} 
  c_{\mathrm{min}} \boldsymbol{1} \preceq \boldsymbol{c} 
  \preceq c_{\mathrm{max}} \boldsymbol{1} \quad \mathrm{in} 
  \; \overline{\Omega}_h
\end{align}
where $c_{\mathrm{min}}$ and $c_{\mathrm{max}}$ are the minimum and 
maximum concentration values possible in $\overline{\Omega}$. 
The LSB equality constraints can be constructed in 
two different ways. The first approach is based on 
the integral statement of the balance of species on 
an element, and takes the following mathematical 
form: 
\begin{align}
  \label{Eqn:LSB_Second_Method}
   \int \limits_{\Omega_e} \alpha (\mathbf{x}) 
   c(\mathbf{x}) \; \mathrm{d} \Omega_e + \int 
   \limits_{\partial \Omega_e} \mathbf{q}(\mathbf{x}) 
   \bullet \widehat{\mathbf{n}}(\mathbf{x}) \; 
   \mathrm{d} \Gamma_e = \int \limits_{\Omega_e} 
   f(\mathbf{x}) \; \mathrm{d} \Omega_e
\end{align}
The second approach is to enforce equation 
\eqref{Eqn:NN_AD_First_GE_BE} in each mesh 
element $\overline{\Omega}_{\mathrm{e}}$ in 
an integral form:
\begin{align}
  \label{Eqn:LSB_First_Method}
   \int \limits_{\Omega_e} \alpha (\mathbf{x}) 
   c(\mathbf{x}) \; \mathrm{d} \Omega_e + \int 
   \limits_{\Omega_e} \mathrm{div}[\mathbf{q}
   (\mathbf{x})] \; \mathrm{d} \Omega_e = \int 
   \limits_{\Omega_e} f(\mathbf{x}) \; \mathrm{d} 
   \Omega_e
\end{align}
One can obtain equation \eqref{Eqn:LSB_Second_Method} 
by applying the divergence theorem to equation 
\eqref{Eqn:LSB_First_Method}, which means that 
these two approaches are equivalent in the 
continuous setting. This will not always be 
the case in the discrete setting. 
In the case of simplicial and non-simplicial low-order 
finite elements, these approaches are equivalent. However, 
these two approaches can be different in the case of 
higher-order finite elements. This is because in certain 
higher-order finite elements (e.g., nine-node quadrilateral 
element), not all the nodes are on the boundary of the 
element. The flux at an interior node contributes to the 
second integral in equation \eqref{Eqn:LSB_First_Method} 
but not to the corresponding term in equation 
\eqref{Eqn:LSB_Second_Method}. These issues are 
beyond the scope of this paper. Herein, we take 
the first approach given by equation 
\eqref{Eqn:LSB_Second_Method}.

We next construct two different least-squares functionals and 
analyze the influence of various constraints on the performance 
of these LSFEMs. It should be noted that Hsieh and Yang 
\cite{2009_Hsieh_Yang_CMAME_v199_p183_p196} have proposed 
similar least-squares functionals, but they considered homogeneous 
isotropic steady-state advection-diffusion equations. Moreover, 
even in the simple setting of isotropic diffusivity, they did 
not consider general Neumann BCs, spatially varying velocity 
fields, simplicial vs. non-simplicial elements, or the effects 
of DMPs and LSB on the performance of the least-squares functionals. 
This paper investigates all the mentioned aspects: we incorporate 
anisotropy, heterogeneity, transient effects, linear reaction terms, 
non-solenoidal spatially varying velocity fields, and DMP and LSB 
constraints.

\subsection{Weighted primitive LSFEM}
\label{Eqn:Weighted_Primitive_LSFEM}
The weighted primitive LSFEM is the standard way of 
constructing a LSFEM-based formulation. It does not 
contain any additional stabilization terms. The weighted 
primitive least-squares functional $\mathfrak{F}_{\mathrm{Prim}}
(c,\mathbf{q}) :\mathcal{C} \times \mathcal{Q} \rightarrow 
\mathbb{R}$ in $L_2$-norm can be written as:
\begin{align}
  \label{Eqn:Weighted_Primitive_LS_Functional}
  \mathfrak{F}_{\mathrm{Prim}} \left( c,\mathbf{q} \right) &:=
  \frac{1}{2} \Big \| \mathbf{A}(\mathbf{x}) \big( \mathbf{q} 
  - c \mathbf{v} + \mathbf{D} \mathrm{grad}[c] \big) \Big \|^2_
  {\Omega} \nonumber \\
  &+ \frac{1}{2} \Big \| \beta (\mathbf{x}) \big( \alpha c + 
  \mathrm{div}[ \mathbf{q} ] - f \big) \Big \|^2_{\Omega} 
  \nonumber \\
  &+ \frac{1}{2} \Bigg \| \left(\mathbf{q} - \left(\frac{1 + 
  \mathrm{Sign}[\mathbf{v} \bullet \widehat{\mathbf{n}}]}{2} 
  \right) c \mathbf{v} \right) \bullet \widehat{\mathbf{n}} 
  - q^{\mathrm{p}} \Bigg \|^2_{\Gamma^{q}}
\end{align}
where the second-order tensor $\mathbf{A}(\mathbf{x})$ 
and the scalar function $\beta(\mathbf{x})$ are the 
weights, which are defined as follows:
\begin{subequations}
  \begin{align}
    \label{Eqn:Weight_A_Tensor}
    \mathbf{A}(\mathbf{x}) &= \left \{
    \begin{array}{cl}
      \mathbf{I} & \text{LS Type-1 } \\
      \mathbf{D}^{-1/2}(\mathbf{x}) & \text{LS Type-2}
    \end{array} \right. \\
    \label{Eqn:Weight_beta_Scalar}
    \beta(\mathbf{x}) &= \left \{
    \begin{array}{cl}
      1 & \text{LS Type-1} \\
      \left. \begin{aligned}
               1 \qquad  & \quad \text{if } \alpha(\mathbf{x}) = 0 \\
               \alpha^{-1/2}(\mathbf{x}) & \quad \text{if } \alpha(
               \mathbf{x}) \neq 0
             \end{aligned} 
      \right \} & \text{LS Type-2}
    \end{array} \right.
  \end{align}
\end{subequations}
A corresponding weak form can be obtained by 
setting the G\^ateaux variation of the functional 
\eqref{Eqn:Weighted_Primitive_LS_Functional} to 
zero. 
We shall show in Sections \ref{Sec:NN_AD_NumericalConvergence} 
and \ref{Sec:NN_AD_Fast_Reactions} that a naive way of constructing 
LSFEM formulation, just like the weighted primitive LSFEM, does 
not perform well for advection-dominated ADR problems. Moreover, 
enforcing LSB and DMP constraints do not seem to have a profound 
effect. In order to adequately capture steep boundary 
and interior layers, we introduce an alternate stabilized LSFEM 
formulation, which will be referred to as the weighted negatively 
stabilized streamline diffusion LSFEM. This stabilized LSFEM 
formulation will be able to handle a wide spectrum of ADR problems 
ranging from advection-dominated to reaction-dominated problems. 

\subsection{Weighted negatively stabilized streamline diffusion LSFEM}
\label{SubSec:Weighted_NSSD_LSFEM}
The underlying idea of the proposed stabilized LSFEM formulation 
is to combine the streamline diffusion and stabilized Galerkin 
formulations. This is motivated by the prior studies that 
combining these two formulations exhibit enhanced stability 
(for example, see \cite{1997_Lazarov_EWJNM_v5_p249_p264,
2009_Hsieh_Yang_CMAME_v199_p183_p196}). In this formulation, 
we introduce a small element-wise stabilization parameter 
$\delta_{\Omega_{\mathrm{e}}}$ to correct $\mathbf{q}(\mathbf{x})$ 
in the streamline direction by adding second-order derivatives 
of $c(\mathbf{x})$. The modified flux along the streamline 
direction takes the following form:
\begin{align}
  \label{Eqn:Weighted_NSSD_Modified_Flux}
  \mathbf{q} = c \mathbf{v} - \mathbf{D} 
  \mathrm{grad}[c] + \delta_{\Omega_{\mathrm{e}}} 
  \mathbf{v} \left( \mathrm{div}[c \mathbf{v} - 
  \mathbf{D} \mathrm{grad}[c]] \right)
\end{align}
The basic philosophy of the correction to the 
flux given by equation \eqref{Eqn:Weighted_NSSD_Modified_Flux} 
      is in the spirit of stabilized finite element formulations 
      such as SUPG \cite{Franca_Frey_Hughes_CMAME_1992_v95_p253,
        1997_Lazarov_EWJNM_v5_p249_p264}. This flux correction 
      is different from that of the Flux-Corrected Transport 
      (FCT) methods \cite{Kuzmin_Lohner_Turek}.
Correspondingly, the species balance equation 
accounting for these corrections will be:
\begin{align}
  \label{Eqn:Weighted_NSSD_Modified_Species_Balance}
  \alpha c + \mathrm{div}[ \mathbf{q} ] = f 
  + f_{\delta_{\Omega_{\mathrm{e}}}}
\end{align}
where 
\begin{align}
  \label{Eqn:NSSD_Stabilized_VolSource}
  f_{\delta_{\Omega_{\mathrm{e}}}} := \delta_{\Omega_{\mathrm{e}}} \big( 
  \mathrm{grad}[f - \alpha c] \bullet \mathbf{v} + \mathrm{div}[\mathbf{v}] 
  \, \left( f - \alpha c \right) \big)
\end{align}
The modification to the flux (given by equations 
\eqref{Eqn:Weighted_NSSD_Modified_Flux}--\eqref{Eqn:NSSD_Stabilized_VolSource}) 
will present two different ways of constructing Neumann BCs.  

The first way utilizes the quantities $\mathbf{q}
(\mathbf{x})$, $c(\mathbf{x})$, $\alpha(\mathbf{x})$, 
and $f(\mathbf{x})$, and takes the following form: 
\begin{align}
  \label{Eqn:Neumann_BC_NSSD_Method1}
  \left(\mathbf{q} - \left( 
  \frac{1 + \mathrm{Sign}[\mathbf{v} \bullet 
  \widehat{\mathbf{n}}]}{2}\right) c \mathbf{v} 
  - \delta_{\Omega_{\mathrm{e}}} \left(f - \alpha 
  c \right) \mathbf{v} \right) \bullet 
  \widehat{\mathbf{n}}(\mathbf{x}) 
  = q^{\mathrm{p}} (\mathbf{x}) \quad \mathrm{on} \; 
  \Gamma^{q}
\end{align}
The second way utilizes $\mathbf{q}(\mathbf{x})$, 
$c(\mathbf{x})$, and the first and second derivatives 
of $c(\mathbf{x})$. The corresponding expression for 
Neumann BCs takes the following form:
\begin{align}
  \label{Eqn:Neumann_BC_NSSD_Method2}
  \left(\mathbf{q} - \left( \frac{1 +
  \mathrm{Sign}[\mathbf{v} \bullet 
  \widehat{\mathbf{n}}]}{2}\right) c 
  \mathbf{v} - \delta_{\Omega_{\mathrm{e}}} 
  \left(\mathrm{div}[c \mathbf{v} -
  \mathbf{D} \mathrm{grad}[c]] \right) 
  \mathbf{v} \right) \bullet 
  \widehat{\mathbf{n}}(\mathbf{x}) 
  = q^{\mathrm{p}} (\mathbf{x}) \quad \mathrm{on} 
  \; \Gamma^{q}
\end{align}
In the continuous setting, equations \eqref{Eqn:Neumann_BC_NSSD_Method1} 
and \eqref{Eqn:Neumann_BC_NSSD_Method2} are equivalent. However, 
in the discrete setting, the performance of these equations can 
be different based on the kind of (finite) element being employed. 
For example, for simplicial elements
(such as three-node triangular (T3) element and four-node tetrahedral 
(T4) element) and four-node quadrilateral (Q4) element, both $\mathrm{div}
[\mathrm{grad}[c(\mathbf{x})]]$ and $\mathrm{grad}[\mathrm{grad}[c(\mathbf{x})]]$ 
are zero for $\overline{\Omega}_{\mathrm{e}} \in {\Gamma^{q}}$. 
This is because the Hessian of $\boldsymbol{N}$, which is 
$\boldsymbol{DDN}$, is a zero matrix for both two-node 
linear (L2) and three-node triangular elements. For more 
details, see Appendix A. But, this is not the case with 
non-simplicial linear finite elements for $nd = 3$ and 
higher-order finite elements (in any dimension). Hence, 
the Neumann BCs based on equation \eqref{Eqn:Neumann_BC_NSSD_Method2} 
are not always physically consistent. However, Neumann BCs based on 
equation \eqref{Eqn:Neumann_BC_NSSD_Method1} are always consistent 
irrespective of the finite element used. Herein, we have chosen 
Neumann BCs given by equation \eqref{Eqn:Neumann_BC_NSSD_Method1}.

Based on the above set of equations 
\eqref{Eqn:Weighted_NSSD_Modified_Flux}--\eqref{Eqn:Neumann_BC_NSSD_Method1}, 
we construct a $L_2$-norm based least-squares functional. Additionally, 
as in the Galerkin least-squares method, we add a stabilization term 
to this functional. This stabilization term is as follows:
\begin{align}
  \label{Eqn:Stabilization_LS_Term_NSSD}
  \frac{1}{2} \sum_{\Omega_{\mathrm{e}} \in \Omega_h} \tau_{
  \Omega_{\mathrm{e}}} \Big \| \mathrm{div} \big [ c \mathbf{v} - 
  \mathbf{D} \mathrm{grad}[ c ] \big ] + \alpha c - f \Big 
  \|^2_{\Omega_{\mathrm{e}}}
\end{align}
The least-squares functional for the weighted negatively 
stabilized streamline diffusion formulation 
$\mathfrak{F}_{\mathrm{NgStb}}(c,\mathbf{q}): \mathcal{C} \times 
\mathcal{Q} \rightarrow \mathbb{R}$ in $L_2$-norm takes the 
following form: 
\begin{align}
  \label{Eqn:Weighted_NSSD_LS_Functional}
  \mathfrak{F}_{\mathrm{NgStb}} \left( c,\mathbf{q} \right) &:=
  \frac{1}{2} \sum_{\Omega_{\mathrm{e}} \in \Omega_h} \Big \| 
  \mathbf{A}(\mathbf{x}) \big( \mathbf{q} - c \mathbf{v} + \mathbf{D} 
  \mathrm{grad}[c] - \delta_{\Omega_{\mathrm{e}}} \mathbf{v} \left( 
  \mathrm{div}[c \mathbf{v} - \mathbf{D} \mathrm{grad}[c]] \right)  
  \big) \Big \|^2_{\Omega_{\mathrm{e}}} \nonumber \\
  &+ \frac{1}{2} \sum_{\Omega_{\mathrm{e}} \in \Omega_h} \Big \| 
  \beta (\mathbf{x}) \big( \alpha c + \mathrm{div}[ \mathbf{q} ] - f 
  - f_{\delta_{\Omega_{\mathrm{e}}}} \big) \Big \|^2_{\Omega_{\mathrm{e}}} 
  \nonumber \\
  &+ \frac{1}{2} \sum_{\Omega_{\mathrm{e}} \in {\Gamma^{q}}} \Bigg 
  \| \left(\mathbf{q} - \left( \frac{1 + \mathrm{Sign}
  [\mathbf{v} \bullet \widehat{\mathbf{n}}]}{2}\right) c \mathbf{v} - 
  \delta_{\Omega_{\mathrm{e}}} \left(f - \alpha c \right) \mathbf{v} 
  \right) \bullet \widehat{\mathbf{n}} - q^{\mathrm{p}} \Bigg \|^2_
  {\Omega_{\mathrm{e}}} \nonumber \\
  &+\frac{1}{2} \sum_{\Omega_{\mathrm{e}} \in \Omega_h} \tau_{
  \Omega_{\mathrm{e}}} \Big \| \mathrm{div} \big [ c \mathbf{v} - 
  \mathbf{D} \mathrm{grad}[ c ] \big ] + \alpha c - f \Big 
  \|^2_{\Omega_{\mathrm{e}}}
\end{align}
The element dependent parameters $\tau_{\Omega_{\mathrm{e}}} 
\leq 0$ and $\delta_{\Omega_{\mathrm{e}}} \leq 0$ are given as:
\begin{subequations}
  \begin{align}
    \label{Eqn:StrDiff_Parameter}
    \delta_{\Omega_{\mathrm{e}}} &= - \frac{\delta_o 
    \lambda_{\mathrm{min}} h^2_{\Omega_{\mathrm{e}}}}{\left( 
    \lambda^2_{\mathrm{max}} + \delta_1 \displaystyle \mathop{\mathrm{max}}_
    {\mathbf{x} \in \overline{\Omega}} \left [ \left(\alpha + 
    \mathrm{div} [\mathbf{v}] \right)^2 \right ] h^2 + 
    \delta_2 \displaystyle \mathop{\mathrm{max}}_
    {\mathbf{x} \in \overline{\Omega}} \left [ \| 
    \mathrm{div}[\mathbf{D}] \|^2 \right ] h^2 \right)} \\
    \label{Eqn:Stab_Parameter}
    \tau_{\Omega_{\mathrm{e}}} &= - \frac{\tau_o 
    \lambda^2_{\mathrm{min}} h^2_{\Omega_{\mathrm{e}}}}{\left( 
    \lambda^2_{\mathrm{max}} + \tau_1 \displaystyle \mathop{\mathrm{max}}_
    {\mathbf{x} \in \overline{\Omega}} \left [ \left(\alpha + 
    \mathrm{div} [\mathbf{v}] \right)^2 \right ] h^2 + 
    \tau_2 \displaystyle \mathop{\mathrm{max}}_
    {\mathbf{x} \in \overline{\Omega}} \left [ \| 
    \mathrm{div}[\mathbf{D}] \|^2 \right ] h^2 \right)}
  \end{align}
\end{subequations}
where $\delta_o$, $\delta_1$, $\delta_2$, $\tau_o$, 
$\tau_1$, and $\tau_2$ are non-negative constants. 
Appendix B provides a thorough mathematical 
justification behind the above stabilization 
parameters.

For unconstrained LSFEMs, the errors incurred 
in satisfying LSB and GSB can be calculated as:
\begin{subequations}
  \label{Eqn:Errors_in_LSB_GSB}
  \begin{align}
    \label{Eqn:Error_in_LSB}
     \epsilon^{(e)}_{\text{\tiny {LSB}}} &= \int \limits_{\Omega_e} 
     \alpha (\mathbf{x}) c(\mathbf{x}) \; \mathrm{d} \Omega + 
     \int \limits_{\partial \Omega_e} \mathbf{q}(\mathbf{x}) 
     \bullet \widehat{\mathbf{n}}(\mathbf{x}) \; \mathrm{d} 
     \Gamma - \int \limits_{\Omega_e} f(\mathbf{x}) \; 
     \mathrm{d} \Omega \\
     \label{Eqn:Error_in_GSB}
     \epsilon_{\text{\tiny {GSB}}} &= \sum_{e = 1}^{Nele} 
     \epsilon^{(e)}_{\text{\tiny {LSB}}}
  \end{align}
\end{subequations}
where $c(\mathbf{x})$ and $\mathbf{q}(\mathbf{x})$ are the 
solutions obtained by solving a given unconstrained LSFEM. 
In numerical $h$-convergence study, we are interested in 
the following quantities with respect to $h$-refinement:
\begin{align}
  \label{Eqn:AbsError_in_LSB}
  \epsilon_{\text{\tiny {MaxAbsLSB}}} &= 
  \mathop{\mathrm{max}}_{\Omega_{e} \in 
    \Omega_{h}} \left [ | \epsilon^{(e)}_{\text{\tiny 
        {LSB}}} | \right ] \\
  \label{Eqn:AbsError_in_GSB}
  \epsilon_{\text{\tiny {AbsGSB}}} &= 
  |\epsilon_{\text{\tiny {GSB}}}|
\end{align}
Few remarks about the species balance are in order.
In writing equation \eqref{Eqn:AbsError_in_GSB}, 
we have assumed that the mesh is conforming, and 
the test and trial functions belong to $C^{0}(\Omega)$ 
(i.e., there is inter-element continuity of the
functions). Under the proposed computational 
framework, we place explicit (equality) constraints 
to meet $\epsilon^{(e)}_{\text{\tiny {LSB}}} = 0 \quad 
\forall e = 1, \cdots, Nele$. By meeting the 
local species balance, the global species 
balance is trivially met.

\subsection{Discrete equations}
\label{Eqn:Discrete_CQP_KKT}
Let $\boldsymbol{K}_{cc}$ denote the stiffness matrix 
obtained by lower-order finite element discretization 
of the LSFEM terms involving $c(\mathbf{x})$ and $w(\mathbf{x})$. Similarly, 
we can define the stiffness matrices $\boldsymbol{K}_{c \mathbf{q}}$, 
$\boldsymbol{K}_{\mathbf{q} c}$, and $\boldsymbol{K}_{\mathbf{q} \mathbf{q}}$.
The load vectors are denoted by $\boldsymbol{r}_{c}$ and $\boldsymbol{r}_
{\mathbf{q}}$, respectively. These vectors are obtained from the finite 
element discretization of the LSFEM terms involving $w(\mathbf{x})$ and 
$\mathbf{p}(\mathbf{x})$. It should be noted that the stiffness matrices 
$\boldsymbol{K}_{cc}$ and $\boldsymbol{K}_{\mathbf{q} \mathbf{q}}$ are 
symmetric and positive definite. Furthermore, $\boldsymbol{K}_{\mathbf{q} 
c} = \boldsymbol{K}_{c \mathbf{q}}^{\mathrm{T}}$. For more details, see 
Appendix C.

The corresponding constrained optimization problem in the 
discrete setting for the proposed locally conservative 
DMP-preserving LSFEMs can be written as follows:
\begin{subequations}
  \begin{align}
    \label{Eqn:Constrained_Proposed_LSFEMs}
    \mathop{\mathrm{minimize}}_{\substack{\boldsymbol{c} \, 
    \in \, \mathbb{R}^{ncdofs} \\ \boldsymbol{q} \, 
    \in \, \mathbb{R}^{nqdofs}}} & \quad \frac{1}{2} 
    \langle \boldsymbol{c}; \boldsymbol{K}_{cc} 
    \boldsymbol{c} \rangle + \langle \boldsymbol{c}; 
    \boldsymbol{K}_{c \mathbf{q}} \boldsymbol{q} \rangle 
    +  \frac{1}{2} \langle \boldsymbol{q}; \boldsymbol{K}
    _{\mathbf{q} \mathbf{q}} \boldsymbol{q} \rangle - 
    \langle \boldsymbol{c}; \boldsymbol{r}_c \rangle - 
    \langle \boldsymbol{q}; \boldsymbol{r}_{\mathbf{q}} 
    \rangle \\
    \label{Eqn:Constrained_Proposed_LSFEMs_Constraints}
    \mbox{subject to} & \quad 
    \begin{cases}
      \boldsymbol{A}_c \boldsymbol{c} + 
      \boldsymbol{A}_{\mathbf{q}} 
      \boldsymbol{q} = \boldsymbol{b}_f \\
      c_{\mathrm{min}} \boldsymbol{1} \preceq 
      \boldsymbol{c} \preceq c_{\mathrm{max}} 
      \boldsymbol{1}
    \end{cases}
  \end{align}
\end{subequations}
where ``$nqdofs$'' denotes the number of degrees-of-freedom 
for the flux vector, and ``$ncdofs$'' denotes the number of 
degrees-of-freedom for the concentration.
The vector of size $ncdofs \times 1$ with all 
entries to be unity is denoted as $\mathbf{1}$. 
Recall that $\langle \bullet;\bullet\rangle$ 
denotes the standard inner-product on the Euclidean 
spaces.  
The finite element discretization of the local species 
balance equation gives rise to the global LSB matrices 
$\boldsymbol{A}_c$ and $\boldsymbol{A}_{\mathbf{q}}$, and 
the global LSB vector $\boldsymbol{b}_f$.
The matrices $\boldsymbol{A}_c$ and $\boldsymbol{A}_{\mathbf{q}}$ 
are of sizes $Nele \times ncdofs$ and $Nele \times nqdofs$, 
respectively. Similar inference can be drawn on the sizes of 
$\boldsymbol{b}_f$, $\boldsymbol{r}_c$, $\boldsymbol{r}_{\mathbf{q}}$, 
$\boldsymbol{K}_{c \mathbf{q}}$, and $\boldsymbol{K}_{\mathbf{q} \mathbf{q}}$.
Since $\boldsymbol{K}_{\mathbf{q} c} = \boldsymbol{K}_{c \mathbf{q}}
^{\mathrm{T}}$ and the matrices $\boldsymbol{K}_{cc}$ and 
$\boldsymbol{K}_{\mathbf{q} \mathbf{q}}$ are symmetric and 
positive definite, the constrained optimization problem 
\eqref{Eqn:Constrained_Proposed_LSFEMs}--\eqref{Eqn:Constrained_Proposed_LSFEMs_Constraints} 
belongs to \emph{convex quadratic programming} and has a 
unique global minimizer \cite{Boyd_convex_optimization}. 
The corresponding first-order optimality conditions -- 
popularly known as the Karush-Kuhn-Tucker (KKT) conditions 
-- for this discrete optimization problem take the following 
form:
\begin{subequations}
  \begin{align}
    \label{Eqn:KKT_Cond_wrt_Conc}
    &\boldsymbol{K}_{cc} \boldsymbol{c} + \boldsymbol{K}
    _{c \mathbf{q}} \boldsymbol{q} = \boldsymbol{r}_c 
    - \boldsymbol{A}^{\mathrm{T}}_c \boldsymbol{\lambda}_c
    + \boldsymbol{\mu}_{\mathrm{min}} - \boldsymbol{\mu}_
    {\mathrm{max}} \\
    \label{Eqn:KKT_Cond_wrt_Flux}
    &\boldsymbol{K}_{c \mathbf{q}}^{\mathrm{T}} 
    \boldsymbol{c} + \boldsymbol{K}_{\mathbf{q} 
    \mathbf{q}} \boldsymbol{q} = \boldsymbol{r}
    _{\mathbf{q}} - \boldsymbol{A}^{\mathrm{T}}
    _{\mathbf{q}} \boldsymbol{\lambda}_{\mathbf{q}}
    \\
    \label{Eqn:KKT_Cond_wrt_LSB}
    &\boldsymbol{A}_c \boldsymbol{c} + 
    \boldsymbol{A}_{\mathbf{q}} 
    \boldsymbol{q} = \boldsymbol{b}_f \\
    \label{Eqn:KKT_Cond_KKT_MinMultiplier}
    &\boldsymbol{\mu}_{\mathrm{min}} \succeq
    \boldsymbol{0} \\
    \label{Eqn:KKT_Cond_KKT_MaxMultiplier}
    &\boldsymbol{\mu}_{\mathrm{max}} \succeq
    \boldsymbol{0} \\
    \label{Eqn:KKT_Cond_wrt_ConcMin}
    &\left(\boldsymbol{c} - c_{\mathrm{min}} 
    \boldsymbol{1} \right) \bullet 
    \boldsymbol{\mu}_{\mathrm{min}} = 0 \\
    \label{Eqn:KKT_Cond_wrt_ConcMax}
    &\left( c_{\mathrm{max}} \boldsymbol{1} 
    - \boldsymbol{c} \right) \bullet 
    \boldsymbol{\mu}_{\mathrm{max}} = 0
  \end{align}
\end{subequations}
where $\boldsymbol{\lambda}_c$ and $\boldsymbol{\lambda}_{\mathbf{q}}$ 
are the Lagrange multipliers enforcing the LSB equality constraints, 
which stem from equation \eqref{Eqn:KKT_Cond_wrt_LSB}. 
$\boldsymbol{\mu}_{\mathrm{min}}$ and $\boldsymbol{\mu}_{\mathrm{max}}$ 
are the KKT multipliers enforcing the DMP inequality constraints 
given by $c_{\mathrm{min}} \boldsymbol{1} \preceq \boldsymbol{c}$ 
and $\boldsymbol{c} \preceq c_{\mathrm{max}} \boldsymbol{1}$. Note 
that the non-negative constraint is a subset of the DMP inequality 
constraints. To wit, setting $c_{\mathrm{min}} = 0$ and $c_{\mathrm{max}} 
= +\infty$ will result in explicit non-negative constraints on the 
nodal concentrations.

\begin{remark}
  Note that the continuous problem, in general, cannot \emph{always} 
  be written as an optimization problem. This is certainly the case 
  with respect to ADR equation \cite{2010_Nakshatrala_Valocchi_IJCM_v7_p559_p572}. 
  Moreover, in the continuous setting, the non-negative and local species 
  balance constraints are satisfied trivially. Therefore, the Lagrange 
  multipliers are zero in the continuous setting (if one can write the 
  problem as an optimization problem). 
  The violations of the non-negative and local species balance 
  constraints are only in the discrete setting. This is the reason 
  why one needs to obtain the discrete form before one can fix 
  the deficiencies in solving the discrete equations.
\end{remark}

In the next two sections, we illustrate the performance of the 
proposed computational framework for advection-dominated ADR problems 
and transport-controlled irreversible fast bimolecular reactions. 
In all the numerical simulations reported in this paper, the constrained 
optimization problem is solved using the \textsf{MATLAB's} \cite{MATLAB_2015a} 
built-in function handler \textsf{quadprog}, which has a robust solver 
based on an interior-point numerical algorithm presented in References 
\cite{2004_Gould_Toint_MPSB_v100_p95_p132,1992_Mehrotra_SIAMJO_v2_p575_p601,
1996_Gondzio_COA_v6_p137_p156}. 
One can alternatively employ the open-source optimization solvers 
such as \texttt{COBYLA}, \texttt{SLSQP}, \texttt{L-BFGS-B}, or 
\texttt{TNC} from \textsf{SciPy} \cite{jones2014scipy}. 
The tolerance in the stopping criterion for solving convex 
quadratic programming problems is taken as $100 \epsilon_
{\mathrm{mach}}$, where $\epsilon_{\mathrm{mach}} \approx 2.22 
\times 10^{-16}$ is the machine precision for a 64-bit machine.

There are various approaches to numerically solve 
transient diffusion-type systems. It is desirable 
to have a numerical strategy that converts and 
utilizes the solvers for steady-state diffusion-type 
equations to solve transient systems. It has been 
recently shown that the method of horizontal lines 
using the backward Euler time-stepping scheme is 
one of the viable approaches to respect maximum 
principles and the non-negative constraint in the 
discrete setting 
\cite{2013_Nakshatrala_Nagarajan_Shabouei_arXiv}.
The method of horizontal lines discretizes the time domain 
first, and thereby converts the transient ADR equations at 
each time-level into a system of governing equations similar 
to \eqref{Eqn:NN_AD_GE_BE}--\eqref{Eqn:NN_AD_GE_NBC}. This 
methodology, thus, helps us to use the computational framework 
provided in Section \ref{Sec:NN_AD_Proposed}. One can employ 
a numerical procedure similar to Algorithm $1$ provided in 
Reference \cite{2013_Nakshatrala_Nagarajan_Shabouei_arXiv} 
to advance the numerical solution over the time. Numerical 
results for transient systems are presented in Section 
\ref{Sec:NN_AD_Fast_Reactions}.

\section{NUMERICAL $h$-CONVERGENCE AND BENCHMARK PROBLEMS}
\label{Sec:NN_AD_NumericalConvergence}
We shall employ a popular problem from the literature, 
which is commonly used to assess the accuracy of numerical 
formulations for advective-diffusive systems (e.g., see 
\cite{1998_Franca_Nesliturk_Stynes_CMAME_v166_p35_p49,
2009_Hsieh_Yang_CMAME_v199_p183_p196}). The test problem 
is constructed using the method of manufactured solutions. 
The computational domain is a bi-unit square: $\Omega = 
(0,1) \times (0,1)$. The advection velocity vector field 
is taken as $\mathbf{v}(\mathbf{x}) = \hat{\mathbf{e}}_y$, 
where $\hat{\mathbf{e}}_y$ is the unit vector along the 
$y$-direction. The scalar diffusivity is denoted by 
$D(\mathbf{x})$. The concentration field is taken as 
follows:
\begin{align}
  \label{Eqn:NumConv_ConcField_AnalyticExp}
  c(x,y)= \frac{\sin(\pi x)}{e^{m_2 - m_1} - 1} \left(e^{m_2 - m_1} 
  e^{m_1 y} - e^{m_2 y} \right)
\end{align}
where the constants $m_1$ and $m_2$ are given 
in terms of the scalar diffusivity:
\begin{subequations}
  \begin{align}
    \label{Eqn:m1_hConv}
    m_1 = \frac{1 - \sqrt{1 + 4 \pi^2 D^2}}{2D} \\
    \label{Eqn:m2_hConv}
    m_2 = \frac{1 + \sqrt{1 + 4 \pi^2 D^2}}{2D}
  \end{align}
\end{subequations}
We have taken $D(\mathbf{x}) = 10^{-2}$ in our numerical 
simulations. This choice is arbitrary, and is primarily 
motivated to check whether the proposed framework gives 
stable, reliable, and accurate numerical results for 
advection-dominated problems. For the chosen value 
of the diffusivity, the solution 
\eqref{Eqn:NumConv_ConcField_AnalyticExp} exhibits 
steep gradients near the boundary of the domain. A 
pictorial description of the boundary value problem 
is provided by Figure \ref{Fig:2D_Numerical_Convergence}. 

Numerical solutions for the concentration and the 
flux vector are obtained by prescribing Dirichlet 
boundary conditions on all the four sides of the 
computational domain. These conditions are enforced 
strongly and are given as follows: 
\begin{align}
  \label{Eqn:hConv_BC}
  &c(\mathbf{x}) = \; 
  \begin{cases}
    \sin(\pi x) \quad & \mathrm{for} \; y = 0 \\
    0 \quad & \mathrm{for} \; x = 0 \; \mathrm{or} 
    \; x = 1 \; \mathrm{or} \; y = 1 
  \end{cases} 
\end{align}
Using equation \eqref{Eqn:NumConv_ConcField_AnalyticExp}, 
one can calculate the corresponding flux vector and 
volumetric source needed for the convergence analysis. 

\subsection{Convergence analysis for $D(x,y) = 10^{-2}$}
\label{SubSec:hConvergence_Analysis_AD}
Herein, we will discuss the performance of negatively 
stabilized streamline diffusion LSFEM with and without 
LSB constraints. In case of unconstrained setting, we 
also quantify the errors incurred in satisfying LSB 
and GSB. Numerical simulations are performed using 
a series of hierarchical structured meshes based on 
three-node triangular (T3) and four-node quadrilateral 
(Q4) elements with XSeed and YSeed ranging from 11 to 81. 
Figure \ref{Fig:2D_Numerical_Convergence_HierarchicalMeshes} 
provides the typical computational meshes used in the numerical 
$h$-convergence analysis.

The weights for the primitive and negatively stabilized 
streamline diffusion LSFEMs are taken to be of LS Type-1 
(i.e., $\mathbf{A}(\mathbf{x}) = 
\mathbf{I}$ and $\beta(\mathbf{x}) = 1$). The element stabilization 
parameters for negatively stabilized streamline diffusion LSFEM are 
taken as $\delta_o = 0.01$ and $\tau_o = 0.01$. The 
convergence of the proposed computational framework with respect to 
$L_2$-norm and $H^1$-semi-norm is illustrated in Figure 
\ref{Fig:2D_NumConvDiff1by100_ConcFluxT3Q4_NoAndLSB_Constraints}. 
From this figure, one can notice that near optimal convergence 
rates are achieved for the concentration field in both $L_2$-norm and 
$H^1$-semi-norm for unconstrained negatively stabilized streamline 
diffusion formulation. For the flux vector, near optimal convergence 
rate is obtained in $L_2$-norm but not in 
$H^1$-semi-norm. This is because of the steep gradients in the concentration 
field at the boundary $y = 1$, which is due to the small value for the 
diffusivity. Enforcing LSB constraints considerably improves the $H^1$-semi-norm 
convergence rate for the flux vector. However, for the flux variables, 
there is a slight decrease in $L_2$-norm convergence rate as compared 
to the unconstrained negatively stabilized streamline diffusion LSFEM. 
Similar decrease in convergence rates of $L_2$-norm and $H^1$-semi-norm 
for the concentration has been observed. This can be attributed to the 
fact that LSB constraints improve the accuracy of the flux vector 
inside the boundary layers but has little effect away from it. 

\begin{remark}
  It should be noted that the convergence rates reported in Figure 
  \ref{Fig:2D_NumConvDiff1by100_ConcFluxT3Q4_NoAndLSB_Constraints}   
  for the unconstrained negatively stabilized streamline diffusion LSFEM 
  are in accordance with the mathematical analysis provided by Kopteva 
  \cite{2004_Kopteva_CMAME_v193_p4875_p4889} and Stynes \cite{2013_Stynes_arXiv}. 
  These results are obtained for singularly perturbed advection-diffusion 
  equation based on a class of unconstrained streamline diffusion finite 
  element formulations. Kopteva \cite{2004_Kopteva_CMAME_v193_p4875_p4889} 
  shows that one can get at best first-order convergence inside 
  boundary and characteristic layers even on special meshes.
\end{remark}

From Figure \ref{Fig:2D_NumConvDiff1by100_ConcFluxT3Q4_NoAndLSB_Constraints} 
one can also conclude that the Q4 element performs better than the T3 element. 
These trends in the convergence rates for different meshes are due to the 
fact that higher-order derivatives (e.g., $\mathrm{div} [ \mathrm{grad} 
[c(\mathbf{x})]$) in the stabilization terms for negatively stabilized 
streamline diffusion LSFEM vanish for T3 element. But these stabilization 
terms are non-zero for a Q4 element. The reason is that the shape functions 
for a T3 element are affine while that of a Q4 element are bilinear. 

Another important aspect of this numerical $h$-convergence study is to 
quantify the errors incurred in satisfying LSB and GSB for unconstrained 
LSFEMs. The contours of the error distribution in LSB and the Lagrange 
multipliers enforcing the LSB constraints are shown in Figure 
\ref{Fig:2D_NumConvDiff1by100_LocMass_LambdaLMBT3MeshAllFourLSFEM}. It 
is apparent that errors incurred in satisfying LSB are smaller under 
Q4 meshes than under T3 meshes. 
The decrease in $\epsilon_{\text{\tiny {MaxAbsLSB}}}$ and $\epsilon_
{\text{\tiny {AbsGSB}}}$ on $h$-refinement is shown in Figure 
\ref{Fig:2D_NumConvDiff1by100_ErrorsLMBandGMB_AllFourLSFEMs}.
From this figure, one can notice that the errors in LSB and GSB 
for a Q4 mesh are lesser than that of a T3 mesh. On $h$-refinement, 
the decrease in $\epsilon_{\text{\tiny {MaxAbsLSB}}}$ and $\epsilon_
{\text{\tiny {AbsGSB}}}$ is slow and not close to machine precision.

Finally, the computational cost of the unconstrained and 
constrained LSFEMs for both T3 and Q4 meshes are shown in Figures 
\ref{Fig:2D_NumConvDiff1by100_TicTocNoConsT3Q4_AllFourLSFEMs} 
and \ref{Fig:2D_NumConvDiff1by100_TicTocLocMassT3Q4_AllFourLSFEMs}. 
It is clear that the computational cost associated with a Q4 mesh is 
higher than that of a T3 mesh. This can be again be attributed to 
the non-vanishing stabilization terms (e.g., $\mathrm{div}[\mathrm{grad} 
[c(\mathbf{x})]$) in the negatively stabilized streamline diffusion 
LSFEM for Q4 meshes. 
For constrained LSFEMs, the maximum additional computational 
cost (for both LSFEMs) did not exceed $15\%$, which has been 
tested on a hierarchy of meshes.    

\subsection{Thermal boundary layer problem}
\label{SubSec:2D_Thermal_BoundaryLayer_Problem}
This benchmark problem has wide practical applications 
in the areas of heat and mass transfer. Herein, we shall 
use this benchmark problem to study the performance of 
unconstrained and constrained LSFEM formulations in 
capturing steep gradients near the boundary for 
advection-dominated scenarios. 
Consider a rectangular domain $\Omega = \{ (x,y) \in [0,1] 
\times [0,0.5] \} $ with velocity field $\mathbf{v}(x,y) = 
2 y \mathbf{\hat{e}}_x$, where $\mathbf{\hat{e}}_x$ is the 
unit vector along the $x$-direction. 
The volumetric source is assumed to be homogeneous (i.e., 
$f(x,y) = 0$), and the scalar diffusivity is taken to be 
$D(x,y) = 10^{-4}$. The boundary conditions are:
\begin{align} 
  c(x,y) &= 
  \begin{cases}
    0 \quad \mathrm{for} \; 0 < x \leq 1 \; \mathrm{and} \;  y = 0 \\
    2y \; \; \; \mathrm{for} \; x = 1 \; \mathrm{and} \; 0 \leq y \leq 0.5 \\
    1 \quad \mathrm{for} \; 0 \leq x \leq 1 \; \mathrm{and} \;  y = 1 \\
    1 \quad \mathrm{for} \; x = 0 \; \mathrm{and} \; 0 \leq y \leq 0.5 \\
  \end{cases}
\end{align}
A pictorial description of the boundary value problem is 
provided in Figure \ref{Fig:2D_ThermalBoundaryLayer}. The 
weights are taken to be that of LS Type-1 (see equations 
\eqref{Eqn:Weight_A_Tensor} and \eqref{Eqn:Weight_beta_Scalar}).
The element-level stabilization parameters for negatively 
stabilized streamline diffusion LSFEM are taken to be 
$\delta_o = 0.01$ and $\tau_o = 0.001$.
Numerical simulations are performed using four-node quadrilateral 
mesh with XSeed = 41 and YSeed = 21. The element P\'eclet number 
will then be $\mathbb{P}\mathrm{e}_{h} = 125$. 
The obtained concentration contours are shown in Figure 
\ref{Fig:2D_ThermalBoundaryLayer_Results}. It is evident 
from these figures that numerical solution obtained from 
the primitive LSFEM contains node-to-node spurious 
oscillations. These oscillations did not reduce even 
after enforcing the LSB and NN constraints. But the 
negatively stabilized streamline diffusion LSFEM is 
able to capture the steep gradients near the boundary 
without producing spurious oscillations. The errors 
incurred in satisfying LSB for \emph{unconstrained} 
LSFEM formulations are shown in Figure 
\ref{Fig:2D_ThermalBoundaryLayer_LSB}.

\section{TRANSPORT-CONTROLLED BIMOLECULAR CHEMICAL REACTIONS}
\label{Sec:NN_AD_Fast_Reactions}
In this section, we shall apply the proposed mixed LSFEM-based 
computational framework to study transport-controlled bimolecular 
chemical reactions. Specifically, we are interested in the spatial 
distribution, plume formation, and chaotic mixing of chemical 
species at high P\'{e}clet numbers. To this end, consider the 
following irreversible bimolecular chemical reaction: 
\begin{align}
  \label{Eqn:BimolFastReaction}
    n_{A} \, A \, + \, n_{B} \, B \longrightarrow n_{C} \, C
\end{align}
where $A$, $B$, and $C$ are the species involved in the 
chemical reaction; $n_A$, $n_B$, and $n_C$ are their 
respective (positive) stoichiometric coefficients. 
The fate of these chemical species are governed by 
the following coupled advective-diffusive-reactive 
system:
\begin{subequations}
  \label{Eqn:ADR_for_A_B_C}
  \begin{align}
    \label{Eqn:ADR_for_A}
    &\frac{\partial c_A}{\partial t} + \mathrm{div}[\mathbf{v} 
    c_A - \mathbf{D}(\mathbf{x},t) \, \mathrm{grad}[c_A]] = f_A
    (\mathbf{x},t) - n_{\small{A}} \, r(\mathbf{x},t,c_A,c_B,c_C) 
    \quad \mathrm{in} \; \Omega \times ]0, \mathcal{I}[ \\
    \label{Eqn:ADR_for_B} 
    &\frac{\partial c_B}{\partial t} + \mathrm{div}[\mathbf{v} 
    c_B - \mathbf{D}(\mathbf{x},t) \, \mathrm{grad}[c_B]] = f_B
    (\mathbf{x},t) - n_{\small{B}} \, r(\mathbf{x},t,c_A,c_B,c_C) 
    \quad \mathrm{in} \; \Omega \times ]0, \mathcal{I}[ \\
    \label{Eqn:ADR_for_C} 
    &\frac{\partial c_C}{\partial t} + \mathrm{div}[\mathbf{v} 
    c_C - \mathbf{D}(\mathbf{x},t) \, \mathrm{grad}[c_C]] = f_C
    (\mathbf{x},t) + n_{\small{C}} \, r(\mathbf{x},t,c_A,c_B,c_C) 
    \quad \mathrm{in} \; \Omega \times ]0, \mathcal{I}[ \\
    \label{Eqn:ADR_for_Dirchlet_ABC}
    &c_i(\mathbf{x},t) = c^{\mathrm{p}}_i(\mathbf{x},t) \quad 
    \mathrm{on} \; \Gamma^{c}_{i} \times ]0, \mathcal{I}[ \\
    \label{Eqn:ADR_for_Neumann_ABC}
    &\left(\left( \frac{1 - \mathrm{Sign}[\mathbf{v} \bullet 
    \widehat{\mathbf{n}}]}{2} \right) \mathbf{v}(\mathbf{x},t) 
    c_i(\mathbf{x},t) - \mathbf{D}(\mathbf{x},t) \mathrm{grad}
    [c_i(\mathbf{x},t)] \right) \bullet \widehat{\mathbf{n}}
    (\mathbf{x}) = h^{\mathrm{p}}_i(\mathbf{x},t) \quad 
    \mathrm{on} \; \Gamma^{q}_{i} \times ]0, \mathcal{I}[ \\
    \label{Eqn:ADR_for_IC_ABC}
    &c_i(\mathbf{x},t=0) = c^{0}_i(\mathbf{x}) \quad 
    \mathrm{in} \; \Omega
  \end{align}
\end{subequations}
where $i = A, \, B$, and $C$. $\mathbf{v}(\mathbf{x},t)$ is 
the advection velocity vector field, $f_i(\mathbf{x},t)$ 
constitutes the non-reactive volumetric source, $c^{\mathrm{p}}
_i(\mathbf{x},t)$ is the Dirichlet boundary condition, and $h^{\mathrm{p}}
_i(\mathbf{x},t)$ is the Neumann boundary condition of the $i$-th chemical 
species. $r(\mathbf{x},t,c_A,c_B,c_C)$ is the bimolecular chemical reaction 
rate, which is a non-linear function of the concentrations of the 
chemical species involved in the reaction. $c^{0}_i(\mathbf{x})$ is the 
initial condition of $i$-th chemical species. $t \in [0,\mathcal{I}]$ 
denote the time, where $\mathcal{I}$ is the total time of interest. The 
coupled governing equations \eqref{Eqn:ADR_for_A}--\eqref{Eqn:ADR_for_Neumann_ABC} 
can be converted to a set of uncoupled advection-diffusion equations 
using the following linear algebraic transformation:
\begin{subequations}
  \begin{align}
    \label{Eqn:ADR_Def_of_F}
    c_F &:= c_A + \left( \frac{n_A}{n_C} \right) c_C \\ 
    \label{Eqn:ADR_Def_of_G}
    c_G &:= c_B + \left( \frac{n_B}{n_C} \right) c_C 
  \end{align}
\end{subequations}
As we are interested in \emph{fast} bimolecular chemical 
reactions, it is acceptable to assume that the chemical 
species $A$ and $B$ cannot co-exist at any given location 
$\mathbf{x}$ and time $t$. Hence, $c_A$, $c_B$, and $c_C$ 
can be evaluated as follows:
\begin{subequations}
  \label{Eqn:ADR_Fast_A_B_C}
  \begin{align}
    \label{Eqn:ADR_Fast_A}
    &c_A(\mathbf{x},t) = \mathrm{max} \left[c_F(\mathbf{x},t) 
    - \left(\frac{n_A}{n_B}\right) c_G(\mathbf{x},t), \, 0 
    \right] \\
    \label{Eqn:ADR_Fast_B}
    &c_B(\mathbf{x},t) = \mathrm{max}\left[c_G(\mathbf{x},t) - 
    \left( \frac{n_B}{n_A} \right) c_F(\mathbf{x},t), \, 0 
    \right] \\
    \label{Eqn:ADR_Fast_C}
    &c_C(\mathbf{x},t) = \left( \frac{n_C}{n_A} \right) \; 
    \big(c_F(\mathbf{x},t) - c_A(\mathbf{x},t) \big)
  \end{align}
\end{subequations}

In Reference \cite{2013_Nakshatrala_Mudunuru_Valocchi_JCP_v253_p278_p307}, 
a similar mathematical model has been studied in the context of maximum 
principles and the non-negative constraint. However, the study has 
neglected the advection, and did not address local and global species 
balance. These aspects are very important and cannot be neglected in 
the numerical simulations of chemically reacting systems. In particular, 
advection can play a predominant role in the study of bioremediation 
\cite{1986_Borden_Bedient_AWR_v22_p_1983_p1990}, transverse mixing-controlled 
chemical reactions in hydro-geological media \cite{2008_Willingham_etal_EST_v42_p3185_p3193}, 
and contaminant degradation problems \cite{Dentz_Borgne_Englert_Bijeljic_JHC_2011_v120_p1}. 
This paper precisely addresses such problems in which advection is 
dominant, and satisfying species balance at both local and global 
levels is extremely important. 

\begin{remark}
  Non-linear chemical dynamics is a huge field with 
  various interesting artifacts, which include chaos 
  and limit cycles \cite{Neufeld_Garcia,Epstein_Pojman,
    Erdi_Toth}. Non-linear reactions will bring many 
  additional complications, which need to be addressed 
  systematically. Our approach can handle zeroth-order 
  and first-order kinetics, as these two cases do not 
  bring additional challenges. Other reaction kinetic 
  models need to be addressed case-by-case. A general 
  treatment of non-linear chemical dynamics is not 
  trivial, and is beyond the scope of this paper.
\end{remark}

Herein, we perform numerical simulations for highly 
spatially varying advection velocity fields and 
time-periodic flows. See Reference \cite{Neufeld_Garcia} 
for a discussion on time-periodic flows. For such problems 
in 2D, the following quantity is of considerable 
importance, which is referred as the position 
weighted second moment of the product $C$ 
concentration:
\begin{align}
  \label{Eqn:Conc_Width_Second_Moment}
  \Theta_C^2(t) = \frac{\displaystyle \int \limits_{\Omega} 
  (y - y_0)^2 c_C(\mathbf{x},t) \; \mathrm{d} \Omega}{
  \displaystyle \int \limits_{\Omega} c_C(\mathbf{x},t) 
  \; \mathrm{d} \Omega}
\end{align}
where $y_0$ is the location of a convenient reference 
horizontal line. In our numerical simulations, we have 
taken $y_0$ to be the $y$-coordinate of the start of 
the formation of product $C$. Since $c_C(\mathbf{x},t) 
\geq 0$, $\Theta^{2}_C(t)$ is a non-negative quantity. 
In subsequent sections, we study the utility of this 
quantity as \emph{a posteriori} criterion to assess 
numerical accuracy. We also analyze the variation of 
$\Theta^{2}_C$ with respect to $\mathbb{P}\mathrm{e}_{h}$. 
We also present the numerical results that shed light on 
the impact of advection on the formation of the product $C$. 
In all our numerical simulations, we have taken the weights 
in primitive and negatively stabilized streamline diffusion 
LSFEMs to be that of LS Type-1 (i.e., $\mathbf{A}(\mathbf{x}) 
= \mathbf{I}$ and $\beta(\mathbf{x}) = 1$).

\begin{remark}
  In the literature, to study mixing processes due to advection, 
  spectral methods \cite{2002_Adrover_etal_CCM_v26_p125_p139}, 
  pseudospectral methods \cite{2009_Tsang_PRE_v80_p026305}, and 
  model reduction methods \cite{Neufeld_Garcia} are commonly 
  employed. However, such methods are limited to time-periodic 
  flows, periodic initial and boundary conditions, simple geometries, 
  and homogeneous isotropic diffusivity. Extending these methods to 
  complicated geometries, general initial and boundary conditions, 
  complicated advection velocity fields, and heterogeneous isotropic 
  and anisotropic diffusivity is not trivial and may not even be 
  possible. Moreover, these methods do not guarantee the satisfaction 
  of non-negativity, DMPs, LSB, and GSB. The proposed computational 
  framework is aimed at filling this lacuna. 
\end{remark}

\subsection{One-dimensional steady-state analysis of product formation in fast reactions}
\label{SubSec:1D_SS_ADR_Bimol}
Analysis is performed for two different advection velocities: $v = 
0.25$ and $v = 1.0$. Diffusivity is assumed to be $2.5 \times 10^{-3}$. 
The stoichiometric coefficients are assumed to be: $n_A = 2$, $n_B = 1$, 
and $n_C = 1$. Numerical simulations are performed for two different 
cases as described below.
\subsubsection{\textbf{Case \#1}}
A pictorial description of the boundary value problem is shown 
in Figure \ref{Fig:1D_Bimolecular_Type1_Type2}. The objective of this 
case study is to analyze whether the proposed negatively stabilized 
streamline diffusion LSFEM can produce physically meaningful values 
for $c_i(x)$ on coarse meshes. Based on the linear algebraic transformation 
given by equations \eqref{Eqn:ADR_Def_of_F}--\eqref{Eqn:ADR_Def_of_G}, 
the analytical solution for invariants $F$ and $G$ can be written as 
follows:
\begin{subequations}
  \begin{align}
    \label{Eqn:ADR_1D_Bimolecular_invF_Type1}
    c_F(x) &= \left( 1 - \frac{1 - \exp(vx/D)}{1 - 
    \exp(v/D)} \right) \\ 
    \label{Eqn:ADR_1D_Bimolecular_invG_Type1}
    c_G(x) &=   \frac{f_G}{v} \left( x - \frac{1 - 
    \exp(vx/D)}{1 - \exp(v/D)} \right)
  \end{align}
\end{subequations}
Using equations \eqref{Eqn:ADR_Fast_A}--\eqref{Eqn:ADR_Fast_C}, 
one can obtain the analytical solution for product $C$. 

For the numerical solution, we have taken XSeed = 11. The 
element stabilization parameters for negatively stabilized 
streamline diffusion LSFEM are taken as $\delta_o = 0.08$ 
and $\tau_o = 0.04$ when $v = 0.25$. For $v = 1.0$, $\delta
_o$ and $\tau_o$, are assumed to equal to 0.083 and 0.0121, 
respectively. The analytical and numerical solutions are 
compared in Figure \ref{Fig:1D_Bimolecular_Type1_concC_Unconstrained}. 
As per this figure, the primitive LSFEM produces node-to-node 
oscillations near the boundaries of the domain. Furthermore, 
its numerical solution considerably deviates from the analytical 
solution in the entire domain. For $\mathbb{P}\mathrm{e}_{h} 
= 5$ and $\mathbb{P}\mathrm{e}_{h} = 20$, the negative value 
for the concentration is as low as $-1.25$ and  $-0.47$. 
On the other hand, the negatively stabilized streamline 
diffusion LSFEM is able capture the analytical solution 
profile in the entire domain without producing negative 
values in the concentration field.

\subsubsection{\textbf{Case \#2}}
A pictorial description of the boundary value problem is 
provided in Figure \ref{Fig:1D_Bimolecular_Type1_Type2}. 
The objective of this case study is to examine whether 
the proposed LSFEM can capture steep gradients in the 
solution near the boundary. The analytical solution 
for the invariants $F$ and $G$ can be written as 
follows:
\begin{subequations}
  \begin{align}
    \label{Eqn:ADR_1D_Bimolecular_invF_Type2}
    c_F(x) &= \left( 1 - \frac{1 - \exp(vx/D)}{1 - 
    \exp(v/D)} \right) \\ 
    \label{Eqn:ADR_1D_Bimolecular_invG_Type2}
    c_G(x) &= \left( \frac{1 - \exp(vx/D)}{1 - 
    \exp(v/D)} \right)
  \end{align}
\end{subequations}
Figure \ref{Fig:1D_Bimolecular_Type2OtherSet_concABC_Unconstrained} 
compares the obtained the numerical solution with the analytical 
solution. The negatively stabilized streamline diffusion LSFEM 
is able to accurately capture the steep gradients near the boundary.

\subsection{Steady-state plume formation from boundary in a reaction tank}
\label{SubSec:2D_SS_RxnTank}
A pictorial description of the boundary value problem is 
provided in Figure \ref{Fig:2D_Bimolecular_ScalDiff_Plume}. The 
computational domain is a rectangle with $L_x = 2$ and $L_y = 1$. 
Dirichlet boundary conditions with $c^{\mathrm{p}}_{A} = c^{\mathrm{p}}
_{B} = 1$ are specified on the left side of the domain. Elsewhere, 
$c^{\mathrm{p}}_{i}(\mathbf{x})$ is taken to be zero for all the chemical 
species involved in the bimolecular reaction. The non-reactive 
volumetric source is assumed to be zero in the entire domain for all the chemical 
species. The stoichiometric coefficients are taken as $n_A = 1$, 
$n_B = 1$ and $n_C = 1$. The advection velocity field is defined 
through the following multi-mode stream function 
\cite{2013_Nakshatrala_Mudunuru_Valocchi_JCP_v253_p278_p307}:
\begin{align}
  \label{Eqn:Plume_Subsurface_Stream_Function}
  \psi(\mathbf{x}) &= -\mathrm{y} - \sum_{k=1}^{3} A_k \cos 
  \left(\frac{{p}_k \pi \mathrm{x}}{L_x} - \frac{\pi}{2}\right) 
  \sin \left( \frac{q_k \pi \mathrm{y}}{L_y}\right) 
\end{align}
where $\mathbf{x} = (x,y)$, $(p_1, p_2, p_3) = (4, 5, 10)$, 
$(q_1, q_2, q_3) = (1, 5, 10)$, and $(A_1, A_2, A_3) = 
(0.08, 0.02, 0.01)$. The corresponding components of 
the advection velocity can be written as follows:
\begin{subequations}
  \begin{align}
    \label{Eqn:Plume_Subsurface_VelX}
    \mathrm{v}_{x}(\mathbf{x}) &= -\frac{\partial \psi}{\partial 
    \mathrm{y}} = 1 + \sum_{k=1}^{3} A_k \frac{q_k \pi}{L_y} \cos 
    \left(\frac{p_k \pi \mathrm{x}}{L_x} - \frac{\pi}{2} \right) 
    \cos\left(\frac{q_k \pi \mathrm{y}}{L_y}\right) \\
    \label{Eqn:Plume_Subsurface_VelY}
    \mathrm{v}_{y}(\mathbf{x}) &= +\frac{\partial \psi}{\partial 
    \mathrm{x}} = \sum_{k=1}^{3} A_k \frac{p_k \pi}{L_x} \sin 
    \left(\frac{p_k \pi \mathrm{x}}{L_x} - \frac{\pi}{2}\right) 
    \sin\left(\frac{q_k \pi \mathrm{y}}{L_y}\right)
  \end{align}
\end{subequations}
It is easy to check that $\mathrm{div}[\mathbf{v}(\mathbf{x})] 
= 0$. The contours of the stream function and the corresponding 
advection velocity vector field are shown in Figure 
\ref{Fig:2D_Bimolecular_ScalDiff_Plume}. Numerical 
simulations are performed using the following two 
different types of diffusivities:
\begin{itemize}
  \item \emph{Type \#1}: $D(\mathbf{x}) = 10^{-2}$
  \item \emph{Type \#2}: $\mathbf{D}(\mathbf{x}) = 
    \mathbf{R} \mathbf{D}_{0} \mathbf{R}^{\mathrm{T}}$, 
    where $\mathbf{R}$ and $\mathbf{D}_{0}$ are given 
    as follows:
    \begin{subequations}
      \begin{align}
        \label{Eqn:Rotation_Tensor}
        \mathbf{R} &= \left(\begin{array}{cc}
        \cos(\theta) & -\sin(\theta) \\
        \sin(\theta) & \cos(\theta) \\
        \end{array}\right) \\
        \label{Eqn:LePotier_AnisoDiff}
        \mathbf{D}_0(\mathbf{x}) &= \omega_0
        \left(\begin{array}{cc}
        y^2_* + \omega_2 x^2_* & - (1 - \omega_2) 
        x_* y_* \\
        -(1 - \omega_2) x_* y_* & \omega_2 
        y^2_* + x^2_* \\
        \end{array}\right)
      \end{align}
    \end{subequations}
    where $x_* = x + \omega_1$ and $y_* = y + \omega_1$. The 
    parameters $\theta$, $\omega_0$, $\omega_1$, and $\omega_2$ 
    are equal to $\frac{\pi}{6}$, $1.0$, $10^{-3}$, and $10^{-3}$. 
    Correspondingly, the eigenvalues of $\mathbf{D}(\mathbf{x})$ 
    are $\omega_0(x^2_* + y^2_*)$ and $\omega_0 \omega_2 \left(
    x^2_* + y^2_* \right)$. The contrast/anisotropic ratio of the 
    media (which is the ratio of maximum to minimum eigenvalue) is 
    as high as $10^3$.
\end{itemize} 
Herein, we employed a structured mesh based on Q4 elements. 
Numerical simulations are performed with varying mesh sizes 
and polynomial orders ($p = 1, 2, 3$) to demonstrate the 
pros and cons of various unconstrained and constrained LSFEMs. 
The stabilization parameters are taken as $\delta_o = \tau_o 
= 10^{-3}$ and $\delta_2 = \tau_2 = 10^{-4}$. The contours of 
the concentration of the product $C$ are shown in Figures 
\ref{Fig:2D_Bimolecular_ScalDiff_ConcC_PrimNoCons}--\ref{Fig:2D_Bimolecular_AnisoDiff_ConcC_NegStabStrDiff} 
for both the primitive and negatively stabilized streamline diffusion 
LSFEMs. The white patches in the figures denote the regions in 
which the non-negative constraint has been violated. The variation 
of $\Theta^2_C$ with respect to XSeed and $Pe_L$ are shown in Figures 
\ref{Fig:2D_Bimolecular_ScalDiff_ConcC_NegStabStrDiff_NSSD}--\ref{Fig:2D_BiMolScalDiff_ConcC_NSSD}.
From these figures, the following inferences can be drawn:
\begin{enumerate}[(i)]
  \item It is clear that both low-order and higher-order 
    polynomials violate the non-negative constraint and 
    DMPs under unconstrained formulations. Moreover, mesh 
    refinement and polynomial refinement do not seem to 
    reduce the amount of violated region for DMP constraints.
  \item The proposed framework based on $p = 1$ is 
    able to satisfy all the desired properties, and is 
    able to predict physically meaningful values for 
    the concentration and the flux. 
  \item The primitive LSFEM and the unconstrained negatively 
    stabilized streamline diffusion LSFEM give unphysical 
    values for the position weighted second moment of the 
    product $C$ (i.e., $\Theta_C^2$). 
    On the other hand, the proposed computational 
    framework is able to accurately describe the 
    variation of $\Theta_C^2$ with respect to mesh 
    refinement. In addition, the numerical values 
    for $\Theta_C^2$ reaches a plateau on $h$-refinement, 
    which indicates convergence. However, this is not 
    observed with the unconstrained primitive and 
    negatively stabilized streamline diffusion LSFEMs.
\end{enumerate}
Finally, it should be emphasized that placing explicit 
non-negative constraints on the nodal concentrations 
does not ensure non-negativity of the concentration 
in the entire computational domain. This is due to 
the fact that higher-order shape functions change 
their sign within an element 
\cite{Payette_Nakshatrala_Reddy_IJNME_2012}.

\subsection{Transient analysis of non-chaotic and chaotic vortex stirred mixing in a reaction tank}
\label{SubSec:2D_Trans_Mixing}
Figure \ref{Fig:2D_Bimolecular_ScalDiff_Mixing_Type1} provides 
a pictorial description of the problem with appropriate initial 
and boundary conditions. The computational domain is a square 
with $L_x = L_y = 1$. For all chemical species, zero flux boundary 
condition is prescribed on the entire boundary. The non-reactive 
volumetric source is zero in the entire domain for all the chemical 
species $A$, $B$, and $C$. Reactant $A$ is placed at the center of 
vortices, which are positioned at $(0.25,0.75)$ and $(0.75,0.25)$. 
The width of the square slug $A$ is equal to 0.25. Within this 
slug, $c_A(\mathbf{x},t=0) = 8$. Elsewhere, the initial condition 
for $A$ is equal to zero. Correspondingly, the initial condition 
for reactant $B$ is zero in these two square regions centered at 
$(0.25,0.75)$ and $(0.75,0.25)$. Elsewhere, $c_B(\mathbf{x},t=0) 
= 1.5$. The stoichiometric coefficients are taken as $n_A = 1$, 
$n_B = 1$, and $n_C = 1$. The total time of interest is taken as 
$\mathcal{I} = 5$. We assume scalar diffusivity to be $D = 10^{-2}$. 
The stabilization parameters are taken as $\delta_o = \tau_o = 10^{-3}$ 
and $\delta_1 = \tau_1 = 10^{-4}$. For advection velocity, we employ 
the following vortex-based flow field:
\begin{itemize}
  \item \emph{Type \#1}: Non-chaotic vortex-based advection velocity 
    field
    \begin{align}
      \label{Eqn:TransMixing_VortexFlowModel}
      \mathbf{v}(\mathbf{x}) = \cos(2 \pi y) \mathbf{\hat{e}}_x
      + \cos(2 \pi x) \mathbf{\hat{e}}_y 
    \end{align}  
  \item \emph{Type \#2}: Chaotic vortex-based advection velocity 
    field
    \begin{align}
      \label{Eqn:TransMixing_ChaoticVortexFlowModel1}
      \mathrm{v}_{x}(\mathbf{x},t) = 
      \begin{cases}
        \cos(2 \pi y) + v_o \sin(2 \pi y) &\quad \mathrm{if} \; 
        \nu T \leq t < \left( \nu + \frac{1}{2} \right) T  \\
        \cos(2 \pi y) &\quad \mathrm{if} \; \left( \nu + \frac{1}{2} 
        \right) T \leq t < \left( \nu + 1 \right) T
      \end{cases}
    \end{align}
    \begin{align} 
      \label{Eqn:TransMixing_ChaoticVortexFlowModel2}
      \mathrm{v}_{y}(\mathbf{x},t) = 
      \begin{cases}
        \cos(2 \pi x) &\quad \mathrm{if} \; \nu T \leq t < 
        \left( \nu + \frac{1}{2} \right) T \\
        \cos(2 \pi x) + v_o \sin(2 \pi x) &\quad \mathrm{if} \;
        \left( \nu + \frac{1}{2} \right) T \leq t < \left( 
        \nu + 1 \right) T
      \end{cases}
    \end{align}
\end{itemize}
where $\nu = 0, 1, 2, \cdots$ \cite{2002_Adrover_etal_CCM_v26_p125_p139,
2009_Tsang_PRE_v80_p026305}. $T$ denotes the period of the 
motion of the flow field. $v_{o}$ is an a-priorly chosen chaotic 
flow perturbation parameter. Herein, we choose $T = 0.8$ and 
$v_{o} = 1.0$ \cite{2004_Cox_PDNP_v199_p369_p386}.

Figures 
\ref{Fig:2D_Bimolecular_ScalDiff_FirstTimeStep_ConcC}--\ref{Fig:2D_Bimolecular_ScalDiff_Mixing_ConcC1} 
provide the concentration profiles of unconstrained and constrained 
negatively stabilized streamline diffusion LSFEM with NN constraints.
Herein, XSeed = YSeed = 121. Numerical simulations are performed for 
various different time steps. These are equal to 0.0001, 0.001, 0.01, 
and 0.1. Figure \ref{Fig:2D_Bimolecular_ScalDiff_FirstTimeStep_ConcC} 
shows the concentration profile of the product $C$ and Figure 
\ref{Fig:2D_Bimolecular_ScalDiff_y0dot5_ConcC} shows the values of 
$c_C$ at $y = 0.5$ at the first time-step. Analysis is performed 
using the unconstrained weighted negatively stabilized streamline 
diffusion LSFEM. From these figures, it is clear that unphysical negative values for $c_C$ are 
obtained even for small time steps. Furthermore, these violations 
are significant and not close to machine precision for both small 
and large time-steps. Non-negative constraints have to be enforced 
in order to get meaningful values for $c_C$.

Figures \ref{Fig:2D_Bimolecular_ScalDiff_Mixing_ConcC} and
\ref{Fig:2D_Bimolecular_ScalDiff_Mixing_ConcC1} show the concentration 
profiles of $c_C$ for both non-chaotic and chaotic vortex flow fields 
at various time levels. For non-chaotic advection, product $C$ is 
initially formed away from the vortex field. As time progresses, it 
slowly gets accumulated in the closed streamlines of the two vortices. 
Regions of higher concentration are located at the center of vortices. 
From Figure \ref{Fig:2D_Bimolecular_ScalDiff_Mixing_ConcC}, it is evident 
that $c_C$ contour is symmetric along the line $y = x$. This is because 
the non-chaotic vortex-based advection velocity vector field is 
symmetric along this line. However, this is not the case with chaotic 
vortex flow field. Qualitatively, there is no symmetry associated 
with the concentration field. This is because of the time-periodic 
sinusoidal terms given by equations \eqref{Eqn:TransMixing_ChaoticVortexFlowModel1} 
and \eqref{Eqn:TransMixing_ChaoticVortexFlowModel2}. They provide 
chaotic features for the chosen value of period of motion $T$. An 
interesting feature observed is that mixing of chemical species is 
enhanced in chaotic flow as compared to non-chaotic flow. This is 
because $c_C$ is not equal to zero in the non-chaotic flow for late 
times. 

Finally, from these figures it is evident that existing numerical 
formulations do not provide accurate information on the fate of 
reactants and products for all times. On the other hand, the 
proposed methodology predicts results accurately for both early 
and late times.

\subsection{Transient analysis of species mixing in cellular flows}
\label{SubSec:2D_Trans_SteadyUnsteady_CellFlows}
A pictorial description of the initial boundary value problem 
is provided in Figure \ref{Fig:2D_CellFlow_ScalDiff_Mixing}.
The prescribed diffusive/total flux for each chemical species 
is taken to be equal to zero. The initial condition is such 
that $c^{0}_A(\mathbf{x}) = 1$ in the bottom half of the domain
and vanishes elsewhere while $c^{0}_B(\mathbf{x}) = 1$ only in 
the upper half. The non-reactive part of the volumetric source 
is equal to zero for all the species. It should be noted that 
the concentration of the product $C$ should be between $0$ and 
$1$ because $f_i(\mathbf{x},t) = 0$. 

For numerical simulations, we have taken $L_x = 1$, $L_y = 0.5$, 
XSeed = 61 and YSeed = 241. The stoichiometric coefficients are 
taken as $n_A = 1$, $n_B = 1$, and $n_C = 1$. The time-step is 
taken as $\Delta t = 0.1$. The total time of interest is taken 
as $\mathcal{I} = 5$. The scalar diffusivity is taken as $D = 5 
\times 10^{-3}$. The stabilization parameters are taken as $\delta_o 
= \tau_o = 10^{-3}$ and $\delta_1 = \tau_1 = 10^{-4}$. 
The advection velocity vector field for the cellular flow 
is given by \cite{1982_Smolarkiewicz_MWR_v110_p1968_p1983}:
\begin{align}
  \label{Eqn:Steady_CellFlows}
  \mathbf{v}(\mathbf{x}) &= - \sin \left(\frac{2 \pi x}{L_
  {\text{\tiny {Cell}}}} \right) \cos \left(\frac{2 \pi y}{L_
  {\text{\tiny {Cell}}}} \right) \mathbf{\hat{e}}_x + \cos 
  \left(\frac{2 \pi x}{L_{\text{\tiny {Cell}}}} \right) 
  \sin \left(\frac{2 \pi y}{L_{\text{\tiny {Cell}}}} 
  \right) \mathbf{\hat{e}}_y
\end{align}
where $L_{\text{\tiny {Cell}}}$ is the cell length of 
a pair of vortices. The velocity field given by 
equation \eqref{Eqn:Steady_CellFlows} has a set 
of symmetrical vortices, and the neighboring 
vortices rotate in opposite directions. It is 
well-known that the advection velocity field given 
by equation \eqref{Eqn:Steady_CellFlows} causes numerical 
difficulties (if the underlying numerical scheme is not 
properly designed). This is because the advection velocity 
is strongly non-uniform \cite[Section 8]{1982_Smolarkiewicz_MWR_v110_p1968_p1983},
(which happens to be in our case). That is,
\begin{align}
  \label{Eqn:Strongly_NonUniform_VelField}
  \mathop{\mathrm{max}}_{\mathbf{x} \in \overline{\Omega}} 
  \left [ \frac{\partial \mathrm{v}_x}{\partial x} - 
  \frac{\partial \mathrm{v}_y}{\partial y} \right ] 
  \times \Delta t > 1
\end{align}
The main objective of this test problem is show that the proposed 
formulation is robust and can analyze velocity fields that are 
strongly non-uniform without causing numerical instabilities/oscillations.

Analysis is performed for a series of hierarchical cell lengths. 
That is $L_{\text{\tiny {Cell}}}$ equal to 0.5, 0.25, 0.125, 0.0625, 
and 0.03125. In all the cases, \emph{as the input data and the position 
of the cellular vortices is symmetric about the line $y = 0.25$, it 
may be expected that formation of product $C$ will be symmetric along 
this line}. Additional information on the symmetry of the formation 
of product $C$ can be inferred based on the cell length of adjacent 
pair of vortices, which rotate in opposite directions. This is apparent 
from the numerical results presented for product $C$ provided in Figures 
\ref{Fig:2D_CellFlow_ScalDiffMixing_cCNSSDNoCons}--\ref{Fig:2D_CellFlow_ScalDiffMixing_cCNSSDAvgLCells}

Figure \ref{Fig:2D_CellFlow_ScalDiffMixing_cCNSSDNoCons} 
shows the concentration profiles of the product $C$ under 
the unconstrained and constrained negatively stabilized 
streamline diffusion for $L_{\text{\tiny {Cell}}} = 0.5$.
From this figure, it is evident that the \emph{unconstrained} 
LSFEM violates both the non-negative and maximum constraints. 
In addition, both undershoots and overshoots are observed. 
On the other hand, the proposed computational framework with 
LSB and DMP constraints provides physically meaningful profiles 
for the concentration of the product $C$. 

For $L_{\text{\tiny {Cell}}} = 0.5$, immediately after time $t = 0$, 
we observe wing-like concentration profiles. This is because diffusion
controls the species mixing rather than advection across the adjacent 
cells along the line $y = 0.25$ (note that $\mathrm{v}_y(x,y = 0.25) = 
0$). However, once the species $A$ and $B$ enter the closed streamlines 
where advection dominates, mixing happens at much faster rate. Furthermore, 
product $C$ spreads in time along the array of counter-rotating vortices. 
After a considerable time ($t \approx 5$), we observe that formation of 
product $C$ is symmetric along the lines $x = 0.5$ and $y = 0.25$. In addition, 
product $C$ accumulates near the region where $\|\mathbf{v}(\mathbf{x})\|$ 
is close to zero (which happens to be at the center of vortices and hyperbolic 
points). This happens to be in a good agreement with the inferences drawn from 
the numerical simulations performed on cellular flows 
\cite{Neufeld_Garcia,1982_Smolarkiewicz_MWR_v110_p1968_p1983}.

Figure \ref{Fig:2D_CellFlow_ScalDiffMixing_cCNSSDNoCons} and 
\ref{Fig:2D_CellFlow_ScalDiffMixing_cCNSSDAvgLCells} 
show that the separatrices connecting the hyperbolic points inhibit 
long range transport of chemical species from one cell to another. 
By decreasing $L_{\text{\tiny {Cell}}}$, species mixing can be enhanced. 
This is evident from Figure \ref{Fig:2D_CellFlow_ScalDiffMixing_cCNSSDAvgLCells}.
Qualitatively, the numerical results presented here agree with the 
analysis presented by Neufeld and Garcia \cite{Neufeld_Garcia}, which 
tells us that mixing of chemical species is fast within a cell but 
the transport of reactants/products between the cells is controlled 
by diffusion only. In order to enhance the efficient species mixing 
in different regions in these type of flows, $L_{\text{\tiny {Cell}}}$ 
has to be as smaller.
To conclude, we would like to emphasize that the numerical solution 
based on the proposed methodology does not exhibit numerical instabilities 
and is able to capture the essential features even when the advection 
velocity is strongly non-uniform.

\section{SUMMARY AND CONCLUDING REMARKS}
\label{Sec:NN_AD_Conclusions}
We presented a robust computational framework for (steady-state 
and transient) advection-diffusion-reaction equations that satisfies 
the non-negative constraint, maximum principles, local species balance, 
and global species balance. The framework can handle general computational 
grids, anisotropic diffusivity, highly heterogeneous velocity fields, and 
provides physically meaningful numerical solutions without node-to-node 
spurious oscillations even on coarse computational meshes. The main 
\emph{contributions} of the paper can be summarized as follows: 
\begin{enumerate}[(C1)]
\item We constructed and proved a continuous maximum 
principle that includes both Dirichlet and Neumann 
boundary conditions. It also takes into account the 
inflow and outflow Neumann boundary conditions in 
establishing the maximum principle. 
\item We described in detail the shortcomings of 
  several plausible numerical approaches to satisfy the 
  maximum principle, the non-negative constraint, 
  and species balance.
\item We proposed a locally conservative DMP-preserving 
  computational framework and constructed element stabilization 
  parameters that are valid for a general reaction coefficient, 
  advection velocity, and diffusivity. The framework has been 
  carefully constructed using the least-squares finite element 
  method (LSFEM). It is also shown that a naive implementation 
  of LSFEM will not meet the desired properties. 
\item The discrete problem under the proposed 
    framework is well-posed, and it can be shown 
    that a unique solution exists.
\item We performed numerical convergence studies on the 
  computational framework. We also systematically analyzed 
  and documented the performance of the proposed framework 
  with various benchmark problems and realistic examples.
\item We obtained numerically a scaling law for 
  a transport-controlled bimolecular reaction.
\item In chemically reactive systems, it is important to predict 
  the fate of reactants and products during the early times. 
  We have shown that the existing formulations may not provide 
  accurate information for such scenarios. On the other hand, 
  using numerical experiments, we have shown that the proposed 
  framework predicts accurate results for both early and late 
  times. 
\end{enumerate}

The \emph{salient features and performance} of the 
proposed computational framework can be summarized 
as follows:
\begin{enumerate}[(S1)]
  \item The rate of decrease of errors in LSB and GSB 
    for \emph{unconstrained} negatively stabilized 
    streamline diffusion LSFEM with $h$-refinement 
    is slow and is about $\mathcal{O}(h)$. Furthermore, 
    this numerical formulation violates various discrete 
    principles and the non-negative constraint for both 
    isotropic and anisotropic diffusivities. On the 
    other hand, the proposed non-negative computational 
    framework is able to satisfy LSB and GSB up to machine 
    precision on an arbitrary computational mesh.
  \item The proposed computational framework with NN and LSB 
    constraints eliminates the spurious node-to-node oscillations 
    and provides physically meaningful values for concentration. 
    Furthermore, it is able to furnish reasonable answers with 
    various time-steps and at various time levels even on coarse 
    computational grids.
  \item It has been shown that existing formulations may 
    \emph{fail} to give acceptable results for non-negative 
    statistical quantities such as $\Theta_C^2$, which is 
    defined in Section \ref{SubSec:2D_SS_RxnTank}. However, 
    the proposed methodology always provides non-negative 
    values for $\Theta_C^2$. The quantity $\Theta_C^2$ can 
    be used as \emph{a posteriori} criteria to assess 
    accuracy of numerical solutions for complex initial 
    and boundary value problems for which non-negativity 
    and species conservation are important. 
  \item Due to the aforementioned desired properties, our 
    proposed computational framework can be an ideal candidate 
    to numerically obtain scaling laws for complicated problems 
    with non-trivial initial and boundary conditions. Therefore, 
    the proposed framework will be vital for predictive simulations 
    in groundwater modeling, reactive transport, environmental 
    fluid mechanics, and modeling of degradation of materials.
\end{enumerate}

A possible future research work is to implement and 
analyze the performance of the proposed numerical 
methodology in a parallel environment. A related 
research is to design tailored iterative solvers 
and associated pre-conditioners for our proposed 
numerical methodology.

\section{APPENDIX A:~Element-level discretization of stabilization terms}
\label{Sec:NN_AD_Appendix}
The weighted negatively stabilized streamline diffusion 
least-squares formulation requires the evaluation of 
$\mathrm{div}[\mathrm{grad}[c(\mathbf{x})]]$ and 
$\mathrm{grad}[\mathrm{grad}[c(\mathbf{x})]]$ terms 
at the element-level. 
Since these terms are not typical, we present a compact 
way of discretizing these terms under the finite element 
method. 
It should be noted that one need not evaluate these terms 
for lowest-order simplicial elements (i.e., three-node 
triangular element and four-node tetrahedron element), 
as these terms will be identically zero. 
However, $\mathrm{div}[\mathrm{grad}[c(\mathbf{x})]]$ 
and $\mathrm{grad}[\mathrm{grad}[c(\mathbf{x})]]$ will 
be non-zero for high-order simplicial and non-simplicial 
elements (i.e., four-node quadrilateral element and 
eight-node brick element).

\subsection{Fourth-order tensors}
\label{SubSec:FourthOrderTensors_Properties}
Let $\mathbf{R}$ and $\mathbf{S}$ be two second-order 
tensors. A fourth-order tensor product $\mathbf{R} 
\boxtimes \mathbf{S}$ is defined as follows:
\begin{align}
  \label{Eqn:FourthOrderTensor_Box_Product}
  \left(\mathbf{R} \boxtimes \mathbf{S}\right) \mathbf{T} 
  = \mathbf{R} \mathbf{T} \mathbf{S}^{\mathrm{T}}
\end{align}
for any second-order tensor $\mathbf{T}$. The 
fourth-order identity tensor can then be written 
as: 
\begin{align}
  \label{Eqn:FourthOrderIdentityTensor}
  \mathbb{I} = \mathbf{I} \boxtimes \mathbf{I}
\end{align}
where $\mathbf{I}$ is the second-order identity 
tensor. The fourth-order transposer $\mathbb{T}$ 
and symmetrizer $\mathbb{S}$ tensors are defined 
as follows:
\begin{subequations}
  \begin{align}
    \label{Eqn:Transposer}
    \mathbb{T} \mathbf{A} &= \mathbf{A}^{\mathrm{T}} \\
    \label{Eqn:Symmetrizer}
    \mathbb{S} \mathbf{A} &= \frac{1}{2} \left(\mathbf{A} + 
    \mathbf{A}^{\mathrm{T}} \right)
  \end{align}
\end{subequations} 
where $\mathbf{A}$ is a general second-order tensor. 

\subsection{Properties of Kronecker products relevant 
  to finite element discretization}
\label{SubSec:KronProduct_Properties}
Let $\boldsymbol{A}$ and $\boldsymbol{B}$ be matrices of 
sizes $n \times m$ and $p \times q$, respectively. Let 
$a_{ij}$ and $b_{ij}$ be, respectively, the $ij^{\mathrm{th}}$ 
components of $\boldsymbol{A}$ and $\boldsymbol{B}$. The 
\emph{Kronecker product} of these two matrices is an 
$np \times mq$ matrix that is defined as follows:
\begin{align}
  \label{Eqn:Kronecker_Product_AB}
  \boldsymbol{A} \odot \boldsymbol{B} :=
  \left[ \begin{array}{ccc} a_{11} 
  \boldsymbol{B} & \hdots & a_{1m} 
  \boldsymbol{B} \\
  \vdots & \ddots & \vdots \\
  a_{n1}\boldsymbol{B} & \hdots & 
  a_{nm}\boldsymbol{B} \end{array} \right]
\end{align}
Another operator that will be useful is the 
$\mathrm{vec}[\bullet]$ operator, which is 
defined as follows:
\begin{align}
  \label{Eqn:Vec_Operator}
  \mathrm{vec}[\boldsymbol{A}] := \left[ \begin{array}{c}
  a_{11} \\
  \vdots \\
  a_{1m} \\
  \vdots \\
  a_{n1} \\
  \vdots \\
  a_{nm} \end{array} \right]
\end{align}
The following standard properties can be established 
for $\mathrm{mat}[\bullet]$ and $\mathrm{vec}[\bullet]$ 
operators:
\begin{align}
  &\mathrm{vec}[\boldsymbol{A} + \boldsymbol{B}] = 
  \mathrm{vec}[\boldsymbol{A}] + \mathrm{vec}[\boldsymbol{B}] \\
  &\left(\boldsymbol{A} \odot \boldsymbol{B} 
  \right) \left(\boldsymbol{C} \odot \boldsymbol{D} 
  \right) = \left(\boldsymbol{A} \boldsymbol{C} 
  \odot \boldsymbol{B} \boldsymbol{D} \right) \\
  &\mathrm{vec}[\boldsymbol{A C B}] = 
  \left(\boldsymbol{B}^{\mathrm{T}} \odot \boldsymbol{A} 
  \right) \mathrm{vec}[\boldsymbol{C}]
\end{align}
For general properties of Kronecker products, see the 
book by Graham \cite{Graham_Kronecker}. However, this 
book does not contain many of the results presented 
below, which are useful for the current study.  

For an effective computer implementation of LSFEM-based 
formulations, we shall represent four-dimensional and 
three-dimensional arrays as two-dimensional matrices. 
To this end, consider a four-dimensional array 
$\boldsymbol{P}$ of size $m \times n \times p \times q$. 
The matrix representation of $\boldsymbol{P}$ is denoted 
by $\mathrm{mat}[\bullet]$, and is defined as follows:
\begin{align}
  \label{Eqn:mat_Operator_FourthOrderTensors}
  \mathrm{mat}[\boldsymbol{P}] &:=
  \left[ \begin{array}{cccccc} 
      P_{1111} & \hdots & P_{11p1} & P_{1112} & \hdots & P_{111q} \\
      \vdots & \ddots & \vdots & \vdots & \ddots & \vdots \\
      P_{mn11} & \hdots & P_{mnp1} & P_{mn12} & \hdots & P_{mnpq}
    \end{array} \right] 
\end{align}
The following properties can be established for the matrix 
representation of fourth-order tensors:
\begin{align}
  &\mathrm{vec}[\boldsymbol{P} \boldsymbol{X}] = 
  \mathrm{mat}[\boldsymbol{P}] \mathrm{vec}[\boldsymbol{X}] \\
  &\mathrm{mat}[\boldsymbol{A} \boxtimes \boldsymbol{B}] 
  = \boldsymbol{B} \odot \boldsymbol{A} \\
  &\mathrm{mat}[\mathbb{S}] = \frac{1}{2} \left(\boldsymbol{I} 
    \odot \boldsymbol{I} + \mathrm{mat}[\mathbb{T}] \right)
\end{align}
where the matrix representation of $\mathbb{T}$ 
takes the following form:
\begin{align}
  \label{Eqn:mat_Transposer}
  \mathrm{mat}[\mathbb{T}] := \left[ \begin{array}{c}
  \boldsymbol{I} \odot [1,0,0, \hdots, 0]_n \\
  \boldsymbol{I} \odot [0,1,0, \hdots, 0]_n \\
  \vdots \\
  \boldsymbol{I} \odot [0,0,0, \hdots, 1]_n 
  \end{array} \right]_{mn \times mn}
\end{align}
In the above equation, $\boldsymbol{I}$ is an 
identity matrix of size $m \times m$.
It can be shown that for a matrix $\boldsymbol{Z}$ 
of size $m \times n$, $\mathrm{vec}[\boldsymbol{Z}^{\mathrm{T}}] 
= \mathrm{mat}[\mathbb{T}] \mathrm{vec}[\boldsymbol{Z}]$. 

For a three-dimensional arrays, there are two useful 
matrix representations, which will be denoted as 
$\mathrm{mat}_1[\bullet]$ and $\mathrm{mat}_2[\bullet]$. 
Consider a three-dimensional array $\boldsymbol{Q}$ of 
size $m \times n \times p$. The corresponding matrix 
representations of $\boldsymbol{Q}$ are defined as 
follows:
\begin{align}
  \label{Eqn:mat1_Operator_ThirdOrderTensors}
  \mathrm{mat}_1[\boldsymbol{Q}] &:=
  \left[ \begin{array}{cccccccc} 
      Q_{111} & \hdots & Q_{1n1} & \hdots & \hdots & Q_{11p} & \hdots & Q_{1np} \\
      \vdots & \ddots & \vdots & \vdots & \vdots & \vdots & \ddots & \vdots \\
      Q_{m11} & \hdots & Q_{mn1} & \hdots & \hdots & Q_{m1p} & \hdots & Q_{mnp}
    \end{array} \right] 
\end{align}
\begin{align}
  \label{Eqn:mat2_Operator_ThirdOrderTensors}
  \mathrm{mat}_2[\boldsymbol{Q}] &:=
  \left[ \begin{array}{ccc}
      Q_{111} & \hdots & Q_{m11} \\
      \vdots & \ddots & \vdots \\
      Q_{11p} & \hdots & Q_{m1p} \\
      Q_{121} & \hdots & Q_{m21} \\
      \vdots & \ddots & \vdots \\
      Q_{12p} & \hdots & Q_{m2p} \\
      \vdots & \ddots & \vdots \\
      Q_{1np} & \hdots & Q_{mnp}
    \end{array} \right] 
\end{align}
The following properties can be established for 
the matrix representations of three-dimensional 
arrays: 
\begin{align}
  &\mathrm{vec}[\boldsymbol{Q} \boldsymbol{Y}] = 
  \mathrm{mat}_1[\boldsymbol{Q}] \mathrm{vec}[\boldsymbol{Y}] \\
  &\mathrm{vec}[\boldsymbol{Q} \boldsymbol{z}] = 
  \mathrm{mat}_2[\boldsymbol{Q}] \mathrm{vec}[\boldsymbol{z}]
\end{align}

\subsection{Finite element discretization of $\mathrm{div}
  [\mathrm{grad}[c(\mathbf{x})]]$ and $\mathrm{grad}
  [\mathrm{grad}[c(\mathbf{x})]]$ terms}
\label{SubSec:FEM_ScalAnisoDiff_divgrad_gradgrad}
Let $\boldsymbol{\xi}$ denote the position vector in the 
reference finite element. The row vector containing the 
shape functions is denoted by $\boldsymbol{N}$, which 
is a function of $\boldsymbol{\xi}$. The derivatives 
and Hessian of $\boldsymbol{N}$ with respect to 
$\boldsymbol{\xi}$ are, respectively, denoted as 
$\boldsymbol{DN}$ and $\boldsymbol{DDN}$. That 
is, in indicial notation we have: 
\begin{subequations}
  \label{Eqn:DN_DDN_Components}
  \begin{align}
    &\left(\boldsymbol{DN} \right)_{ij} = 
    \frac{\partial N_i}{\partial \xi_j} \\
    &\left(\boldsymbol{DDN} \right)_{ijk} = 
    \frac{\partial^2 N_i}{\partial \xi_j \xi_k}
  \end{align}
\end{subequations}
where $N_i$ and $\xi_i$ are, respectively, 
the $i^{\mathrm{th}}$-component of the vectors 
$\boldsymbol{N}$ and $\boldsymbol{\xi}$. The 
concentration field $c(\mathbf{x})$ and total 
flux vector $\mathbf{q}(\mathbf{x})$ are 
interpolated as follows:
\begin{subequations}
  \label{Eqn:Conc_Flux_LinFEM}
  \begin{align}
    c(\mathbf{x}) &= \widehat{\boldsymbol{c}}^{\mathrm{T}} 
    \boldsymbol{N}^{\mathrm{T}}\left(\boldsymbol{\xi}
    (\boldsymbol{x}) \right) \\ 
    \mathbf{q}(\mathbf{x}) &= 
    \widehat{\boldsymbol{q}}^{\mathrm{T}} \boldsymbol{N}^
    {\mathrm{T}}\left(\boldsymbol{\xi}(\boldsymbol{x}) 
    \right)
  \end{align}
\end{subequations}
where $\widehat{\boldsymbol{c}}^{\mathrm{T}}$ and $\widehat{\boldsymbol{q}}
^{\mathrm{T}}$ denote the matrix containing nodal concentration and 
total flux vector. In negatively stabilized streamline diffusion LSFEM,
$\mathrm{div}[\mathrm{grad}[c(\mathbf{x})]]$ and $\mathrm{grad}[\mathrm{grad}
[c(\mathbf{x})]]$ terms arise from the following expansions:
\begin{subequations}
\begin{align}
  \label{Eqn:ScalDiff_divDgradc}
  \mathrm{div}[D(\mathbf{x}) \mathrm{grad}[c(\mathbf{x})]] &= 
  \mathrm{grad}[D(\mathbf{x})] \bullet \mathrm{grad}[c(\mathbf{x})] + 
  D(\mathbf{x}) \, \mathrm{div}[\mathrm{grad}[c(\mathbf{x})]] \\
  \label{Eqn:AnisoDiff_divDgradc}
  \mathrm{div}[\mathbf{D}(\mathbf{x}) \mathrm{grad}[c(\mathbf{x})]] &= 
  \mathrm{div}[\mathbf{D}(\mathbf{x})] \bullet \mathrm{grad}[c(\mathbf{x})] 
  + \mathbf{D}(\mathbf{x}) \bullet \mathrm{grad}[\mathrm{grad}[c(\mathbf{x})]]
\end{align}
\end{subequations}
Based on the regularity of diffusivity, $\mathrm{grad}
[D(\mathbf{x})]$ and $\mathrm{div}[\mathbf{D}(\mathbf{x})]$ 
can be calculated in multiple ways. If the diffusivity is 
\emph{continuously differentiable}, then $\mathrm{grad}
[D(\mathbf{x})]$ and $\mathrm{div}[\mathbf{D}(\mathbf{x})]$ 
can be directly evaluated analytically. This will considerably 
reduce the computational cost in the evaluation of local 
stiffness matrices. If $D(\mathbf{x})$ is not differentiable 
(but is square integrable), then $\mathrm{grad}[D(\mathbf{x})]$ 
and $\mathrm{div}[\mathbf{D}(\mathbf{x})]$ can be evaluated as 
follows:
\begin{subequations}
  \begin{align}
    \label{Eqn:gradD_FEM_Expansion}
    \mathrm{grad}[D(\mathbf{x})] &= \widehat{\boldsymbol{D}}
    ^{\mathrm{T}} (\boldsymbol{DN}) \boldsymbol{J}^{-1} \\
    \label{Eqn:divD_FEM_Expansion}
    \mathrm{div}[\mathbf{D}(\mathbf{x})] &= \mathrm{mat}_1
    [\widehat{\boldsymbol{D}}] \mathrm{vec}[\mathrm{mat}[\mathbb{T}] 
    (\boldsymbol{DN}) \boldsymbol{J}^{-1}]
  \end{align}
\end{subequations}
where $\boldsymbol{J}$ is the Jacobian matrix, and 
$\widehat{\boldsymbol{D}}$ is the nodal values for 
the diffusivity, whose size can be inferred based 
on the context (whether the diffusivity is either 
$D(\mathbf{x})$ or $\mathbf{D}(\mathbf{x})$). Using 
equations \eqref{Eqn:DN_DDN_Components} and 
\eqref{Eqn:Conc_Flux_LinFEM}, the Laplacian and 
Hessian of $c(\mathbf{x})$ can be calculated as follows:
\begin{align}
  \label{Eqn:ScalDiff_divgradc_FEM}
  \mathrm{div}[\mathrm{grad}[c(\mathbf{x})]] &= 
  \Bigg( \left(\boldsymbol{I} - (\boldsymbol{DN}) 
  \boldsymbol{J}^{-1} \widehat{\boldsymbol{x}}^{\mathrm{T}} 
  \right) \mathrm{mat}_1[\boldsymbol{DDN}] \mathrm{vec} 
  [\boldsymbol{J}^{-1} \boldsymbol{J}^{-\mathrm{T}}] 
  \Bigg)^{\mathrm{T}} \mathrm{vec}[\widehat{\boldsymbol{c}}
  ^{\mathrm{T}}] \\
  \label{Eqn:AnisoDiff_gradgradc_FEM}
  \mathrm{vec}[\mathrm{grad}[\mathrm{grad}[c(\mathbf{x})]]] 
  &= \Bigg( \left(\boldsymbol{J}^{-\mathrm{T}} \odot 
  \boldsymbol{J}^{-\mathrm{T}} \right) \mathrm{mat}_2
  [\boldsymbol{DDN}] \left(\boldsymbol{I} \odot 
  \boldsymbol{I} - \widehat{\boldsymbol{x}} \boldsymbol{J}
  ^{\mathrm{T}} (\boldsymbol{DN})^{\mathrm{T}} \right) \Bigg)
  \mathrm{vec}[\widehat{\boldsymbol{c}}^{\mathrm{T}}]
\end{align}
where $\widehat{\boldsymbol{x}}$ is the 
matrix containing nodal coordinates.
\section{APPENDIX B:~Coercivity, error estimates, and stabilization parameters}
\label{Sec:NN_AD_Appendix2}
Herein, we shall establish coercivity and error estimates 
for the homogeneous weighted LSFEMs (i.e., $c^{\mathrm{p}}
(\mathbf{x}) = 0$ on $\partial \Omega$). Based on this 
mathematical analysis, we obtain the stabilization 
parameters that are used in this paper for the weighted 
negatively stabilized streamline diffusion LSFEM. 
To this end, let $H^{m}(\Omega)$ denote the standard 
Sobolev space for a given non-negative integer $m$ 
\cite{Evans_PDE}. The associated standard inner product 
and norm are denoted by $(\bullet;\bullet)_m$ and $\| \bullet 
\|_m$, respectively. On the function space $\mathcal{Q}$, the 
inner product $(\bullet;\bullet)_{\mathrm{div}}$ and 
norm $\| \bullet \|_{\mathrm{div}}$ are defined as follows: 
\begin{subequations}
  \begin{align}
    \label{Eqn:Hdiv_Inner_Product}
    (\mathbf{p};\mathbf{q})_{\mathrm{div}} &:=
    (\mathbf{p};\mathbf{q})_0 + (\mathrm{div}
    [\mathbf{p}]; \mathrm{div}[\mathbf{q}])_0 
    \quad \forall \mathbf{p}, \mathbf{q} \in 
    \mathcal{Q} \\
    \label{Eqn:Hdiv_Norm}
    \| \mathbf{p} \|^2_{\mathrm{div}} &:= \| 
    \mathbf{p} \|^2_0 + \| \mathrm{div}[\mathbf{p}] 
    \|^2_0 \quad \forall \mathbf{p} \in \mathcal{Q}
  \end{align}
\end{subequations}
The Poincar\'{e}-Friedrichs inequality takes the 
following form \cite{Bochev_Gunzburger}: there 
exists a constant $C_{pf} > 0$ such that we have
\begin{align}
  \label{Eqn:Poincare_Friedrichs_Inequality}
  \| u \|_0 \leq C_{pf} \| \mathrm{grad}[u] \|_0 
  \quad \forall \, u \in H^{1}_0(\Omega)
\end{align}

Consider the classical weighted primitive least-squares 
functional $\mathfrak{F}_{\mathrm{Prim}} \left( \left( 
c,\mathbf{q} \right), f \right)$ given by equation 
\eqref{Eqn:Weighted_Primitive_LS_Functional} with 
$c(\mathbf{x}) = 0$ on $\partial \Omega$. If $(c,
\mathbf{q}) \in \mathcal{C} \times \mathcal{Q}$ is 
an exact solution of the equations 
\eqref{Eqn:NN_AD_GE_BE}--\eqref{Eqn:NN_AD_GE_NBC}, 
then $(c,\mathbf{q})$ must be a unique zero minimizer 
of $\mathfrak{F}_{\mathrm{Prim}} \left( \left( c,\mathbf{q} 
\right), f \right)$ on $\mathcal{C} \times \mathcal{Q}$. 
Hence, for any $\epsilon \in \mathbb{R}$, we have
\begin{align}
  \label{Eqn:Gateaux_Variation_PrimLSFEM}
  \frac{d}{d \epsilon} \mathfrak{F}_{\mathrm{Prim}} 
  \left(\left(c,\mathbf{q} \right) + \epsilon \left(
  w,\mathbf{p} \right), f \right) \Big |_{\epsilon = 0} 
  = 0 \quad \forall \left(w,\mathbf{p} \right) \in 
  \mathcal{W} \times \mathcal{Q}
\end{align}
which is identical to the following:
\begin{align}
  \label{Eqn:Bilinear_Linear_Forms_PrimLSFEM}
  \mathfrak{B}_{\mathrm{Prim}} \left(\left(c,\mathbf{q} 
  \right); \left(w,\mathbf{p} \right) \right) = 
  \mathfrak{L}_{\mathrm{Prim}} \left(\left(w,\mathbf{p} 
  \right) \right) \quad \forall \left(w,\mathbf{p} \right) 
  \in \mathcal{W} \times \mathcal{Q}
\end{align}
where $\mathfrak{B}_{\mathrm{Prim}}(\bullet;\bullet)$ and 
$\mathfrak{L}_{\mathrm{Prim}}(\bullet)$ are the corresponding
bilinear and linear forms for the weighted primitive least-squares 
functional $\mathfrak{F}_{\mathrm{Prim}}$. It should be noted that
\begin{align}
  \label{Eqn:BilinearForm_PrimLSFEM_MainResult}
  \mathfrak{B}_{\mathrm{Prim}} \left(\left(w,\mathbf{p} 
  \right); \left(w,\mathbf{p} \right) \right) = \mathfrak{F}_
  {\mathrm{Prim}} \left( \left(w,\mathbf{p} \right), f=0 
  \right) \quad \forall \left(w,\mathbf{p} \right) \in 
  \mathcal{W} \times \mathcal{Q}
\end{align}
Equation \eqref{Eqn:BilinearForm_PrimLSFEM_MainResult} is 
used to prove coercivity and boundedness estimates for the 
bilinear form $\mathfrak{B}_{\mathrm{Prim}}$. Now consider 
the finite element discretization of the equation 
\eqref{Eqn:Bilinear_Linear_Forms_PrimLSFEM}. Let $\mathcal{C}
_h \subseteq \mathcal{C}$, $\mathcal{W}_h \subseteq \mathcal{W}$, 
and $\mathcal{Q}_h \subseteq \mathcal{Q}$ be the finite element 
function spaces spanned by piecewise polynomials of degree less 
than or equal to $r$ over $\Omega_h$. It should be noted that 
$r$ is an integer and $r \geq 1$. Then, the discrete weighted 
primitive LSFEM can be written as follows: Find $\left(c_h,
\mathbf{q}_h \right) \in \mathcal{C}_h \times \mathcal{Q}_h$ 
such that 
\begin{align}
  \label{Eqn:Bilinear_Linear_Forms_DiscretePrimLSFEM}
  \mathfrak{B}_{\mathrm{Prim}} \left(\left(c_h,\mathbf{q}_h 
  \right); \left(w_h,\mathbf{p}_h \right) \right) = 
  \mathfrak{L}_{\mathrm{Prim}} \left(\left(w_h,\mathbf{p}_h 
  \right) \right) \quad \forall \left(w_h,\mathbf{p}_h \right) 
  \in \mathcal{W}_h \times \mathcal{Q}_h
\end{align}
where $\left(c_h,\mathbf{q}_h \right)$ is the finite element 
solution with respect to the chosen basis functions spanning 
the finite element space $\mathcal{C}_h \times \mathcal{Q}_h$.
Similar inference holds for $\mathfrak{F}_{\mathrm{NgStb}}((c,\mathbf{q}),f)$, 
$\mathfrak{B}_{\mathrm{NgStb}}$, and $\mathfrak{L}_{\mathrm{NgStb}}$. 

We assume that $\Omega_h$ is quasi-uniform \cite{Jiang_LSFEM,
Bochev_Gunzburger}. That is, there exists a constant $\widehat{C} 
> 0$ independent of $h$ such that $h \leq \widehat{C} h_
{\Omega_{e}}$ for all $\Omega_{e} \in \Omega_{h}$. Additionally, 
we assume that the following inverse inequality holds on these 
quasi-uniform meshes. There exists a constant $\widetilde{C} > 
0$ independent of $h$ such that 
\begin{subequations}
  \begin{align}
    \label{Eqn:Inverse_Estimate1}
    &\widetilde{C} \sum_{\Omega_{\mathrm{e}} \in \Omega_h} 
    h^2_{\Omega_{e}} \Big \| \mathrm{div}[\mathrm{grad}[c_h]] \Big \|^2_
    {0,\Omega_{e}} \leq \| \mathrm{grad}[c_h] \|^2_0 \quad 
    \forall c_h \in \mathcal{C}_h \\
    \label{Eqn:Inverse_Estimate2}
    \Big \|\mathbf{D} \bullet \mathrm{grad}[\mathrm{grad}[c_h]] 
    \Big \|_0 &\leq \Big \|\mathrm{tr}[\mathbf{D}] \mathrm{tr}
    [\mathrm{grad}[\mathrm{grad}[c_h]]] \Big \|_0 = \mathrm{tr}
    [\mathbf{D}] \Big \| \mathrm{div}[\mathrm{grad}[c_h]] \Big 
    \|_0 \quad \forall c_h \in \mathcal{C}_h
  \end{align}
\end{subequations}
where $\mathrm{tr}[\bullet]$ is the trace of a matrix.
In proposing equation \eqref{Eqn:Inverse_Estimate2}, 
we assumed that the Hessian of $c_h$, $\mathrm{grad}[
\mathrm{grad}[c_h]]$, is positive semi-definite. 

All the results presented here are applicable for 
a general anisotropic diffusivity tensor, advection 
velocity vector field, and linear reaction coefficient. 
One can obtain simplified results for isotropy by taking 
$\mathbf{D}(\mathbf{x}) = D(\mathbf{x}) \mathbf{I}$, 
where $\mathbf{I}$ is an identity tensor.

\begin{theorem}[\texttt{Coercivity for weighted primitive LSFEM}]
  \label{Thm:Weighted_PrimLSFEM_CoerciveEstimates}
  There exist constants $C_{\mathrm{Prim}1} > 0$ and $C_
  {\mathrm{Prim}2} > 0$ independent of $\mathbf{D}$ and 
  $h$ such that for all $\left(w_h,\mathbf{p}_h \right) 
  \in \mathcal{W}_h \times \mathcal{Q}_h$:
  \begin{subequations}
    \begin{align}
      \label{Eqn:Coercivity_PrimLSFEM1}
      \mathfrak{F}_{\mathrm{Prim}} \left( \left(w_h,\mathbf{p}_h 
      \right), f=0 \right) &\geq C_{\mathrm{Prim}1} \gamma^2
      _{\mathrm{min}} \lambda^2_{\mathrm{min}} \| \mathrm{grad}[w_h] 
      \|^2_0 \\
      \label{Eqn:Coercivity_PrimLSFEM2}
      \mathfrak{F}_{\mathrm{Prim}} \left( \left(w_h,\mathbf{p}_h 
      \right), f=0 \right) &\geq C_{\mathrm{Prim}2} \gamma^2
      _{\mathrm{min}} \lambda^2_{\mathrm{min}} \left( \| w_h \|^2_1 + \frac{\| 
      \mathbf{p}_h \|^2_{\mathrm{div}}}{1 + \lambda^2_{\mathrm{min}} + 
      \lambda^2_{\mathrm{max}}} \right)
    \end{align}
  \end{subequations} 
  where the positive constant $\gamma_{\mathrm{min}}$ is:
  \begin{align}
    \label{Eqn:MinWeight_PrimLSFEM}
    \gamma_{\mathrm{min}} := \mathop{\mathrm{min}} \left [ 1,
    \mathop{\mathrm{min}}_{\mathbf{x} \in 
    \overline{\Omega}} \left [ \beta(\mathbf{x}) 
    \right ], \mathop{\mathrm{min}}_{\mathbf{x} 
    \in \overline{\Omega}} \left [ \lambda_{\mathrm{\mathrm{min}},
    \mathbf{A}(\mathbf{x})} \right ] \right ]
  \end{align}
  where $\lambda_{\mathrm{\mathrm{min}},\mathbf{A}(\mathbf{x})}$ is the 
  minimum eigenvalue of $\mathbf{A}(\mathbf{x})$ at a 
  given point $\mathbf{x} \in \overline{\Omega}$.
\end{theorem}
\begin{proof}
  Consider the weighted primitive least-squares functional 
  \eqref{Eqn:Weighted_Primitive_LS_Functional} with $f = 0$. 
  Equation \eqref{Eqn:MinWeight_PrimLSFEM} implies: 
  \begin{align}
    \label{Eqn:WgtPrim_BiLinForm1}
    \frac{2 \mathfrak{F}_{\mathrm{Prim}}}{\gamma^2
    _{\mathrm{min}}} &\geq \Big \| \mathbf{p}_h - w_h \mathbf{v} 
    + \mathbf{D} \mathrm{grad}[w_h] - \mu \mathrm{grad}[w_h] 
    \Big \|^2_{0,\Omega} + \Big \| \alpha w_h + \mathrm{div}[ 
    \mathbf{p}_h ] - \mu w_h \Big \|^2_{0,\Omega} \nonumber \\
    &+ 2 \mu \left(\mathbf{p}_h - w_h \mathbf{v} + \mathbf{D} 
    \mathrm{grad}[w_h]; \mathrm{grad}[w_h] \right)_{0,\Omega} 
    + 2 \mu \left(\alpha w_h + \mathrm{div}[ \mathbf{p}_h ]; 
    w_h \right)_{0,\Omega} \nonumber \\
    &- \mu^2 \| w_h \|^2_{0,\Omega} - \mu^2 \| \mathrm{grad}[w_h] 
    \|^2_{0,\Omega}
  \end{align}
  where $\mu$ is a positive constant, which will be 
  determined later. Using Poincar\'{e}-Friedrichs 
  inequality and Green's formulae, equation 
  \eqref{Eqn:WgtPrim_BiLinForm1} can be written as:
  \begin{align}
    \label{Eqn:WgtPrim_BiLinForm2}
    \frac{2 \mathfrak{F}_{\mathrm{Prim}}}{\gamma^2
    _{\mathrm{min}}} &\geq \mu \left(2\lambda_{\mathrm{min}} - \mu 
    \left(1 + C^2_{pf} \right) \right) \| \mathrm{grad}[w_h] 
    \|^2_{0,\Omega}
  \end{align}
  We obtain equation \eqref{Eqn:Coercivity_PrimLSFEM1} 
  by choosing 
  \begin{align}
  \mu = \frac{\lambda_{\mathrm{min}}}{1 + C^2_{pf}}
  \end{align}
  
  There exist two non-negative constants $C_{\mathbf{v}}$ and $C_{\alpha}$ 
  (for instance, $C_{\mathbf{v}} = \displaystyle \mathop{\mathrm{max}}_
  {\mathbf{x} \in \overline{\Omega}} \left [ \| \mathbf{v} \|^2 
  \right ]$ and $C_{\alpha} = \displaystyle \mathop{\mathrm{max}}_
  {\mathbf{x} \in \overline{\Omega}} \left [ \alpha^2 \right ]$) 
  such that 
  \begin{subequations}
    \begin{align}
      \label{Eqn:CoerciveEst1}
      \| w_h \|^2_1 &= \| w_h \|^2_0 + \| 
      \mathrm{grad}[w_h] \|^2_0 \leq \frac{2 \left(1 
      + C^2_{pf} \right)^2}{\gamma^2_{\mathrm{min}} \lambda^2_{\mathrm{min}}} 
      \mathfrak{F}_{\mathrm{Prim}}  \\
      \label{Eqn:CoerciveEst2}
      \| \mathbf{p}_h \|^2_{0} &\leq 2 \| \mathbf{p}_h - 
      w_h \mathbf{v} + \mathbf{D} \mathrm{grad}[w_h] 
      \Big \|^2_{0,\Omega} + 2 \|- w_h \mathbf{v} 
      + \mathbf{D} \mathrm{grad}[w_h] \Big \|^2_{0,\Omega} 
      \nonumber \\
      &\leq \left(1 + \frac{2 C_{\mathbf{v}} C^2_{pf} \left(1 + 
      C^2_{pf} \right)}{\lambda^2_{\mathrm{min}}} + 2 \left(1 + C^2_{pf} \right)
      \frac{\lambda^2_{\mathrm{max}}}{\lambda^2_{\mathrm{min}}} \right) 
      \frac{4 \mathfrak{F}_{\mathrm{Prim}}}{\gamma^2_{\mathrm{min}}}\\
      \label{Eqn:CoerciveEst3}
      \| \mathrm{div}[\mathbf{p}_h] \|^2_{0} &\leq 2 \big 
      \| \alpha w_h + \mathrm{div}[ \mathbf{p}_h ] \big 
      \|^2_{0,\Omega} + 2 \big \| \alpha w_h \big \|^2_{0,\Omega}
      \leq \left(1 + \frac{C_{\alpha} C^2_{pf} \left(1 + 
      C^2_{pf} \right)}{\lambda^2_{\mathrm{min}}} \right) \frac{4 
      \mathfrak{F}_{\mathrm{Prim}}}{\gamma^2_{\mathrm{min}}}
    \end{align}
  \end{subequations}
  It is easy to check that inequalities 
  \eqref{Eqn:CoerciveEst1}--\eqref{Eqn:CoerciveEst3} 
  imply inequality \eqref{Eqn:Coercivity_PrimLSFEM2}. 
\end{proof}

\begin{theorem}[\texttt{Coercivity and boundedness estimate for NSSD LSFEM}]
  \label{Thm:Weighted_NSSDLSFEM_CoerciveEstimates}
  Given that equations \eqref{Eqn:Inverse_Estimate1}--\eqref{Eqn:Inverse_Estimate2} 
  hold. If for each $\Omega_{e} \in \Omega_{h}$ we take 
  \begin{subequations}
    \begin{align}
      \label{Eqn:StrDiff_EleStabParam}
      \delta_{\Omega_{\mathrm{e}}} &= - \frac{\widetilde{C} 
      \lambda_{\mathrm{min}} h^2_{\Omega_{\mathrm{e}}}}{4 \left( nd^2
      \lambda^2_{\mathrm{max}} + \widetilde{C} C^2_{pf} \delta_{\alpha 
      \mathbf{v}} h^2 + \widetilde{C} \delta_{\mathbf{D}} 
      h^2 \right)} \\
      \label{Eqn:NegStabilized_EleStabParam}
      \tau_{\Omega_{\mathrm{e}}} &= - \frac{\widetilde{C} 
      \lambda^2_{\mathrm{min}} h^2_{\Omega_{\mathrm{e}}}}{32 \left(1 + 
      C^2_{pf} \right) \left( nd^2 \lambda^2_{\mathrm{max}} + \widetilde{C} 
      C^2_{pf} \delta_{\alpha \mathbf{v}} h^2 + \widetilde{C} 
      \delta_{\mathbf{D}} h^2 \right)}
    \end{align}
  \end{subequations}
  then for all $\left(w_h,\mathbf{p}_h \right) \in 
  \mathcal{W}_h \times \mathcal{Q}_h$ there exist 
  two constants $C_{\mathrm{NgStb}0} > 0$ and 
  $C_{\mathrm{NgStb}4} > 0$ independent of $\mathbf{D}$ 
  and $h$ such that we have:
  
  \noindent \textbf{Coercivity}
  \begin{align}
    \label{Eqn:Coercivity_NSSD_LSFEM1}
    \mathfrak{F}_{\mathrm{NgStb}} \left( \left(w_h,\mathbf{p}_h 
    \right), f=0 \right) &\geq \frac{11 \gamma^2_{\mathrm{min}}
    \lambda^2_{\mathrm{min}} \| \mathrm{grad}[w_h] \|^2_0}{32 
    (1 + C^2_{pf})} \nonumber \\
    &+ \sum_{\Omega_{\mathrm{e}} \in \Omega_h} \frac{ \widetilde{C} 
    \gamma^2_{\mathrm{min}}\lambda^2_{\mathrm{min}} h^2_{\Omega_{\mathrm{e}}} \| 
    \mathbf{v} \bullet \mathrm{grad}[w_h] \|^2_{0,\Omega_{\mathrm{e}}}}{32 
    (1 + C^2_{pf}) \left( nd^2 \lambda^2_{\mathrm{max}} + \widetilde{C} C^2_{pf} 
    \delta_{\alpha \mathbf{v}} h^2 + \widetilde{C} \delta_{\mathbf{D}} 
    h^2 \right)}
  \end{align}

  \noindent \textbf{Boundedness estimate}
  \begin{align}
    \label{Eqn:Coercivity_NSSD_LSFEM2}
    C_{\mathrm{NgStb}1} \| w_h \|^2_1 + C_{\mathrm{NgStb}2} \| 
    \mathbf{p}_h \|^2_{\mathrm{div}} &+ C_{\mathrm{NgStb}3} \| 
    \mathbf{v} \bullet \mathrm{grad}[w_h] \|^2_0 \leq 
    \mathfrak{F}_{\mathrm{NgStb}} \left( \left(w_h,\mathbf{p}_h 
    \right), f=0 \right) \nonumber \\
    &\leq C_{\mathrm{NgStb}4} \left(\| w_h \|^2_1 + \| \mathbf{p}
    _h \|^2_{\mathrm{div}} + \| \mathbf{v} \bullet \mathrm{grad}[w_h] 
    \|^2_0 \right)
  \end{align}
  where the constant $\gamma_{\mathrm{min}}$ is given by equation 
  \eqref{Eqn:MinWeight_PrimLSFEM}. The constants 
  $\delta_{\alpha \mathbf{v}}$, $\delta_{\mathbf{D}}$,
  $C_{\mathrm{NgStb}1}$, $C_{\mathrm{NgStb}2}$, and 
  $C_{\mathrm{NgStb}2}$ are given as follows:
  \begin{subequations}
    \begin{align}
      \label{Eqn:StrDiff_StabConstant1}
      \delta_{\alpha \mathbf{v}} &= \displaystyle 
      \mathop{\mathrm{max}}_{\mathbf{x} \in \overline{\Omega}} 
      \left [ \left(\alpha + \mathrm{div} [\mathbf{v}] \right)^2 
      \right ] \\
      \label{Eqn:StrDiff_StabConstant2}
      \delta_{\mathbf{D}} &= \displaystyle \mathop{\mathrm{max}}_
      {\mathbf{x} \in \overline{\Omega}} \left [ \| 
      \mathrm{div}[\mathbf{D}] \|^2 \right ] \\
      \label{Eqn:Coercivity_Constant_NSSD1}
      C_{\mathrm{NgStb}1} &= C_{\mathrm{NgStb}0} \gamma^2
      _{\mathrm{min}} \lambda^2_{\mathrm{min}} \\
      \label{Eqn:Coercivity_Constant_NSSD2}
      C_{\mathrm{NgStb}2} &=  \frac{C_{\mathrm{NgStb}0} 
      \gamma^2_{\mathrm{min}} \lambda^2_{\mathrm{min}} \delta^2_{\alpha 
      \mathbf{v} \mathbf{D}}}{\left( 1 + \lambda^2_{\mathrm{min}} 
      + \lambda^2_{\mathrm{max}} \right) \delta^2_{\alpha \mathbf{v} 
      \mathbf{D}} + \delta_{\alpha \mathbf{v} \mathbf{D}} 
      \lambda^2_{\mathrm{min}} h^2 + \delta_{\alpha \mathbf{v}1} 
      \lambda^2_{\mathrm{min}} h^4} \\
      \label{Eqn:Coercivity_Constant_NSSD3}
      C_{\mathrm{NgStb}3} &=\frac{C_{\mathrm{NgStb}0} 
      \gamma^2_{\mathrm{min}} \lambda^2_{\mathrm{min}} h^2}{\delta_{\alpha 
      \mathbf{v} \mathbf{D}}}  
    \end{align}
  \end{subequations}
  The constants $\delta_{\alpha \mathbf{v}1}$ and 
  $\delta_{\alpha \mathbf{v} \mathbf{D}}$ in the above 
  equations are defined as follows: 
  \begin{subequations}
    \begin{align}
      \label{Eqn:NSSD_CoerEstStabConst1}
      \delta_{\alpha \mathbf{v}1} &= \displaystyle 
      \mathop{\mathrm{max}}_{\mathbf{x} \in \overline{\Omega}} 
      \left [ \left(\mathrm{grad}[\alpha] \bullet \mathbf{v} + 
      \alpha\mathrm{div} [\mathbf{v}] \right)^2 \right ] \\
      \label{Eqn:NSSD_CoerEstStabConst2}
      \delta_{\alpha \mathbf{v} \mathbf{D}} &= \lambda^2
      _{\mathrm{max}} + \delta_{\alpha \mathbf{v}} h^2 + \delta_
      {\mathbf{D}} h^2
    \end{align}
  \end{subequations}
\end{theorem}
\begin{proof}
  The boundedness estimate is a direct consequence of the 
  triangle inequality. Herein, we shall proceed to show 
  the validity of coercivity estimates, specifically, 
  equation \eqref{Eqn:Coercivity_NSSD_LSFEM1} and the 
  left hand side of \eqref{Eqn:Coercivity_NSSD_LSFEM2}. 
  Let $\mu > 0$ be a constant, which will be determined 
  later. Using equation \eqref{Eqn:MinWeight_PrimLSFEM} 
  and \eqref{Eqn:Weighted_NSSD_LS_Functional} with 
  $f = 0$, we have
  \begin{align}
    \label{Eqn:BiLin_NSSD_LSFEM1}
    \frac{2\mathfrak{F}_{\mathrm{NgStb}}}{\gamma^2_{\mathrm{min}}} &\geq
    \sum_{\Omega_{\mathrm{e}} \in \Omega_h} \Big \| \mathbf{p}_h 
    - w_h \mathbf{v} + \mathbf{D} \mathrm{grad}[w_h] - \delta_
    {\Omega_{\mathrm{e}}} \mathbf{v} \left( \mathrm{div}[w_h 
    \mathbf{v} - \mathbf{D} \mathrm{grad}[w_h]] \right) - \mu 
    \mathrm{grad}[w_h] \Big \|^2_{0,\Omega_{\mathrm{e}}} \nonumber \\
    &+ \sum_{\Omega_{\mathrm{e}} \in \Omega_h} \Big \| 
    \alpha w_h + \mathrm{div}[ \mathbf{p}_h ] + \delta_
    {\Omega_{\mathrm{e}}} \mathrm{div}[\alpha w_h \mathbf{v}] 
    - \mu w_h \Big \|^2_{0,\Omega_{\mathrm{e}}} - \mu^2 
    \| w_h \|^2_{0,\Omega} - \mu^2 \| \mathrm{grad}[w_h]
    \|^2_{0,\Omega} \nonumber \\
    &+ \sum_{\Omega_{\mathrm{e}} \in \Omega_h} \tau_{
    \Omega_{\mathrm{e}}} \Big \| \alpha w_h + \mathrm{div} 
    \big [ w_h \mathbf{v} - \mathbf{D} \mathrm{grad}[ w_h ] 
    \big ] \Big \|^2_{0,\Omega_{\mathrm{e}}}
    + \sum_{\Omega_{\mathrm{e}} \in \Omega_h} 2 \mu \left( 
    \alpha w_h + \mathrm{div}[ \mathbf{p}_h + \delta_
    {\Omega_{\mathrm{e}}} \alpha w_h \mathbf{v}]; w_h \right)
    _{0,\Omega} \nonumber \\
    &+ \sum_{\Omega_{\mathrm{e}} \in \Omega_h} 2\mu \left(\mathbf{p}_h 
    - w_h \mathbf{v} + \mathbf{D} \mathrm{grad}[w_h] - \delta_{\Omega_
    {\mathrm{e}}} \mathbf{v} \left( \mathrm{div}[w_h \mathbf{v} - 
    \mathbf{D} \mathrm{grad}[w_h]] \right) ; \mathrm{grad}[w_h] \right)
    _{0,\Omega}
  \end{align}
  Using Theorem \ref{Thm:Weighted_PrimLSFEM_CoerciveEstimates}, 
  equation \eqref{Eqn:StrDiff_EleStabParam}--\eqref{Eqn:NegStabilized_EleStabParam}, 
  Cauchy-Schwartz inequality, Poincar\'{e}-Friedrichs inequality, 
  Green's formulae, and following inequalities
  \begin{subequations}
    \begin{align}
      \label{Eqn:Inequality_NSSD_1}
      2 \tau_{\Omega_{\mathrm{e}}} \left( \left(\alpha + 
      \mathrm{div}[\mathbf{v}] \right) w_h ; \mathbf{v} 
      \bullet \mathrm{grad}[w_h] \right)_{0,\Omega_e} &\geq 
      \tau_{\Omega_{\mathrm{e}}} \| \left(\alpha 
      + \mathrm{div}[\mathbf{v}] \right) w_h \|^2_{0,\Omega_e} 
      \nonumber \\
      &+ \tau_{\Omega_{\mathrm{e}}} \| \mathbf{v} \bullet 
      \mathrm{grad}[w_h] \|^2_{0,\Omega_e} \\
      \label{Eqn:Inequality_NSSD_2}
      -2 \tau_{\Omega_{\mathrm{e}}} \left( \left(\alpha + 
      \mathrm{div}[\mathbf{v}] \right) w_h; \mathbf{D} \bullet 
      \mathrm{grad}[\mathrm{grad}[w_h]] \right)_{0,\Omega_e} 
      &\geq \tau_{\Omega_{\mathrm{e}}} \| \left(\alpha + 
      \mathrm{div}[\mathbf{v}] \right) w_h \|^2_{0,\Omega_e} 
      \nonumber \\
      &+ \tau_{\Omega_{\mathrm{e}}} \| \mathbf{D} \bullet 
      \mathrm{grad}[\mathrm{grad}[w_h]] \|^2_{0,\Omega_e} \\
      \label{Eqn:Inequality_NSSD_3}
      -2 \tau_{\Omega_{\mathrm{e}}} \left( \mathbf{v} \bullet 
      \mathrm{grad}[w_h] ; \mathbf{D} \bullet \mathrm{grad}
      [\mathrm{grad}[w_h]] \right)_{0,\Omega_e} &\geq \tau_
      {\Omega_{\mathrm{e}}} \| \mathbf{v} \bullet \mathrm{grad}[w_h] 
      \|^2_{0,\Omega_e} \nonumber \\
      &+ \tau_{\Omega_{\mathrm{e}}} \| \mathbf{D} \bullet 
      \mathrm{grad}[\mathrm{grad}[w_h]] \|^2_{0,\Omega_e} \\
      \label{Eqn:Inequality_NSSD_4}
      2 \tau_{\Omega_{\mathrm{e}}} \left( \left(\alpha + 
      \mathrm{div}[\mathbf{v}] \right) w_h ; \mathrm{div}
      [\mathbf{D}] \bullet \mathrm{grad}[w_h] \right)_{0,\Omega_e} 
      &\geq \tau_{\Omega_{\mathrm{e}}} \| \left(\alpha 
      + \mathrm{div}[\mathbf{v}] \right) w_h \|^2_{0,\Omega_e} 
      \nonumber \\
      &+ \tau_{\Omega_{\mathrm{e}}} \| \mathrm{div}
      [\mathbf{D}] \bullet \mathrm{grad}[w_h] \|^2_{0,\Omega_e} \\
      \label{Eqn:Inequality_NSSD_5}
      -2 \tau_{\Omega_{\mathrm{e}}} \left( \mathrm{div}[\mathbf{D}] 
      \bullet \mathrm{grad}[w_h] ; \mathbf{D} \bullet \mathrm{grad}
      [\mathrm{grad}[w_h]] \right)_{0,\Omega_e} &\geq \tau_
      {\Omega_{\mathrm{e}}} \| \mathrm{div}[\mathbf{D}] \bullet 
      \mathrm{grad}[w_h] \|^2_{0,\Omega_e} \nonumber \\
      &+ \tau_{\Omega_{\mathrm{e}}} \| \mathbf{D} \bullet 
      \mathrm{grad}[\mathrm{grad}[w_h]] \|^2_{0,\Omega_e} \\
      \label{Eqn:Inequality_NSSD_6}
      2 \tau_{\Omega_{\mathrm{e}}} \left( \mathbf{v} \bullet 
      \mathrm{grad}[w_h] ; \mathrm{div}[\mathbf{D}] \bullet 
      \mathrm{grad}[w_h] \right)_{0,\Omega_e} &\geq \tau_
      {\Omega_{\mathrm{e}}} \| \mathbf{v} \bullet \mathrm{grad}[w_h] 
      \|^2_{0,\Omega_e} \nonumber \\
      &+ \tau_{\Omega_{\mathrm{e}}} \| \mathrm{div}[\mathbf{D}] 
      \bullet \mathrm{grad}[w_h] \|^2_{0,\Omega_e}
    \end{align}
  \end{subequations}  
  we have the following inequality:
  \begin{align}
    \label{Eqn:Eqn:Inequality_NSSD_7}
    \sum_{\Omega_{\mathrm{e}} \in \Omega_h} \tau_{
    \Omega_{\mathrm{e}}} \Big \| \alpha w_h + \mathrm{div} 
    \big [ w_h \mathbf{v} &- \mathbf{D} \mathrm{grad}[ w_h ] 
    \big ] \Big \|^2_{0,\Omega_{\mathrm{e}}} \geq - 
    \frac{\lambda^2_{\mathrm{min}}}{16 \left(1 + C^2_{pf} \right)} \nonumber \\
    &- \sum_{\Omega_{\mathrm{e}} \in \Omega_h} \frac{ \widetilde{C} 
    \lambda^2_{\mathrm{min}} h^2_{\Omega_{\mathrm{e}}} \| \mathbf{v} \bullet 
    \mathrm{grad}[w_h] \|^2_{0,\Omega_{\mathrm{e}}}}{16 (1 + C^2_{pf}) 
    \left( nd^2 \lambda^2_{\mathrm{max}} + \widetilde{C} C^2_{pf} \delta_{\alpha 
    \mathbf{v}} h^2 + \widetilde{C} \delta_{\mathbf{D}} h^2 \right)}
  \end{align}
  Similarly, using the following equality:
  \begin{align}
    \label{Eqn:Eqn:Inequality_NSSD_8}
    2 \mu \delta_{\Omega_{\mathrm{e}}} \left(\mathrm{div}[
    \alpha w_h \mathbf{v}]; w_h \right)_{0,\Omega_e} = - 2 \mu 
    \delta_{\Omega_{\mathrm{e}}} \left( \alpha w_h; \mathbf{v} 
    \bullet \mathrm{grad}[w_h] \right)_{0,\Omega_e} = \mu 
    \delta_{\Omega_{\mathrm{e}}} \left(\mathrm{div}[\alpha 
    \mathbf{v}] w_h; w_h \right)_{0,\Omega_e}
  \end{align}
  in combination with the following inequalities
  \begin{subequations}
    \begin{align}
      \label{Eqn:Eqn:Inequality_NSSD_9}
      -2 \mu \delta_{\Omega_{\mathrm{e}}} \left( \left(\alpha + 
      \mathrm{div}[\mathbf{v}] \right) w_h ; \mathbf{v} 
      \bullet \mathrm{grad}[w_h] \right)_{0,\Omega_e} &\geq 
      2 \mu \delta_{\Omega_{\mathrm{e}}} \| \left(\alpha 
      + \mathrm{div}[\mathbf{v}] \right) w_h \|^2_{0,\Omega_e} 
      \nonumber \\
      &+ \frac{\mu \delta_{\Omega_{\mathrm{e}}}}{2} \| \mathbf{v} 
      \bullet \mathrm{grad}[w_h] \|^2_{0,\Omega_e} \\
      \label{Eqn:Eqn:Inequality_NSSD_10}
      2 \mu \delta_{\Omega_{\mathrm{e}}} \left(\mathrm{div}
      [\mathbf{D}] \bullet \mathrm{grad}[w_h] ; \mathbf{v} 
      \bullet \mathrm{grad}[w_h] \right)_{0,\Omega_e} &\geq 
      2 \mu \delta_{\Omega_{\mathrm{e}}} \| \mathrm{div}
      [\mathbf{D}] \bullet \mathrm{grad}[w_h] \|^2_{0,\Omega_e} 
      \nonumber \\
      &+ \frac{\mu \delta_{\Omega_{\mathrm{e}}}}{2} \| \mathbf{v} 
      \bullet \mathrm{grad}[w_h] \|^2_{0,\Omega_e} \\
      \label{Eqn:Eqn:Inequality_NSSD_11}
      2 \mu \delta_{\Omega_{\mathrm{e}}} \left(\mathbf{D} 
      \bullet \mathrm{grad}[\mathrm{grad}[w_h]] ; \mathbf{v} 
      \bullet \mathrm{grad}[w_h] \right)_{0,\Omega_e} &\geq 
      2 \mu \delta_{\Omega_{\mathrm{e}}} \| \mathbf{D} \bullet 
      \mathrm{grad}[\mathrm{grad}[w_h]] \|^2_{0,\Omega_e} 
      \nonumber \\
      &+ \frac{\mu \delta_{\Omega_{\mathrm{e}}}}{2} \| \mathbf{v} 
      \bullet \mathrm{grad}[w_h] \|^2_{0,\Omega_e}
    \end{align}
  \end{subequations}
  and choosing $\mu = \frac{\lambda_{\mathrm{min}}}{1 + C^2_{pf}}$, 
  equation \eqref{Eqn:BiLin_NSSD_LSFEM1} reduces to the 
  following:
  \begin{align}
    \label{Eqn:Eqn:Inequality_NSSD_12}
    \frac{2\mathfrak{F}_{\mathrm{NgStb}}}{\gamma^2_{\mathrm{min}}} &\geq
    \frac{3\lambda^2_{\mathrm{min}}}{4 \left(1 + C^2_{pf} \right)} +
    \sum_{\Omega_{\mathrm{e}} \in \Omega_h} \frac{ \widetilde{C} 
    \lambda^2_{\mathrm{min}} h^2_{\Omega_{\mathrm{e}}} \| \mathbf{v} \bullet 
    \mathrm{grad}[w_h] \|^2_{0,\Omega_{\mathrm{e}}}}{8 (1 + C^2_{pf}) 
    \left( nd^2 \lambda^2_{\mathrm{max}} + \widetilde{C} C^2_{pf} \delta_{\alpha 
    \mathbf{v}} h^2 + \widetilde{C} \delta_{\mathbf{D}} h^2 \right)} 
    \nonumber \\
    &+ \sum_{\Omega_{\mathrm{e}} \in \Omega_h} \tau_{
    \Omega_{\mathrm{e}}} \Big \| \alpha w_h + \mathrm{div} 
    \big [ w_h \mathbf{v} - \mathbf{D} \mathrm{grad}[ w_h ] 
    \big ] \Big \|^2_{0,\Omega_{\mathrm{e}}}
  \end{align}
  From equations \eqref{Eqn:Eqn:Inequality_NSSD_7} and \eqref{Eqn:Eqn:Inequality_NSSD_12}, 
  we get the desired result given by equation \eqref{Eqn:Coercivity_NSSD_LSFEM1}.
  The second part of the proof is similar to Theorem \ref{Thm:Weighted_PrimLSFEM_CoerciveEstimates}. 
  These exist a constant $C_{\alpha \mathbf{v} \mathbf{D}} > 0$ (for instance, 
  $C_{\alpha \mathbf{v} \mathbf{D}} = \displaystyle \mathop{\mathrm{max}} 
  \left [nd^2, \widetilde{C}, \widetilde{C}C^2_{pf} \right]$) such that 
  \begin{subequations}
    \begin{align}
      \label{Eqn:Eqn:Inequality_NSSD_12}
      nd^2 \lambda^2_{\mathrm{max}} + \widetilde{C} C^2_{pf} \delta_{\alpha 
      \mathbf{v}} h^2 + \widetilde{C} \delta_{\mathbf{D}} h^2 
      &\leq C_{\alpha \mathbf{v} \mathbf{D}} \delta_{\alpha \mathbf{v} 
      \mathbf{D}}\\
      \label{Eqn:Eqn:Inequality_NSSD_13}
      \| \mathrm{grad}[w_h] \|^2_0 &\leq \frac{32 (1 + C^2_{pf}) 
      \mathfrak{F}_{\mathrm{NgStb}}}{11 \gamma^2_{\mathrm{min}} 
      \lambda^2_{\mathrm{min}}}\\
      \label{Eqn:Eqn:Inequality_NSSD_14}
      \| \mathbf{v} \bullet \mathrm{grad}[w_h] \|^2_0 &\leq 
      \frac{32 C_{\alpha \mathbf{v} \mathbf{D}} \delta_{\alpha 
      \mathbf{v} \mathbf{D}} (1 + C^2_{pf}) \widehat{C}^2 
      \mathfrak{F}_{\mathrm{NgStb}}}{\widetilde{C} \gamma^2_
      {\mathrm{min}} \lambda^2_{\mathrm{min}} h^2}
    \end{align}
  \end{subequations}
  Using Cauchy-Schwartz inequality on $\| \mathbf{v} \bullet 
  \mathrm{grad}[w_h] \|_0$ and \eqref{Eqn:Eqn:Inequality_NSSD_13} 
  gives
  \begin{align}
    \label{Eqn:Eqn:Inequality_NSSD_15}
    \| \mathbf{v} \bullet \mathrm{grad}[w_h] \|^2_0 \leq 
    \| \mathbf{v} \|^2_0 \| \mathrm{grad}[w_h] \|^2_0 \leq
    \frac{32 C_{\mathbf{v}}(1 + C^2_{pf}) \mathfrak{F}_
    {\mathrm{NgStb}}}{11 \gamma^2_{min} \lambda^2_{min}}
  \end{align}
  Now, consider the terms $\| w_h \|^2_1$ and $\| \mathbf{p}_h 
  \|^2_{\mathrm{div}}$:
  \begin{subequations}
    \begin{align}
      \label{Eqn:CoerciveEst2_NSSD1S}
      \| w_h \|^2_1 &= \| w_h \|^2_0 + \| 
      \mathrm{grad}[w_h] \|^2_0 \leq \frac{32 (1 + C^2_{pf})^2 
      \mathfrak{F}_{\mathrm{NgStb}}}{11 \gamma^2_{min} \lambda^2
      _{min}} \\
      \label{Eqn:CoerciveEst2_NSSD2}
      \| \mathbf{p}_h \|^2_{0} &\leq 2 \sum_{\Omega_{\mathrm{e}} 
      \in \Omega_h} \Big \| \mathbf{p}_h - w_h \mathbf{v} + 
      \mathbf{D} \mathrm{grad}[w_h] - \delta_{\Omega_{\mathrm{e}}} 
      \mathbf{v} \left( \mathrm{div}[w_h \mathbf{v} - \mathbf{D} 
      \mathrm{grad}[w_h]] \right) \Big \|^2_{0,\Omega_{\mathrm{e}}}
      \nonumber \\
      &+ 2\sum_{\Omega_{\mathrm{e}} \in \Omega_h} \Big \|- w_h 
      \mathbf{v} + \mathbf{D} \mathrm{grad}[w_h] - \delta_
      {\Omega_{\mathrm{e}}} \mathbf{v} \left( \mathrm{div}[w_h 
      \mathbf{v} - \mathbf{D} \mathrm{grad}[w_h]] \right) \Big 
      \|^2_{0,\Omega_{\mathrm{e}}} \\
      \label{Eqn:CoerciveEst2_NSSD3}
      \| \mathrm{div}[\mathbf{p}_h] \|^2_{0} &\leq 2\sum_{\Omega_
      {\mathrm{e}} \in \Omega_h} \Big \| \alpha w_h + \mathrm{div}[ 
      \mathbf{p}_h ] + \delta_{\Omega_{\mathrm{e}}} \mathrm{div}
      [\alpha w_h \mathbf{v}] \Big \|^2_{0,\Omega_{\mathrm{e}}} 
      \nonumber \\
      &+ 2\sum_{\Omega_{\mathrm{e}} \in \Omega_h} \Big \| \alpha w_h 
      + \delta_{\Omega_{\mathrm{e}}} \mathrm{div}[\alpha w_h \mathbf{v}] 
      \Big \|^2_{0,\Omega_{\mathrm{e}}}
    \end{align}
  \end{subequations}
  Using \eqref{Eqn:Eqn:Inequality_NSSD_12}--\eqref{Eqn:Eqn:Inequality_NSSD_15} 
  and repeated use of triangle inequality on \eqref{Eqn:CoerciveEst2_NSSD2} and 
  \eqref{Eqn:CoerciveEst2_NSSD3} gives the boundedness estimate 
  \eqref{Eqn:Coercivity_NSSD_LSFEM2}.
\end{proof}

\begin{theorem}[\texttt{Error estimate for proposed LSFEM}]
  \label{Thm:Weighted_NSSDLSFEM_ErrorEstimates}
  Given that equations \eqref{Eqn:NN_AD_GE_BE}--\eqref{Eqn:NN_AD_GE_NBC} 
  have a sufficiently smooth solution $\left(c,\mathbf{q} \right) \in 
  \left(\mathcal{C} \times \mathcal{Q} \right) \cap \left(H^{r+1}
  (\Omega) \right)^3$. Then the finite element solution $\left(c_h,
  \mathbf{q}_h \right)$ of the unconstrained weighted negatively 
  stabilized streamline diffusion LSFEM satisfies the following 
  error estimate:
  \begin{align}
    \label{Eqn:Error_Estimate_NSSDLSFEM}
    \sqrt{C_{\mathrm{NgStb}1}} \| c - c_h \|_1 + 
    \sqrt{C_{\mathrm{NgStb}2}} \| \mathbf{q} - 
    \mathbf{q}_h \|_{\mathrm{div}} &+ 
    \sqrt{C_{\mathrm{NgStb}3}} \| \mathbf{v} 
    \bullet \mathrm{grad}[c - c_h] \|_0 \nonumber \\
    &\leq C_{\mathrm{NgStb}} h^r \left( \| c \|_{r+1} 
    + \| \mathbf{q} \|_{r+1} \right)
  \end{align}
  where $C_{\mathrm{NgStb}} > 0$ is a constant independent of 
  $\mathbf{D}$ and $h$.
\end{theorem}
\begin{proof}
  Let $c_{I} \in \mathcal{C}_h$ and $\mathbf{q}_{I} 
  \in \mathcal{Q}_h$ be the standard finite element 
  interpolants of $c$ and $\mathbf{q}$, respectively. 
  From the interpolation theory \cite{Bochev_Gunzburger}, 
  we have
  \begin{subequations}
    \begin{align}
      \label{Eqn:FEM_Error_Estimate1}
      \|c - c_{I} \|_1 &\leq C h^r \| c \|_{r+1} \\
      \label{Eqn:FEM_Error_Estimate3}
      \|\mathbf{q} - \mathbf{q}_{I} \|_{\mathrm{div}} 
      &\leq C h^r \| \mathbf{q} \|_{r+1}
    \end{align}
  \end{subequations}
  for some positive constant $C$ independent of $\mathbf{D}$ 
  and $h$. The error $\left(c - c_h, \mathbf{q} - \mathbf{q}_h 
  \right)$ satisfies the following orthogonality property:
  \begin{align}
    \label{Eqn:OrthoProp_NSSDLSFEM}
    \mathfrak{B}_{\mathrm{NgStb}} \left(\left(c_h - c,
    \mathbf{q}_h -\mathbf{q} \right); \left(w_h,\mathbf{p}_h 
    \right) \right) = 0 \quad \forall \left(w_h,\mathbf{p}_h 
    \right) \in \mathcal{W}_h \times \mathcal{Q}_h
  \end{align}
  Cauchy-Schwartz inequality implies:
   \begin{align}
    \label{Eqn:OrthoProp_NSSDLSFEM_Result}
    \mathfrak{B}^{1/2}_{\mathrm{NgStb}} \left(\left(c_h - c_I,
    \mathbf{q}_h -\mathbf{q}_I \right); \left(c_h - c_I,
    \mathbf{q}_h -\mathbf{q}_I \right) \right)
    \leq \mathfrak{B}^{1/2}_{\mathrm{NgStb}} \left(\left(c - c_I,
    \mathbf{q} -\mathbf{q}_I \right); \left(c - c_I,
    \mathbf{q} -\mathbf{q}_I \right) \right)
  \end{align}
  From Theorem 
  \ref{Thm:Weighted_NSSDLSFEM_CoerciveEstimates} and interpolation 
  estimates \eqref{Eqn:FEM_Error_Estimate1}--\eqref{Eqn:FEM_Error_Estimate3}, 
  one can obtain the desired error estimate \eqref{Eqn:Error_Estimate_NSSDLSFEM}. 
\end{proof}

From the above mathematical analysis, it is evident 
that the element-dependent stabilization parameters 
$\tau_{\Omega_{\mathrm{e}}} \leq 0$ and $\delta_{\Omega_{\mathrm{e}}} 
\leq 0$ can be taken as:
\begin{subequations}
  \begin{align}
    \label{Eqn:StrDiff_Params}
    \delta_{\Omega_{\mathrm{e}}} &= - \frac{\delta_o 
    \lambda_{\mathrm{min}} h^2_{\Omega_{\mathrm{e}}}}{\left( 
    \lambda^2_{\mathrm{max}} + \delta_1 \displaystyle \mathop{\mathrm{max}}_
    {\mathbf{x} \in \overline{\Omega}} \left [ \left(\alpha + 
    \mathrm{div} [\mathbf{v}] \right)^2 \right ] h^2 + 
    \delta_2 \displaystyle \mathop{\mathrm{max}}_
    {\mathbf{x} \in \overline{\Omega}} \left [ \| 
    \mathrm{div}[\mathbf{D}] \|^2 \right ] h^2 \right)} \\
    \label{Eqn:Stab_Params}
    \tau_{\Omega_{\mathrm{e}}} &= - \frac{\tau_o 
    \lambda^2_{\mathrm{min}} h^2_{\Omega_{\mathrm{e}}}}{\left( 
    \lambda^2_{\mathrm{max}} + \tau_1 \displaystyle \mathop{\mathrm{max}}_
    {\mathbf{x} \in \overline{\Omega}} \left [ \left(\alpha + 
    \mathrm{div} [\mathbf{v}] \right)^2 \right ] h^2 + 
    \tau_2 \displaystyle \mathop{\mathrm{max}}_
    {\mathbf{x} \in \overline{\Omega}} \left [ \| 
    \mathrm{div}[\mathbf{D}] \|^2 \right ] h^2 \right)}
  \end{align}
\end{subequations}
where $\delta_o$, $\delta_1$, $\delta_2$, $\tau_o$, $\tau_1$, 
and $\tau_2$ are non-negative constants. 

\begin{remark}
  The mathematical analysis provided by Hsieh and Yang 
  \cite{2009_Hsieh_Yang_CMAME_v199_p183_p196} can be 
  obtained as a special case of the mathematical 
  analysis presented above. Specifically, take 
  $\alpha = 0$, $\mathbf{D}(\mathbf{x})$ to be 
  homogeneous and isotropic, and $\mathbf{v}(\mathbf{x})$ 
  to be solenoidal and constant.
\end{remark}

\section{APPENDIX C:~Finite element stiffness matrices and load vectors}
\label{Sec:NN_AD_Appendix3}
For weighted primitive LSFEM the terms $\boldsymbol{K}_{cc}$, $\boldsymbol{K}
_{c \mathbf{q}}$, $\boldsymbol{K}_{\mathbf{q} \mathbf{q}}$, $\boldsymbol{r}_
{c}$, and $\boldsymbol{r}_{\mathbf{q}}$ are constructed from the local stiffness 
matrices and load vectors $\boldsymbol{K}^e_{cc}$, $\boldsymbol{K}^e_{c \mathbf{q}}$, 
$\boldsymbol{K}^e_{\mathbf{q} \mathbf{q}}$, $\boldsymbol{r}^e_{c}$, and $\boldsymbol{r}
^e_{\mathbf{q}}$ through the standard finite element assembly process. The expressions 
for these element stiffness matrices and element load vectors in terms of shape 
functions and their derivatives are explicitly defined as follows:
\begin{align}
  \label{Eqn:Kcc_WeighPrimLSFEM}
  \boldsymbol{K}^e_{cc} &= \int \limits_{\Omega_e} 
  \left( \beta^2 \alpha^2 \right) \boldsymbol{N}^{\mathrm{T}} 
  \boldsymbol{N} \; \mathrm{d} \Omega_e \, + \, \int \limits_{\Omega_e} 
  \boldsymbol{N}^{\mathrm{T}} \mathbf{v}^{\mathrm{T}} \mathbf{A}^2 
  \mathbf{v} \boldsymbol{N} \; \mathrm{d} \Omega_e \, - \, \int 
  \limits_{\Omega_e} \left(\boldsymbol{DN} \boldsymbol{J}^{-1} \right)
  \mathbf{D} \mathbf{A}^2 \mathbf{v} \boldsymbol{N} \; \mathrm{d} 
  \Omega_e \nonumber \\
  &- \int \limits_{\Omega_e} \boldsymbol{N}^{\mathrm{T}} \mathbf{v}
  ^{\mathrm{T}} \mathbf{A}^2 \mathbf{D} \left(\boldsymbol{DN} 
  \boldsymbol{J}^{-1} \right)^{\mathrm{T}} \; \mathrm{d} \Omega_e 
  + \int \limits_{\Omega_e} \left(\boldsymbol{DN} \boldsymbol{J}^{-1} 
  \right) \mathbf{D} \mathbf{A}^2 \mathbf{D} \left(\boldsymbol{DN} 
  \boldsymbol{J}^{-1} \right)^{\mathrm{T}} \; \mathrm{d} \Omega_e 
  \nonumber \\ 
  &+ \int \limits_{\Omega_{\mathrm{e}} \in {\Gamma^{q}}} \left( 
  \frac{1 + \mathrm{Sign}[\mathbf{v} \bullet \widehat{\mathbf{n}}]}{2} 
  \right)^2 \left(\mathbf{v} \bullet \widehat{\mathbf{n}} \right)^2
  \boldsymbol{N}^{\mathrm{T}} \boldsymbol{N} \; \mathrm{d} \Gamma^q_e
\end{align}
\begin{align}
  \label{Eqn:Kcq_WeighPrimLSFEM}
  \boldsymbol{K}^e_{c \mathbf{q}} &= \int \limits_{\Omega_e} 
  \left( \beta^2 \alpha \right) \boldsymbol{N}^{\mathrm{T}} 
  \left( \mathrm{vec} \left[\left(\boldsymbol{DN} \boldsymbol{J}^{-1} 
  \right)^{\mathrm{T}} \right] \right)^{\mathrm{T}} \; \mathrm{d} 
  \Omega_e \, - \, \int \limits_{\Omega_e} \boldsymbol{N}^{\mathrm{T}} 
  \mathbf{v}^{\mathrm{T}} \mathbf{A}^2 \left(\boldsymbol{N} \odot 
  \boldsymbol{I} \right) \; \mathrm{d} \Omega_e \nonumber \\
  &+ \int \limits_{\Omega_e} \left(\boldsymbol{DN} \boldsymbol{J}^{-1} 
  \right) \mathbf{D} \mathbf{A}^2 \left(\boldsymbol{N} \odot \boldsymbol{I} 
  \right) \; \mathrm{d} \Omega_e \, \nonumber \\
  &- \, \int \limits_{\Omega_{\mathrm{e}} \in {\Gamma^{q}}} \left( \frac{1 + 
  \mathrm{Sign}[\mathbf{v} \bullet \widehat{\mathbf{n}}]}{2} \right) 
  \left(\mathbf{v} \bullet \widehat{\mathbf{n}} \right) \boldsymbol{N}
  ^{\mathrm{T}} \widehat{\mathbf{n}}^{\mathrm{T}} \left(\boldsymbol{N} 
  \odot \boldsymbol{I} \right) \; \mathrm{d} \Gamma^q_e
\end{align}
\begin{align}
  \label{Eqn:Kqq_WeighPrimLSFEM}
  \boldsymbol{K}^e_{\mathbf{q} \mathbf{q}} &= \int \limits_{\Omega_e} 
  \beta^2  \left( \mathrm{vec} \left[\left(\boldsymbol{DN} \boldsymbol{J}
  ^{-1} \right)^{\mathrm{T}} \right] \right) \left( \mathrm{vec} 
  \left[\left(\boldsymbol{DN} \boldsymbol{J}^{-1} \right)^{\mathrm{T}} 
  \right] \right)^{\mathrm{T}} \; \mathrm{d} \Omega_e \, + \, \int 
  \limits_{\Omega_e} \left(\boldsymbol{N}^{\mathrm{T}} \odot \boldsymbol{I} 
  \right) \mathbf{A}^2 \left(\boldsymbol{N} \odot \boldsymbol{I} \right)
  \; \mathrm{d} \Omega_e \nonumber \\
  &+ \int \limits_{\Omega_{\mathrm{e}} \in {\Gamma^{q}}} \left(\boldsymbol{N}
  ^{\mathrm{T}} \odot \boldsymbol{I} \right) \widehat{\mathbf{n}} \, \widehat{\mathbf{n}}
  ^{\mathrm{T}} \left(\boldsymbol{N} \odot \boldsymbol{I} \right) \; \mathrm{d} 
  \Gamma^q_e
\end{align}
Correspondingly, the expressions for the element load vectors in terms of 
shape functions and their derivatives are explicitly defined as follows:
\begin{align}
  \label{Eqn:rc_WeighPrimLSFEM}
  \boldsymbol{r}^e_{c} &= \int \limits_{\Omega_e} \left( \beta^2 \alpha f 
  \right) \boldsymbol{N}^{\mathrm{T}} \; \mathrm{d} \Omega_e \, - \, 
  \int \limits_{\Omega_{\mathrm{e}} \in {\Gamma^{q}}} \left( \frac{1 + 
  \mathrm{Sign}[\mathbf{v} \bullet \widehat{\mathbf{n}}]}{2} \right) 
  \left(\mathbf{v} \bullet \widehat{\mathbf{n}} \right) \boldsymbol{N}
  ^{\mathrm{T}} q^{\mathrm{p}} \; \mathrm{d} \Gamma^q_e
\end{align}
\begin{align}
  \label{Eqn:rq_WeighPrimLSFEM}
  \boldsymbol{r}^e_{\mathbf{q}} &= \int \limits_{\Omega_e} \left( \beta^2 
  f \right) \mathrm{vec} \left[\left(\boldsymbol{DN} \boldsymbol{J}^{-1} 
  \right)^{\mathrm{T}} \right] \; \mathrm{d} \Omega_e \, + \, \int \limits
  _{\Omega_{\mathrm{e}} \in {\Gamma^{q}}} \left(\boldsymbol{N}^{\mathrm{T}} 
  \odot \boldsymbol{I} \right) \widehat{\mathbf{n}} \, q^{\mathrm{p}} \; 
  \mathrm{d} \Gamma^q_e
\end{align}
It should be noted that these terms are obtained from the bilinear 
and linear forms of the weighted primitive least-squares functional 
$\mathfrak{F}_{\mathrm{Prim}}$, which are given as follows:
\begin{align}
  \label{Eqn:BiLin_WeighPrimLSFEM1}
  \mathfrak{B}_{\mathrm{Prim}} \left(\left(c,\mathbf{q} 
  \right); \left(w,\mathbf{p} \right) \right) &= \left(w; 
  \beta^2 \alpha^2 c \right) + \left(w; \left(\mathbf{v} 
  \bullet \mathbf{A}^2 \mathbf{v}  \right) c \right) - 
  \left(\mathrm{grad}[w]; \left(\mathbf{D} A^2 \mathbf{v} 
  \right) c \right) \nonumber \\
  &- \left(w ; \mathbf{v} \bullet A^2 \mathbf{D} \, \mathrm{grad}[c] 
  \right) + \left(\mathrm{grad}[w]; \left(\mathbf{D} A^2 \mathbf{D} 
  \right) \mathrm{grad}[c] \right) \nonumber \\
  &+\left(w; \beta^2 \alpha \, \mathrm{div}[\mathbf{q}] \right) + 
  \left(\mathrm{div}[\mathbf{p}]; \beta^2 \alpha c \right) - \left(w; 
  \mathbf{v} \bullet \mathbf{A}^2 \mathbf{q} \right) - \left(\mathbf{p}; 
  \mathbf{A}^2 \mathbf{v} c \right) \nonumber \\
  &+ \left(\mathrm{grad}[w]; \mathbf{D} A^2 \,\mathbf{q} \right) + 
  \left(\mathbf{p}; A^2 \mathbf{D} \, \mathrm{grad}[c] \right) +
  \left(\mathrm{div}[\mathbf{p}]; \beta^2 \mathrm{div}[\mathbf{q}] 
  \right) \nonumber \\
  &+ \left(\mathbf{p}; A^2 \mathbf{q} \right) + \left(w; \left( 
  \frac{1 + \mathrm{Sign}[\mathbf{v} \bullet \widehat{\mathbf{n}}]}{2} 
  \right)^2 \left(\mathbf{v} \bullet \widehat{\mathbf{n}} \right)^2 
  \, c \right)_{\Gamma^{q}} \nonumber \\
  &- \left(w; \left( \frac{1 + \mathrm{Sign}[\mathbf{v} \bullet
  \widehat{\mathbf{n}}]}{2} \right) \left(\mathbf{v} \bullet 
  \widehat{\mathbf{n}} \right) \, \mathbf{q} \bullet \widehat{\mathbf{n}} 
  \right)_{\Gamma^{q}} + \left(\mathbf{p} \bullet \widehat{\mathbf{n}}; 
  \mathbf{q} \bullet \widehat{\mathbf{n}} \right)_{\Gamma^{q}} \nonumber \\
  &- \left(\mathbf{p} \bullet \widehat{\mathbf{n}}; 
  \left( \frac{1 + \mathrm{Sign}[\mathbf{v} \bullet \widehat{\mathbf{n}}]}{2} 
  \right) \left(\mathbf{v} \bullet \widehat{\mathbf{n}} \right) c \right)_
  {\Gamma^{q}}
\end{align}
\begin{align}
  \label{Eqn:BiLin_WeighPrimLSFEM2}
  \mathfrak{L}_{\mathrm{Prim}} \left(\left(w,\mathbf{p} 
  \right) \right) &= \left(w; \beta^2 \alpha f \right) + 
  \left(\mathrm{div}[\mathbf{p}]; \beta^2 f \right) - 
  \left(w; \left( \frac{1 + \mathrm{Sign}[\mathbf{v} 
  \bullet \widehat{\mathbf{n}}]}{2} \right) \left(\mathbf{v} 
  \bullet \widehat{\mathbf{n}} \right) \; q^{\mathrm{p}} 
  \right)_{\Gamma^{q}} \nonumber \\
  &+ \left(\mathbf{p} \bullet \widehat{\mathbf{n}};q^{\mathrm{p}} 
  \right)_{\Gamma^{q}}
\end{align}
Similarly, one can derive the stiffness matrices and load vectors for weighted 
negatively stabilized streamline diffusion LSFEM $\mathfrak{F}_{\mathrm{NgStb}}$. 
For sake of saving space, herein we shall not explicitly define them as the 
bilinear and linear forms of $\mathfrak{F}_{\mathrm{NgStb}}$ have more than 
fifty terms (from which $\boldsymbol{K}_{c \mathbf{q}}$, $\boldsymbol{K}_{\mathbf{q} 
\mathbf{q}}$, $\boldsymbol{r}_{c}$, and $\boldsymbol{r}_{\mathbf{q}}$ are derived).

\section*{ACKNOWLEDGMENTS}
The authors acknowledge the support from the DOE Nuclear 
Energy University Programs (NEUP). The opinions expressed 
in this paper are those of the authors and do not necessarily 
reflect that of the sponsors. 

\bibliographystyle{unsrt}
\bibliography{Master_References,Books}


\begin{figure}
  \centering
  \includegraphics[scale=0.4,clip]{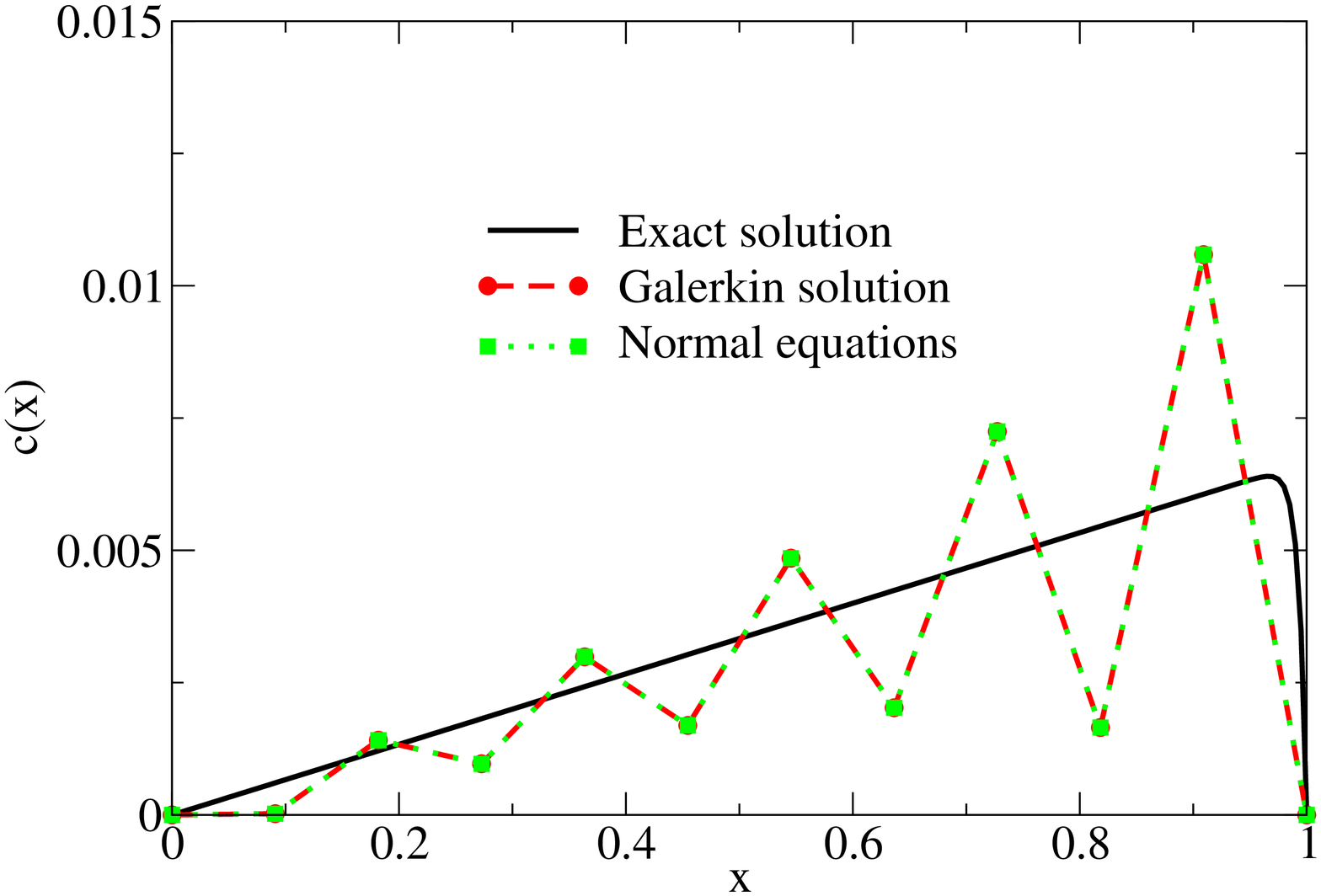}
  \caption{\textsf{Academic problem:}~This figure compares 
    the numerical solutions under the standard single-field 
    Galerkin formulation and the normal equations approach 
    with the exact solution. The normal equations approach 
    does not eliminate node-to-node spurious oscillations. 
    \label{Fig:Normal_spurious}}
\end{figure}

    
\begin{figure}
  \psfrag{A}{$(0,0)$}
  \psfrag{B}{$(1,0)$}
  \psfrag{C}{$(0,1)$}
  \psfrag{BC1}{$c = \sin(\pi x)$}
  \psfrag{BC2}{\rotatebox{-90}{$c = 0$}}
  \psfrag{BC3}{$c = 0$}
  \psfrag{BC4}{\rotatebox{90}{$c = 0$}}
  \psfrag{v}{$\mathbf{v} = \mathbf{\hat{e}}_y$}
  \psfrag{D}{$D = 0.01$}
  \psfrag{x}{$x$}
  \psfrag{y}{$y$}
  \includegraphics[scale=0.9]{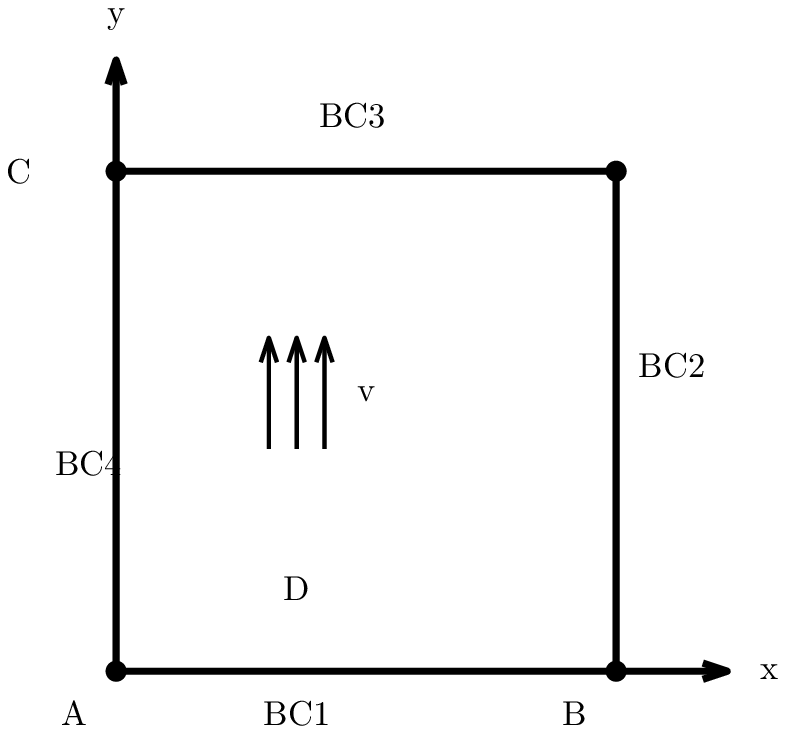}
  \caption{\textsf{Numerical $h$-convergence study:}~A 
    pictorial description of the two-dimensional boundary 
    value problem used in the numerical convergence analysis. 
    Dirichlet boundary conditions are prescribed on the entire 
    boundary. \label{Fig:2D_Numerical_Convergence}}
\end{figure}

\begin{figure}
  \begin{center}
  \subfigure[Mesh using T3 elements]{\includegraphics[scale=0.35]
    {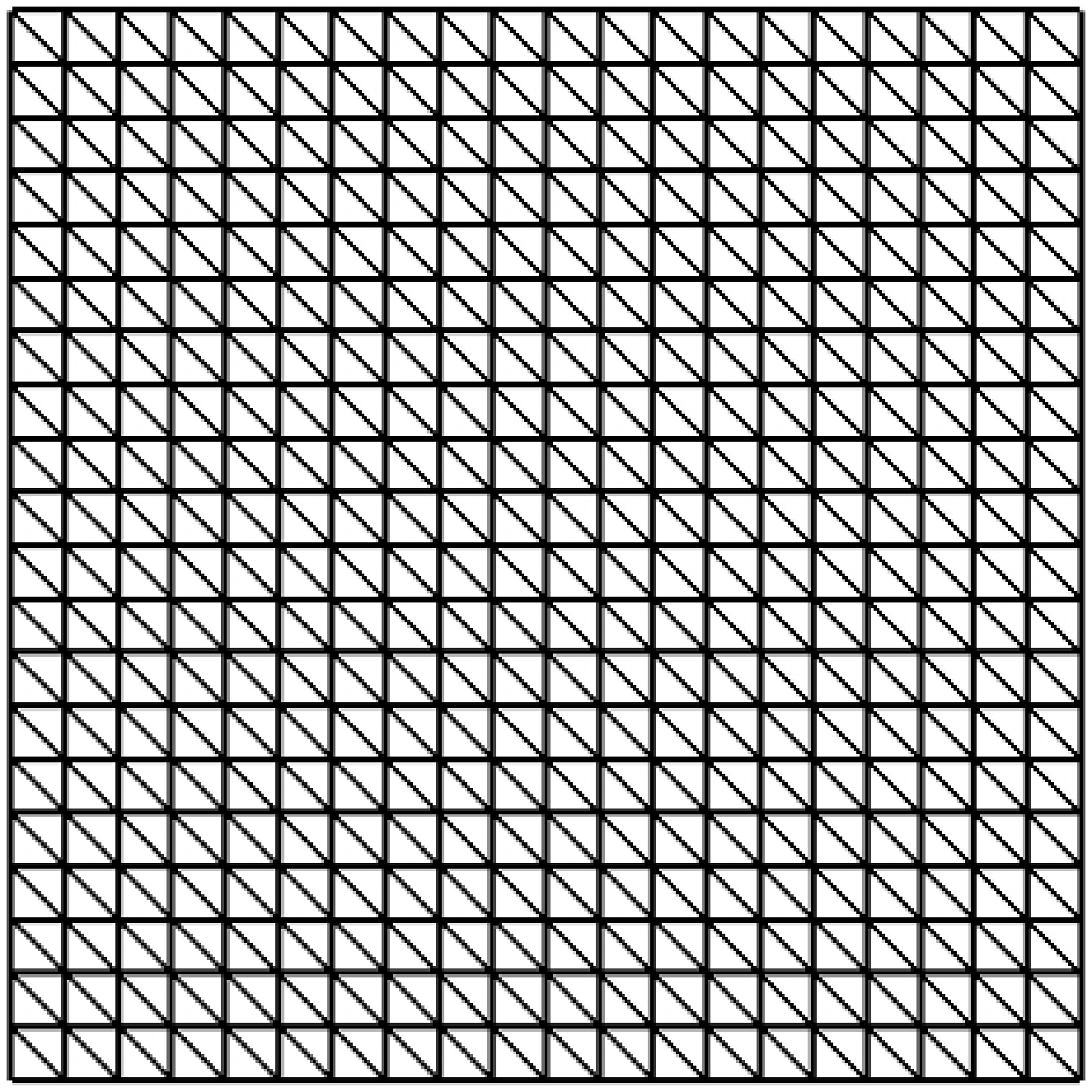}}
  \subfigure[Mesh using Q4 elements]{\includegraphics[scale=0.35]
    {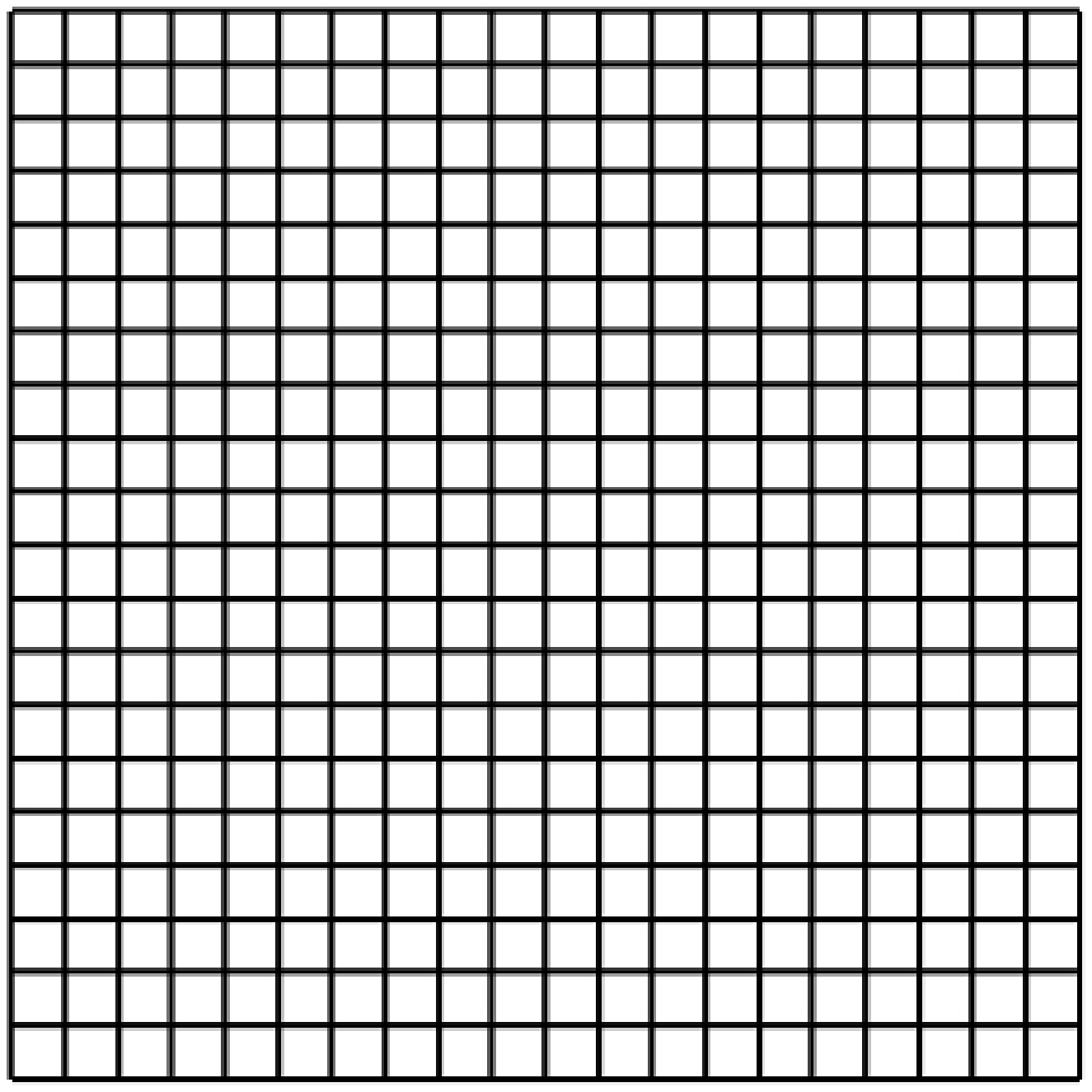}}
  \end{center}
  \caption{\textsf{Numerical $h$-convergence study:}~This 
    figure shows the typical computational meshes used in 
    the numerical convergence analysis. The meshes shown 
    in this figure have 21 nodes along each side of the 
    computational domain (i.e., XSeed = YSeed = 21). A 
    series of hierarchical computational meshes are 
    employed in the study with $11 \times 11$, $21 
    \times 21$, $41 \times 41$ and $81 \times 81$ 
    nodes.
  \label{Fig:2D_Numerical_Convergence_HierarchicalMeshes}}
\end{figure}

\begin{figure}
  \begin{center}
  \subfigure[Concentration: No constraints]{\includegraphics[scale=0.12]
    {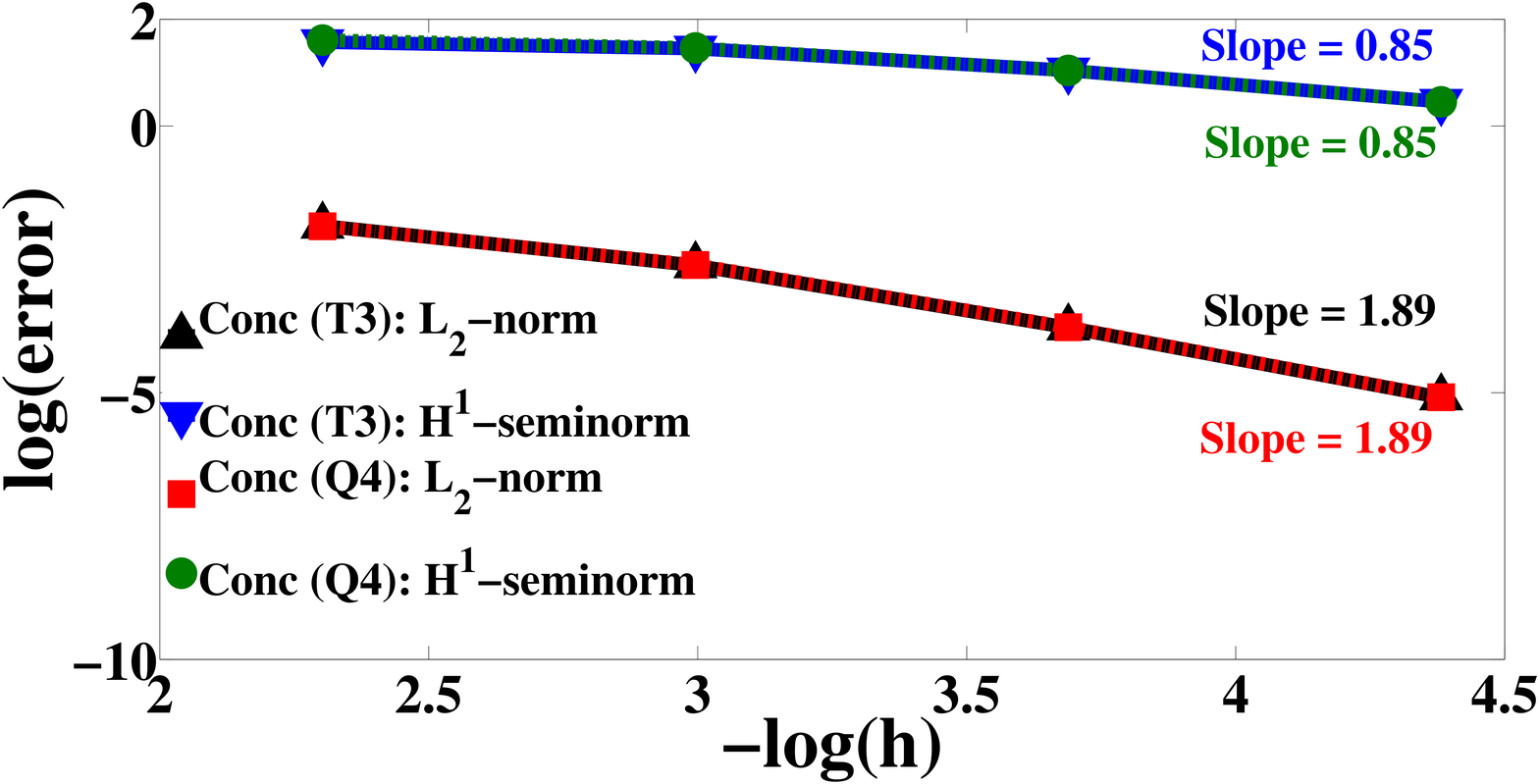}}  
  \subfigure[Concentration: LSB constraints]{\includegraphics[scale=0.12]
    {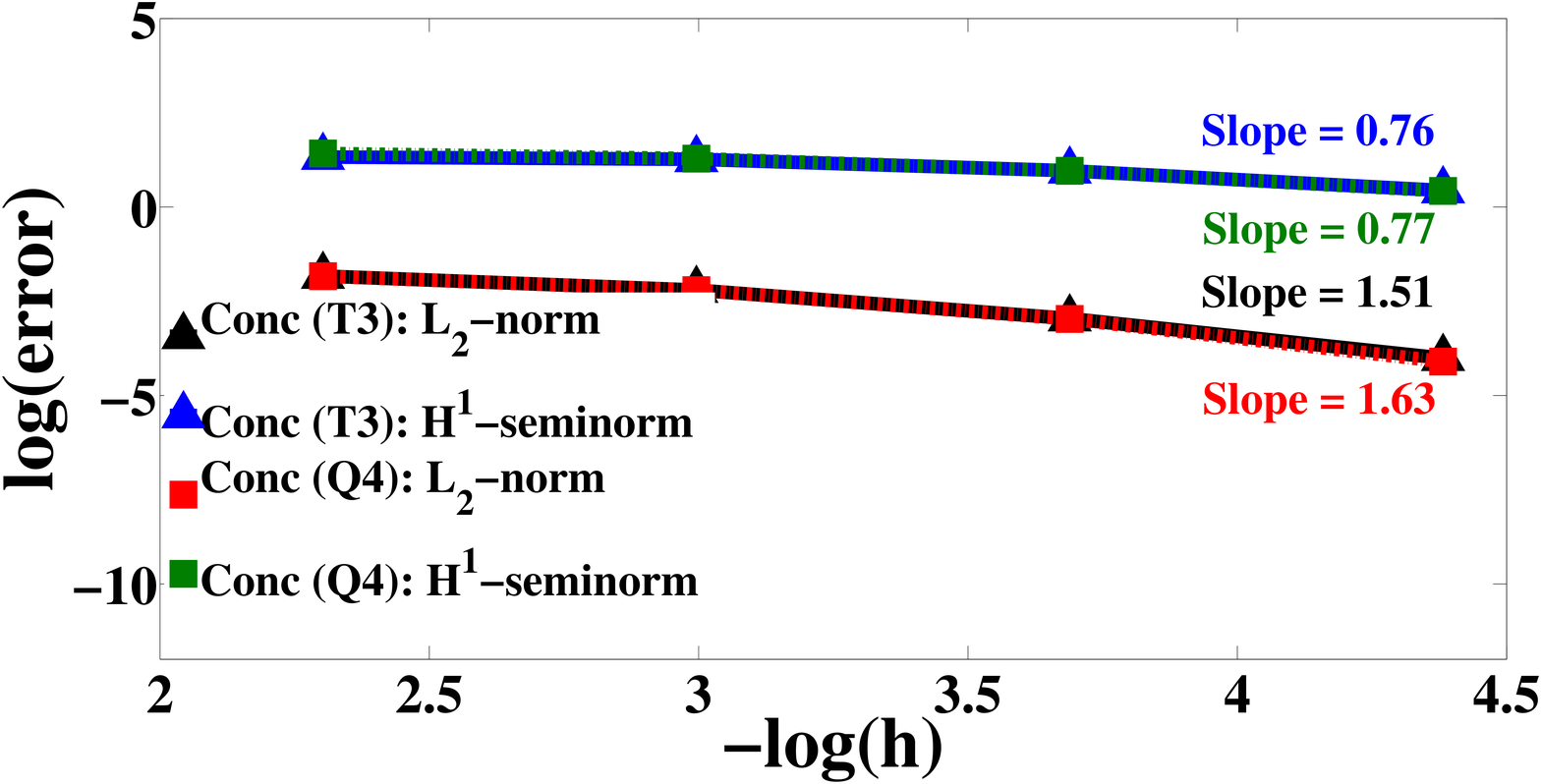}}    
  \subfigure[Flux: No constraints]{\includegraphics[scale=0.12]
    {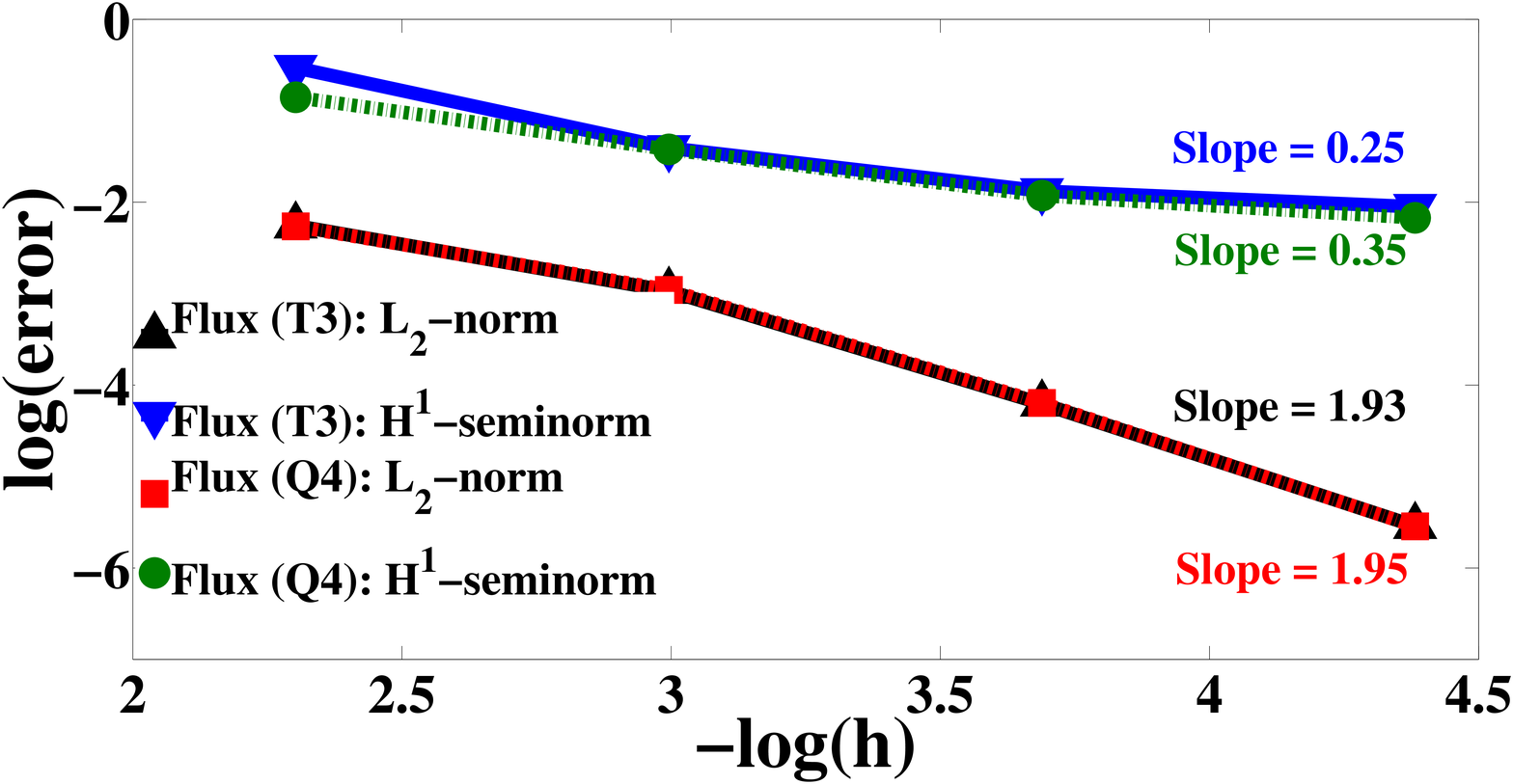}}  
  \subfigure[Flux: LSB constraints]{\includegraphics[scale=0.12]
    {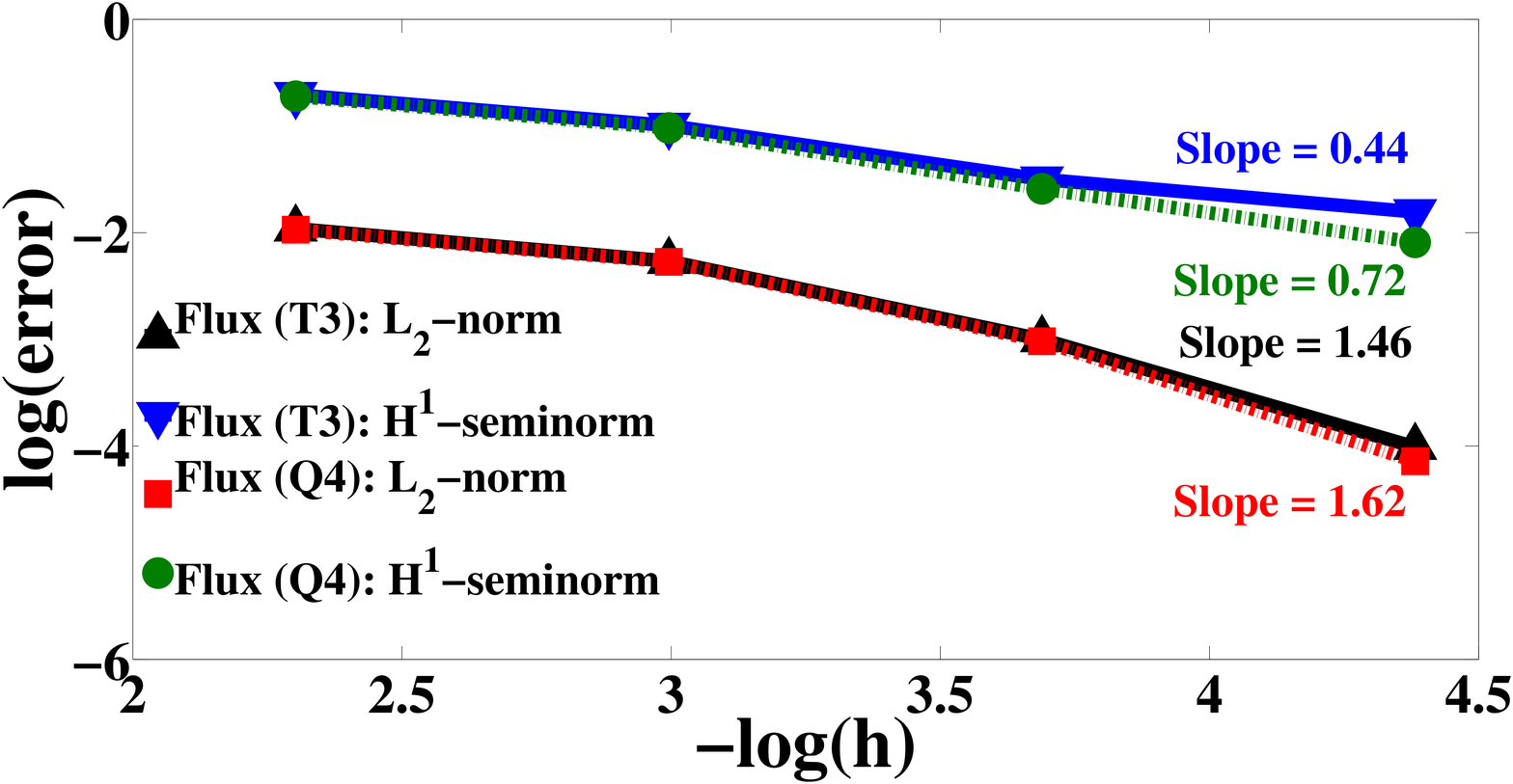}}      
  \end{center}
  \caption{\textsf{Numerical $h$-convergence study:}~This 
    figure shows the convergence rates for the concentration 
    and flux vector in $L_2$-norm and $H^1$-semi-norm with and 
    without LSB constraints. Convergence studies are performed 
    using T3- and Q4-based meshes under the negatively stabilized 
    streamline diffusion LSFEM. It is evident that the Q4 element 
    slightly outperforms the T3 element in terms of rates of convergence.
  \label{Fig:2D_NumConvDiff1by100_ConcFluxT3Q4_NoAndLSB_Constraints}}
\end{figure}


\begin{figure}
  \begin{center}
  \subfigure[T3 mesh: Error in LSB]
    {\includegraphics[scale=0.31]
    {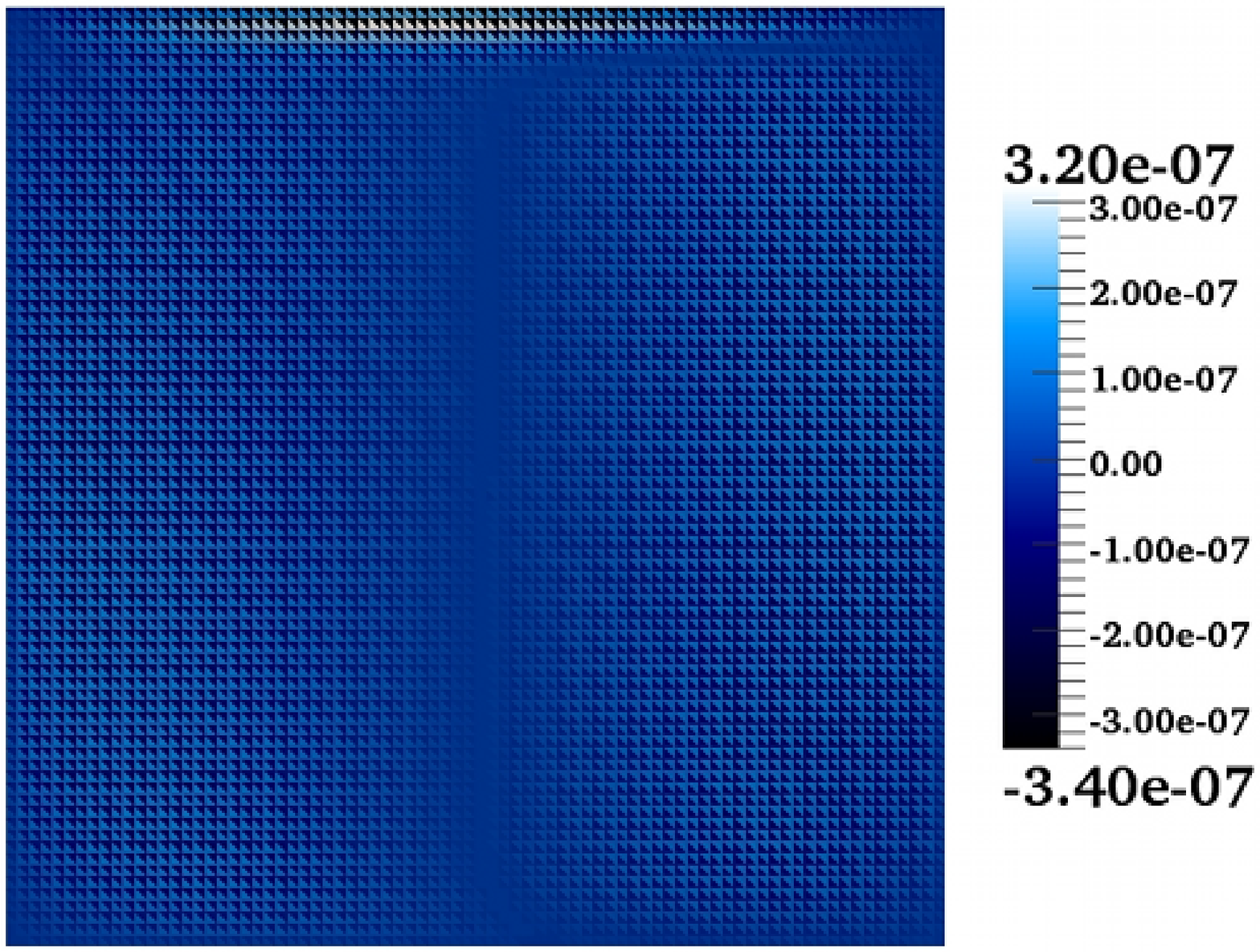}}
  \subfigure[T3 mesh: Lagrange multiplier enforcing LSB]
    {\includegraphics[scale=0.31]
    {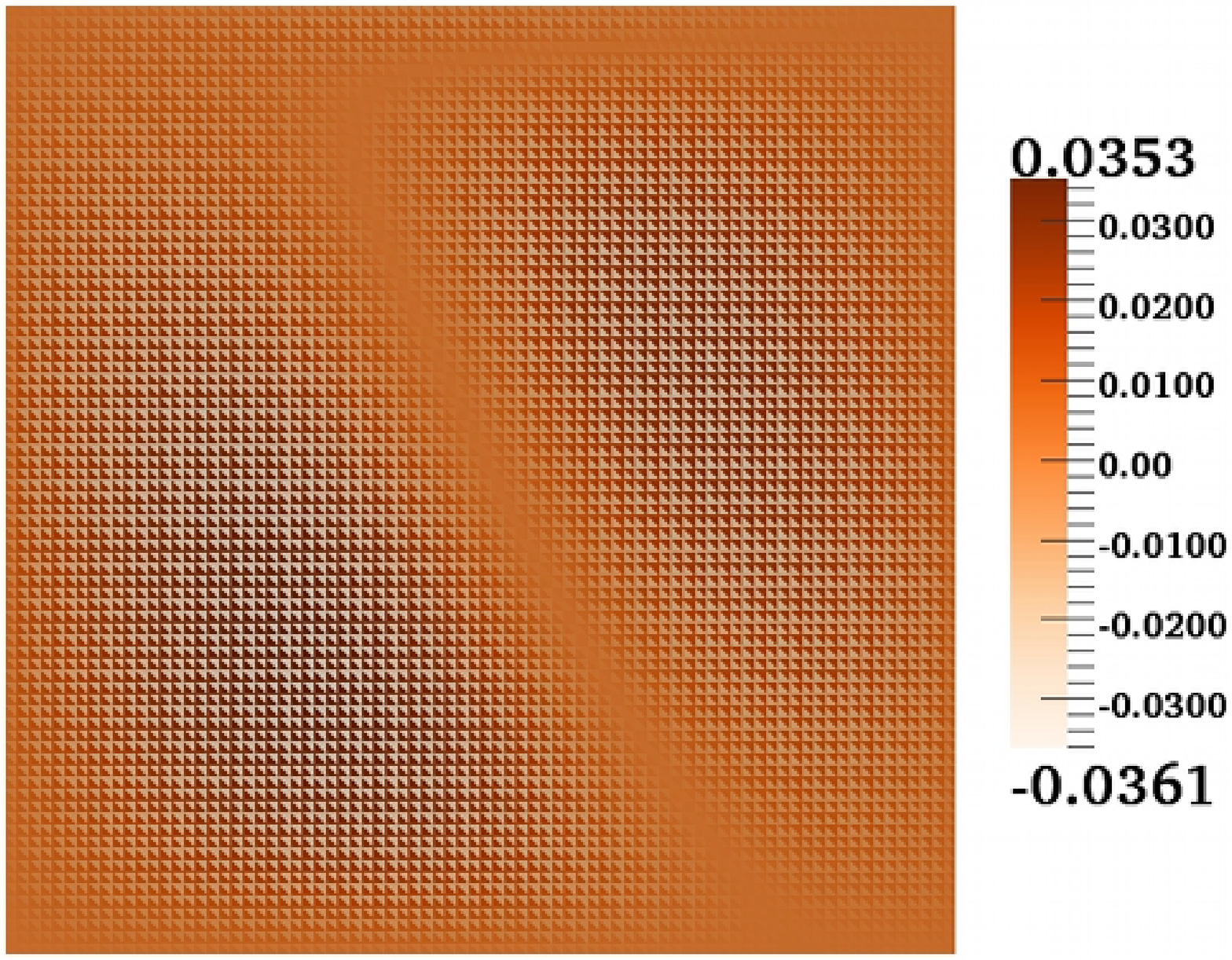}} 
  \subfigure[Q4 mesh: Error in LSB]
    {\includegraphics[scale=0.31]
    {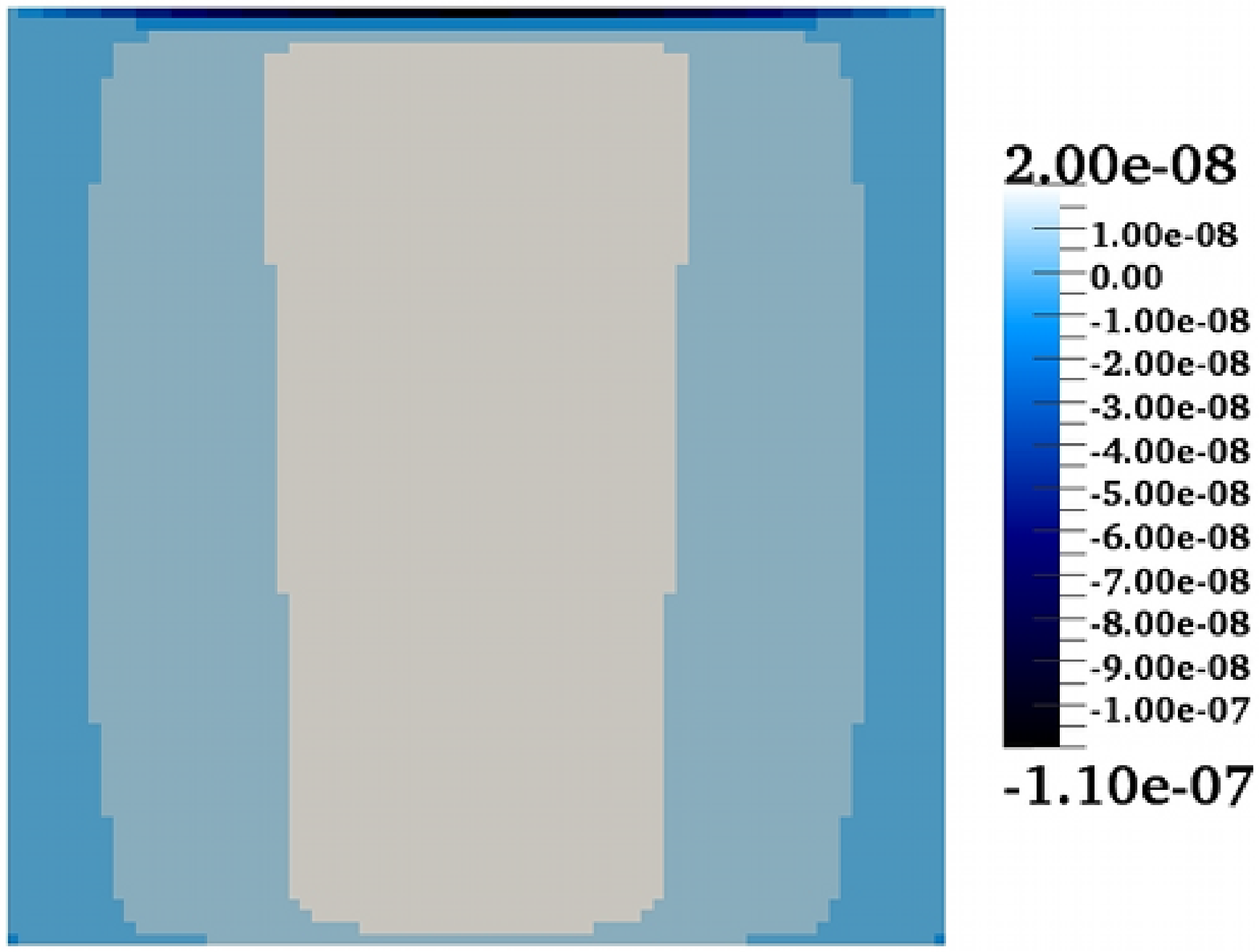}}
  \subfigure[Q4 mesh: Lagrange multiplier enforcing LSB]
    {\includegraphics[scale=0.31]
    {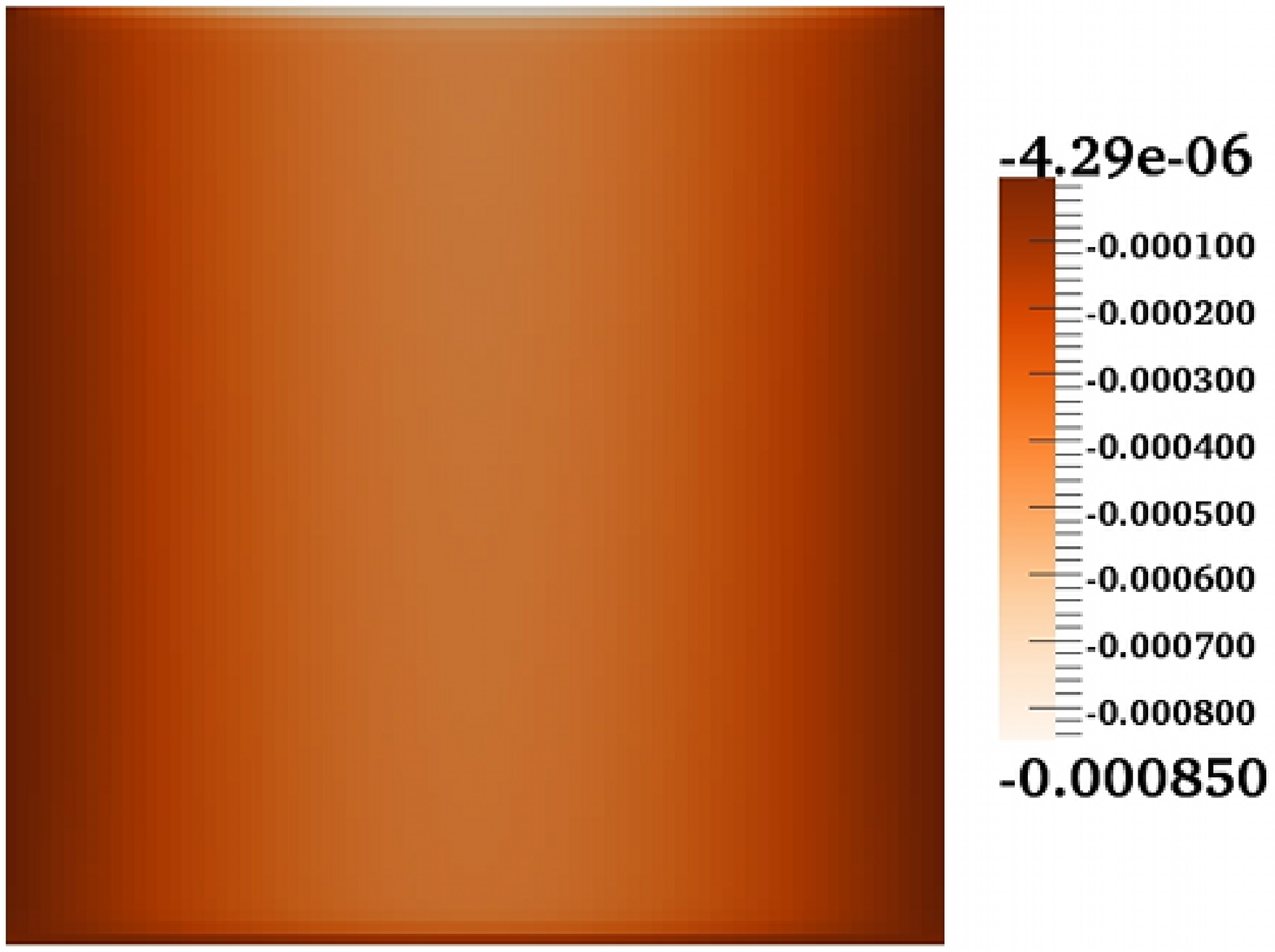}}  
  \end{center}
  \caption{\textsf{Numerical $h$-convergence study:}~The top 
    and bottom left figures show the contours of error incurred in satisfying 
    LSB for unconstrained negatively stabilized streamline diffusion LSFEM. 
    The right set of figures show the contours of Lagrange multiplier enforcing 
    LSB constraint using the proposed LSFEM. Note that the Lagrange multipliers 
    enforcing the LSB constraint can have negative value as opposed to KKT 
    multipliers. Numerical simulations are performed based on three-node 
    triangular mesh and four-node quadrilateral mesh with 81 nodes on each 
    side of the domain. In essence, the LSB errors and Lagrange multipliers 
    enforcing LSB based on a Q4 mesh is lesser than a T3 mesh.
  \label{Fig:2D_NumConvDiff1by100_LocMass_LambdaLMBT3MeshAllFourLSFEM}}
\end{figure}

\begin{figure}
  \subfigure[LSB errors]{\includegraphics[scale=0.12]
    {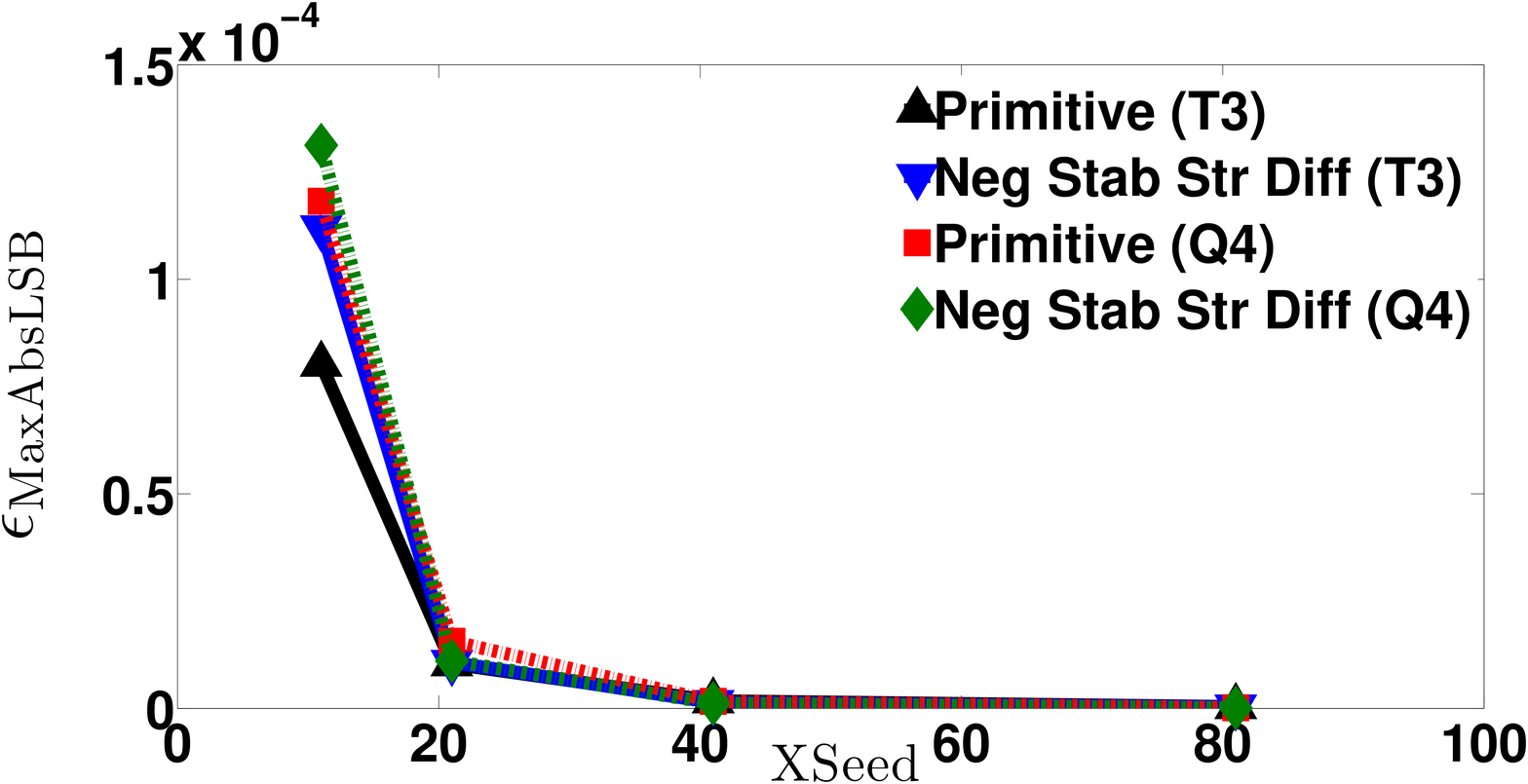}}
  \subfigure[GSB errors]{\includegraphics[scale=0.12]
    {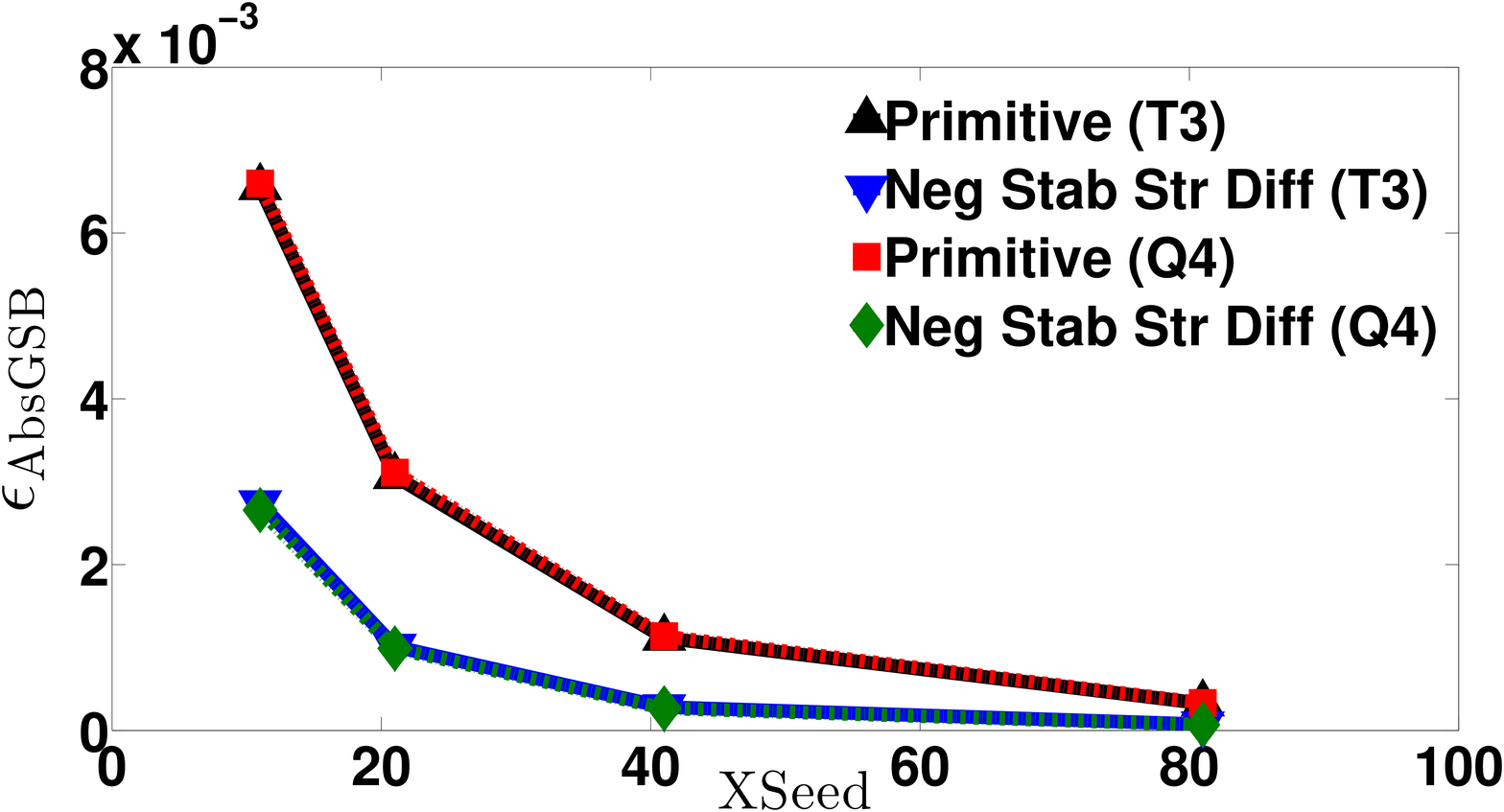}}
  \caption{\textsf{Numerical $h$-convergence study:}~These 
    figures show the decrease of $\epsilon_{\text{\tiny {MaxAbsLSB}}}$ 
    and $\epsilon_{\text{\tiny {AbsGSB}}}$ with respect to XSeed for a 
    series of hierarchical three-node triangular and four-node 
    quadrilateral meshes. (See equations 
    \eqref{Eqn:Errors_in_LSB_GSB}--\eqref{Eqn:AbsError_in_GSB} for 
    the definitions of $\epsilon_{\text{\tiny {MaxAbsLSB}}}$ 
    and $\epsilon_{\text{\tiny {AbsGSB}}}$.)
    Numerical simulations are performed using the 
    unconstrained primitive and negatively stabilized 
    streamline diffusion LSFEMs. 
    For XSeed = 81, $\epsilon_{\text{\tiny {MaxAbsLSB}}}$ and 
    $\epsilon_{\text{\tiny {AbsGSB}}}$ are in $\mathcal{O}(10^{-6})$. 
    In addition, the decrease in LSB and GSB errors with respect 
    to $h$-refinement is slow, and the values are not close to 
    the machine precision. 
  \label{Fig:2D_NumConvDiff1by100_ErrorsLMBandGMB_AllFourLSFEMs}}
\end{figure}

\begin{figure}
  \subfigure[T3 mesh]{\includegraphics[scale=0.12]
    {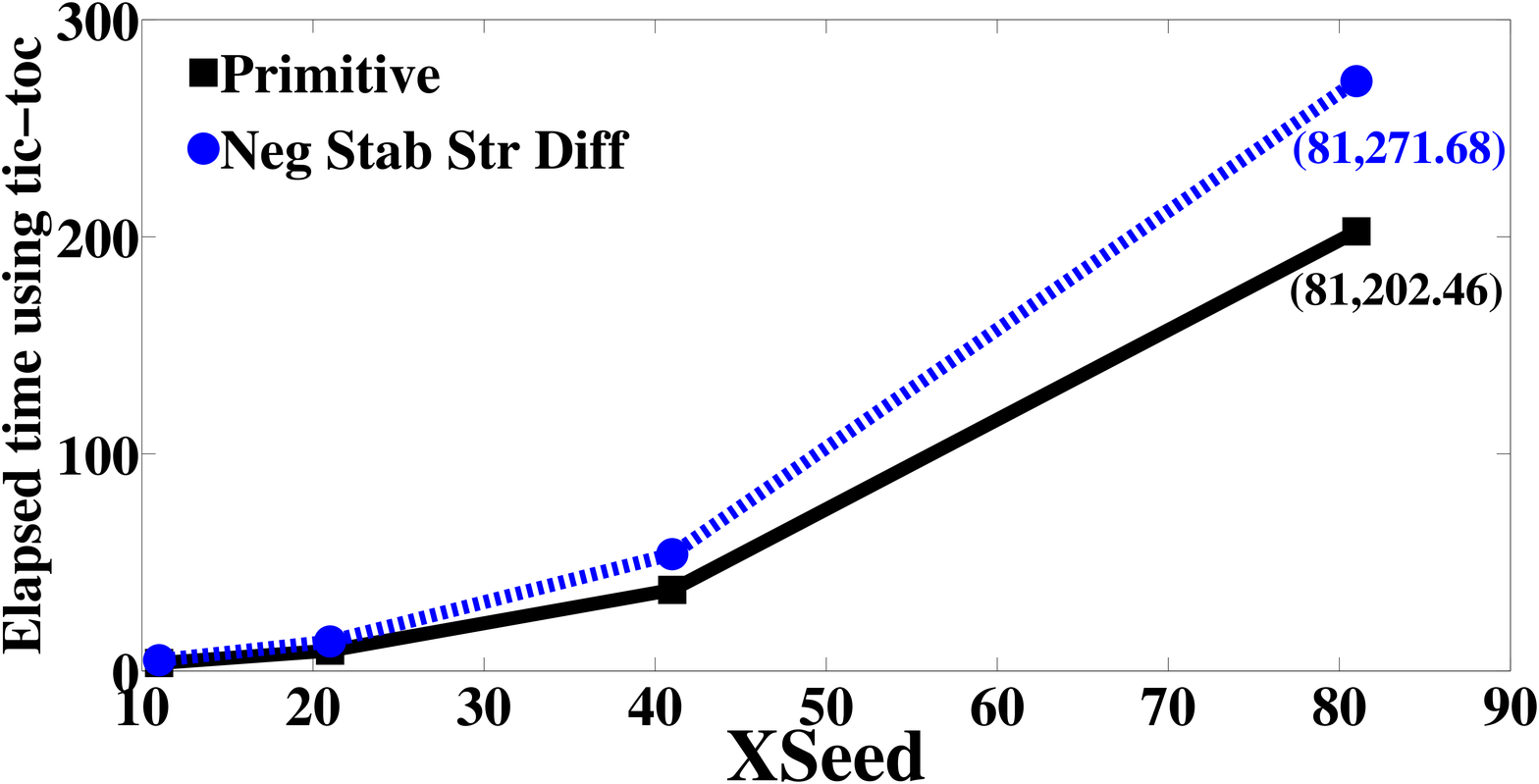}}
  \subfigure[Q4 mesh]{\includegraphics[scale=0.12]
    {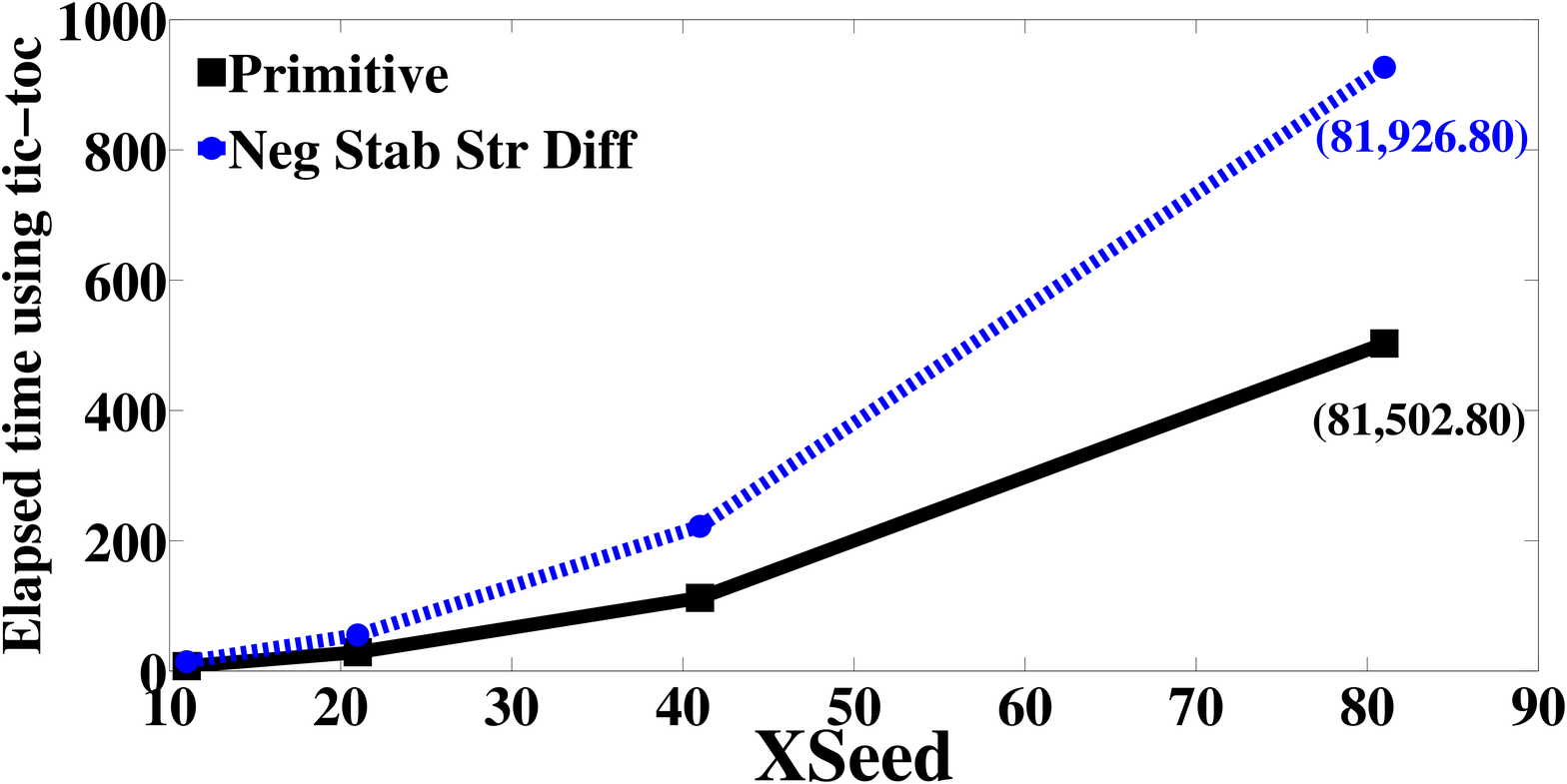}}
  \caption{\textsf{Numerical $h$-convergence study:}~This figure shows 
    the CPU time (in seconds) of the proposed computational framework 
    for unconstrained primitive and unconstrained negatively stabilized 
    streamline diffusion LSFEMs. For Q4 mesh, as 
    $\mathrm{div}[\mathrm{grad}[c]] \neq 0$, the computational cost 
    is higher than that of the T3 mesh.
  \label{Fig:2D_NumConvDiff1by100_TicTocNoConsT3Q4_AllFourLSFEMs}}
\end{figure}

\begin{figure}
  \subfigure[T3 mesh]{\includegraphics[scale=0.12]
    {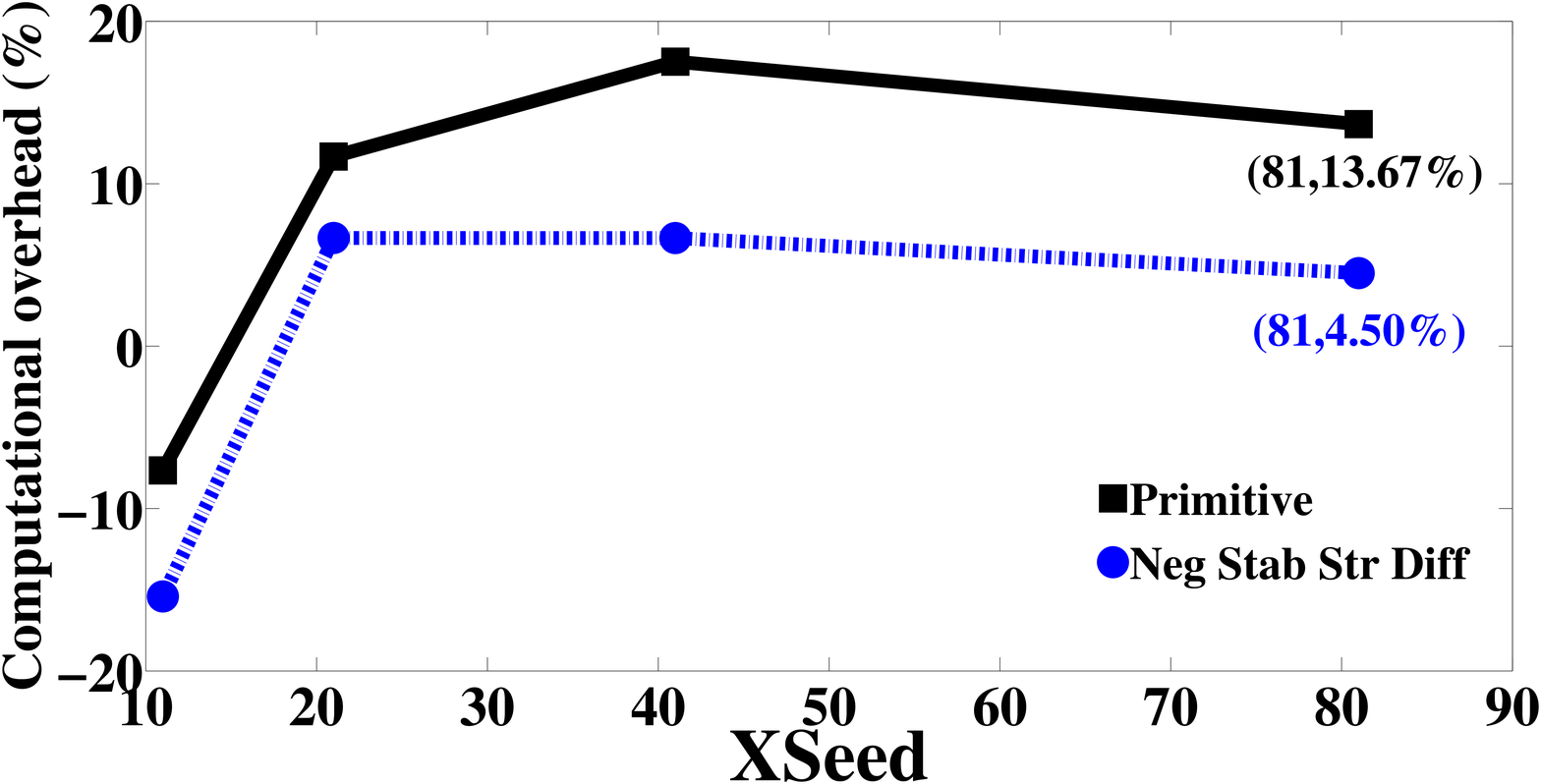}}
  \subfigure[Q4 mesh]{\includegraphics[scale=0.12]
    {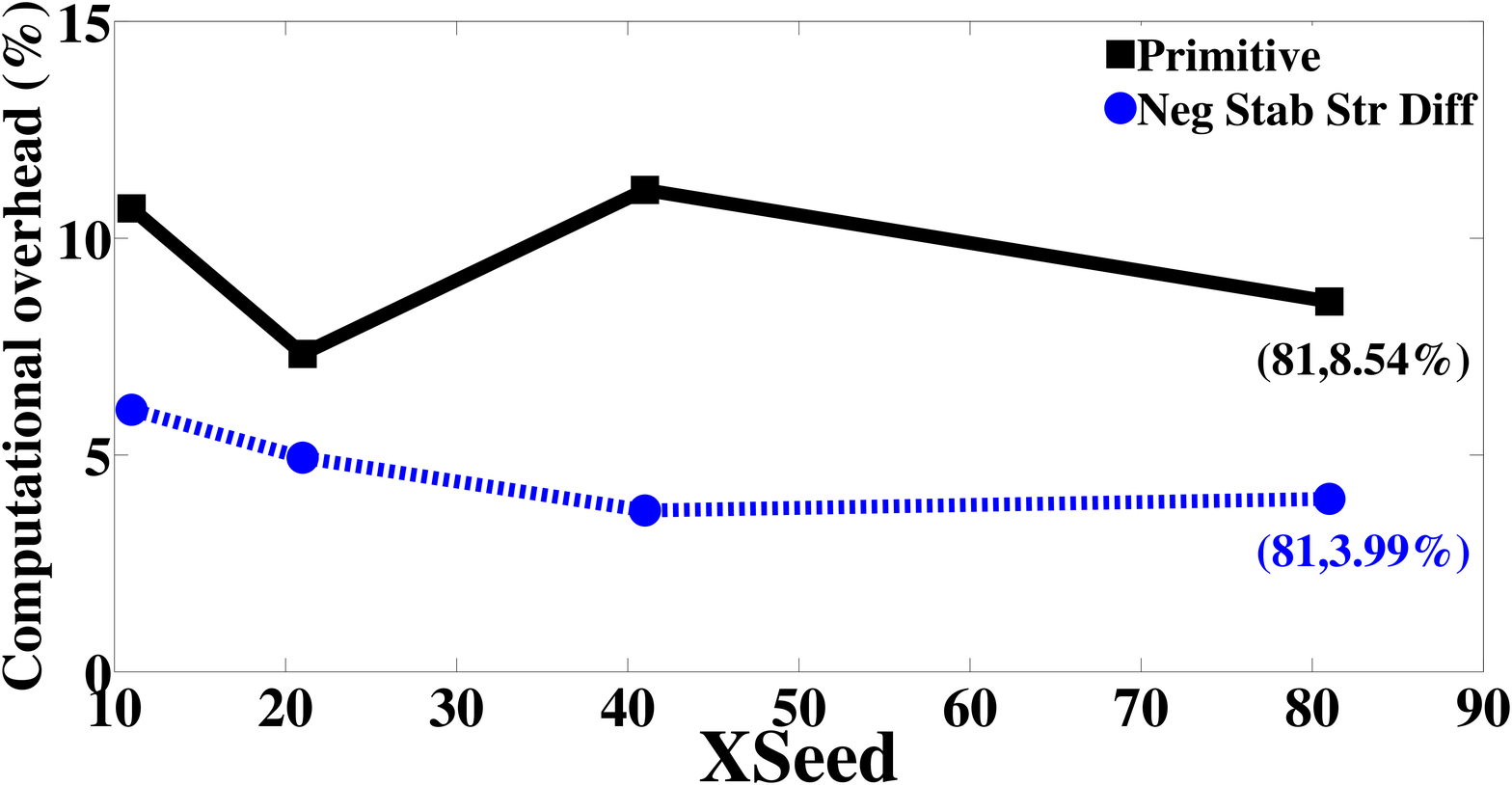}}
  \caption{\textsf{Numerical $h$-convergence study:}~This figure 
    shows the computational overhead incurred in satisfying LSB as 
    compared to that of the corresponding unconstrained formulations. 
    Analysis is performed for primitive and negatively stabilized 
    streamline diffusion LSFEMs. For XSeed = 11, we obtained negative 
    value for the computational overhead. This is because the 
    \texttt{interior point convex algorithm} used in \textsf{MATLAB 
    optimization solver} \cite{MATLAB_2015a} pre-processes the 
    constrained convex quadratic programming problem simplifies 
    the given LSB constraints by removing redundancies. Hence, 
    for very low number of unknowns, the computational cost 
    associated with \texttt{interior point convex algorithm} is 
    much faster than the \texttt{LU solver} for the unconstrained 
    optimization problem.
    \label{Fig:2D_NumConvDiff1by100_TicTocLocMassT3Q4_AllFourLSFEMs}}
\end{figure}


\begin{figure}
  \psfrag{A}{$(0,0)$}
  \psfrag{B}{$(1,0)$}
  \psfrag{C}{$(1,0.5)$}
  \psfrag{D}{$(0,0.5)$}
  \psfrag{BC1}{$c = 0$}
  \psfrag{BC2}{\rotatebox{-90}{$c = 2y$}}
  \psfrag{BC3}{$c = 1$}
  \psfrag{BC4}{\rotatebox{90}{$c = 1$}}
  \psfrag{v}{$\mathbf{v} = 2y \mathbf{\hat{e}}_x$}
  \psfrag{f}{$f = 0$}
  \psfrag{alpha}{$\alpha = 0$}
  \psfrag{d}{$D = 10^{-4}$}
  \includegraphics[scale=0.55]{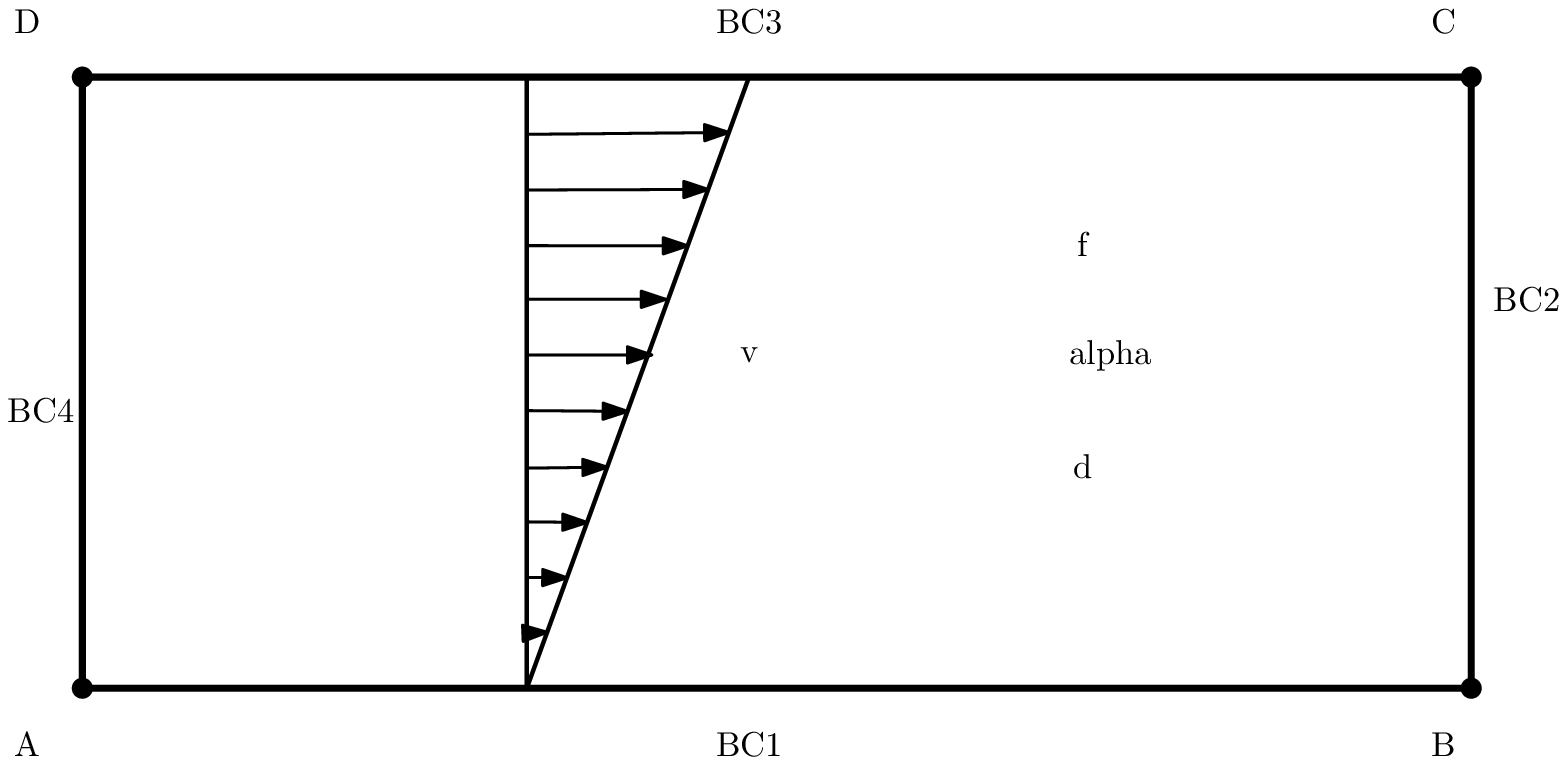}
  \caption{\textsf{Thermal boundary layer problem:}~This 
    figure shows a pictorial description of the boundary 
    value problem. Dirichlet boundary conditions are 
    prescribed on all four sides of the computational 
    domain. We have taken $c(\mathbf{x}) = 1$ at 
    $\mathbf{x} = (0,0)$. 
    \label{Fig:2D_ThermalBoundaryLayer}}
\end{figure}

\begin{figure}
  \centering
  \subfigure[Primitive (No constraints)]{\includegraphics[scale=0.35]
    {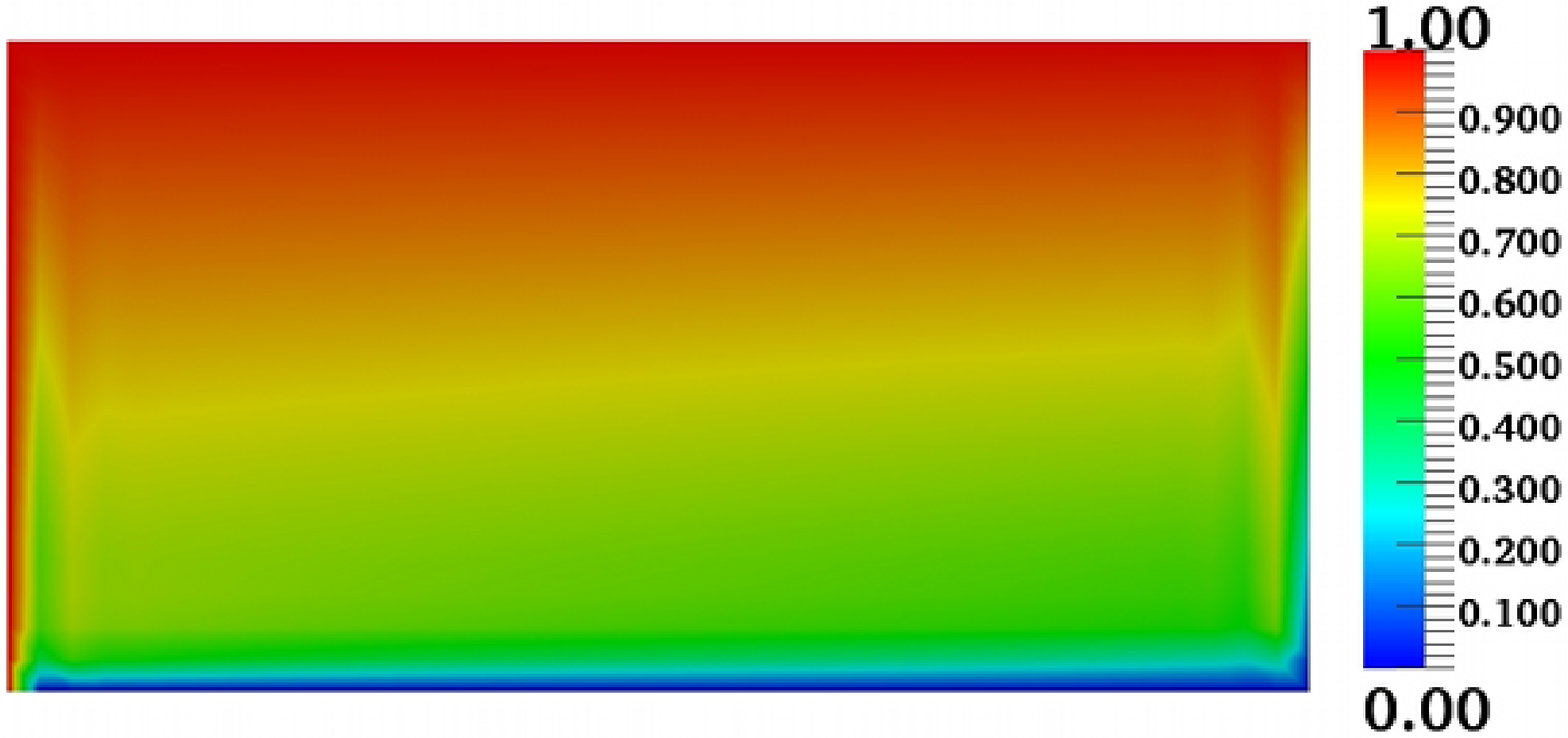}}
  \subfigure[Negatively stabilized streamline diffusion (LSB and NN constraints)]
    {\includegraphics[scale=0.35]
    {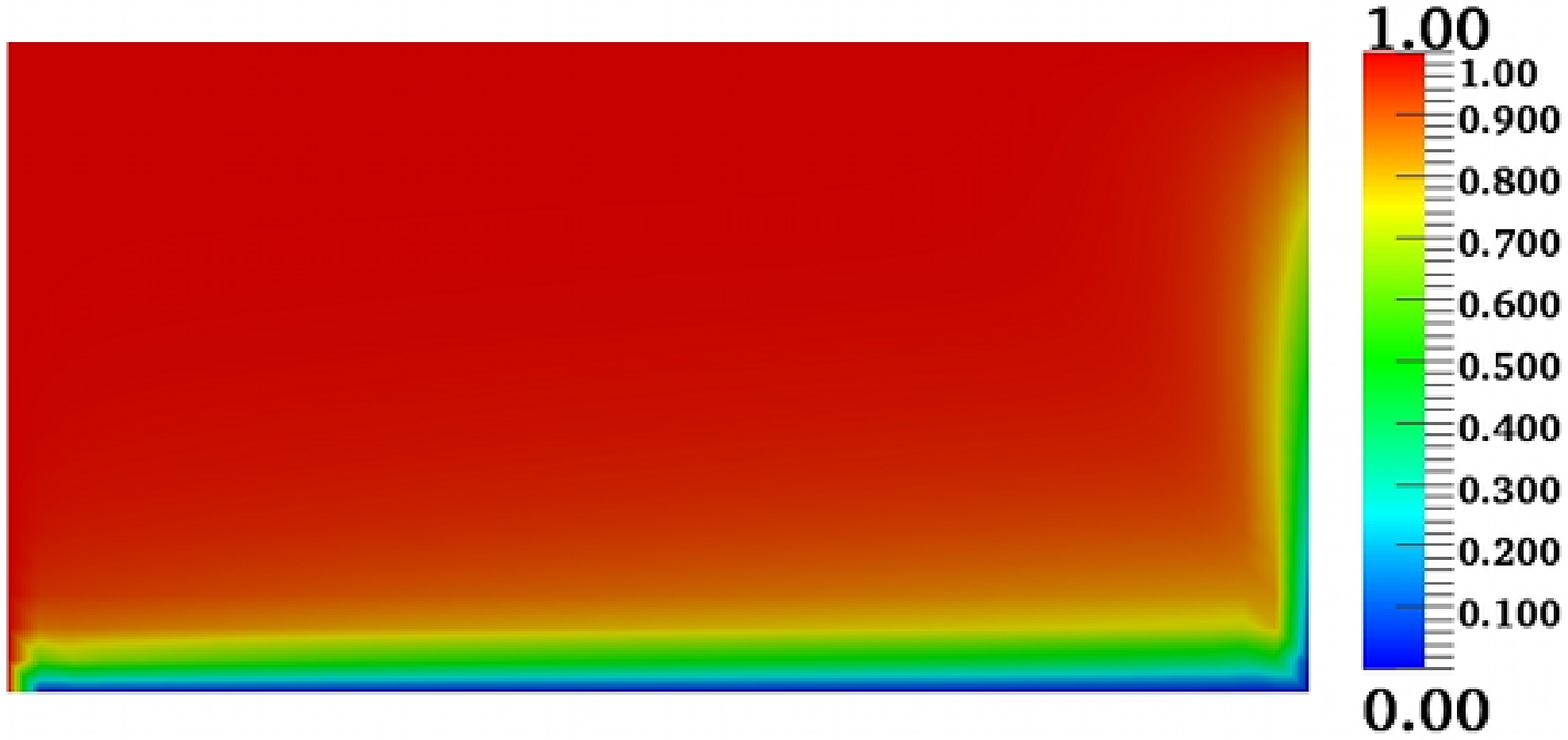}}
  \caption{\textsf{Thermal boundary layer problem:}~This 
    figure shows the contours of concentration obtained 
    for both unconstrained and constrained LSFEMs based 
    on Q4 finite element mesh. The proposed LSFEM-based 
    framework with NN and LSB constraints is able to 
    eliminate spurious oscillations near the boundaries 
    $y = 0$ and $x = 1$. This is not the case with the 
    primitive LSFEM.
    \label{Fig:2D_ThermalBoundaryLayer_Results}}
\end{figure}

\begin{figure}
  \centering
  \subfigure[Primitive]{\includegraphics[scale=0.35]
    {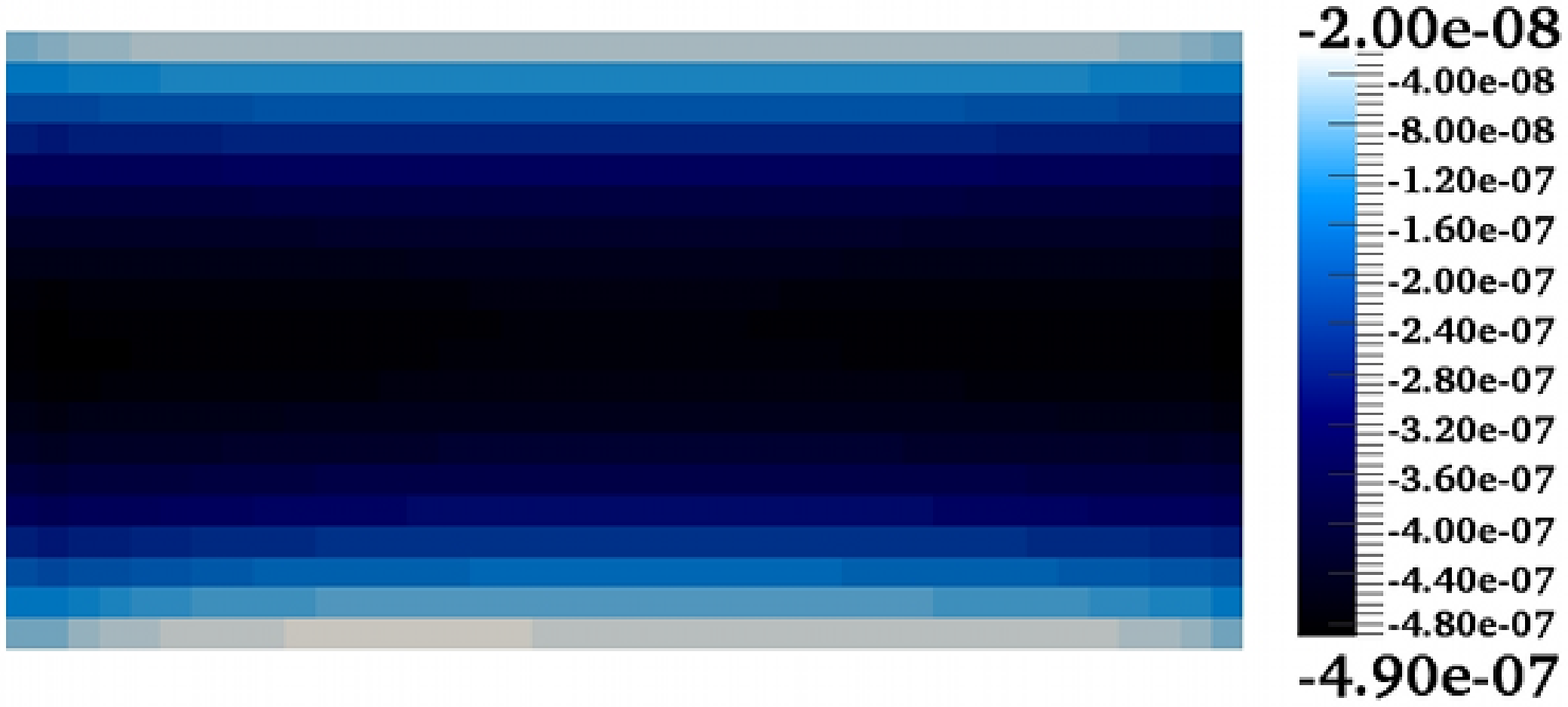}}
  \subfigure[Negatively stabilized streamline diffusion]{\includegraphics[scale=0.35]
    {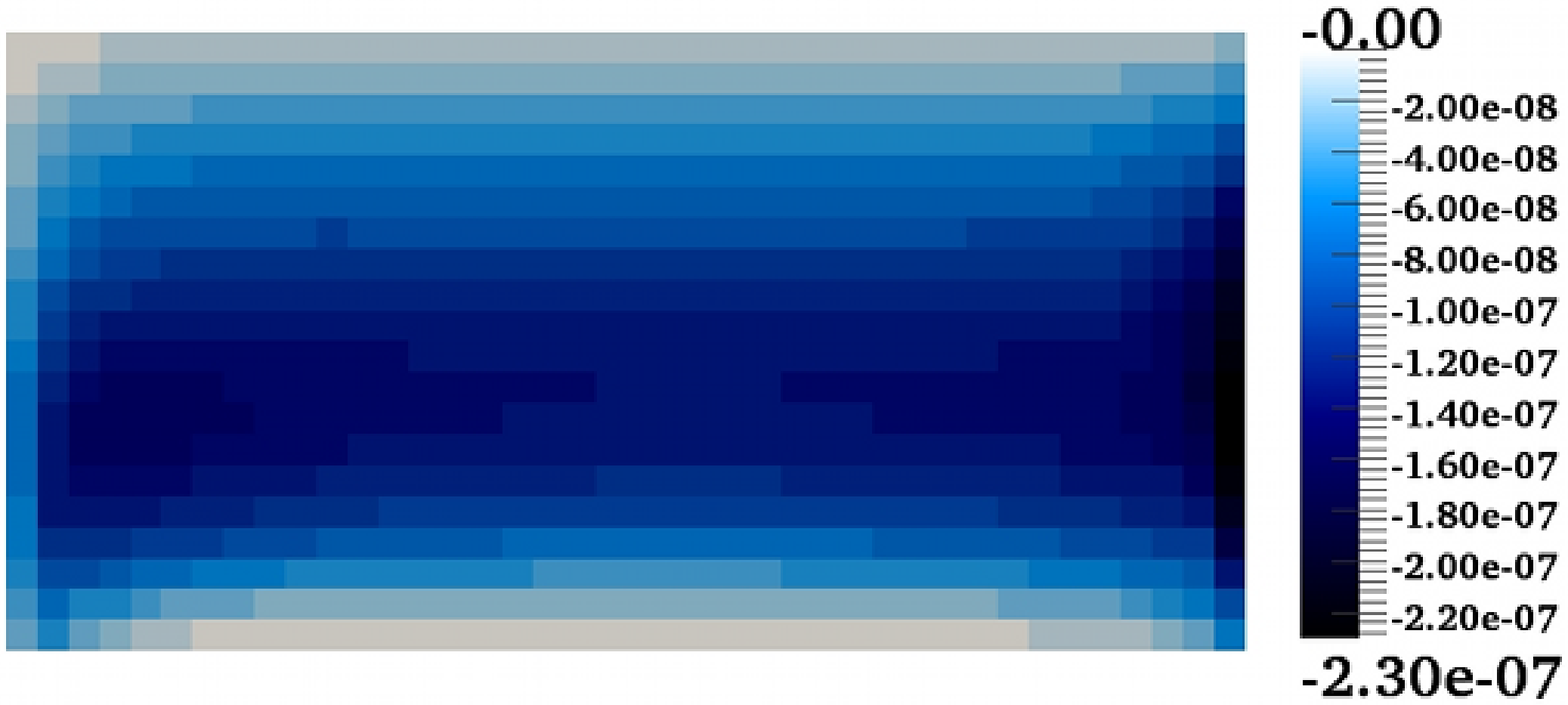}}
  \caption{\textsf{Thermal boundary problem:}~This figure shows 
    the contours of the error incurred in satisfying LSB for 
    various unconstrained LSFEM formulations using Q4 meshes. 
    One can notice that the error is more dominant in the 
    interior of the domain under the primitive LSFEM, whereas 
    the error is dominant at the boundary $x = 1$ under the 
    negatively stabilized streamline diffusion formulation.
    \label{Fig:2D_ThermalBoundaryLayer_LSB}}
\end{figure}


\begin{figure}
  \centering
  \psfrag{A1}{$c^{\mathrm{p}}_A(\mathbf{x} = 0) = 1$}
  \psfrag{A2}{$c^{\mathrm{p}}_B(\mathbf{x} = 0) = 0$}
  \psfrag{A3}{$c^{\mathrm{p}}_C(\mathbf{x} = 0) = 0$}
  \psfrag{B1}{$c^{\mathrm{p}}_A(\mathbf{x} = 1) = 0$}
  \psfrag{B2}{$c^{\mathrm{p}}_B(\mathbf{x} = 1) = 0 \; \mathrm{or} \; 1$}
  \psfrag{B3}{$c^{\mathrm{p}}_C(\mathbf{x} = 1) = 0$}
  \psfrag{C1}{$f_A = 0$}
  \psfrag{C2}{$f_B = 0 \; \mathrm{or} \; 1$}
  \psfrag{C3}{$f_C = 0$}
  \psfrag{x1}{$x = 0$}
  \psfrag{x2}{$x = 1$}
  \psfrag{d1}{$D = 0.0025$}
  \psfrag{v1}{$\mathbf{v} = v \mathbf{\hat{e}}_x$}
  \includegraphics[scale=0.9,clip]{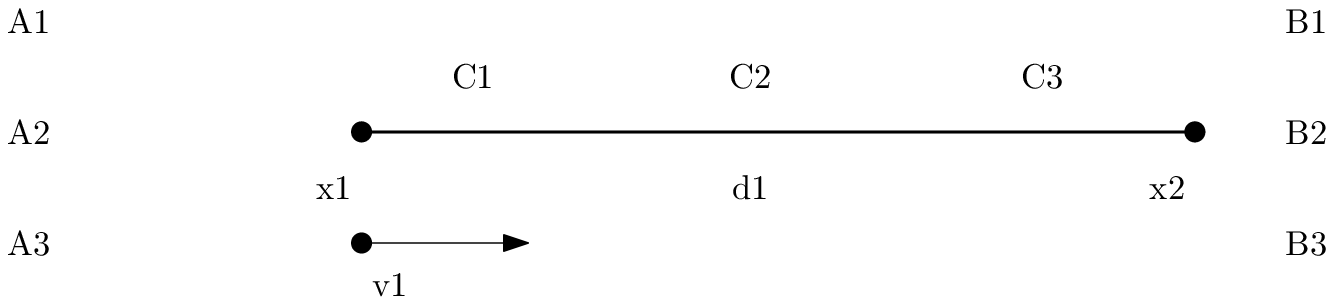}
  \caption{\textsf{1D irreversible bimolecular fast 
      reaction problem:}~A pictorial description 
    of the boundary value problem. For Case \#1: 
    $f_B(\mathbf{x}) = 1$ and $c^{\mathrm{p}}_B(\mathbf{x} 
    = 1) = 0$, and for Case \#2: $f_B(\mathbf{x}) = 0$ 
    and $c^{\mathrm{p}}_B(\mathbf{x} = 1) = 1$.
    \label{Fig:1D_Bimolecular_Type1_Type2}}
\end{figure}

\begin{figure}
  \centering
  \psfrag{cA}{$c_A(\mathrm{x})$}
  \psfrag{cB}{$c_B(\mathrm{x})$}
  \psfrag{cC}{$c_C(\mathrm{x})$}
  \psfrag{Peh5}{$\mathbb{P}\mathrm{e}_{h} = 5$}
  \psfrag{Peh20}{$\mathbb{P}\mathrm{e}_{h} = 20$}
  \subfigure{\includegraphics[scale = 0.31,clip]
    {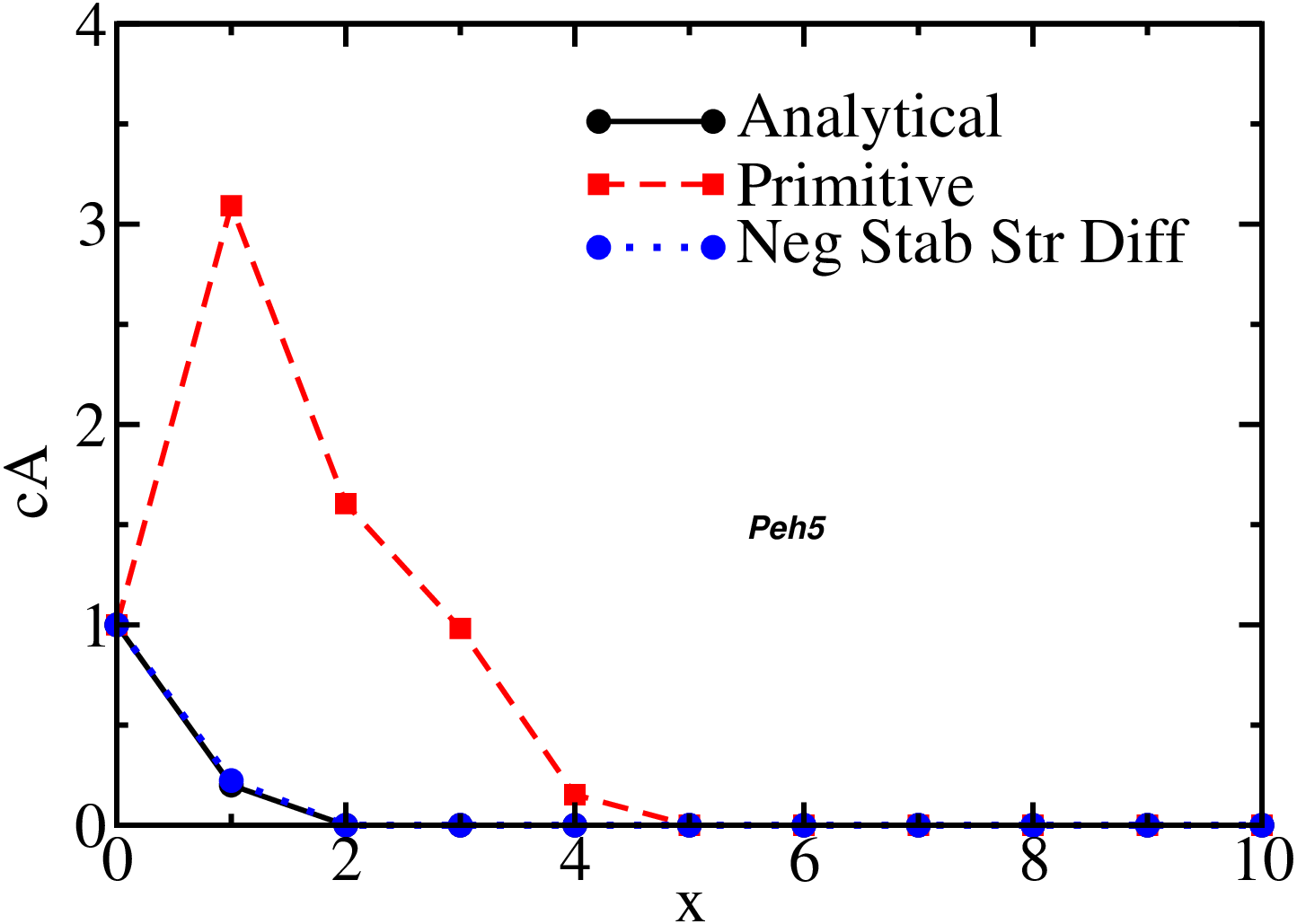}}
  \subfigure{\includegraphics[scale = 0.31,clip]
    {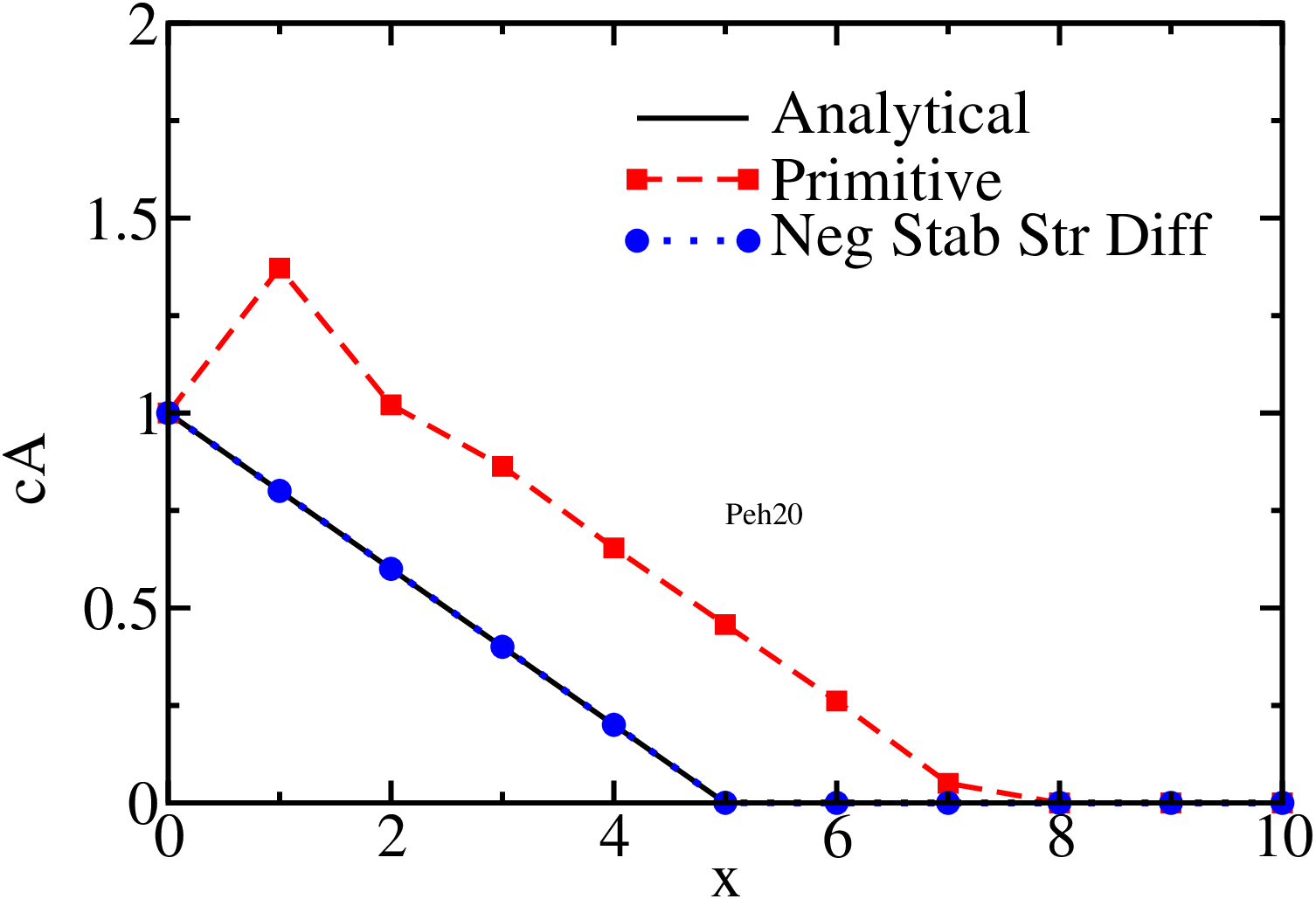}}
  \subfigure{\includegraphics[scale = 0.31,clip]
    {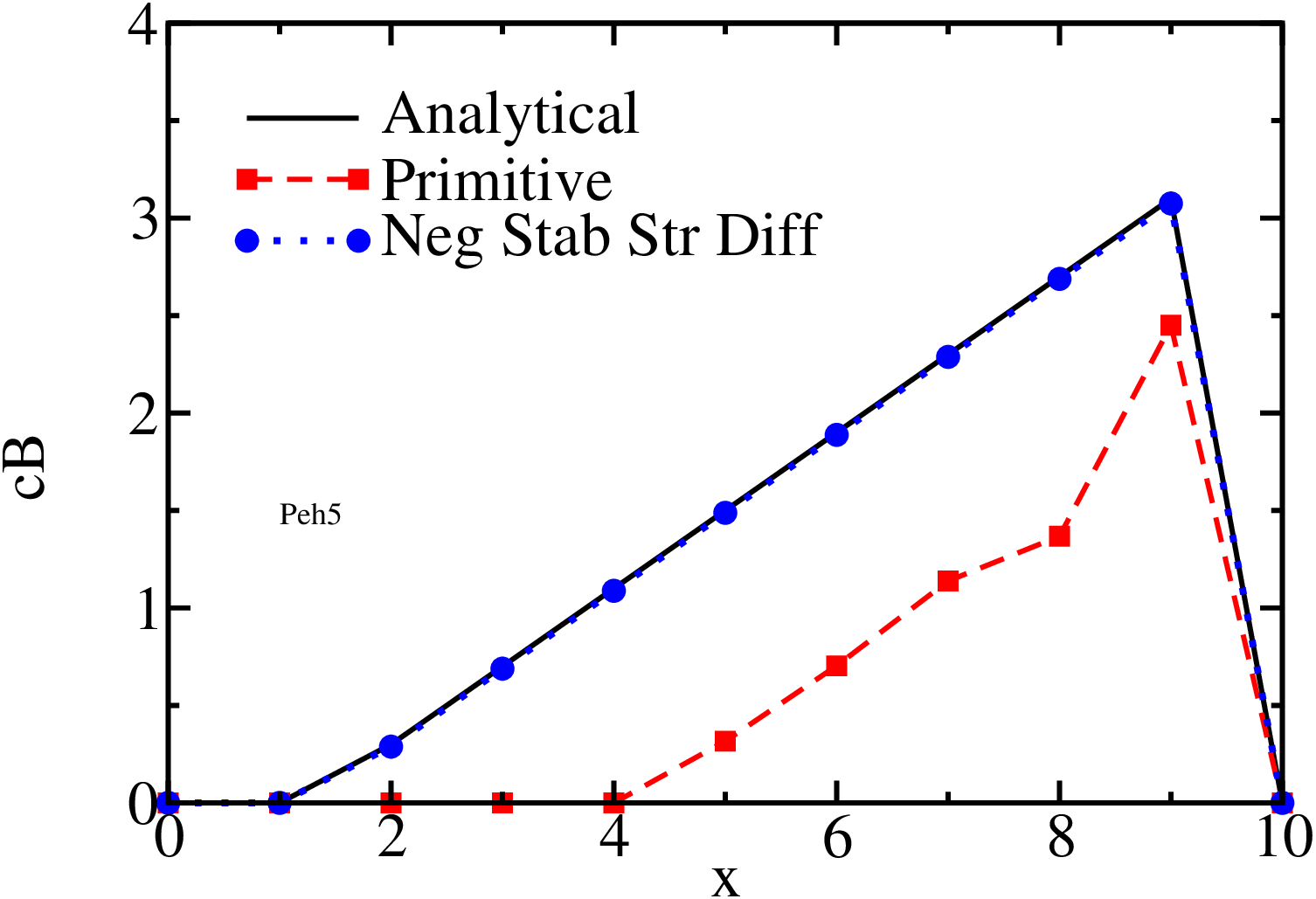}}
  \subfigure{\includegraphics[scale = 0.31,clip]
    {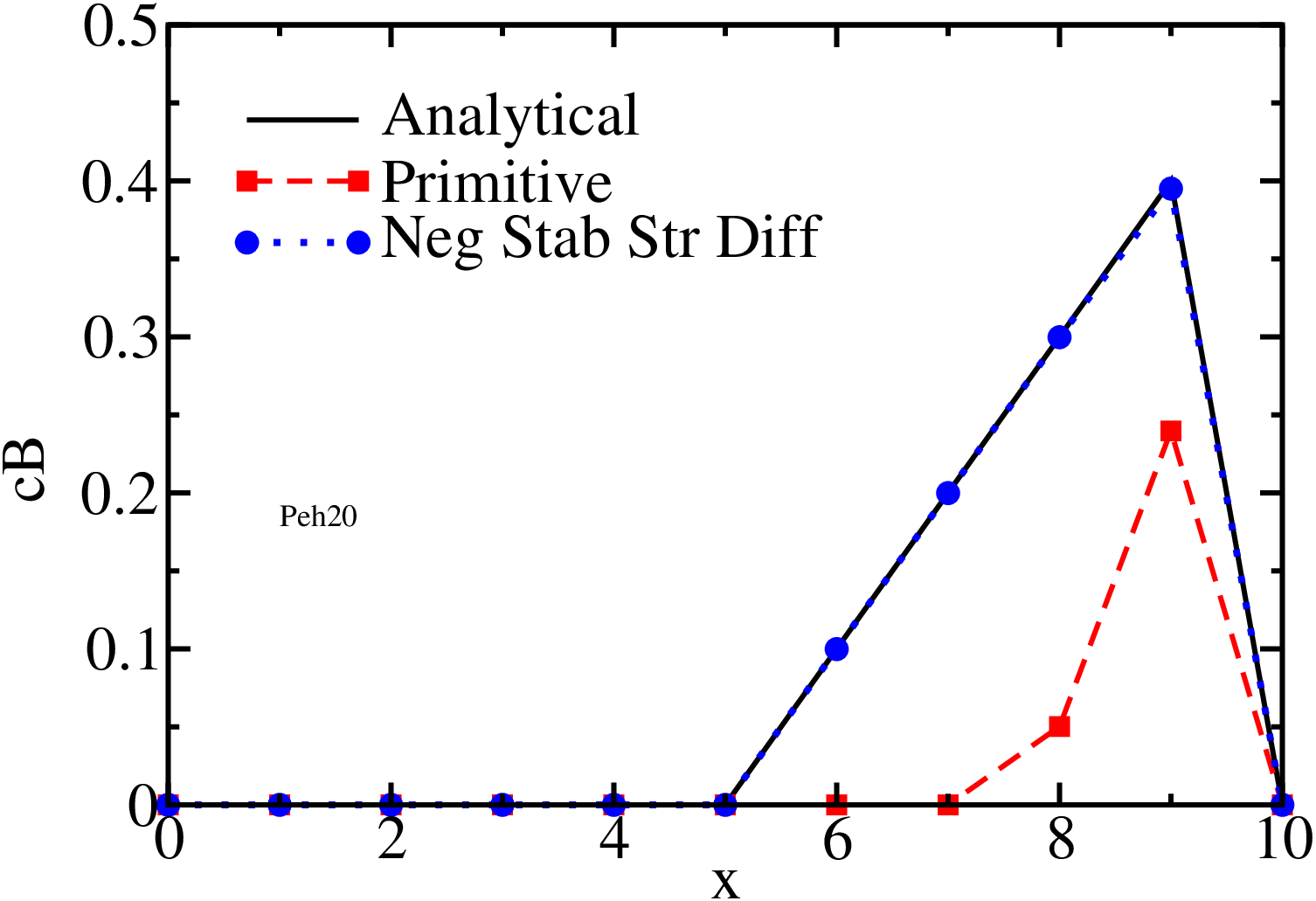}}
  \subfigure{\includegraphics[scale = 0.31,clip]
    {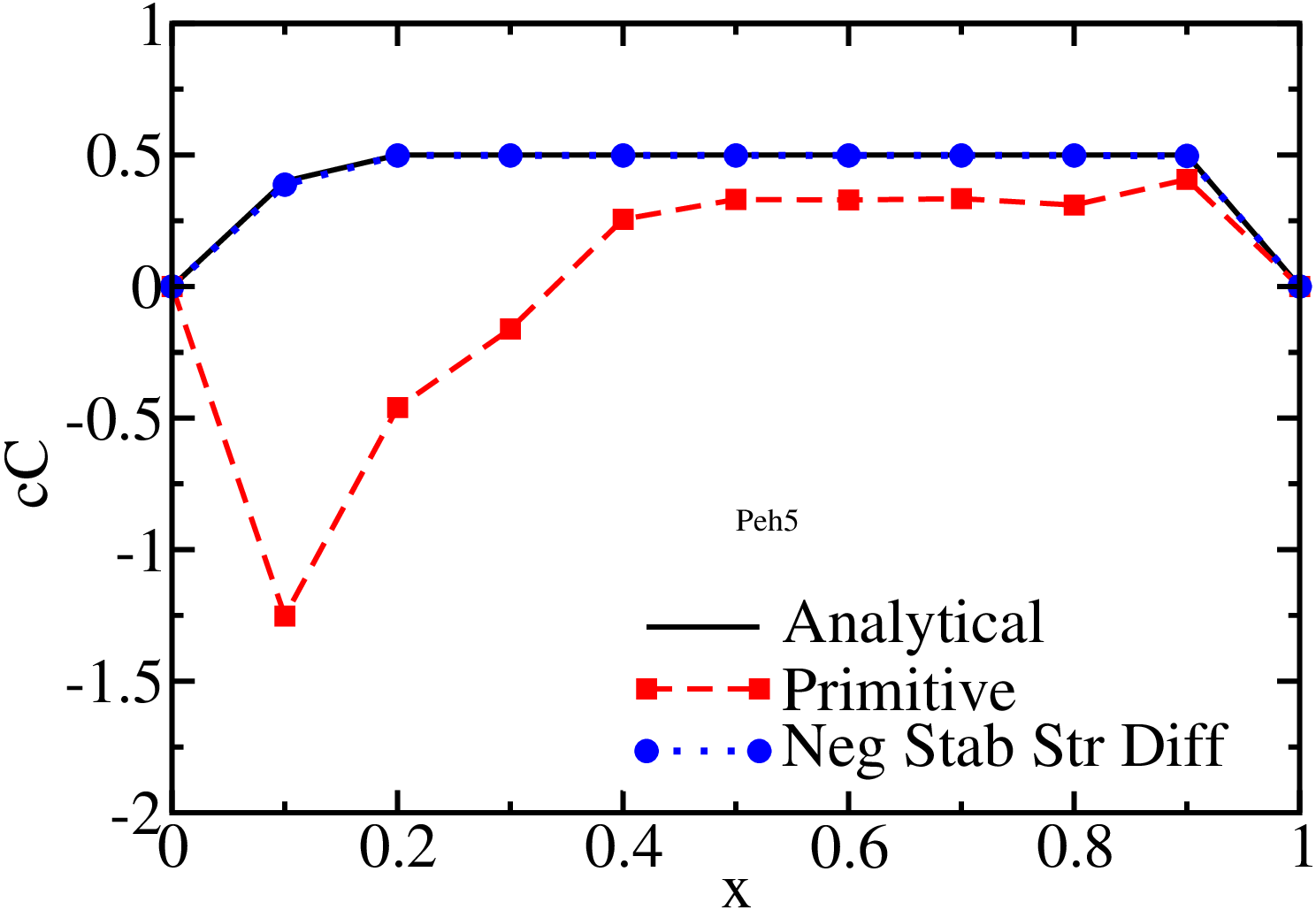}}
  \subfigure{\includegraphics[scale = 0.31,clip]
    {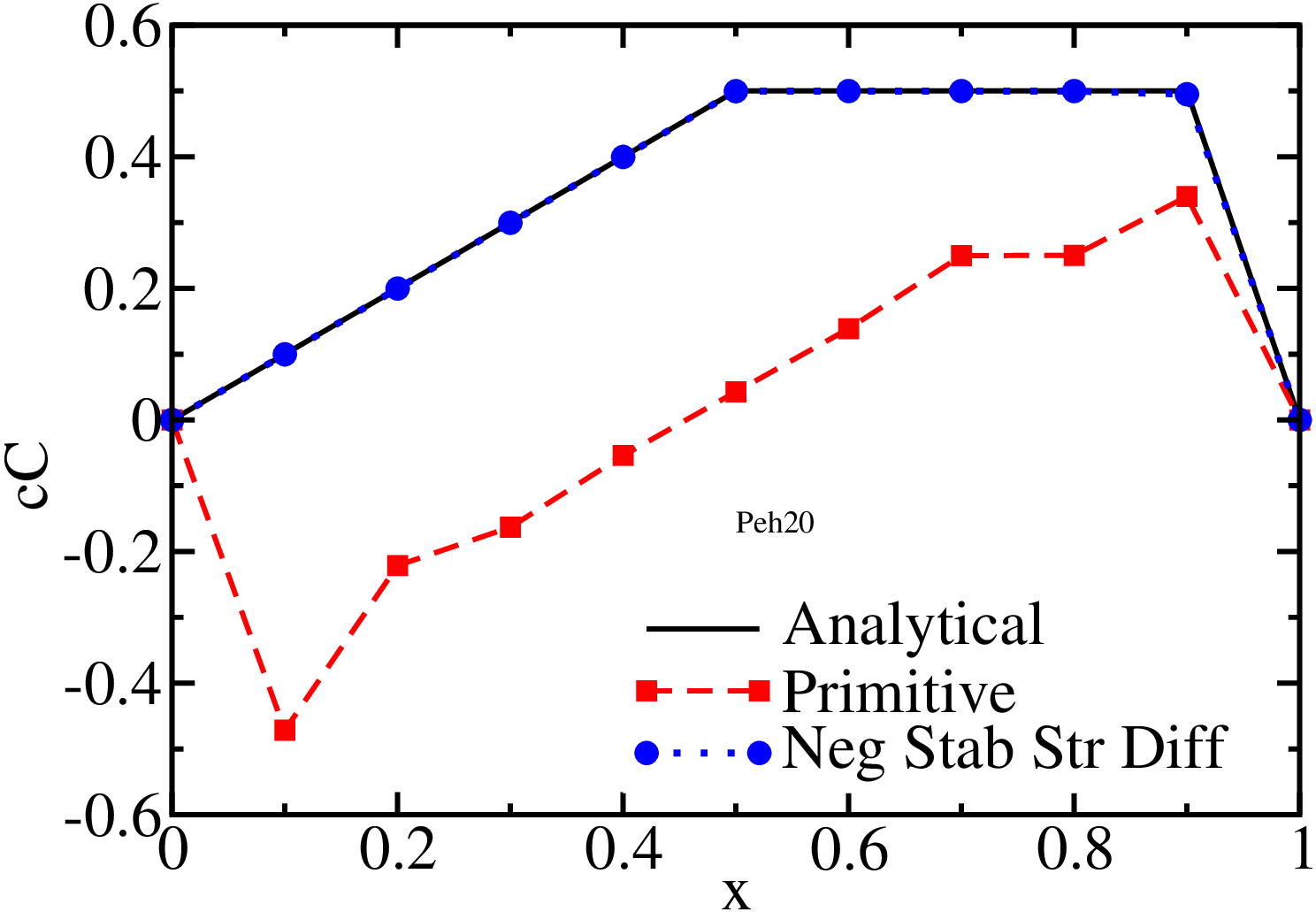}}
  \caption{\textsf{1D irreversible bimolecular fast reaction problem 
    (Case \#1):}~This figure compares the concentration profile of 
    the reactants and the product for various element P\'{e}clet numbers 
    using unconstrained primitive and unconstrained negatively stabilized 
    streamline diffusion LSFEMs to that of the analytical solution. The 
    primitive LSFEM considerably deviated from the analytical solution. 
    Moreover, it violated the non-negative 
    and maximum constraints. On the other hand, the negatively stabilized streamline 
    diffusion LSFEM is able to capture the analytical solution exactly 
    in the entire domain even at high element P\'{e}clet numbers. 
  \label{Fig:1D_Bimolecular_Type1_concC_Unconstrained}}
\end{figure}

\begin{figure}
  \centering
  \psfrag{cA}{$c_A(\mathrm{x})$}
  \psfrag{cB}{$c_B(\mathrm{x})$}
  \psfrag{cC}{$c_C(\mathrm{x})$}
  \psfrag{Peh1}{$\mathbb{P}\mathrm{e}_{h} = 1$}
  \psfrag{Peh5}{$\mathbb{P}\mathrm{e}_{h} = 5$}
  \subfigure{\includegraphics[scale = 0.31,clip]
    {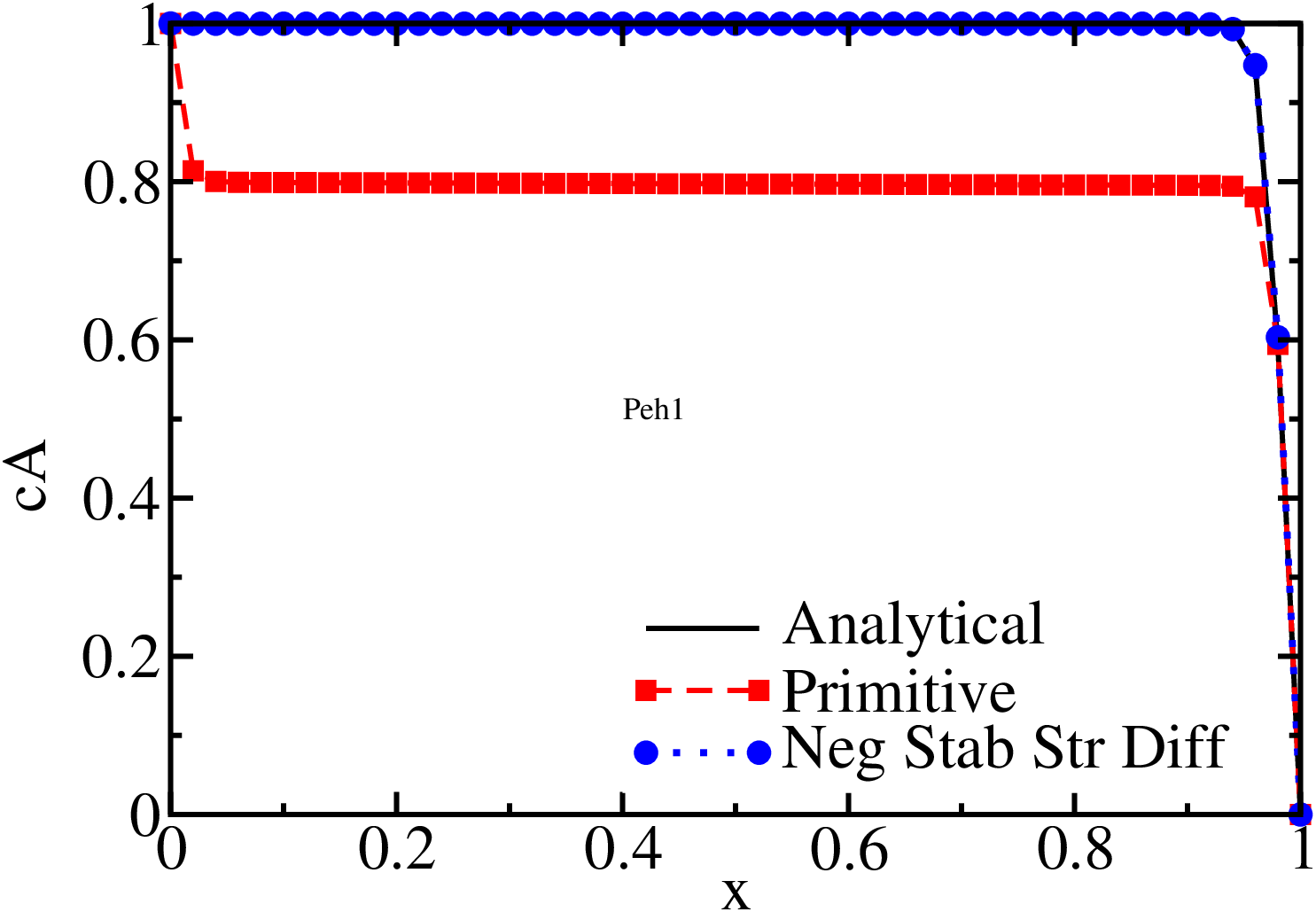}}
  \subfigure{\includegraphics[scale = 0.31,clip]
    {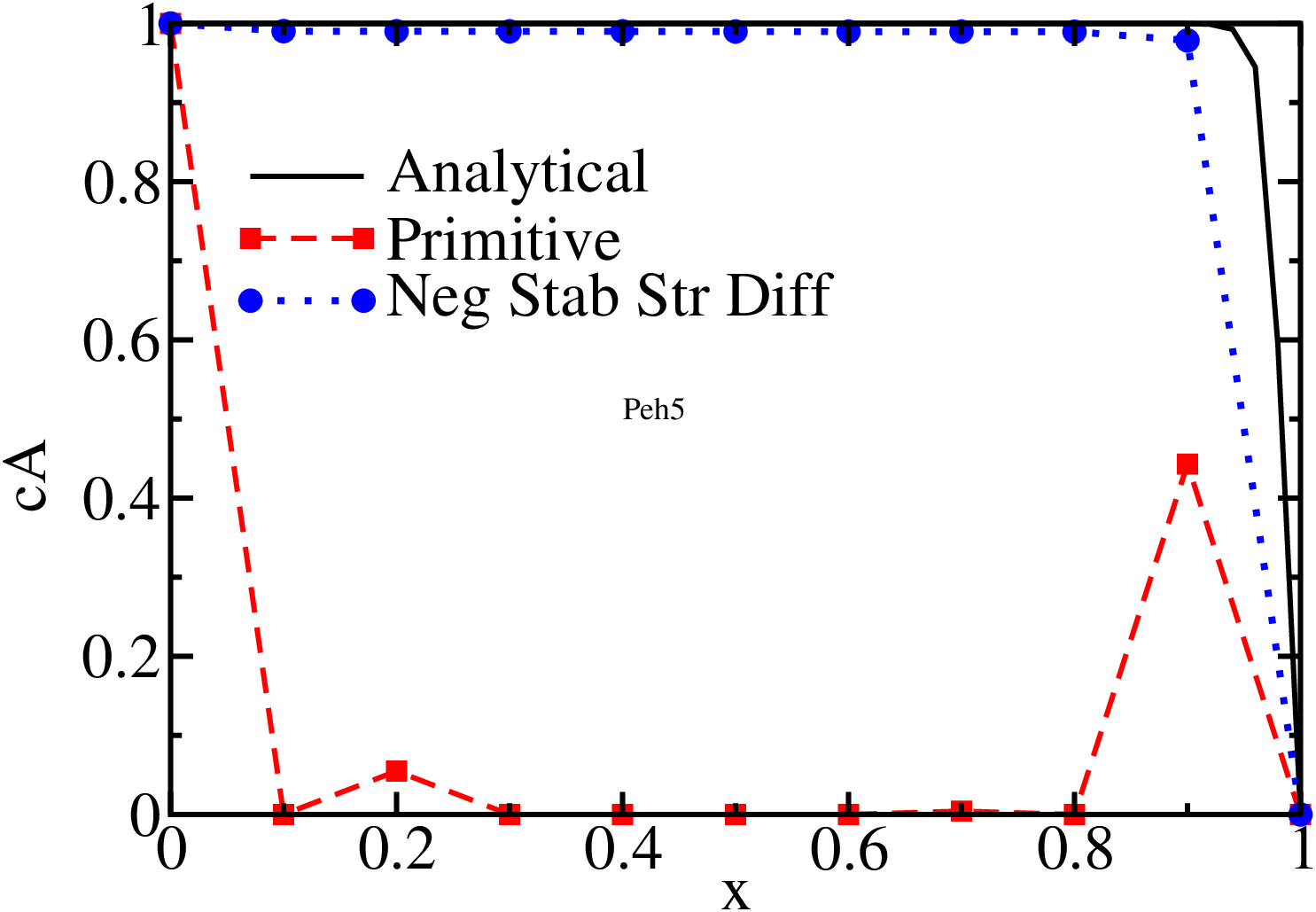}}
  \subfigure{\includegraphics[scale = 0.31,clip]
    {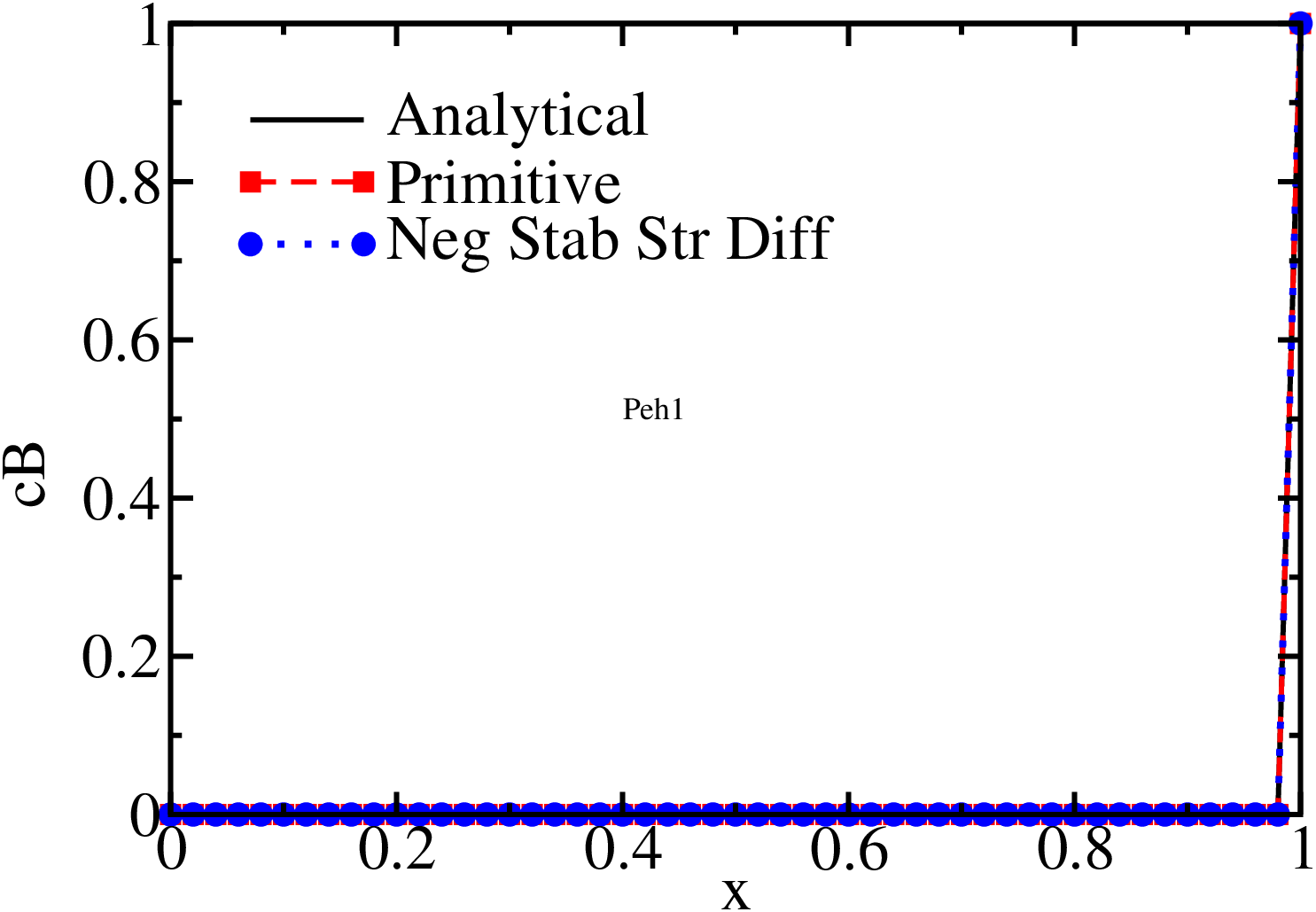}}
  \subfigure{\includegraphics[scale = 0.31,clip]
    {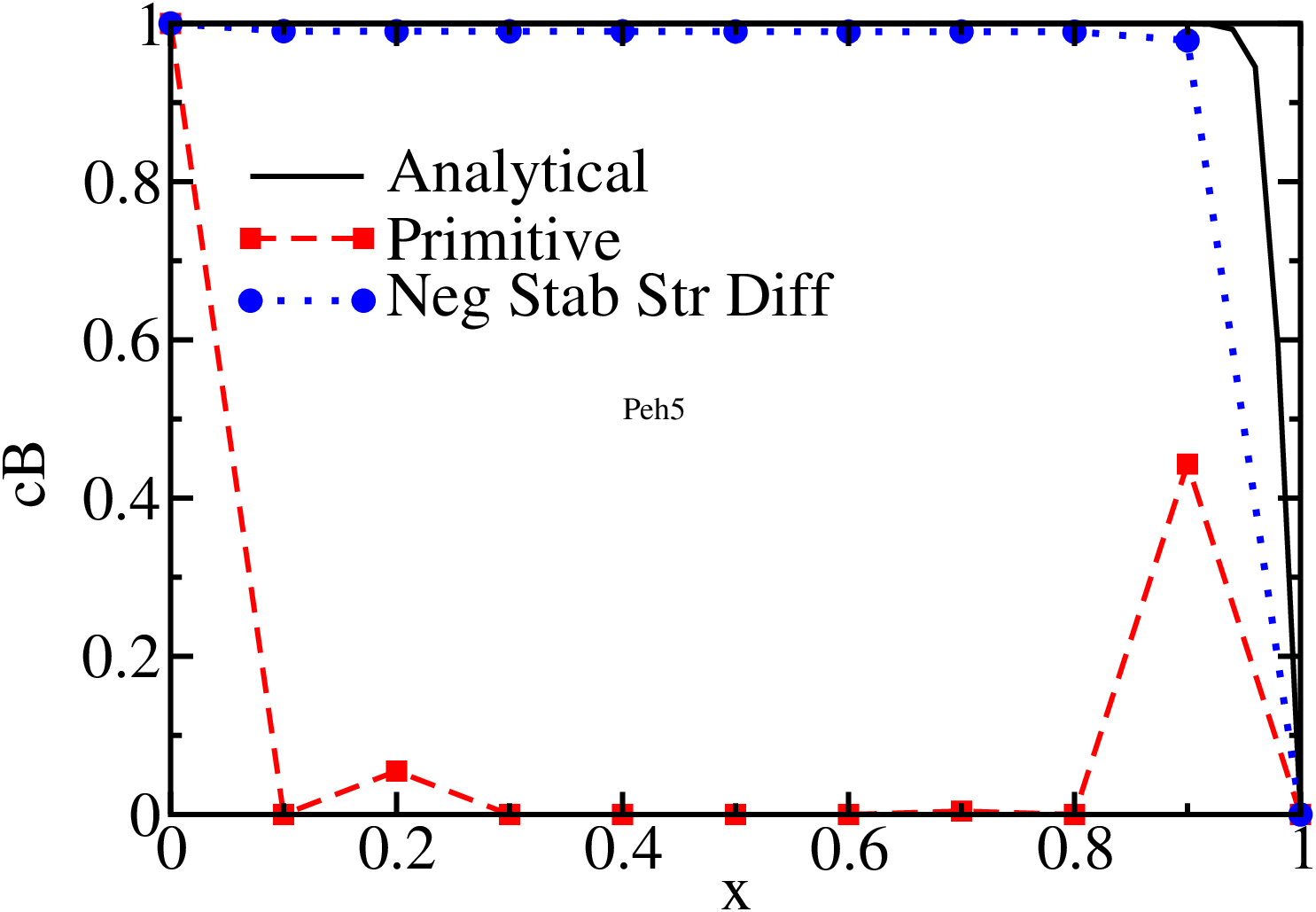}}
  \subfigure{\includegraphics[scale = 0.31,clip]
    {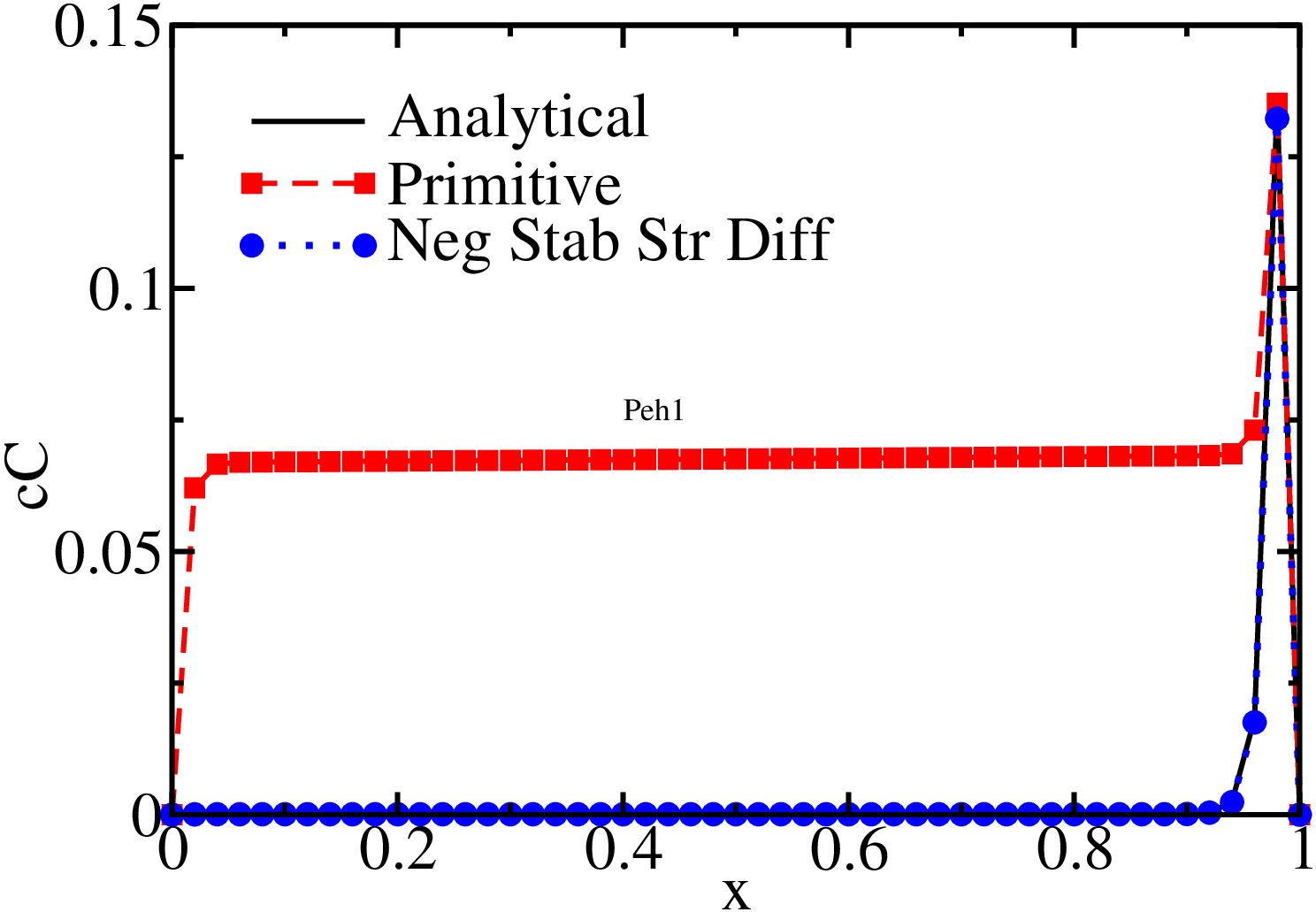}}
  \subfigure{\includegraphics[scale = 0.31,clip]
    {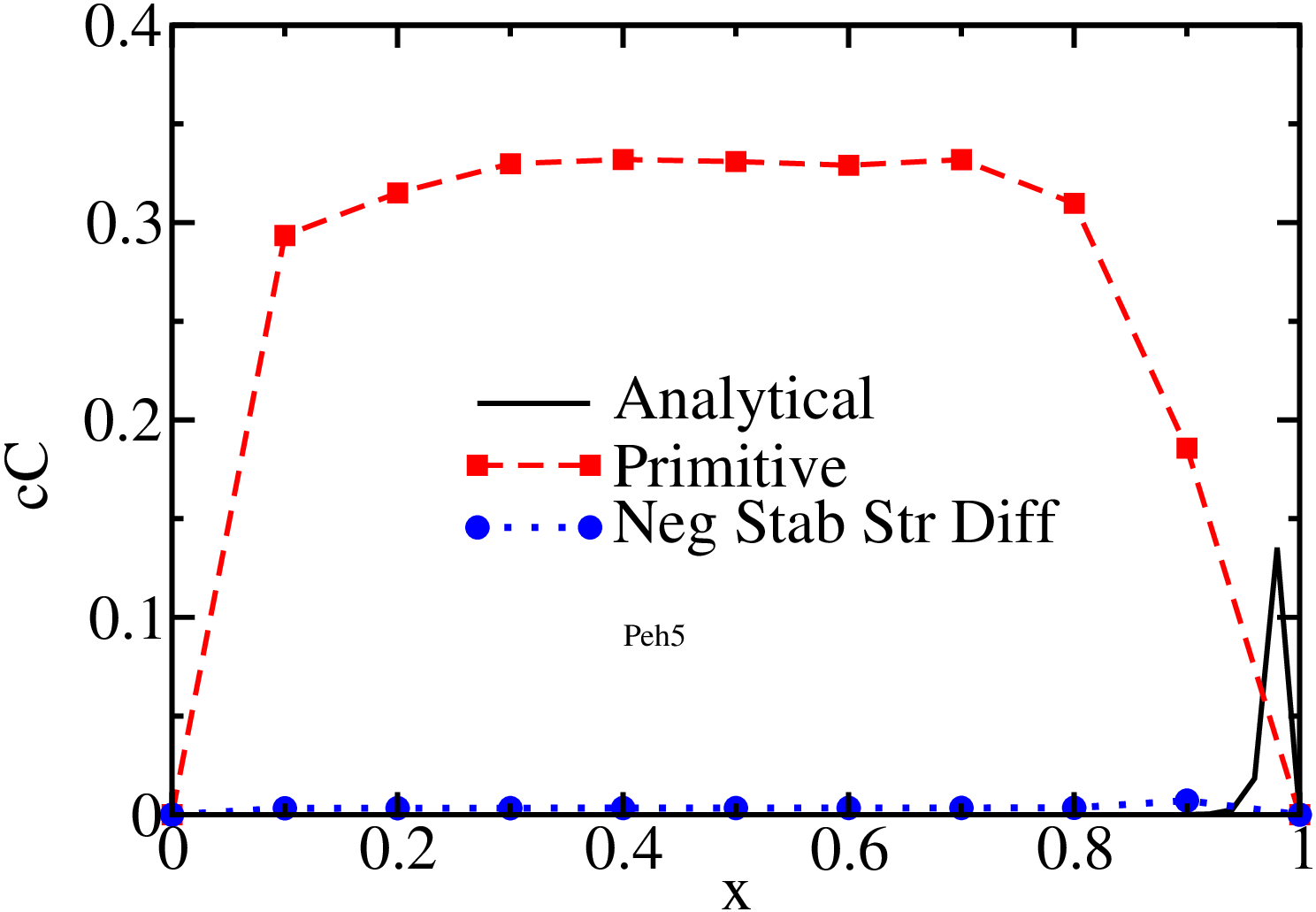}}
  \caption{\textsf{1D irreversible bimolecular fast reaction problem 
    (Case \#2):}~This figure compares the concentration profile of 
    the chemical species $A$, $B$, and $C$ for various element P\'{e}clet 
    numbers using unconstrained primitive and unconstrained negatively 
    stabilized streamline diffusion LSFEMs to that of the analytical 
    solution. The negatively stabilized 
    streamline diffusion LSFEM is able to capture the features near 
    the boundary layer with considerable accuracy even on coarse meshes. 
    \label{Fig:1D_Bimolecular_Type2OtherSet_concABC_Unconstrained}}
  \vspace{0.22in}
\end{figure}


\begin{figure}
  \centering
  \psfrag{A}{$A$}
  \psfrag{B}{$B$}
  \psfrag{O}{$(0,0)$}
  \psfrag{cpa}{$c^{\mathrm{p}}_{A}$}
  \psfrag{cpb}{$c^{\mathrm{p}}_{B}$}
  \psfrag{f1}{$f_i(\mathbf{x}) = 0$}
  \psfrag{h1}{$c^{\mathrm{p}}_i(\mathbf{x}) = 0$}
  \psfrag{h2}{\rotatebox{-90}{$c^{\mathrm{p}}_i(\mathbf{x}) = 0$}}
  \psfrag{h3}{$c^{\mathrm{p}}_i(\mathbf{x}) = 0$}
  \psfrag{Lx}{$L_x$}
  \psfrag{Ly1}{\rotatebox{90}{$L_y/2$}}
  \psfrag{Ly2}{\rotatebox{90}{$L_y/2$}}
  \subfigure[Problem description]{\includegraphics[scale=0.7]
    {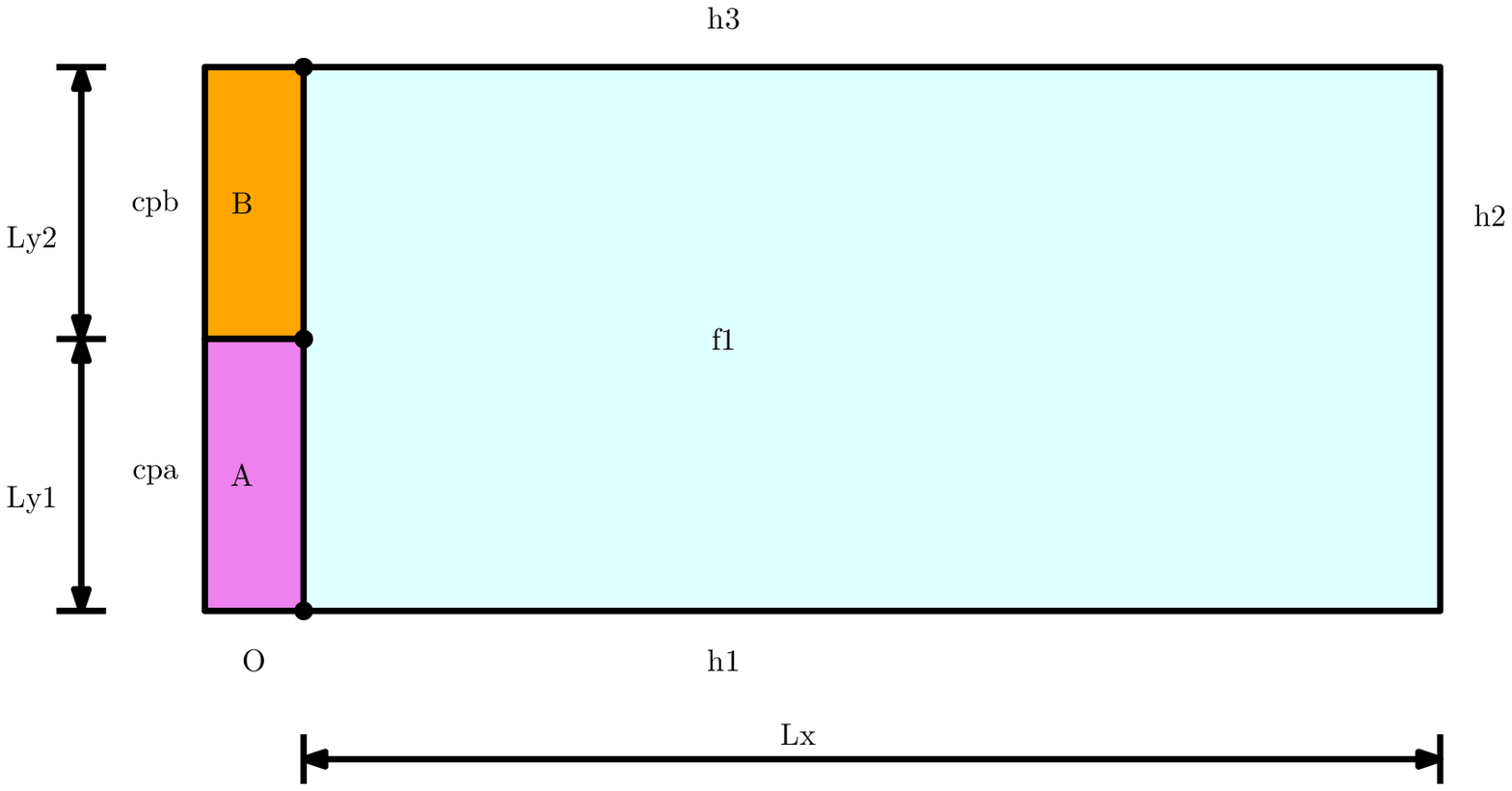}}
  \hspace{0.05in}
  \subfigure[Stream function and advection velocity vector field]
    {\includegraphics[width = 4.5in, height = 2.75in,clip]
    {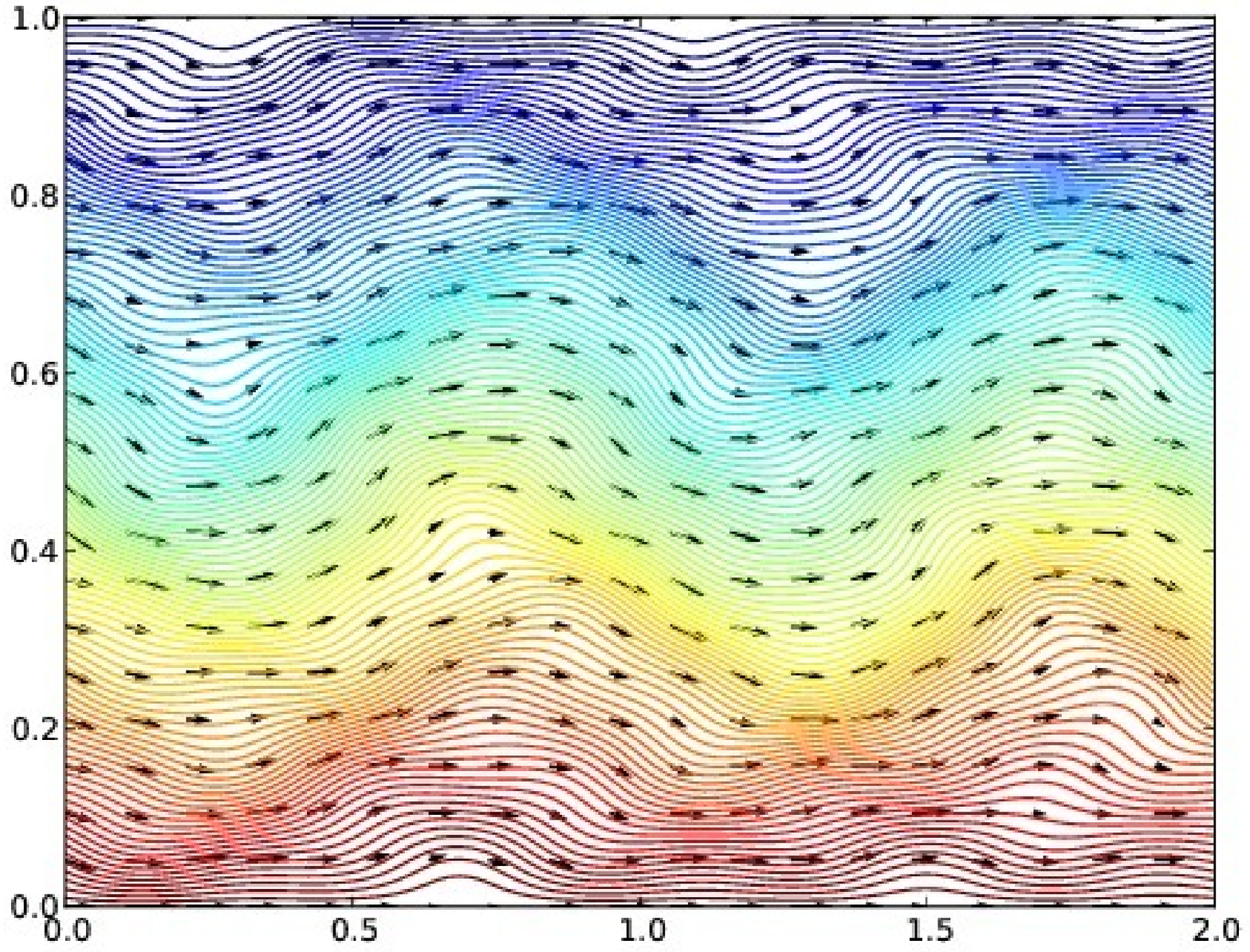}}
  \caption{\textsf{Plume development from boundary in a reaction 
    tank:}~The top figure provides a pictorial description of the 
    boundary value problem. The bottom figure shows the contours 
    of the stream function corresponding to the advection velocity 
    vector field.
  \label{Fig:2D_Bimolecular_ScalDiff_Plume}}
\end{figure}

\begin{figure}
  \centering
  \subfigure[$p = 1$, XSeed = YSeed = 101]
    {\includegraphics[scale = 0.28,clip]
    {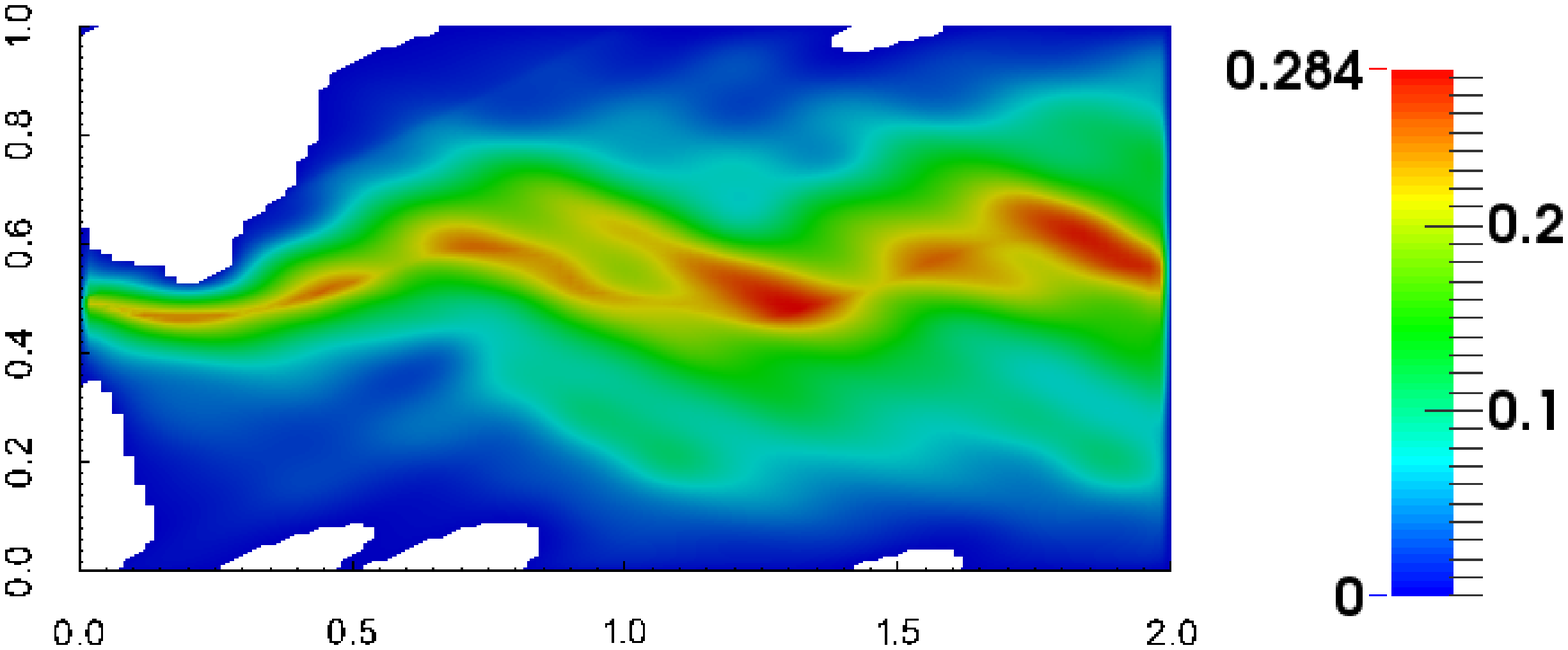}}
  \hspace{0.15in}
  \vspace{0.1in}
  \subfigure[$p = 1$, XSeed = YSeed = 501]
    {\includegraphics[scale = 0.28,clip]
    {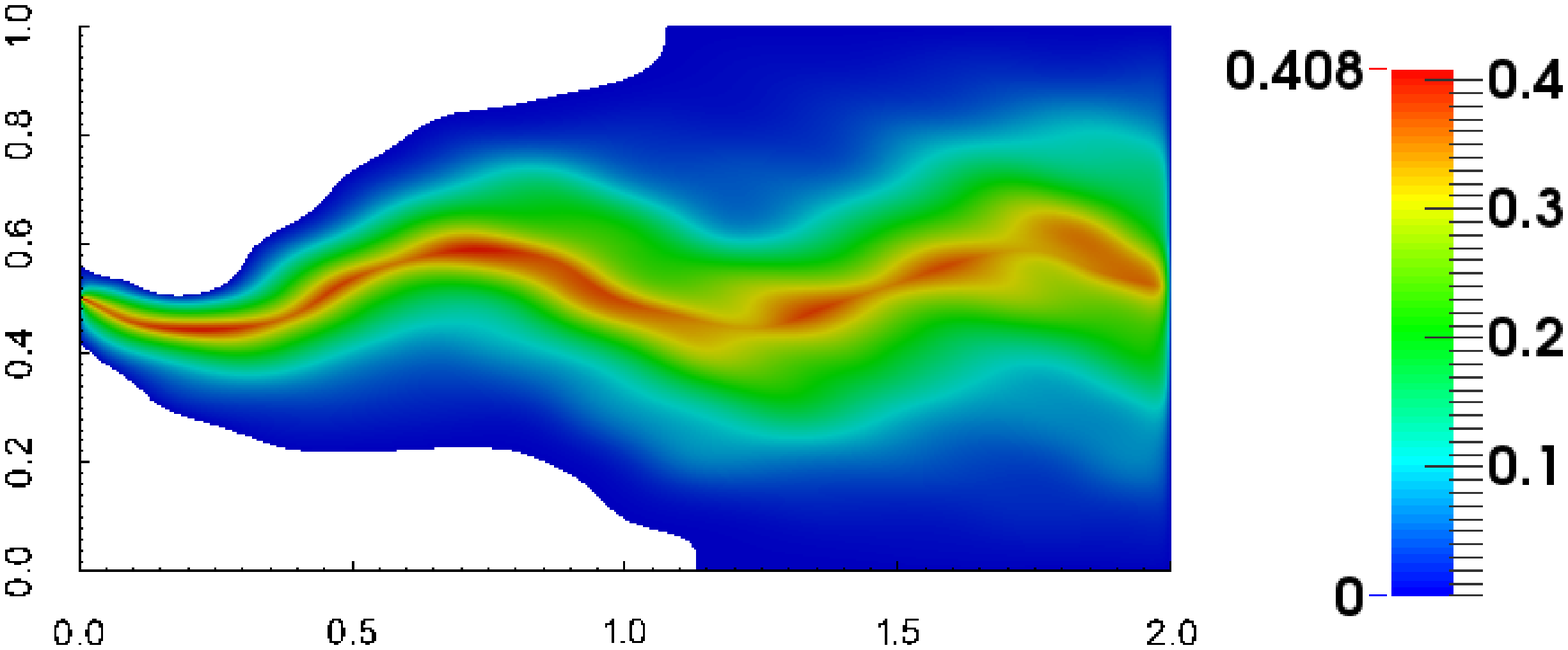}}
  \hspace{0.15in}
  \subfigure[$p = 2$, XSeed = YSeed = 101]
    {\includegraphics[scale = 0.28,clip]
    {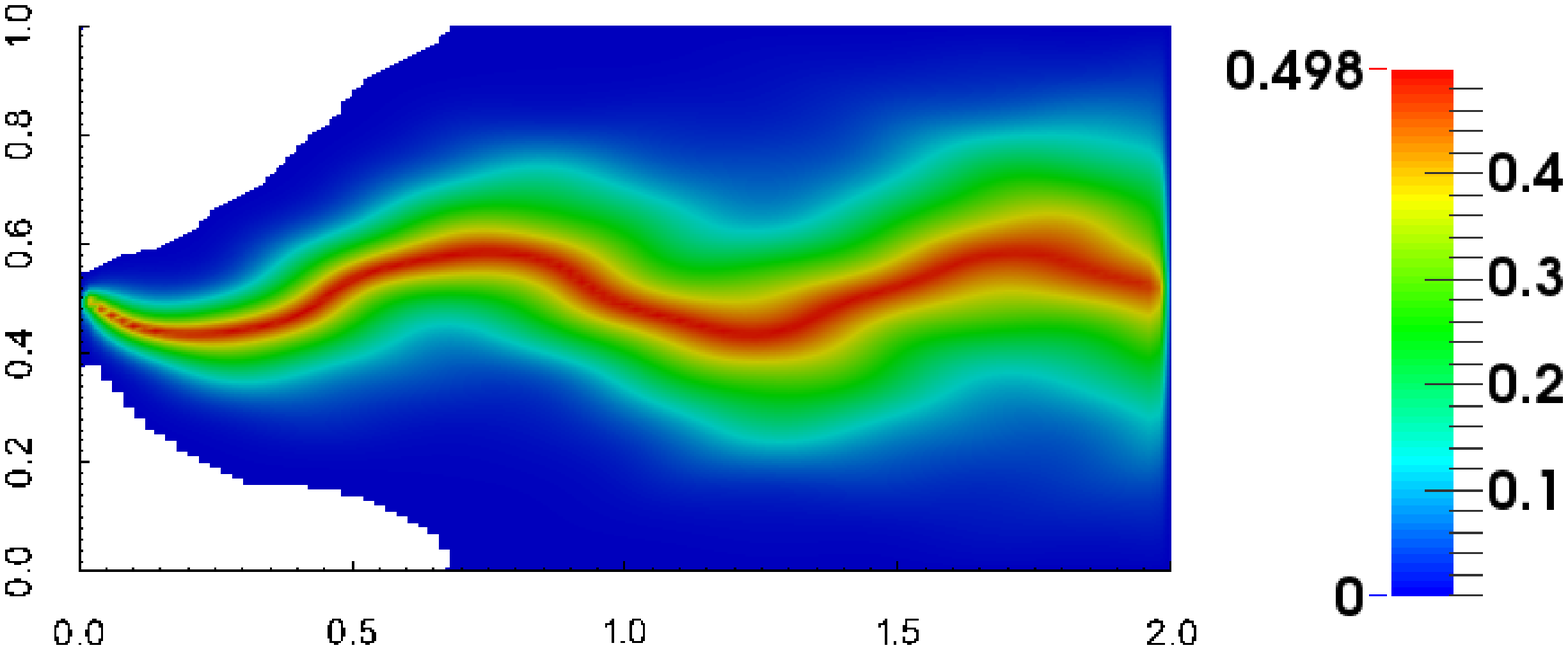}}
  \hspace{0.15in}
  \subfigure[$p = 3$, XSeed = YSeed = 66]
    {\includegraphics[scale = 0.28,clip]
    {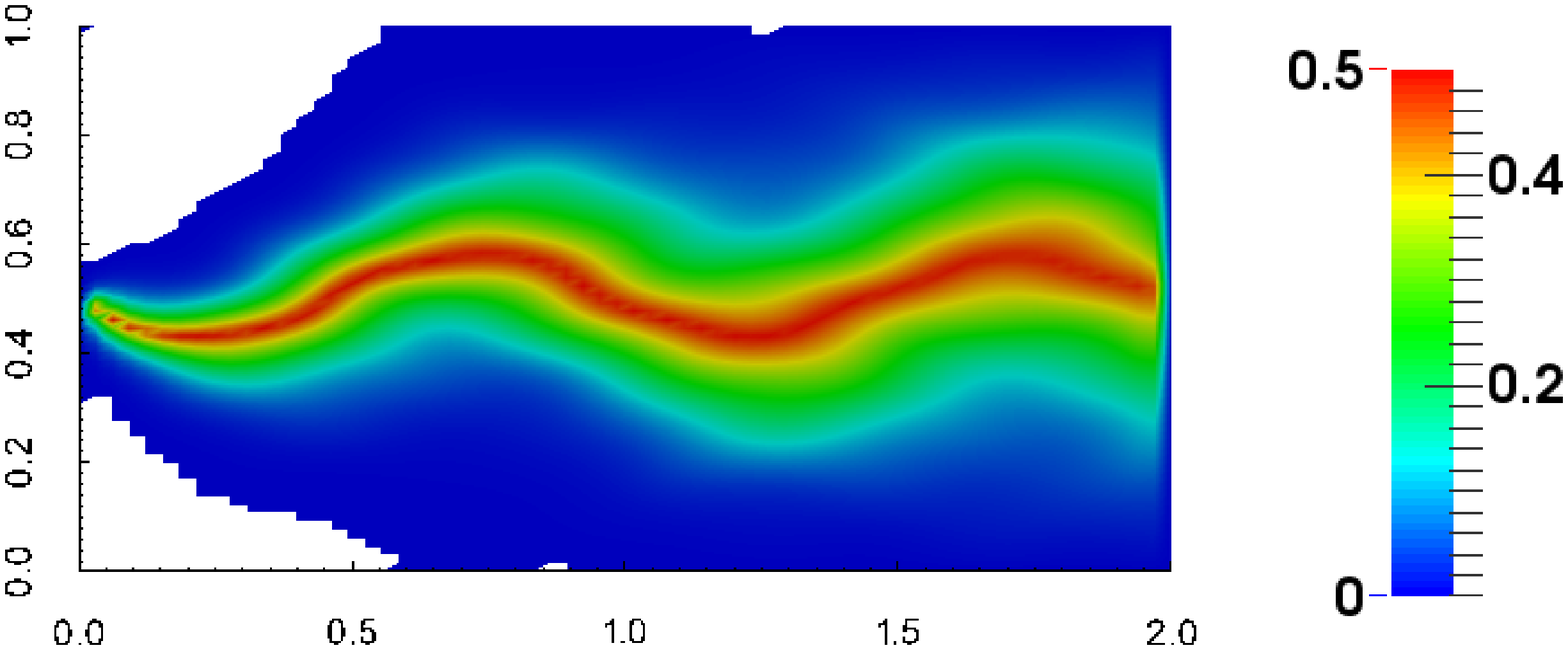}}
  \caption{\textsf{Plume development from boundary in a reaction 
    tank (Type \#1):}~This figure shows the concentration profiles 
    of the product $C$ based on unconstrained primitive LSFEM. The 
    white region shows the area in which concentration is negative. 
    Both the lower-order and higher-order finite elements violate 
    the non-negative constraint. The negative values are in the 
    range $\mathcal{O}(10^{-2})$ to $\mathcal{O}(10^{-4})$, which 
    are not close to the machine precision $\epsilon_{\mathrm{mach}} 
    = \mathcal{O}(10^{-16})$.
  \label{Fig:2D_Bimolecular_ScalDiff_ConcC_PrimNoCons}}
\end{figure}

\begin{figure}
  \centering
  \subfigure[XSeed = YSeed = 501 (No constraints)]
    {\includegraphics[scale = 0.325,clip]
    {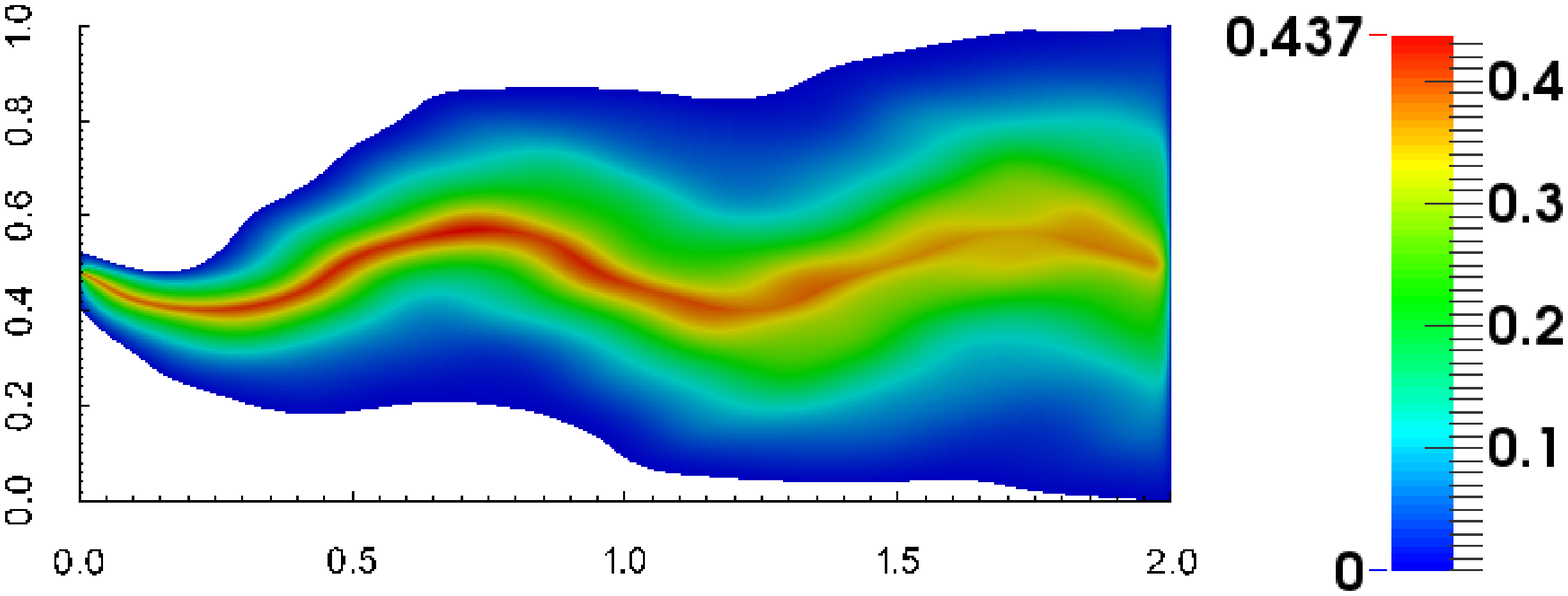}}
  \hspace{0.015in}
  \subfigure[XSeed = YSeed = 251 (NN constraints)]
    {\includegraphics[scale = 0.30,clip]
    {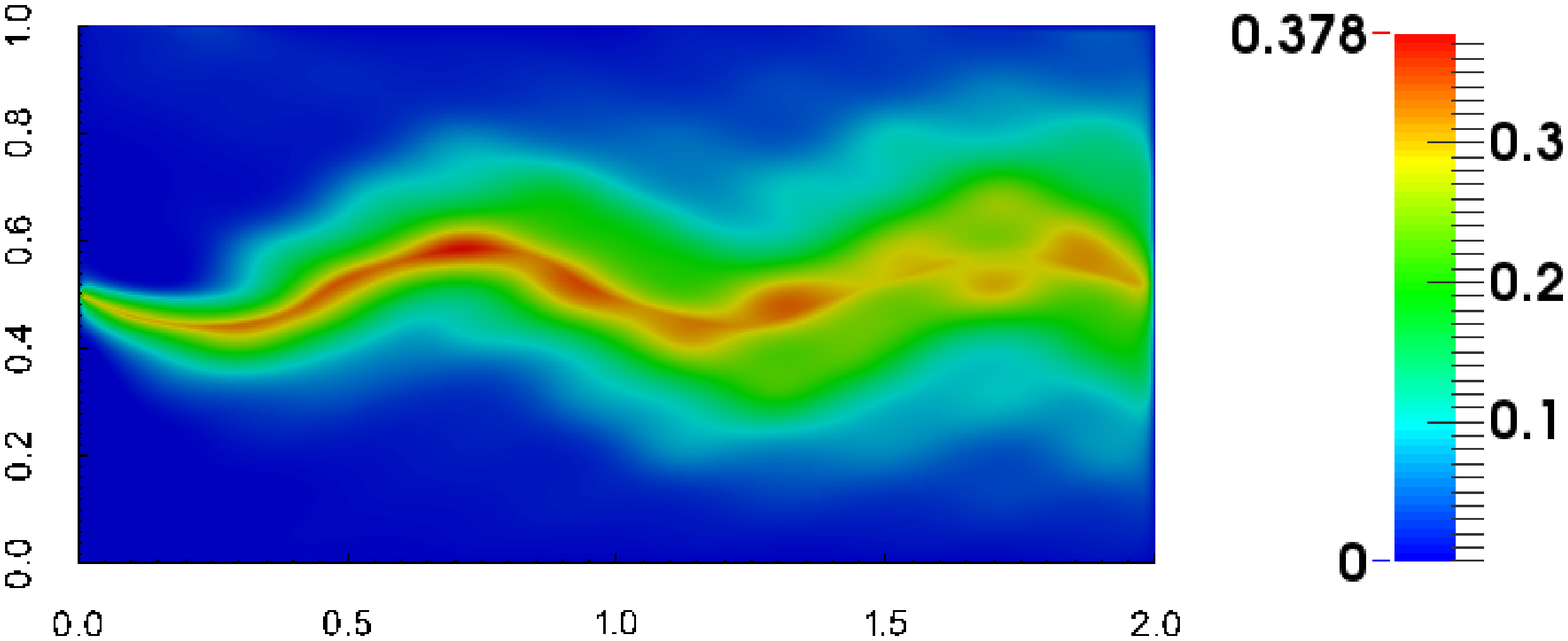}}
  \subfigure[XSeed = YSeed = 251 (LSB and NN constraints)]
    {\includegraphics[scale = 0.30,clip]
    {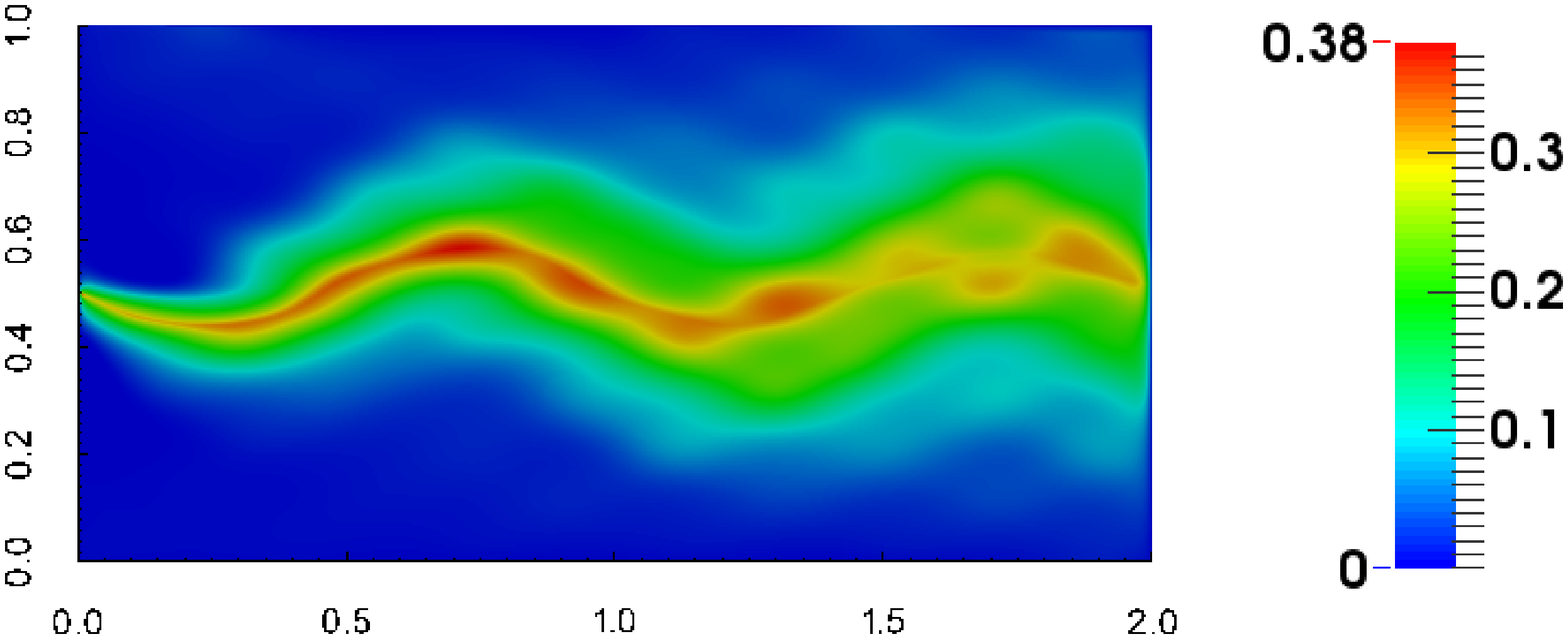}}
  \caption{\textsf{Plume development from boundary in a reaction 
    tank (Type \#1):}~This figure shows the concentration profiles 
    of the product $C$ based on unconstrained and constrained 
    negatively stabilized streamline diffusion LSFEM. The white 
    region shows the area in which concentration is negative. 
    Considerable part of the domain violated the non-negative 
    constraint. The proposed framework with NN and LSB constraints 
    is able to capture the plume formation on a coarse mesh for 
    a highly heterogeneous advection velocity vector field.  
  \label{Fig:2D_Bimolecular_ScalDiff_ConcC_NegStabStrDiff}}
\end{figure}

\begin{figure}
  \centering
  \subfigure[$p = 1$, XSeed = YSeed = 101]
    {\includegraphics[scale = 0.28,clip]
    {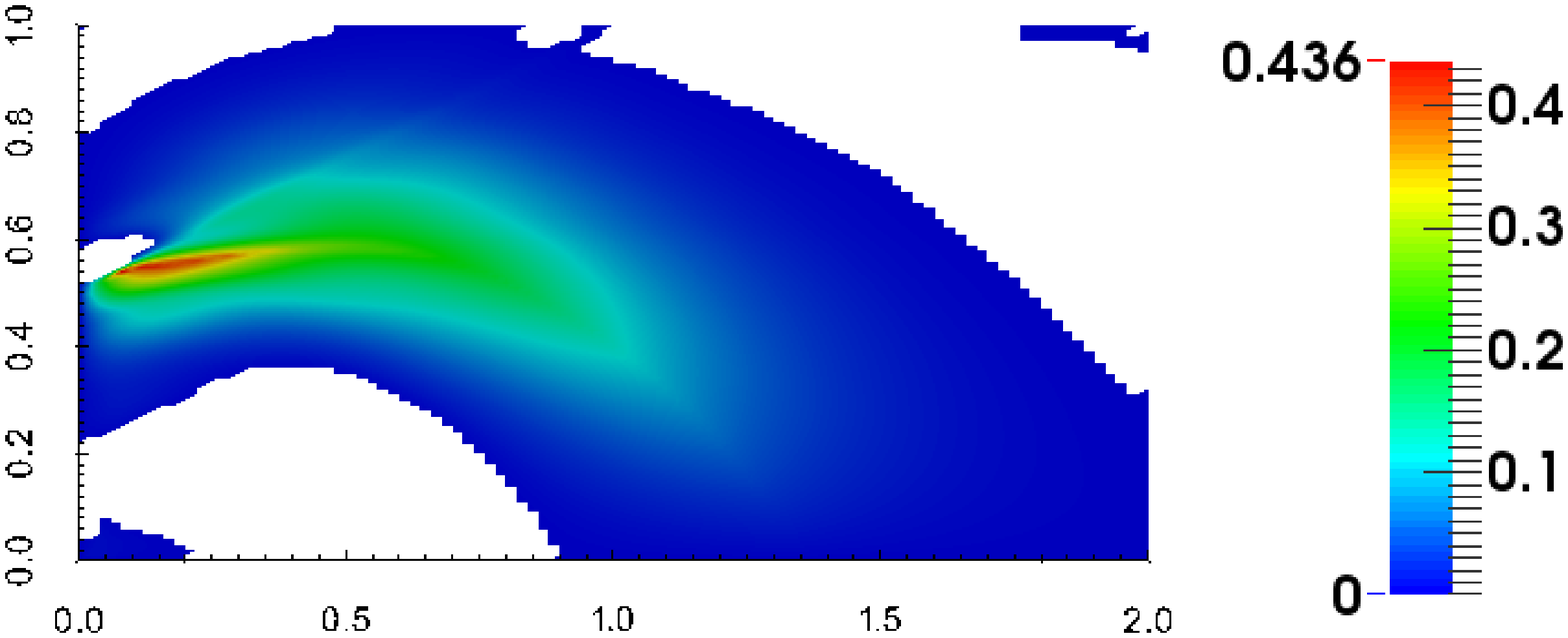}}
  \hspace{0.15in}
  \vspace{0.1in}
  \subfigure[$p = 1$, XSeed = YSeed = 501]
    {\includegraphics[scale = 0.28,clip]
    {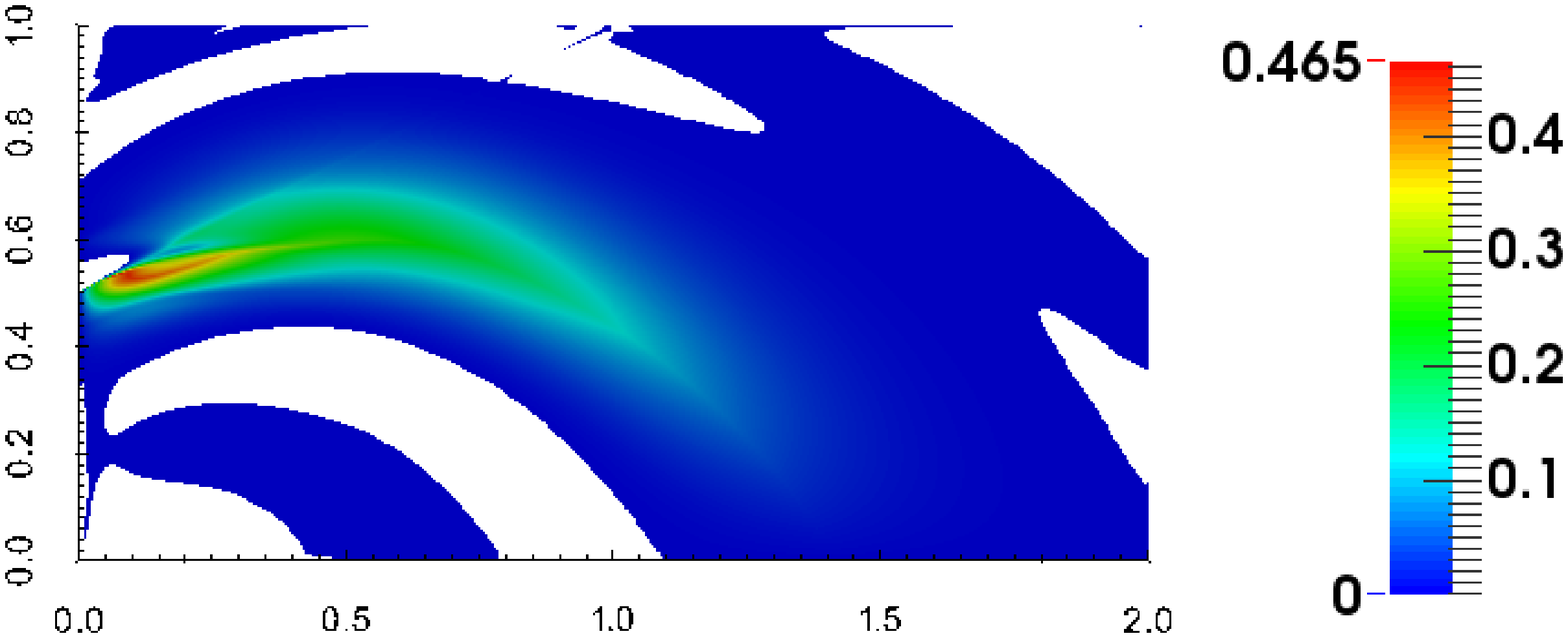}}
  \hspace{0.15in}
  \subfigure[$p = 2$, XSeed = YSeed = 101]
    {\includegraphics[scale = 0.28,clip]
    {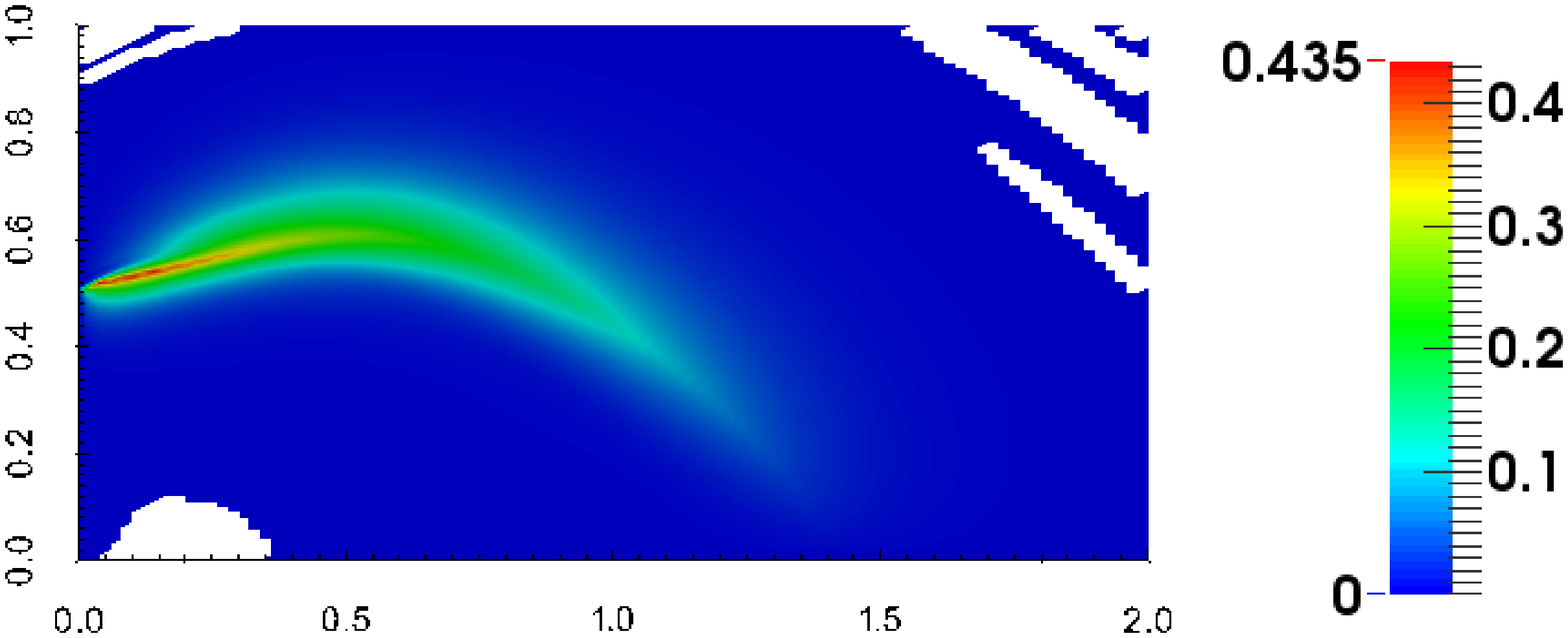}}
  \hspace{0.15in}
  \subfigure[$p = 3$, XSeed = YSeed = 66]
    {\includegraphics[scale = 0.28,clip]
    {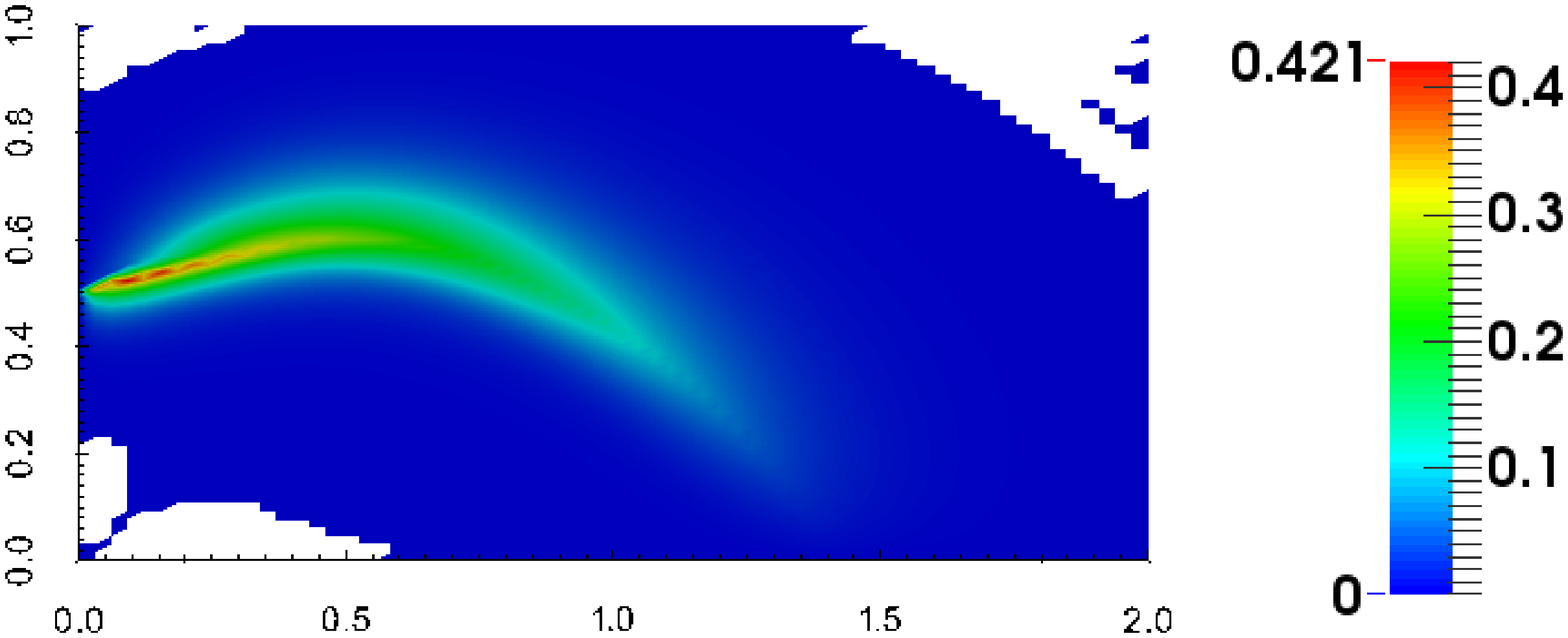}}
  \caption{\textsf{Plume development from boundary in a reaction 
    tank (Type \#2):}~This figure shows the concentration profiles 
    of the product $C$ based on unconstrained primitive LSFEM. The 
    white region indicates the area in which the obtained concentration 
    is negative. The negative values are in the range $\mathcal{O}(10^
    {-3})$ to $\mathcal{O}(10^{-5})$. Both $h$-refinement 
    and $p$-refinement could not eliminate the violation in the non-negative 
    constraint for this problem, which has highly heterogeneous 
    anisotropic diffusivity.
  \label{Fig:2D_Bimolecular_AnisoDiff_ConcC_PrimNoCons}}
\end{figure}

\begin{figure}
  \centering
  \subfigure[XSeed = YSeed = 251 (No constraints)]
    {\includegraphics[scale = 0.325,clip]
    {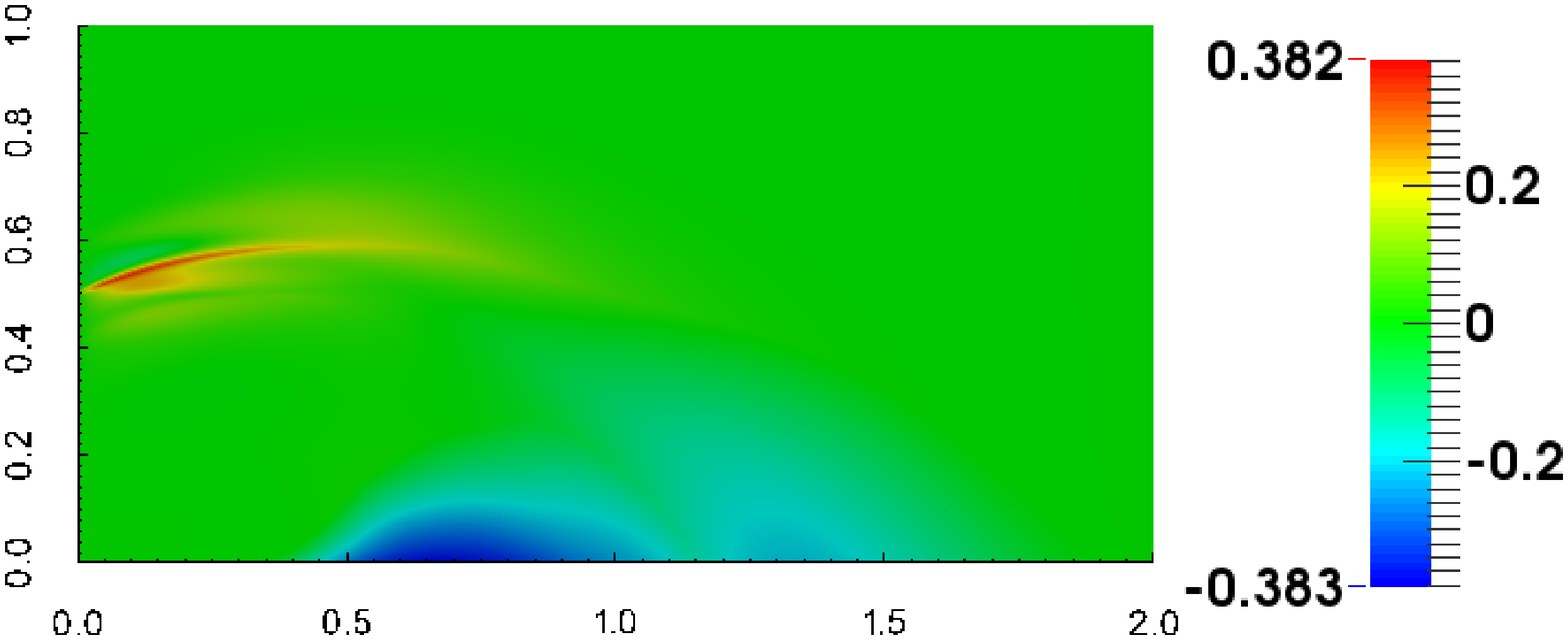}}
  \hspace{0.015in}
  \subfigure[XSeed = YSeed = 251 (NN constraints)]
    {\includegraphics[scale = 0.325,clip]
    {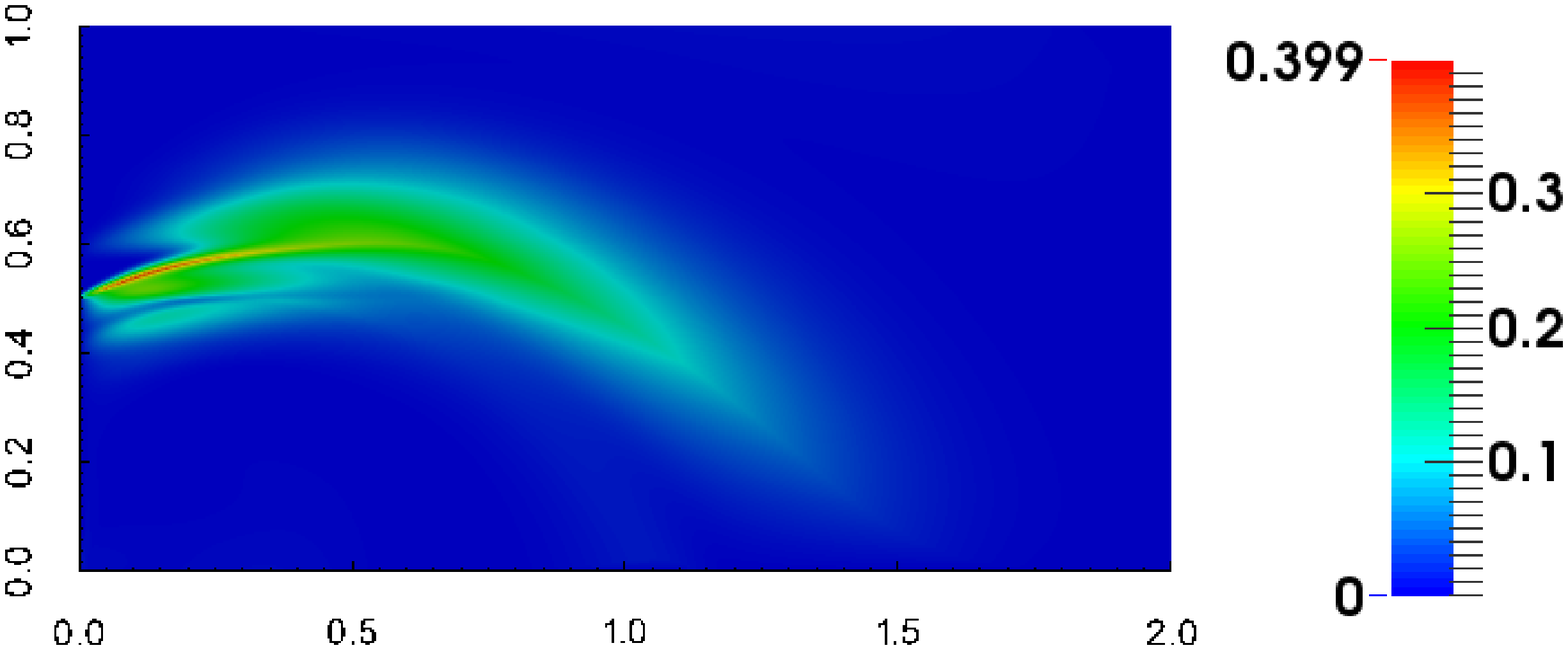}}
  \subfigure[XSeed = YSeed = 251 (LSB and NN constraints)]
    {\includegraphics[scale = 0.325,clip]
    {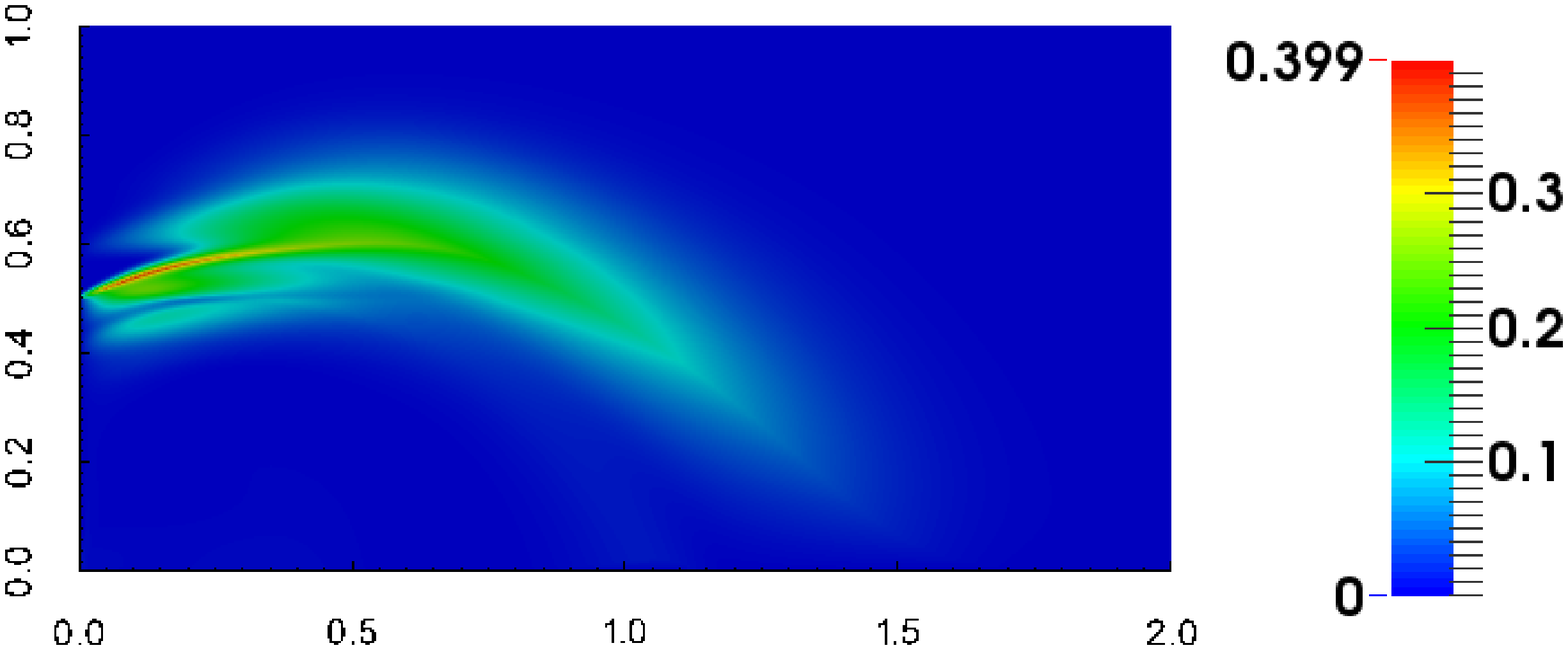}}
  \caption{\textsf{Plume development from boundary in a reaction 
    tank (Type \#2):}~This figure shows the concentration profiles 
    of the product $C$ based on unconstrained and constrained 
    negatively stabilized streamline diffusion LSFEM. Compared 
    to Figure \ref{Fig:2D_Bimolecular_AnisoDiff_ConcC_PrimNoCons}, 
    the proposed methodology with NN and LSB constraints is able 
    to accurately describe the plume formation of the product $C$ 
    even for highly heterogeneous anisotropic diffusivities and 
    highly spatially varying velocity fields.
  \label{Fig:2D_Bimolecular_AnisoDiff_ConcC_NegStabStrDiff}}
\end{figure}

\begin{figure}
  \centering
  \subfigure[No constraints]
    {\includegraphics[scale = 0.12,clip]{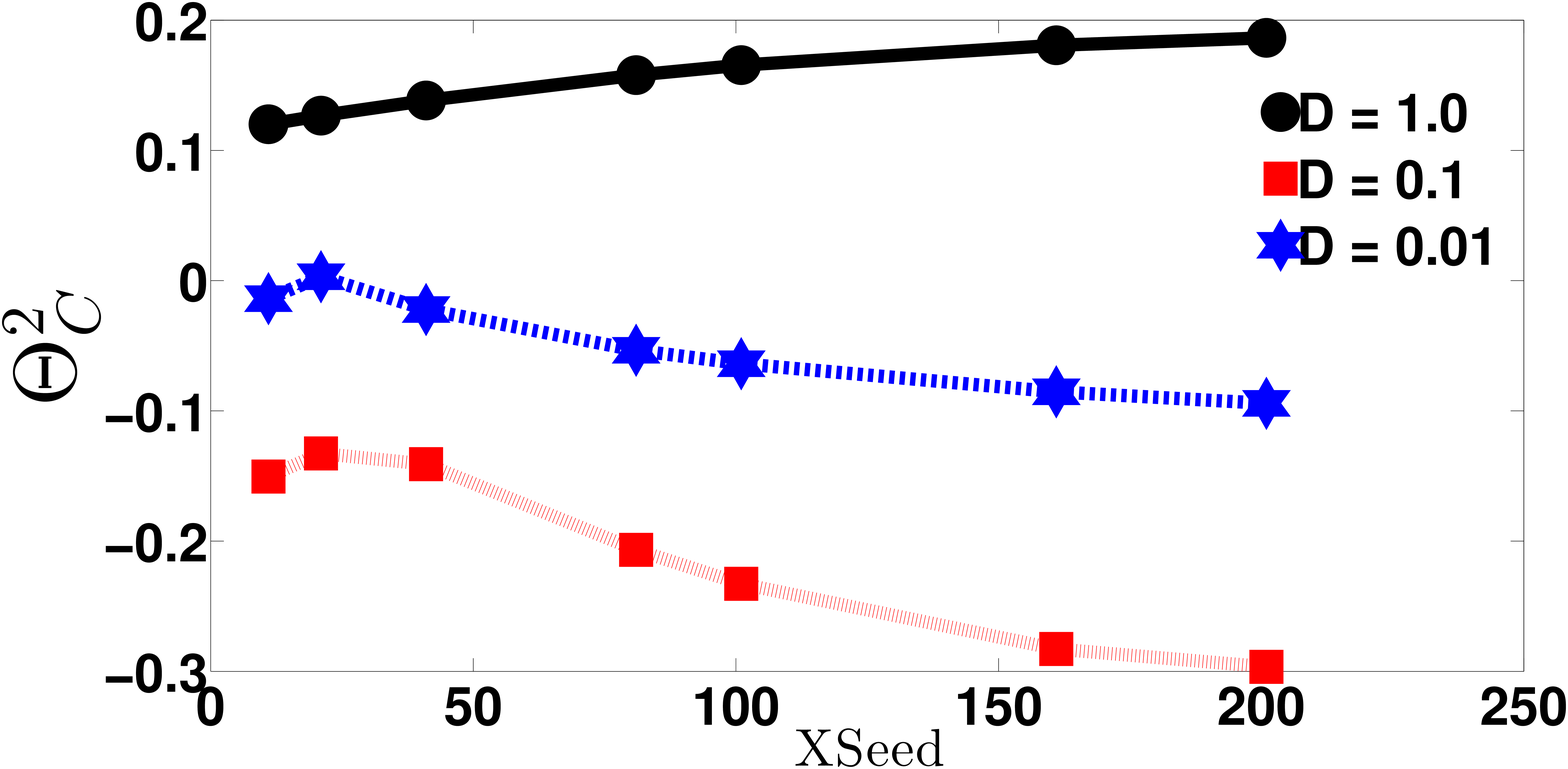}}
  \subfigure[With LSB and DMP constraints]
    {\includegraphics[scale = 0.12,clip]{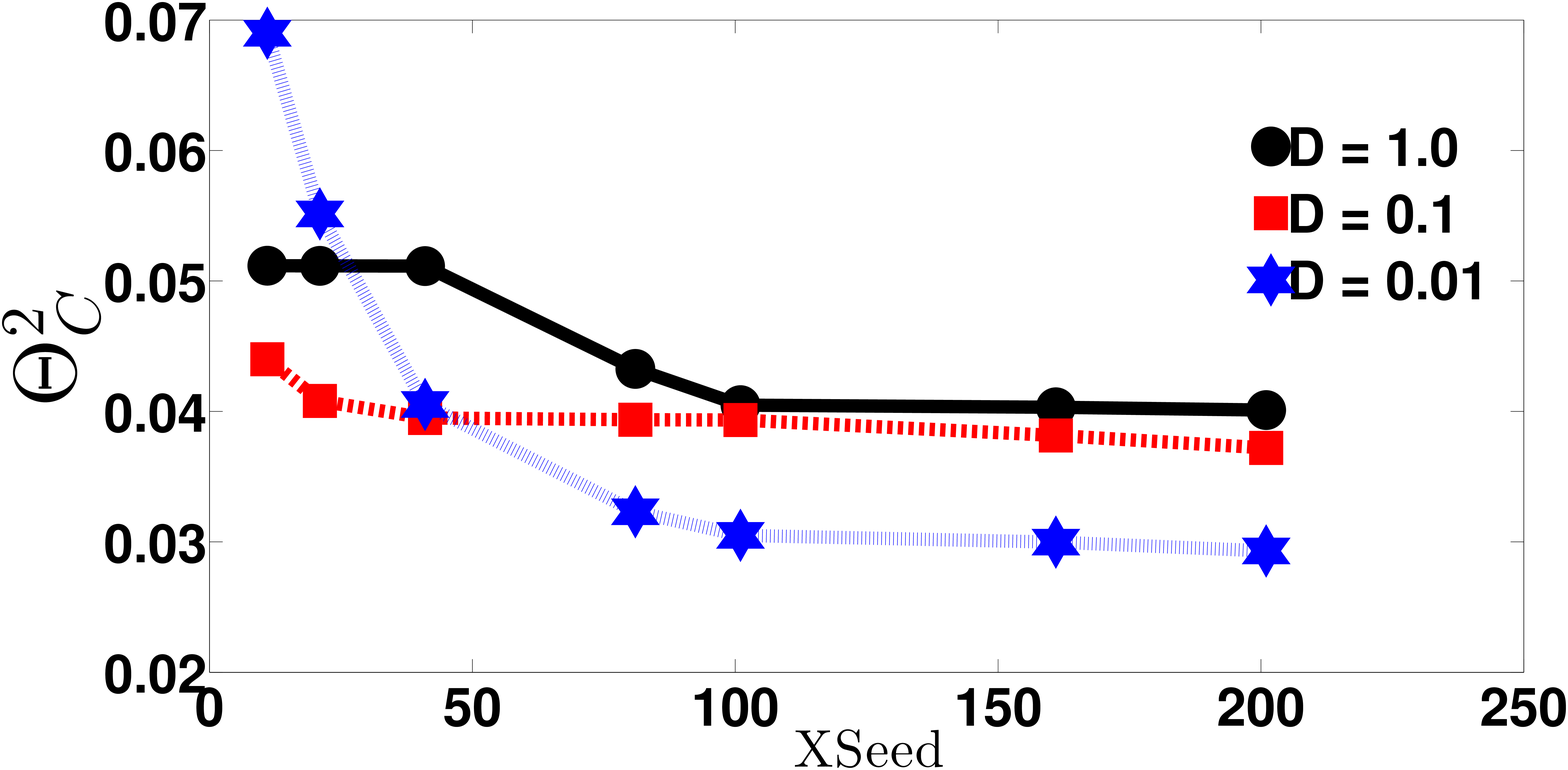}}
  \caption{\textsf{Plume development from boundary in a reaction 
    tank (Type \#1):}~This figure shows the variation $\Theta_C^2$ 
    with mesh refinement under the weighted negatively stabilized 
    streamline diffusion LSFEM. It should be noted that $\Theta_C^2$ 
    is a non-negative quantity. However, the \emph{unconstrained} 
    negatively stabilized streamline diffusion LSFEM gives negative 
    values for $\Theta_C^2$. Mesh refinement did not alleviate this 
    problem. On the other hand, the proposed framework not only 
    gives non-negative values for $\Theta_C^2$ but flattens upon 
    mesh refinement, which indicates convergence.
    \label{Fig:2D_Bimolecular_ScalDiff_ConcC_NegStabStrDiff_NSSD}}
\end{figure}

\begin{figure}
  \centering
  \includegraphics[scale = 0.15,clip]{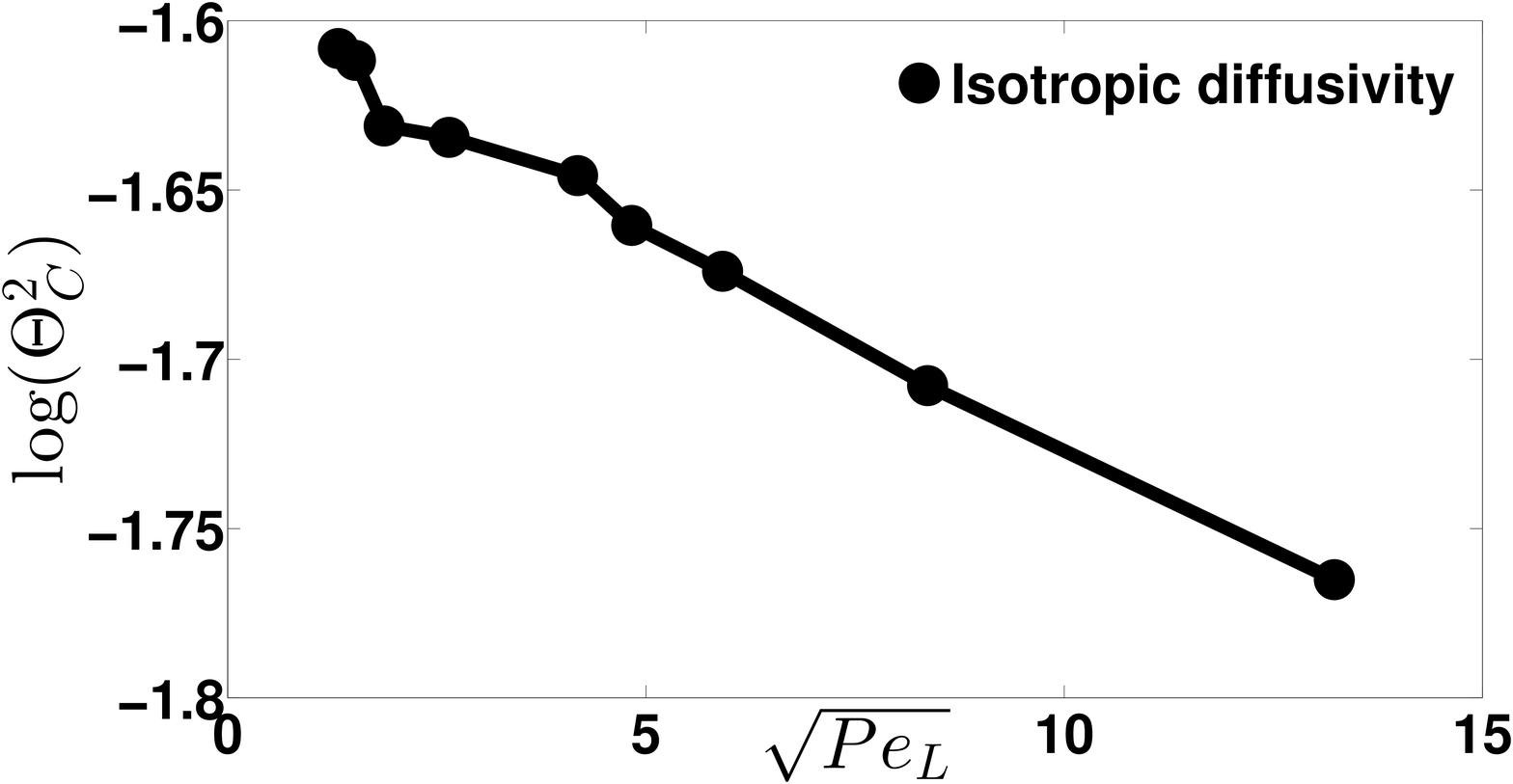}
  \caption{\textsf{Plume development from boundary 
    in a reaction tank (Type \#1):}~This figure shows 
    the variation $\log(\Theta^2_C)$ 
    with respect to $\sqrt{\mathbb{P}\mathrm{e}_L}$ for isotropic 
    diffusivity under the weighted negatively stabilized streamline 
    diffusion LSFEM with LSB and DMP constraints. Herein, analysis 
    is performed using XSeed = YSeed = 201. Through numerical 
    simulations, we observed that $\log(\Theta^2_C) \propto 
    \sqrt{\mathbb{P}\mathrm{e}_L}$.
  \label{Fig:2D_BiMolScalDiff_ConcC_NSSD}}
\end{figure}


\begin{figure}
  \centering
  \psfrag{A}{Reactant $A$}
  \psfrag{B}{Reactant $B$}
  \psfrag{O}{$(0,0)$}
  \psfrag{f1}{$f_i(\mathbf{x},t) = 0$}
  \psfrag{q1}{$h^{\mathrm{p}}_i(\mathbf{x},t) = 0$}
  \psfrag{q2}{\rotatebox{-90}{$h^{\mathrm{p}}_i(\mathbf{x},t) = 0$}}
  \psfrag{q3}{$h^{\mathrm{p}}_i(\mathbf{x},t) = 0$}
  \psfrag{q4}{\rotatebox{90}{$h^{\mathrm{p}}_i(\mathbf{x},t) = 0$}}
  \psfrag{L2}{$L_x$}
  \psfrag{L3}{\rotatebox{-90}{$L_y$}}
  \subfigure[Problem description]{\includegraphics[scale=0.6]
    {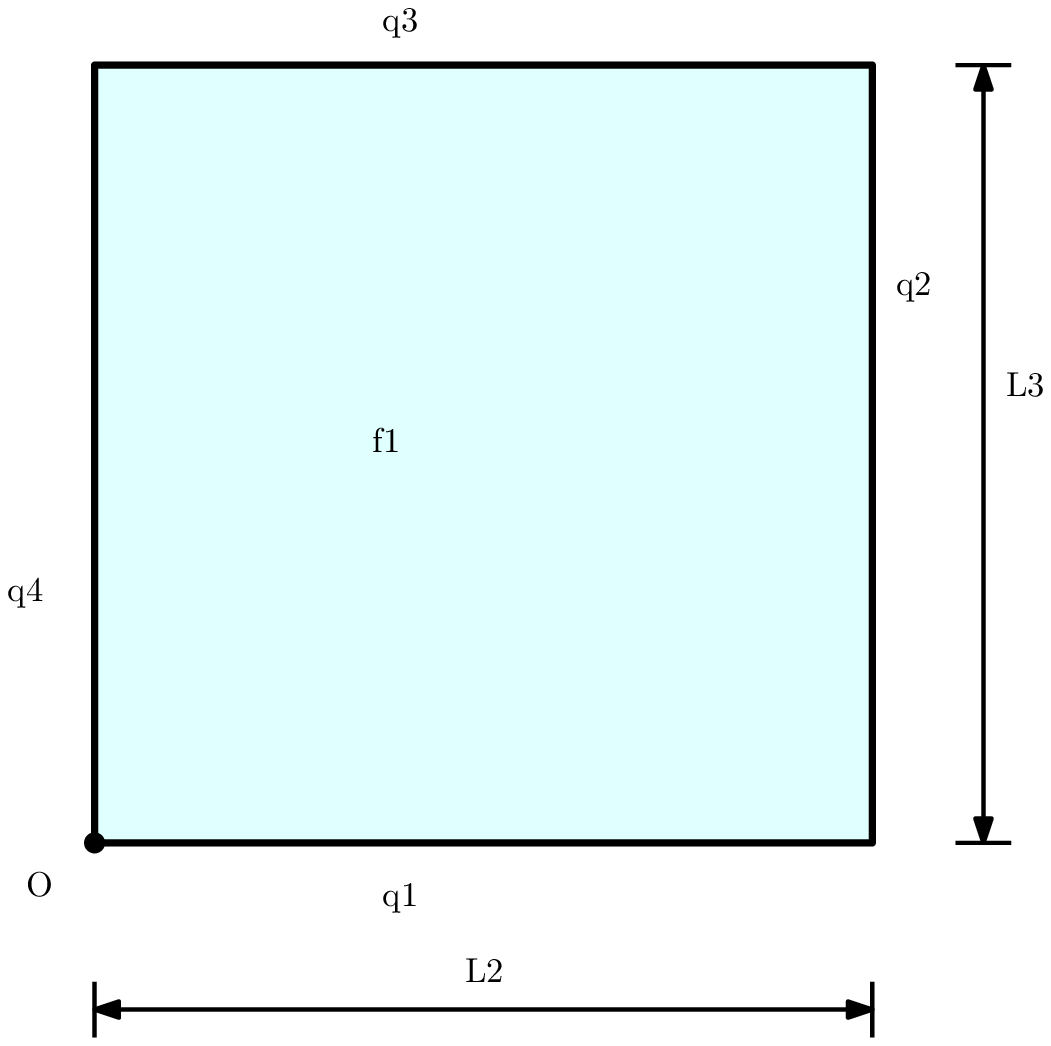}}
  \subfigure[Stream function and associated advection velocity field]{\includegraphics[scale=0.42]
    {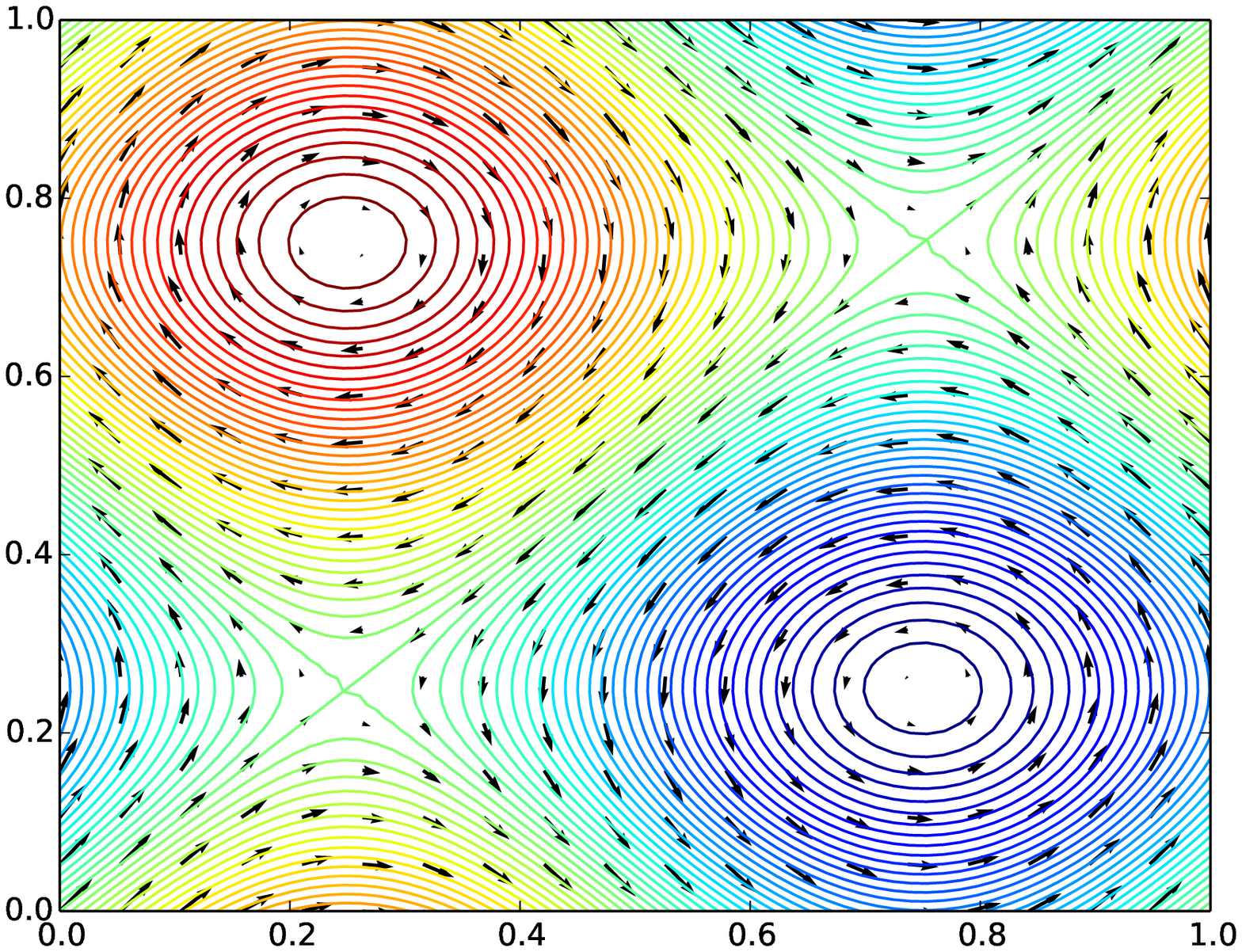}}
  \subfigure[Reactant $A$: Initial condition]{\includegraphics[scale=0.30,clip]
    {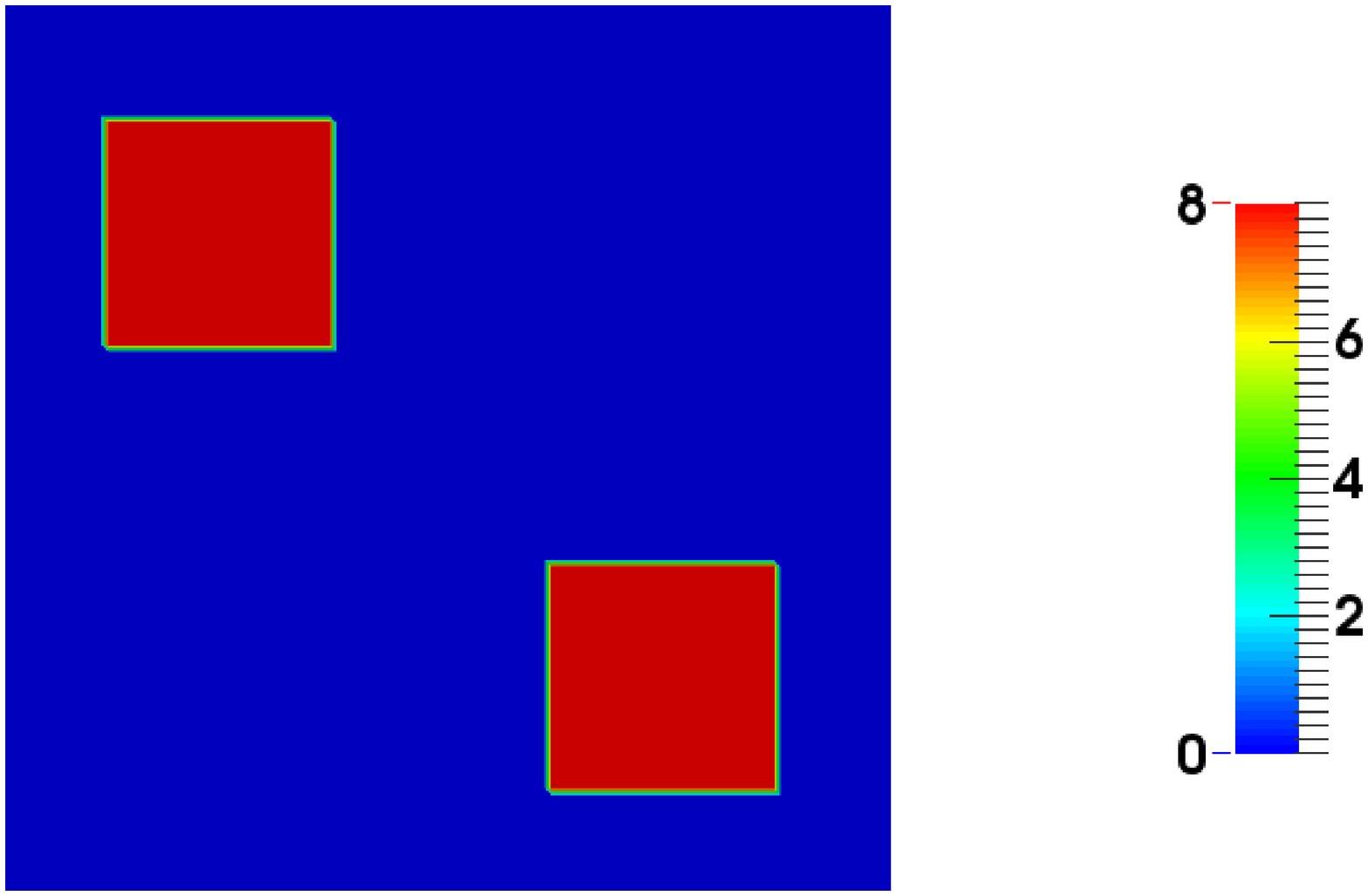}}
  \hspace{0.15in}
  \subfigure[Reactant $B$: Initial condition]{\includegraphics[scale=0.30,clip]
    {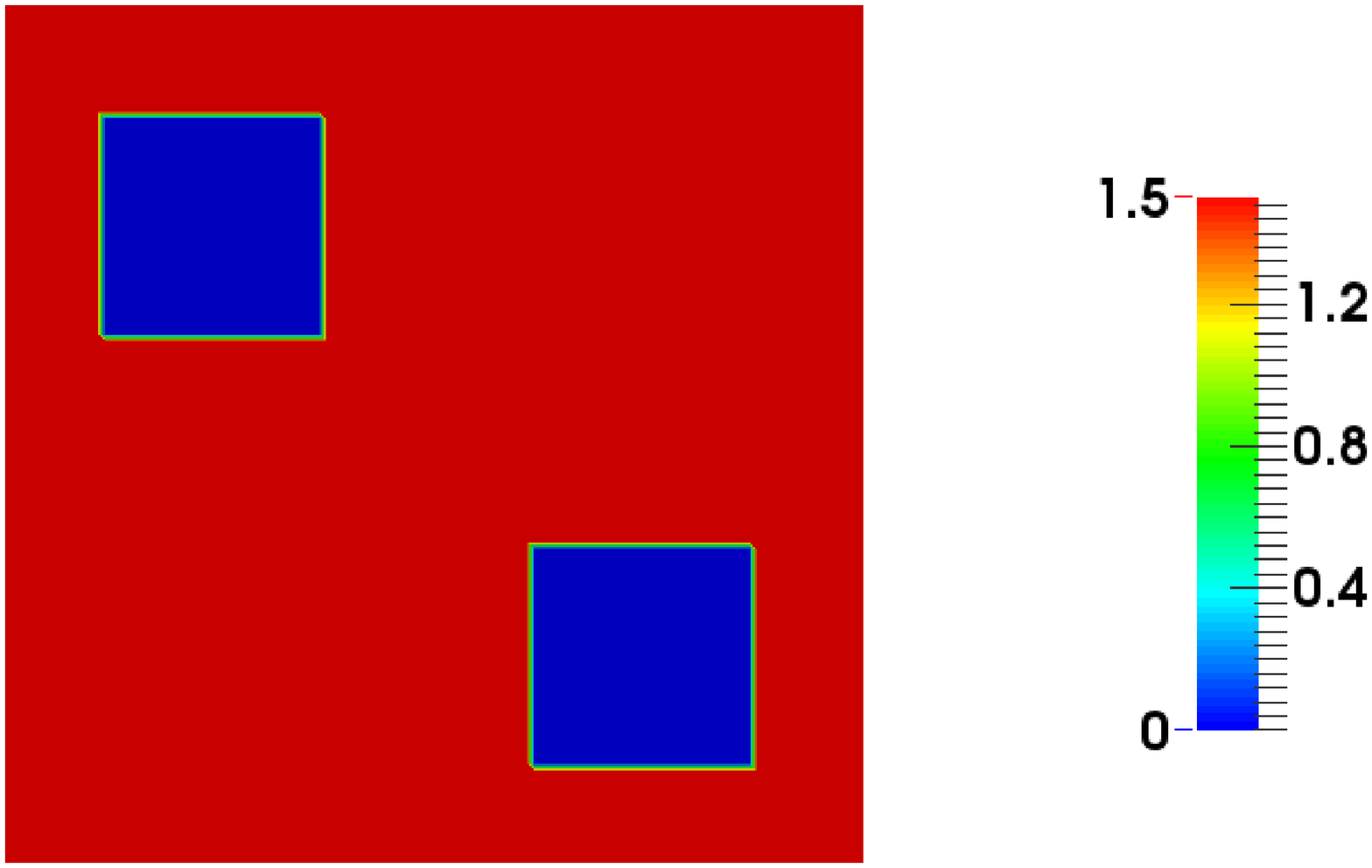}}
  \caption{\textsf{Vortex-stirred mixing in a reaction 
    tank:}~The top-left figure provides a pictorial 
    description of the initial boundary value problem. 
    The top-right figure shows the contours of the stream 
    function corresponding to the advection velocity field 
    given by \eqref{Eqn:TransMixing_VortexFlowModel}. 
    The bottom figures show the initial conditions for the 
    reactants $A$ and $B$ such that $\langle c_A(\mathbf{x},
    t=0) \rangle = \langle c_B(\mathbf{x},t=0) \rangle = 1$.
    \label{Fig:2D_Bimolecular_ScalDiff_Mixing_Type1}}
\end{figure}

\begin{figure}
  \centering
  \subfigure[Product $C$ at $t = 0.0001$]
    {\includegraphics[scale=0.30,clip]
    {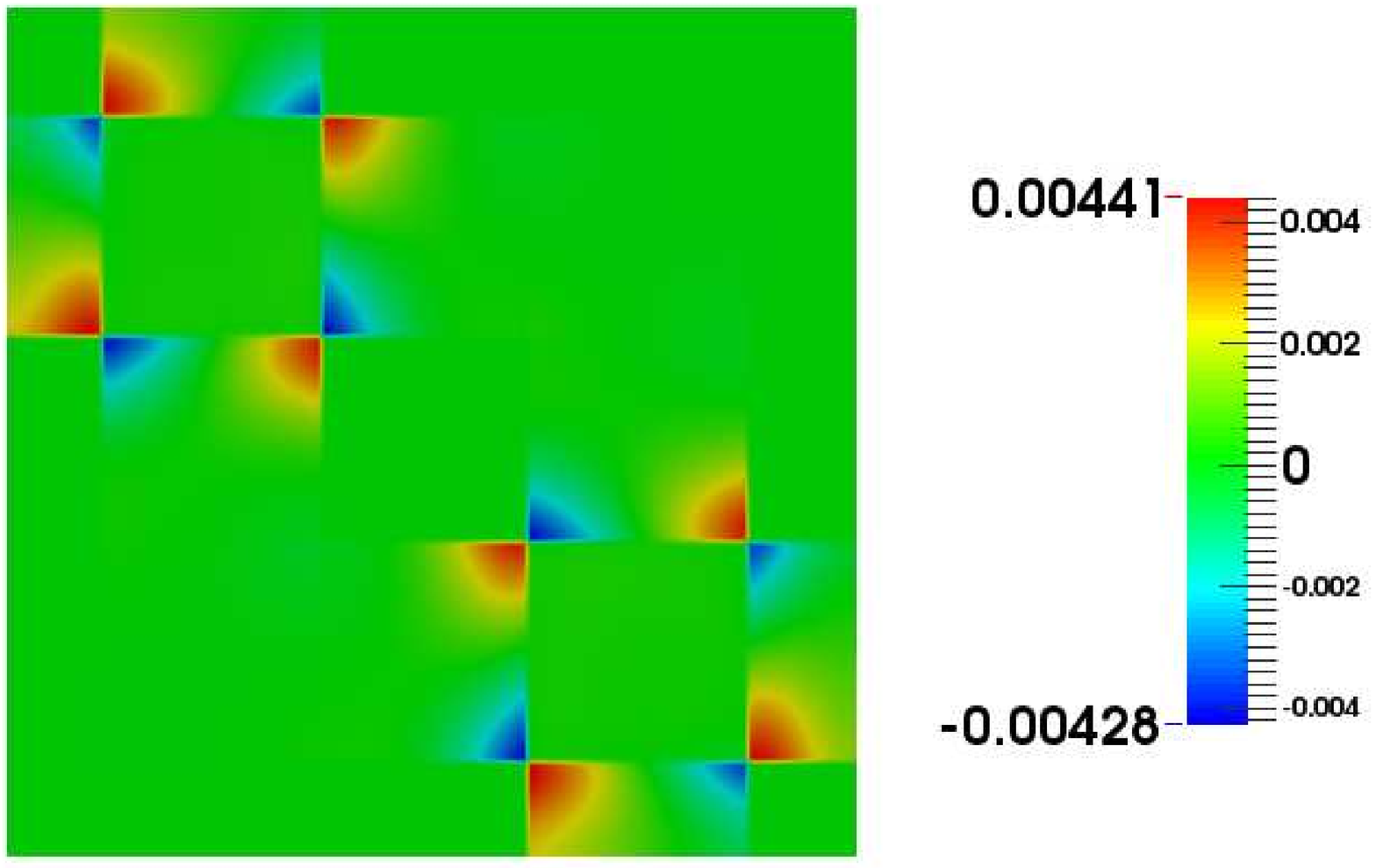}}
  \hspace{0.15in}
  \vspace{0.1in}
  \subfigure[Product $C$ at $t = 0.001$]
    {\includegraphics[scale=0.30,clip]
    {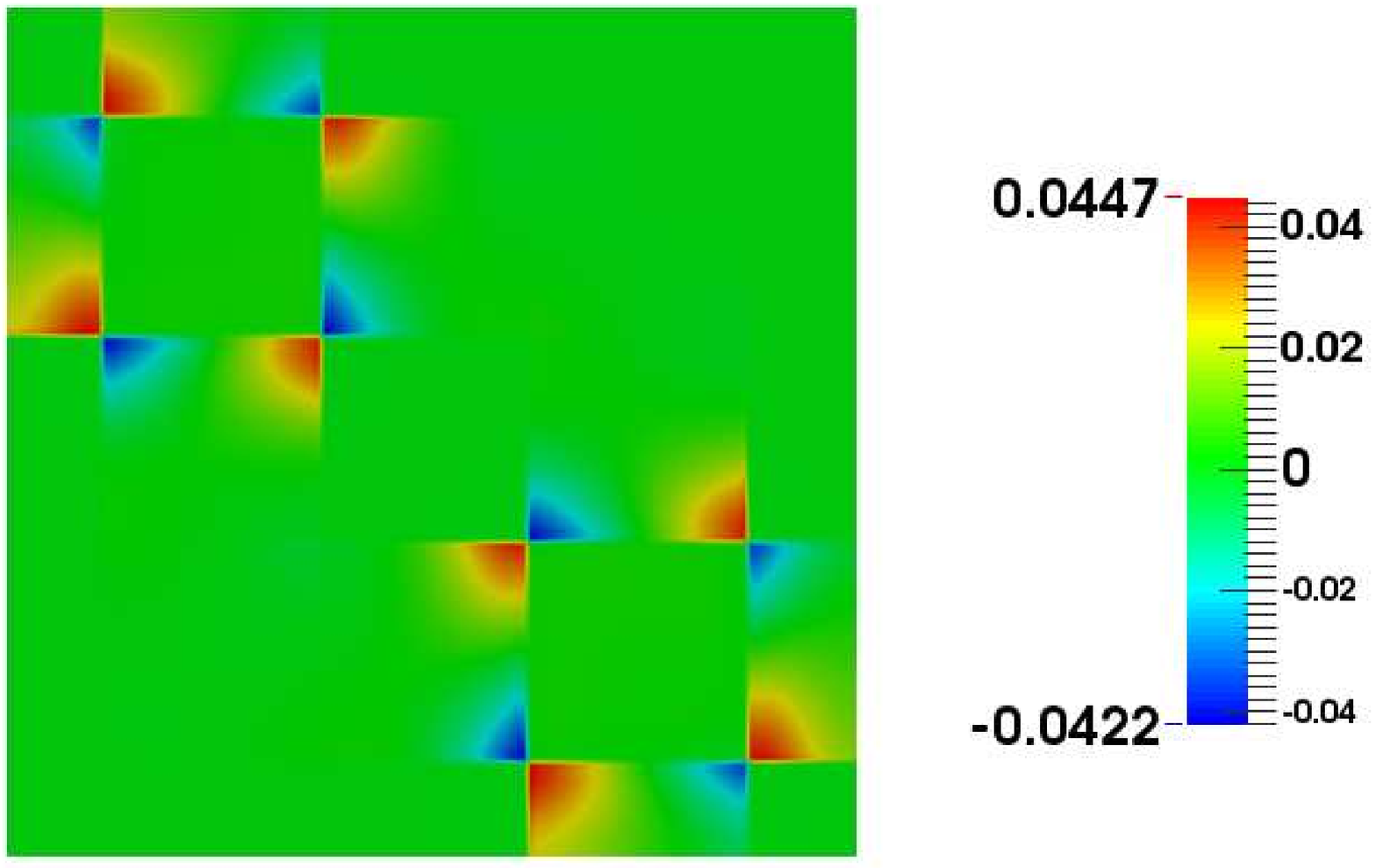}}
  \hspace{0.15in}
  \subfigure[Product $C$ at $t = 0.01$]
    {\includegraphics[scale=0.30,clip]
    {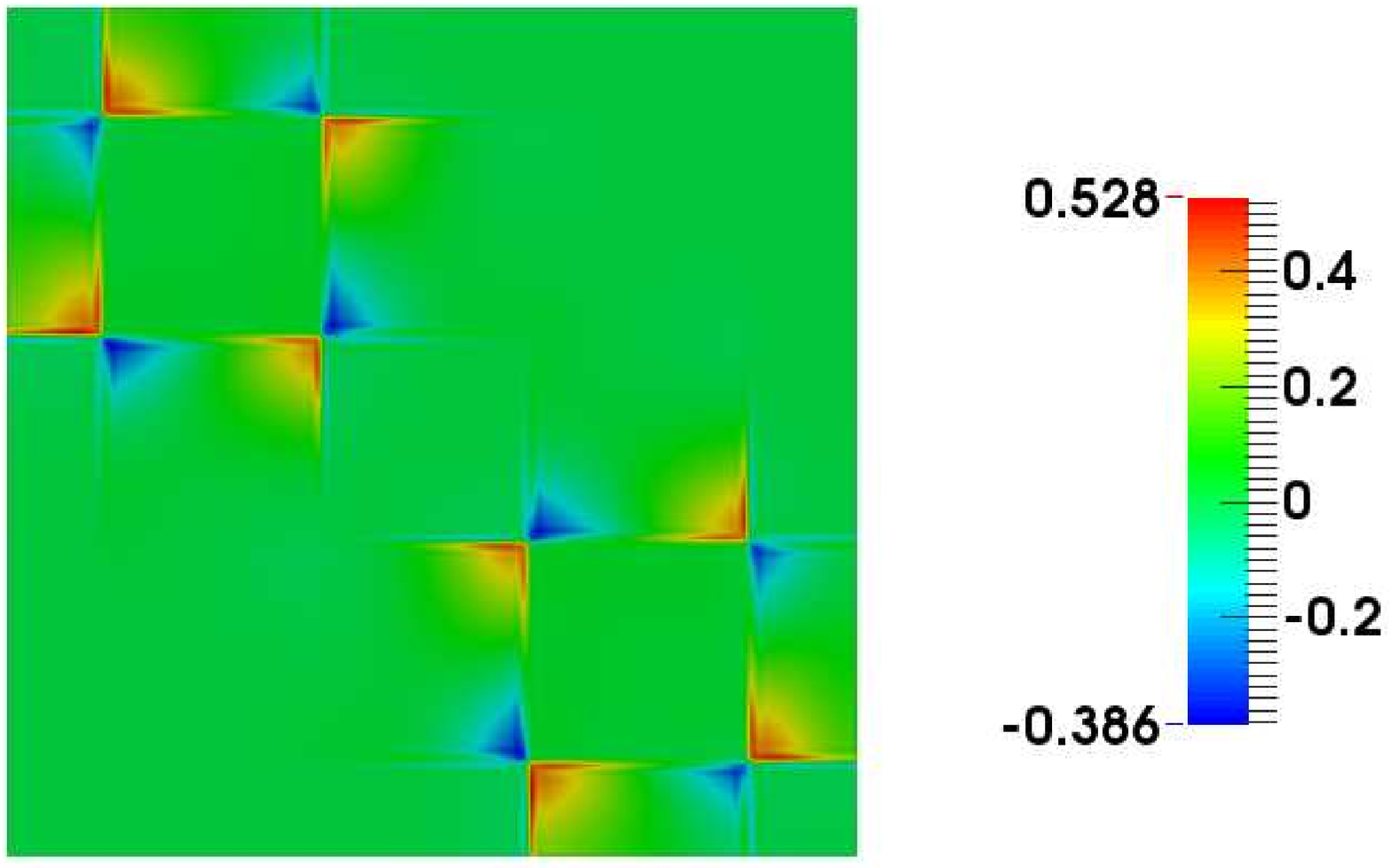}}
  \caption{\textsf{Non-chaotic vortex-stirred mixing in 
    a reaction tank:}~This figure shows the concentration 
    profiles of the product $C$ after the first time-step 
    using the unconstrained weighted negatively stabilized 
    streamline diffusion LSFEM. We have taken XSeed = 
    YSeed = 121. If constraints are not enforced, one 
    gets unphysical negative values for the concentration 
    of product $C$. This will be particularly true in 
    the early times of a numerical simulation.
    \label{Fig:2D_Bimolecular_ScalDiff_FirstTimeStep_ConcC}}
\end{figure}

\begin{figure}
  \centering
  \subfigure[$c_C(y=0.5,t=0.0001)$ with $\Delta t = 0.0001$]
    {\includegraphics[scale=0.12,clip]
    {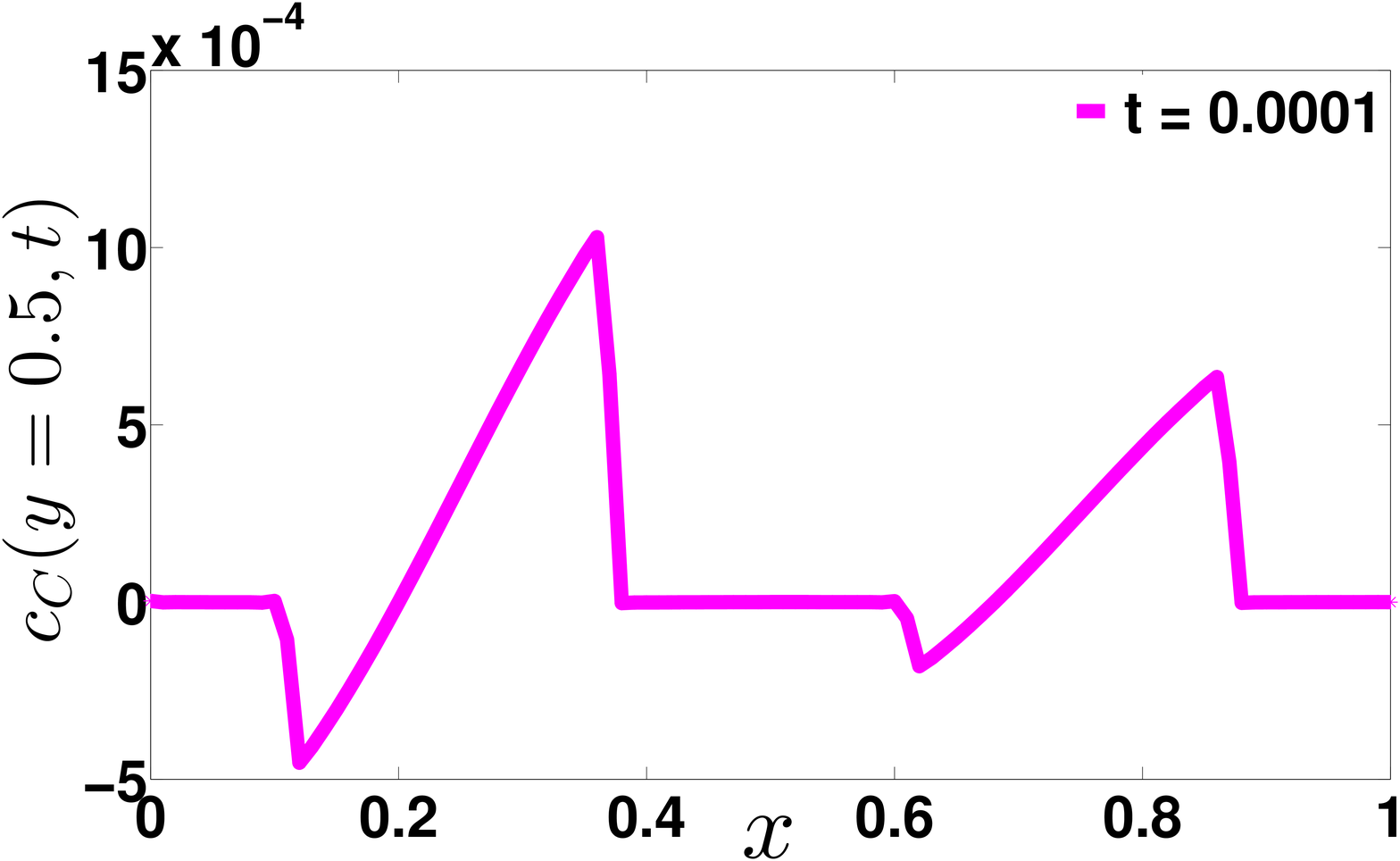}}
  \subfigure[$c_C(y=0.5,t=0.001)$ with $\Delta t = 0.001$]
    {\includegraphics[scale=0.12,clip]
    {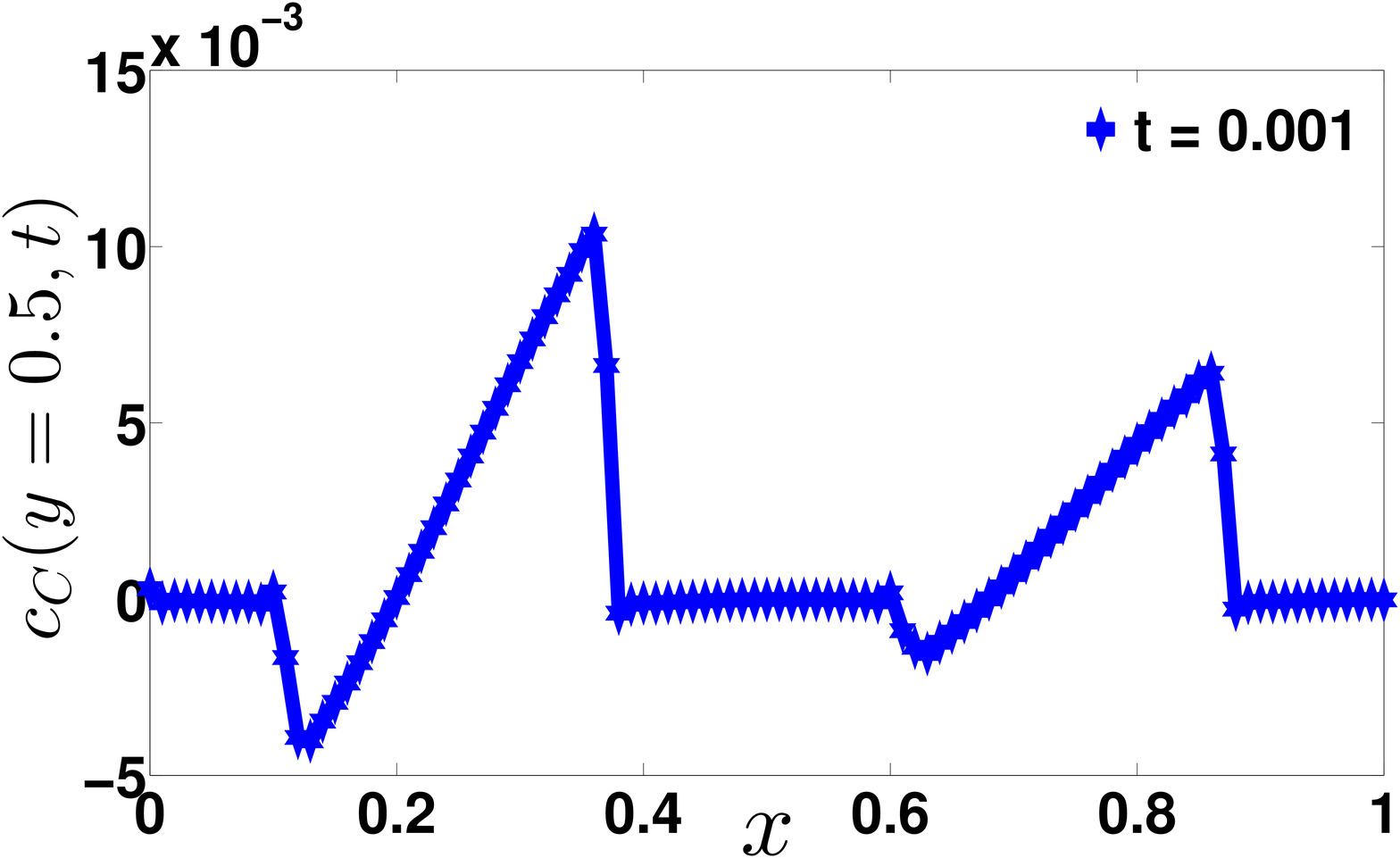}}
  \subfigure[$c_C(y=0.5,t=0.01)$ with $\Delta t = 0.01$]
    {\includegraphics[scale=0.12,clip]
    {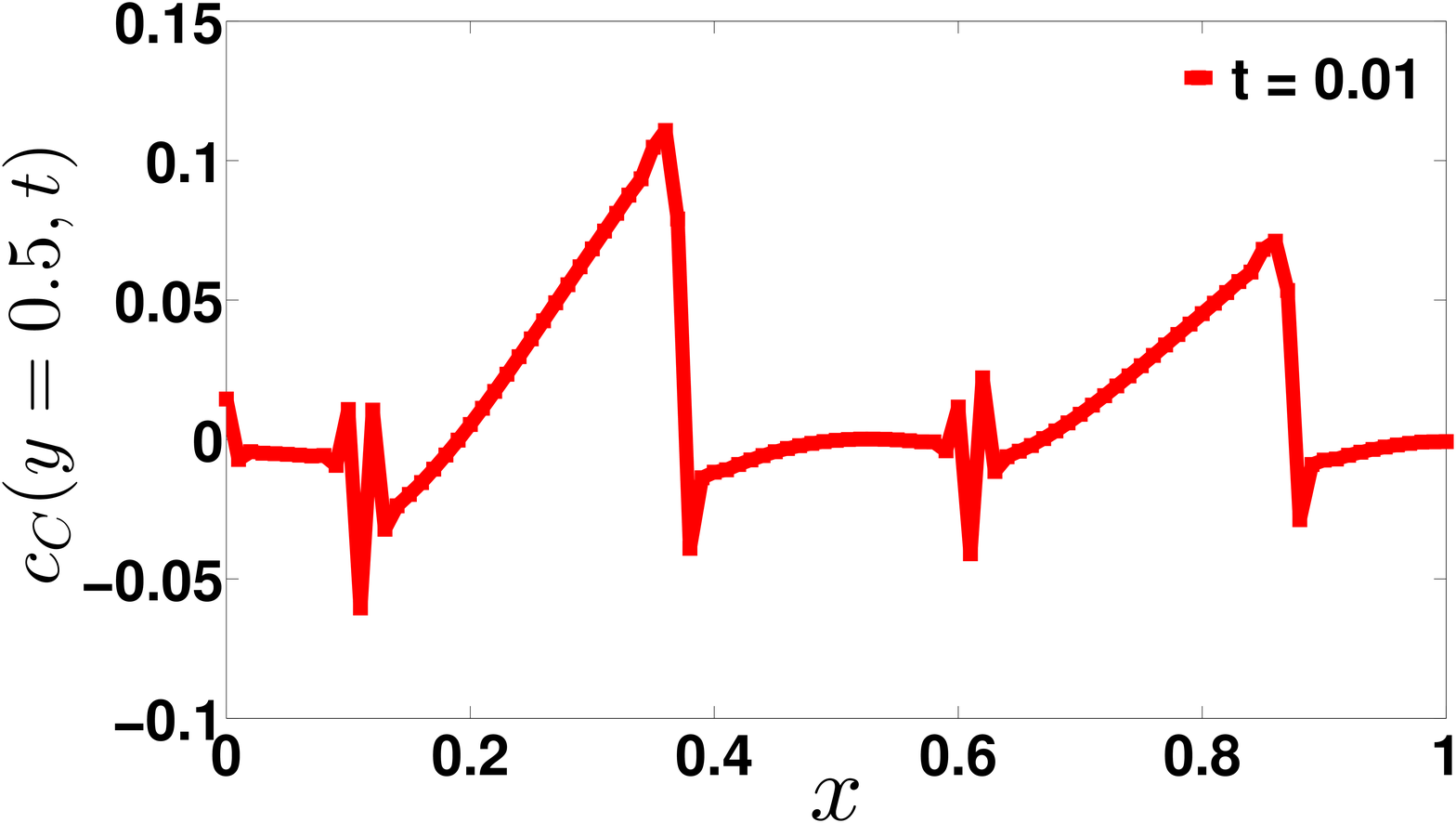}}
  \subfigure[$c_C(y=0.5,t=0.1)$ with $\Delta t = 0.1$]
    {\includegraphics[scale=0.12,clip]
    {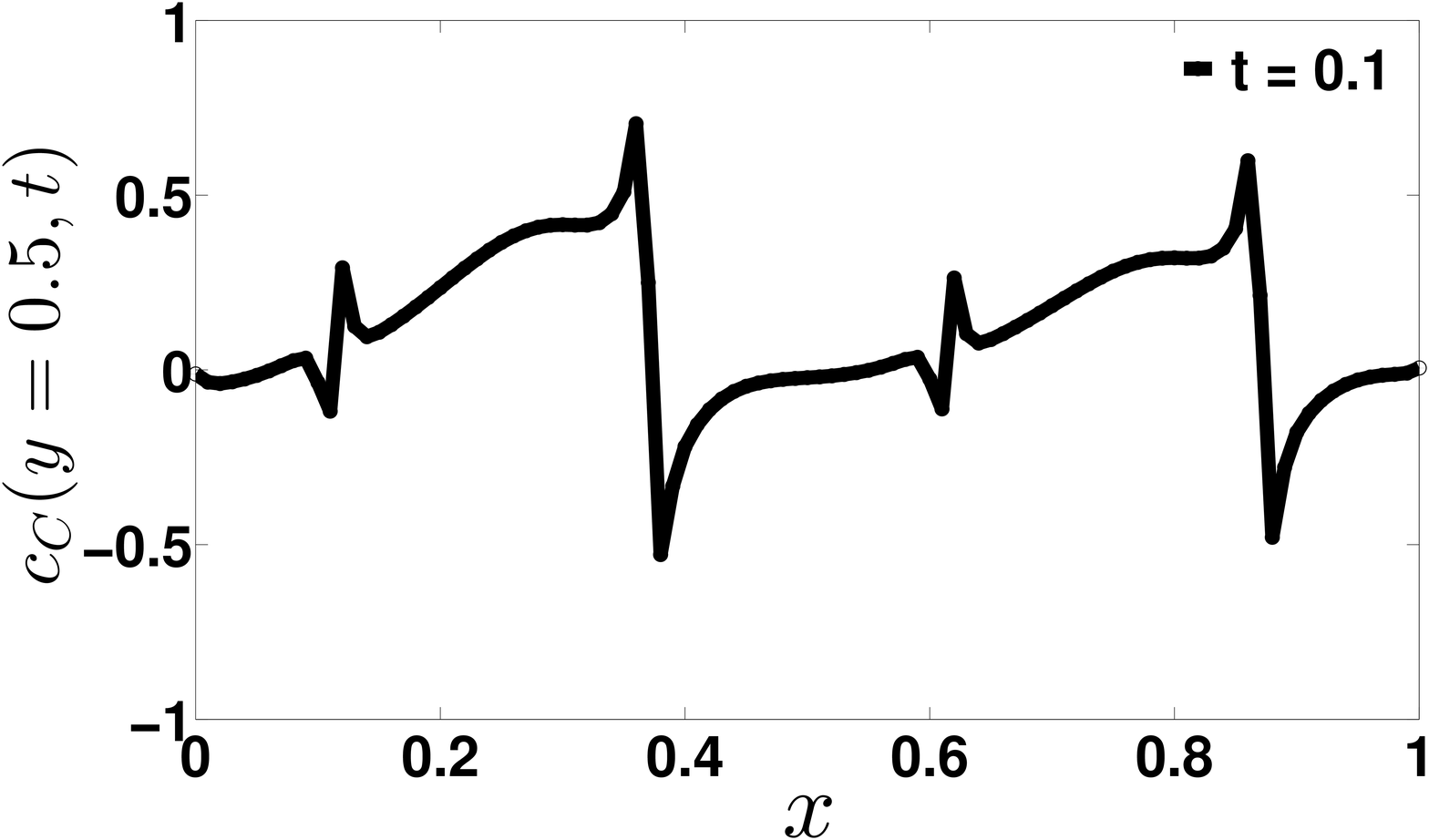}}
  \caption{\textsf{Non-chaotic vortex-stirred mixing in 
    a reaction tank:}~This figure shows the concentration 
    profiles of the product $C$ at $y = 0.5$ after the 
    first time-step using the unconstrained weighted 
    negatively stabilized streamline diffusion LSFEM. 
    The violations of the non-negative constraint are 
    significant, and are present for various choices 
    of the time-step.
    \label{Fig:2D_Bimolecular_ScalDiff_y0dot5_ConcC}}
\end{figure}

\begin{figure}
  \centering
  \subfigure[Product $C$ at $t = 0.1$ (No constraints)]{\includegraphics[scale=0.27,clip]
    {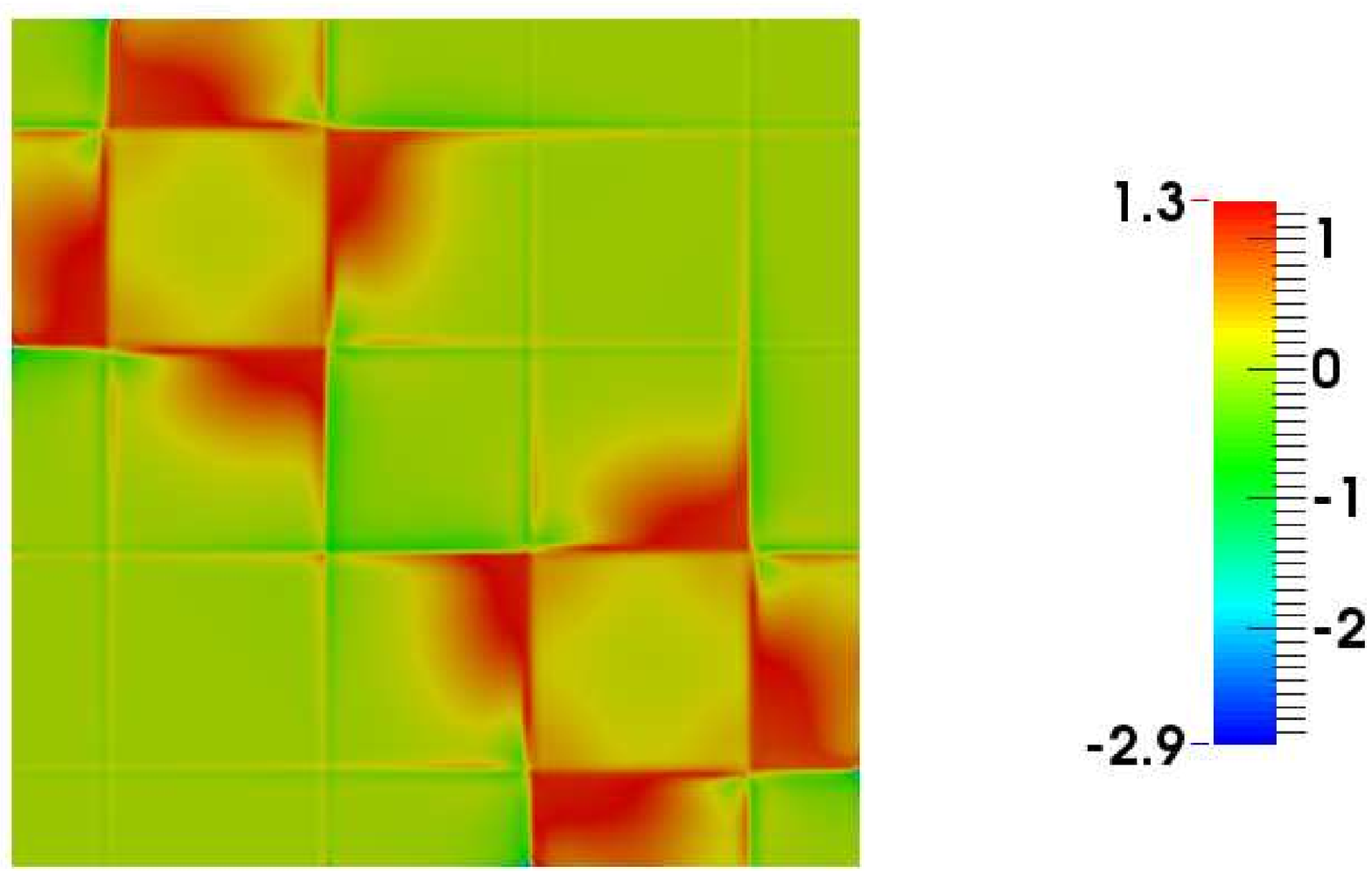}}
  \hspace{0.15in}
  \vspace{0.05in}
  \subfigure[Product $C$ at $t = 0.1$ (NN constraints)]{\includegraphics[scale=0.27,clip]
    {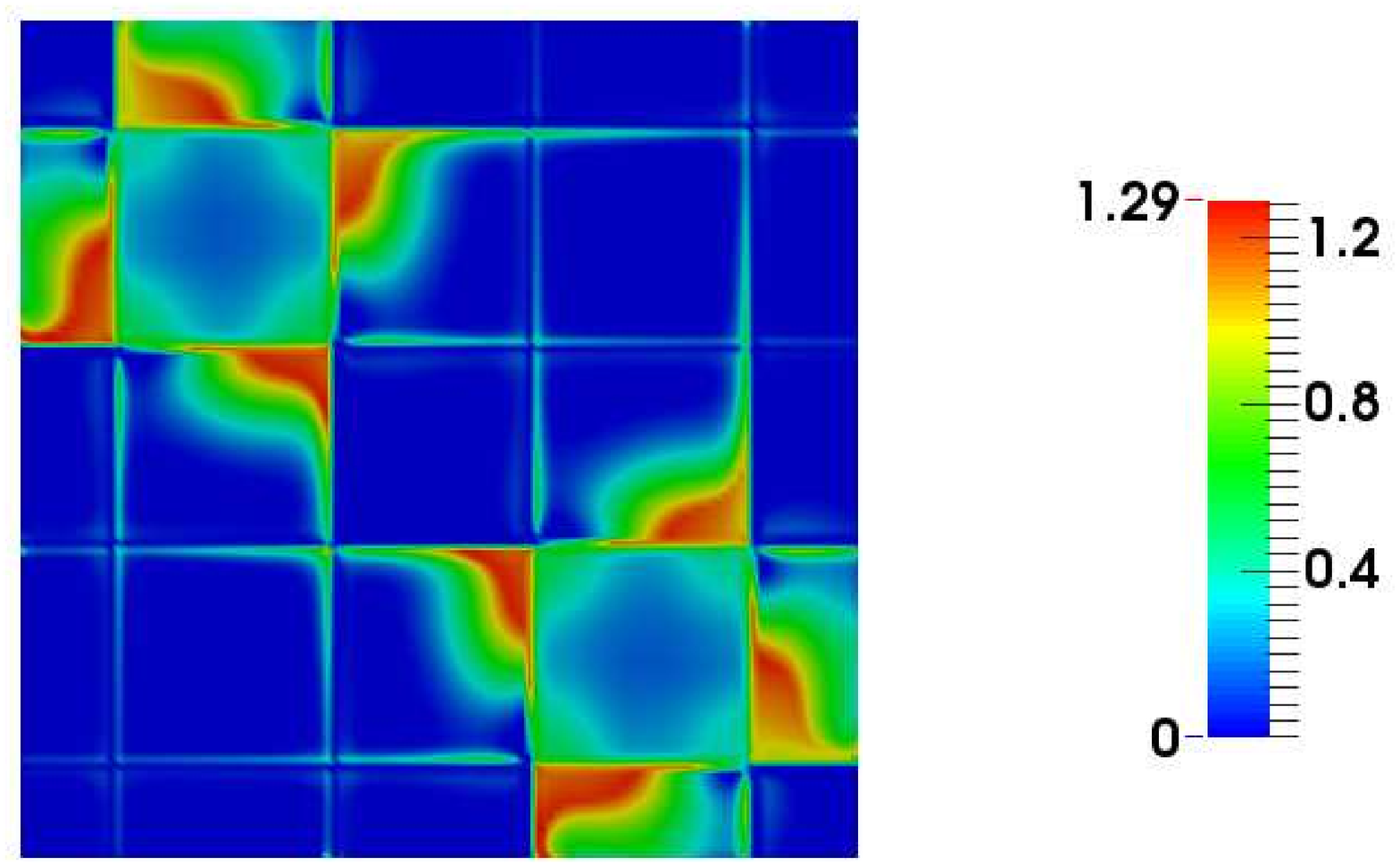}}
  \vspace{0.05in}
  \subfigure[Product $C$ at $t = 0.5$ (No constraints)]{\includegraphics[scale=0.27,clip]
    {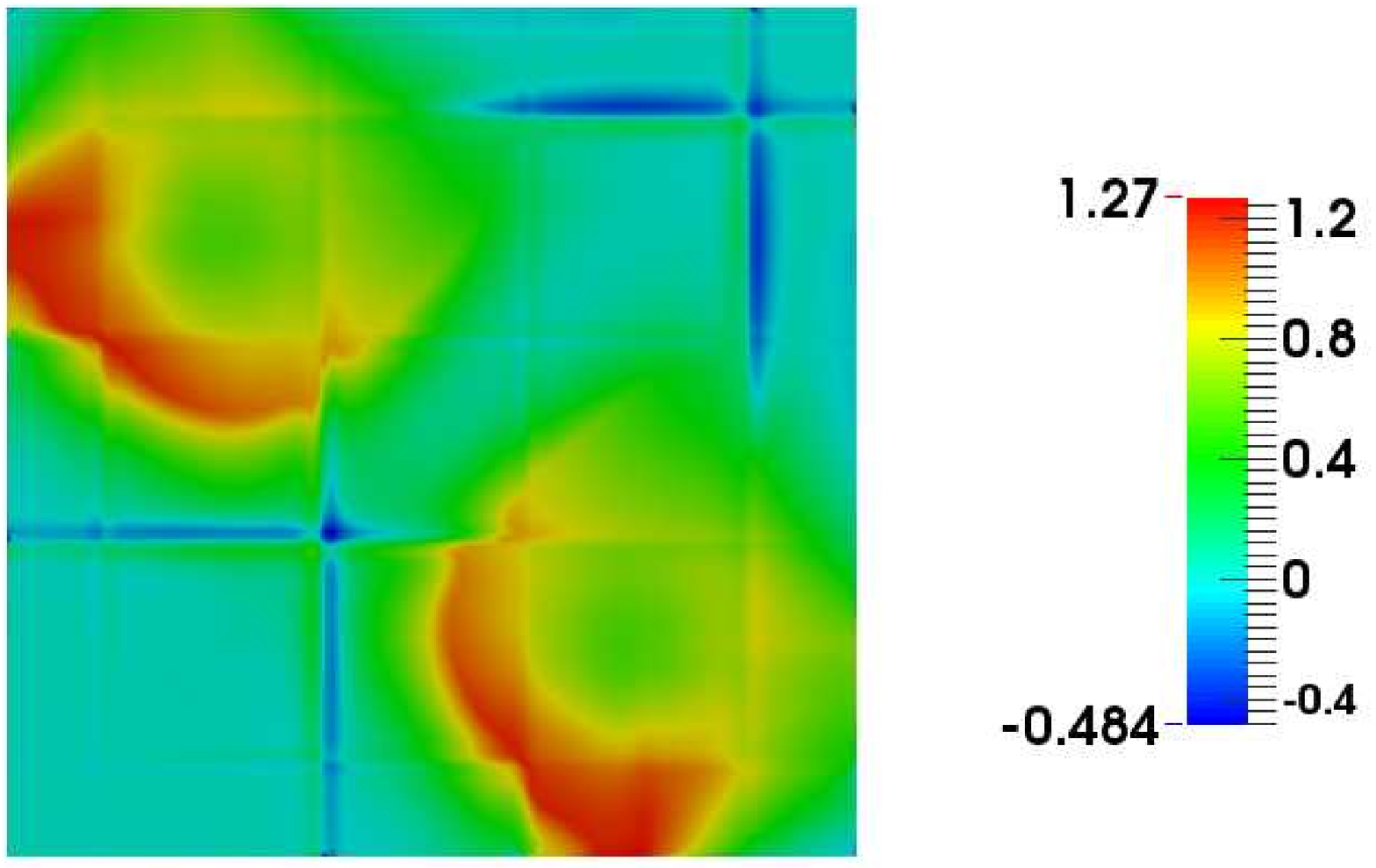}}
  \hspace{0.15in}
  \vspace{0.05in}
  \subfigure[Product $C$ at $t = 0.5$ (NN constraints)]{\includegraphics[scale=0.27,clip]
    {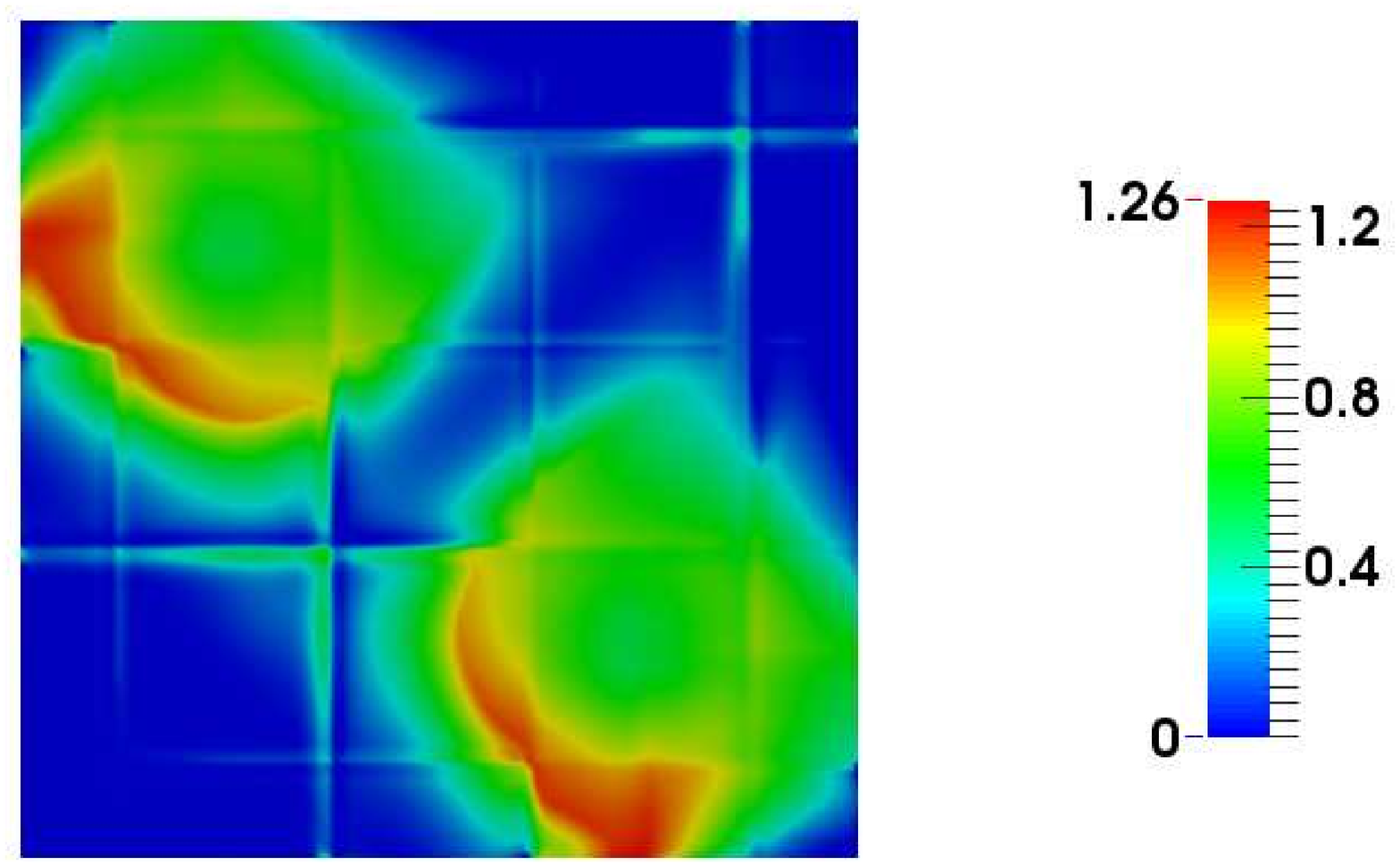}}
  \vspace{0.05in}
  \subfigure[Product $C$ at $t = 1.0$ (No constraints)]{\includegraphics[scale=0.27,clip]
    {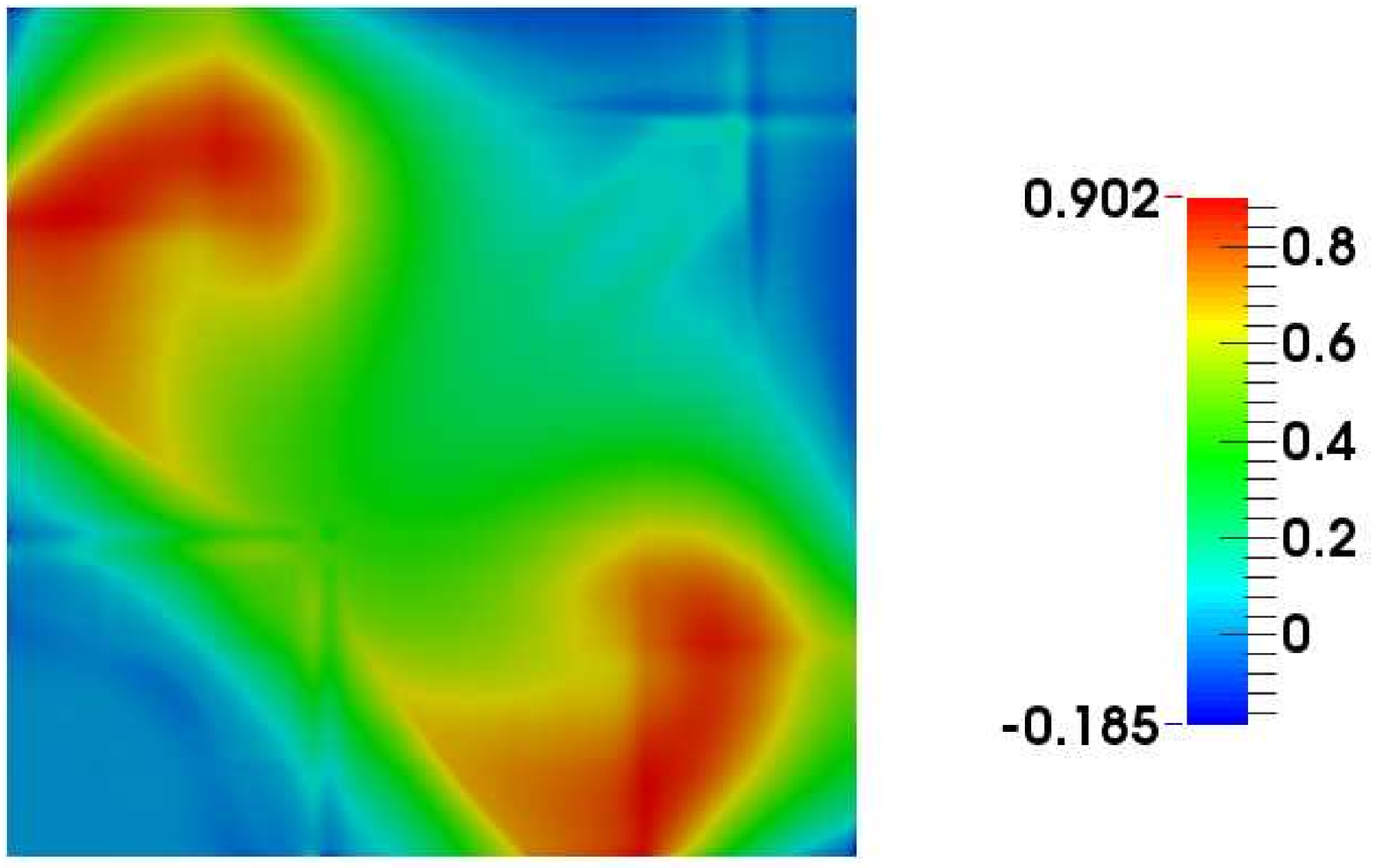}}
  \hspace{0.15in}
  \vspace{0.05in}
  \subfigure[Product $C$ at $t = 1.0$ (NN constraints)]{\includegraphics[scale=0.27,clip]
    {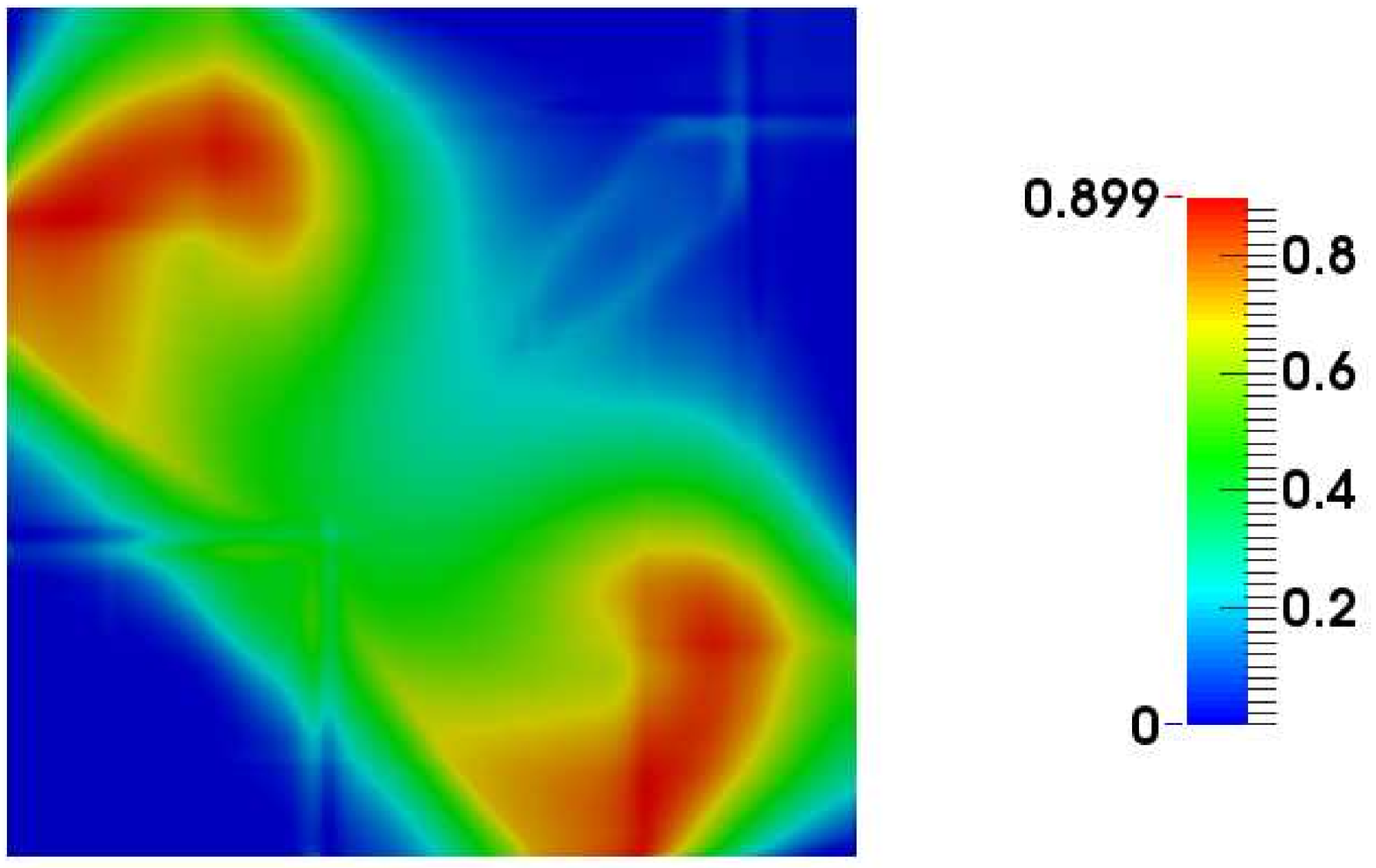}}
  \vspace{0.05in}
  \subfigure[Product $C$ at $t = 5.0$ (No constraints)]{\includegraphics[scale=0.27,clip]
    {ConcC_121by121_Time5dot0_NoCons_Prob2_Trans_Mixing.eps}}
  \hspace{0.15in}
  \vspace{0.05in}
  \subfigure[Product $C$ at $t = 5.0$ (NN constraints)]{\includegraphics[scale=0.27,clip]
    {ConcC_121by121_Time5dot0_NonNeg_Prob2_Trans_Mixing.eps}}
  \caption{\textsf{Non-chaotic vortex-stirred mixing in 
    a reaction tank:}~This figure shows the concentration 
    profiles of the product $C$ at various time levels 
    using the weighted negatively stabilized streamline 
    diffusion LSFEM with and without constraints. The 
    time-step is taken as $\Delta t = 0.1$. Herein, 
    XSeed = YSeed = 121. As $t$ increases, the product 
    $C$ should accumulate near the center of the two 
    vortices. The proposed computational framework is 
    able to accurately capture such features, and the 
    obtained solutions are physical at all times.
    \label{Fig:2D_Bimolecular_ScalDiff_Mixing_ConcC}}
\end{figure}

\begin{figure}
  \centering
  \subfigure[Product $C$ at $t = 0.1$]{\includegraphics[scale=0.27,clip]
    {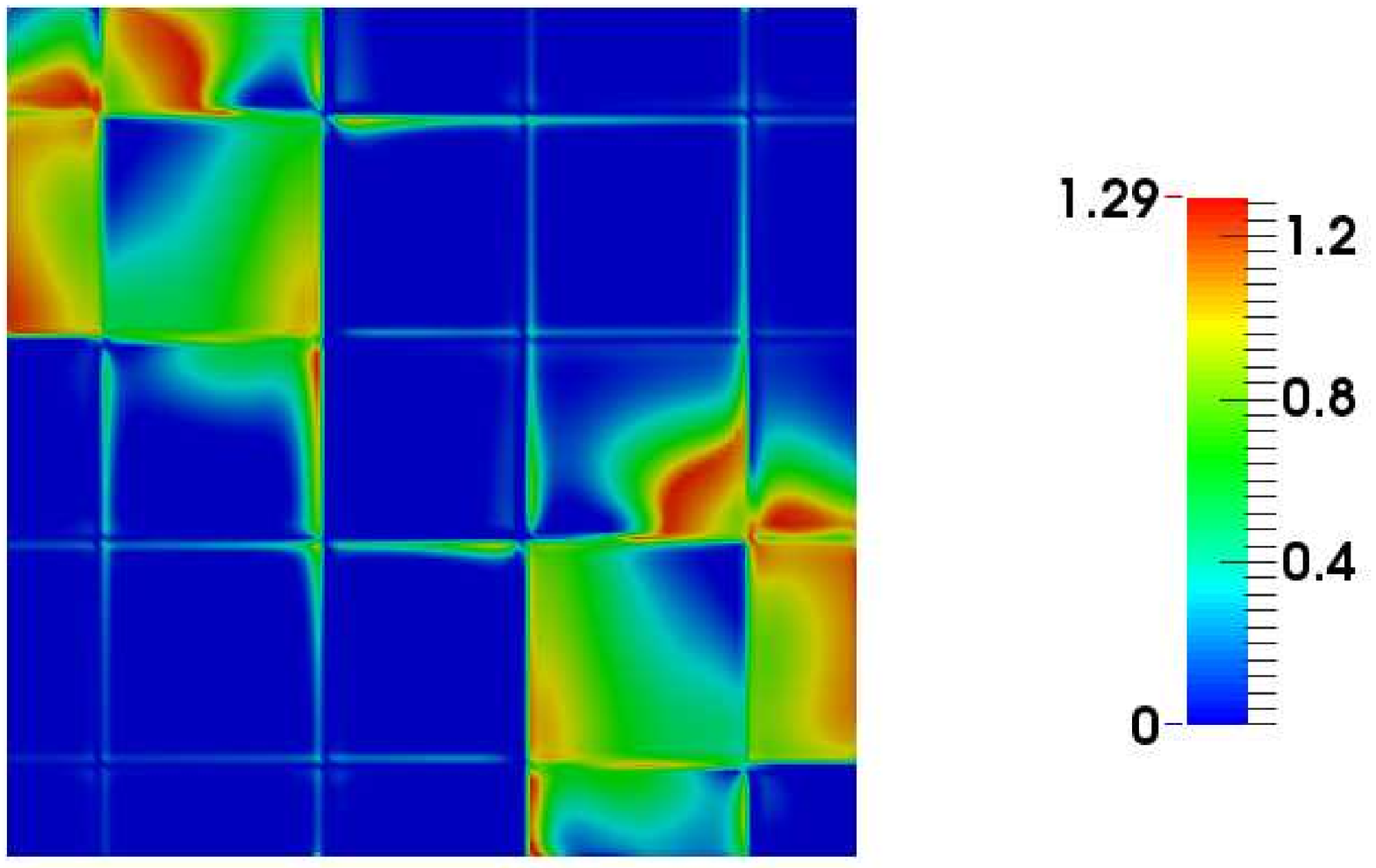}}
  \hspace{0.15in}
  \vspace{0.05in}
  \subfigure[Product $C$ at $t = 0.5$]{\includegraphics[scale=0.27,clip]
    {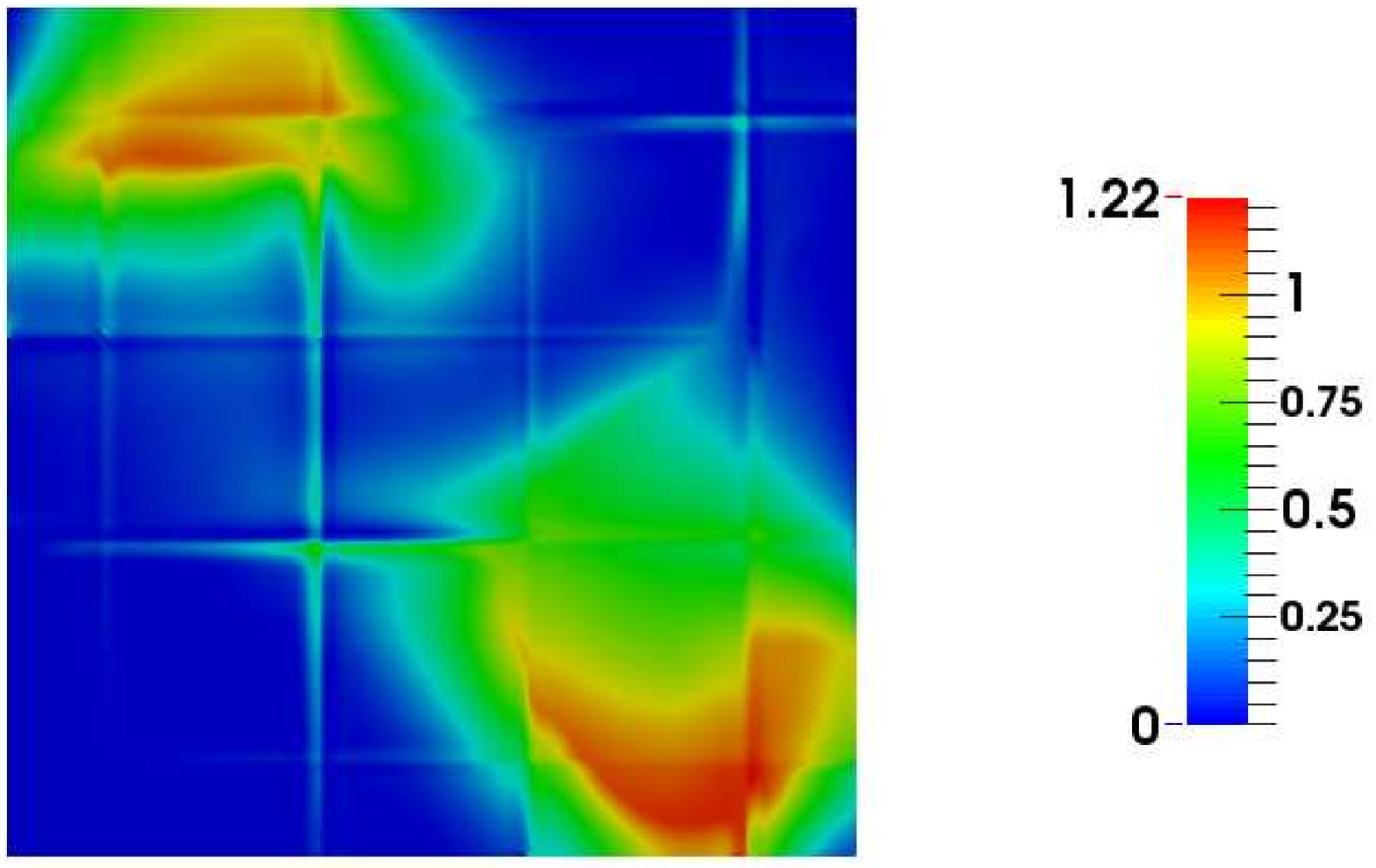}}
  \vspace{0.05in}
  \subfigure[Product $C$ at $t = 1.0$]{\includegraphics[scale=0.27,clip]
    {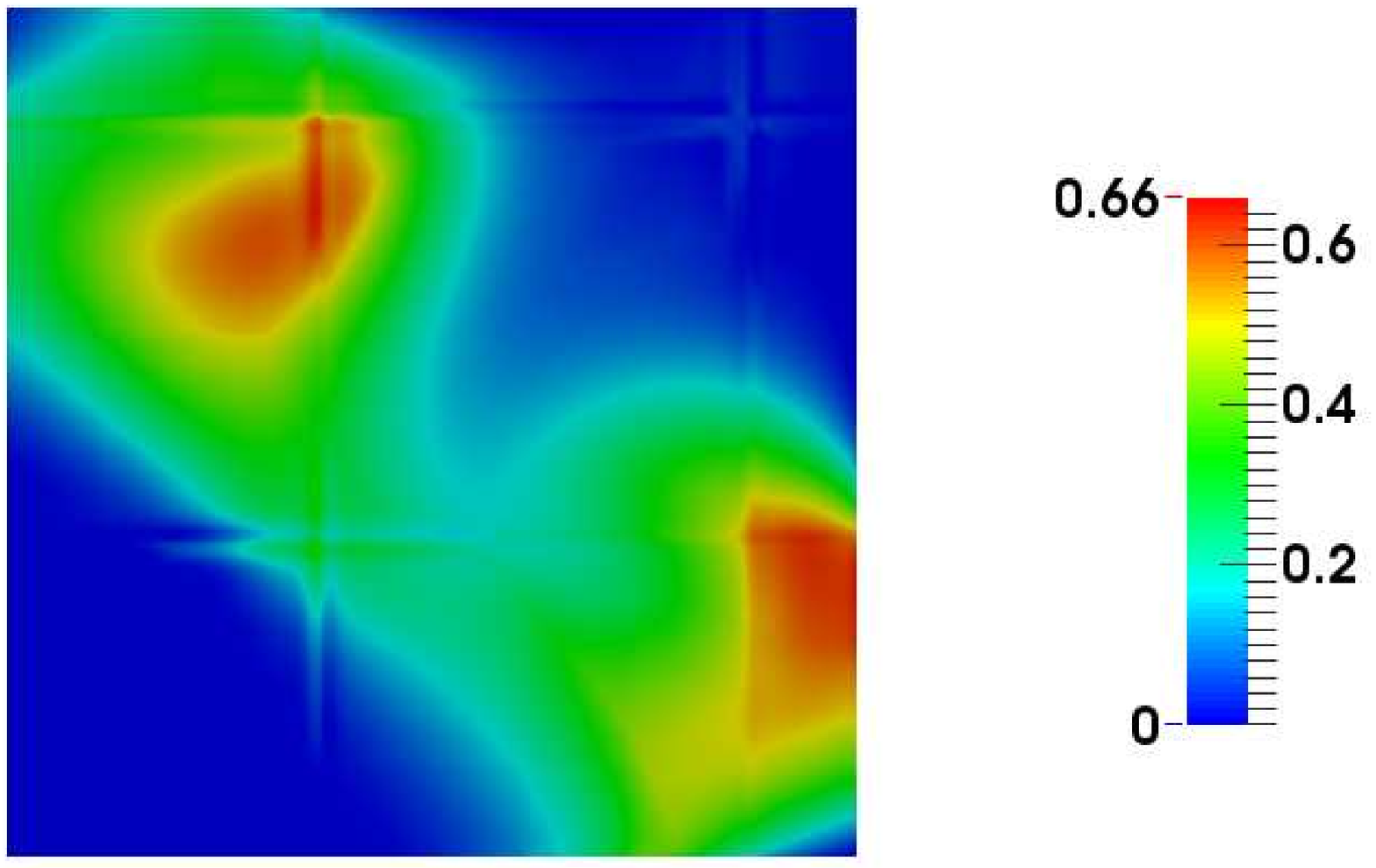}}
  \hspace{0.15in}
  \vspace{0.05in}
  \subfigure[Product $C$ at $t = 1.5$]{\includegraphics[scale=0.27,clip]
    {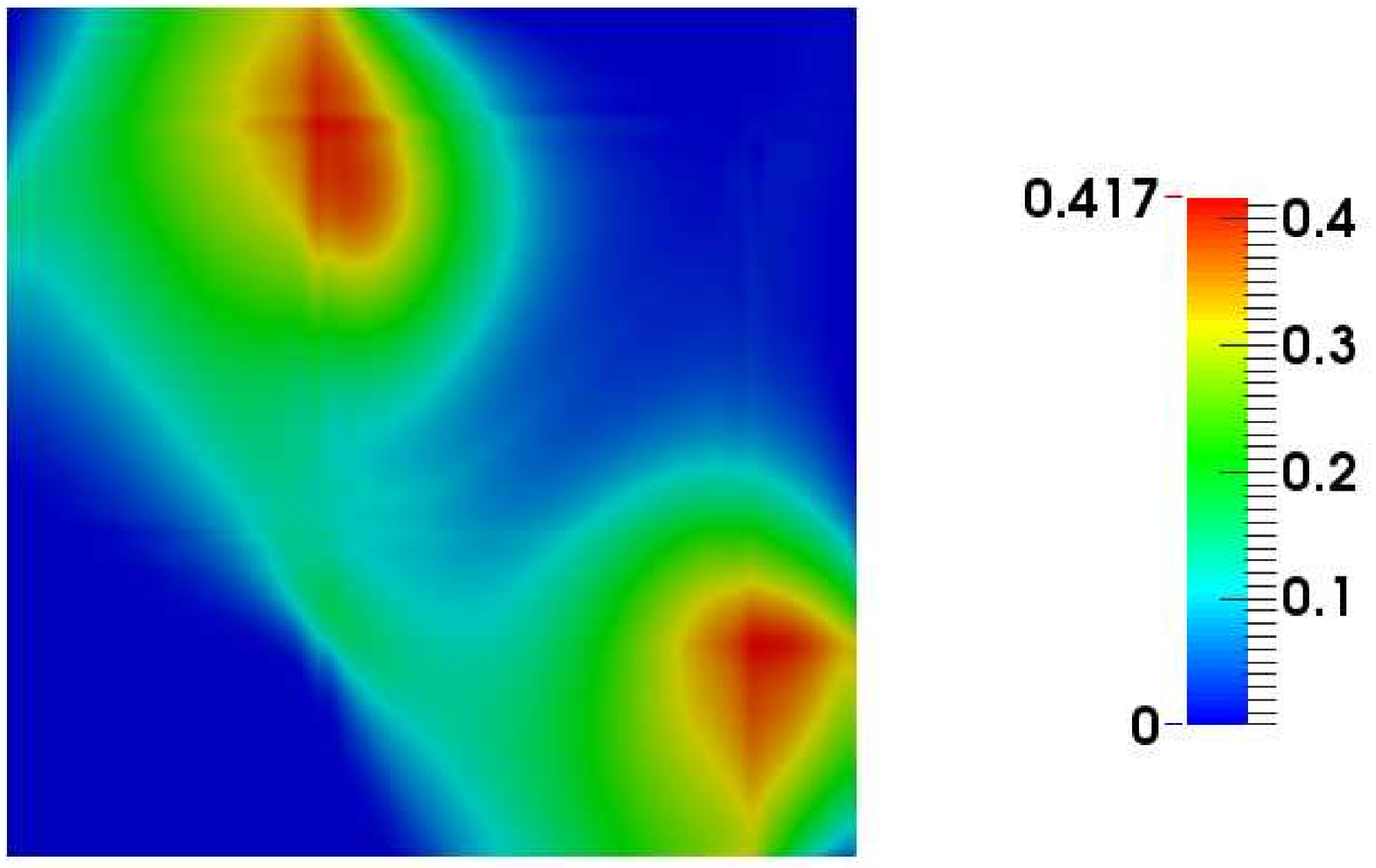}}
  \vspace{0.05in}
  \subfigure[Product $C$ at $t = 2.0$]{\includegraphics[scale=0.27,clip]
    {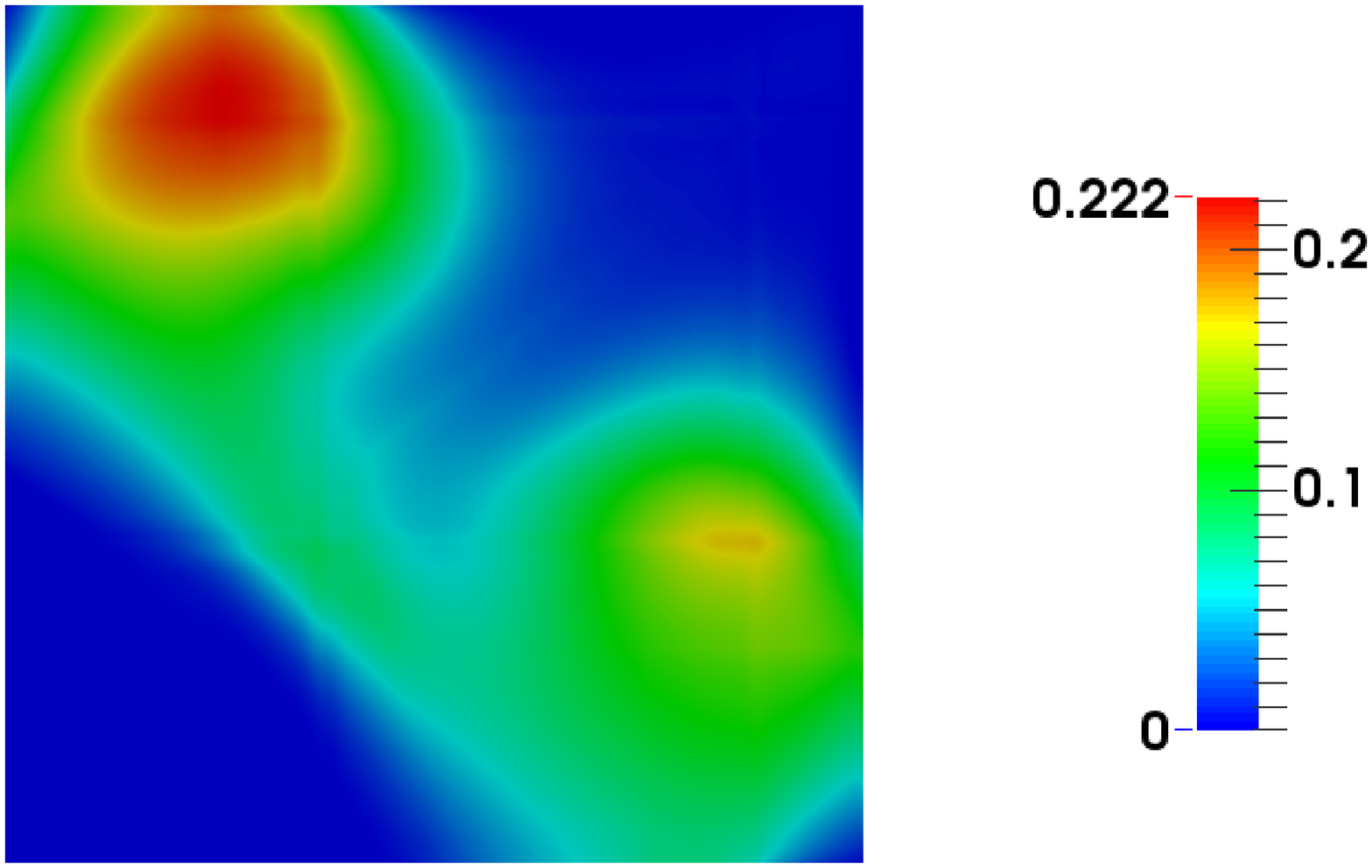}}
  \hspace{0.15in}
  \vspace{0.05in}
  \subfigure[Product $C$ at $t = 3.0$]{\includegraphics[scale=0.27,clip]
    {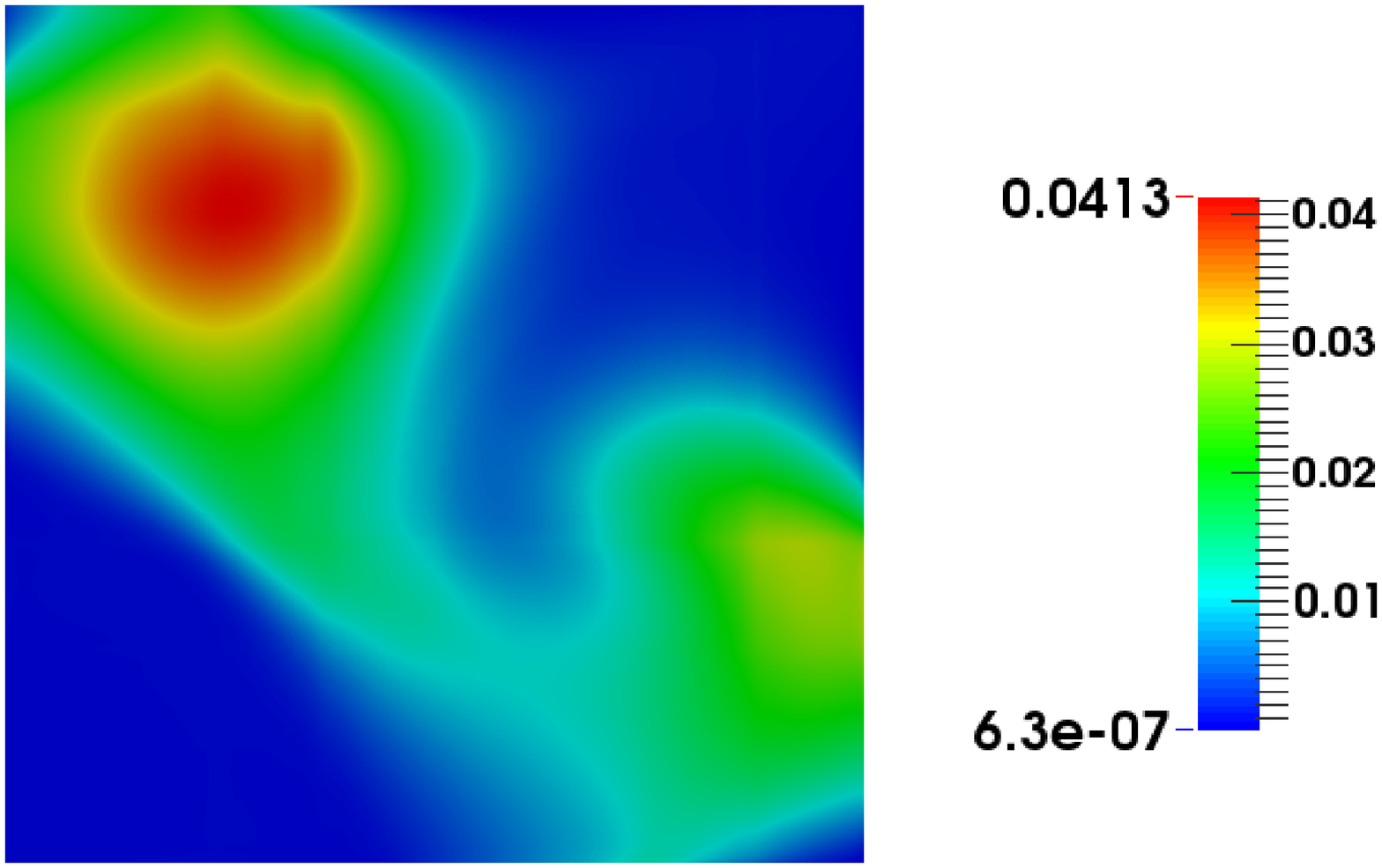}}
  \vspace{0.05in}
  \subfigure[Product $C$ at $t = 4.0$]{\includegraphics[scale=0.27,clip]
    {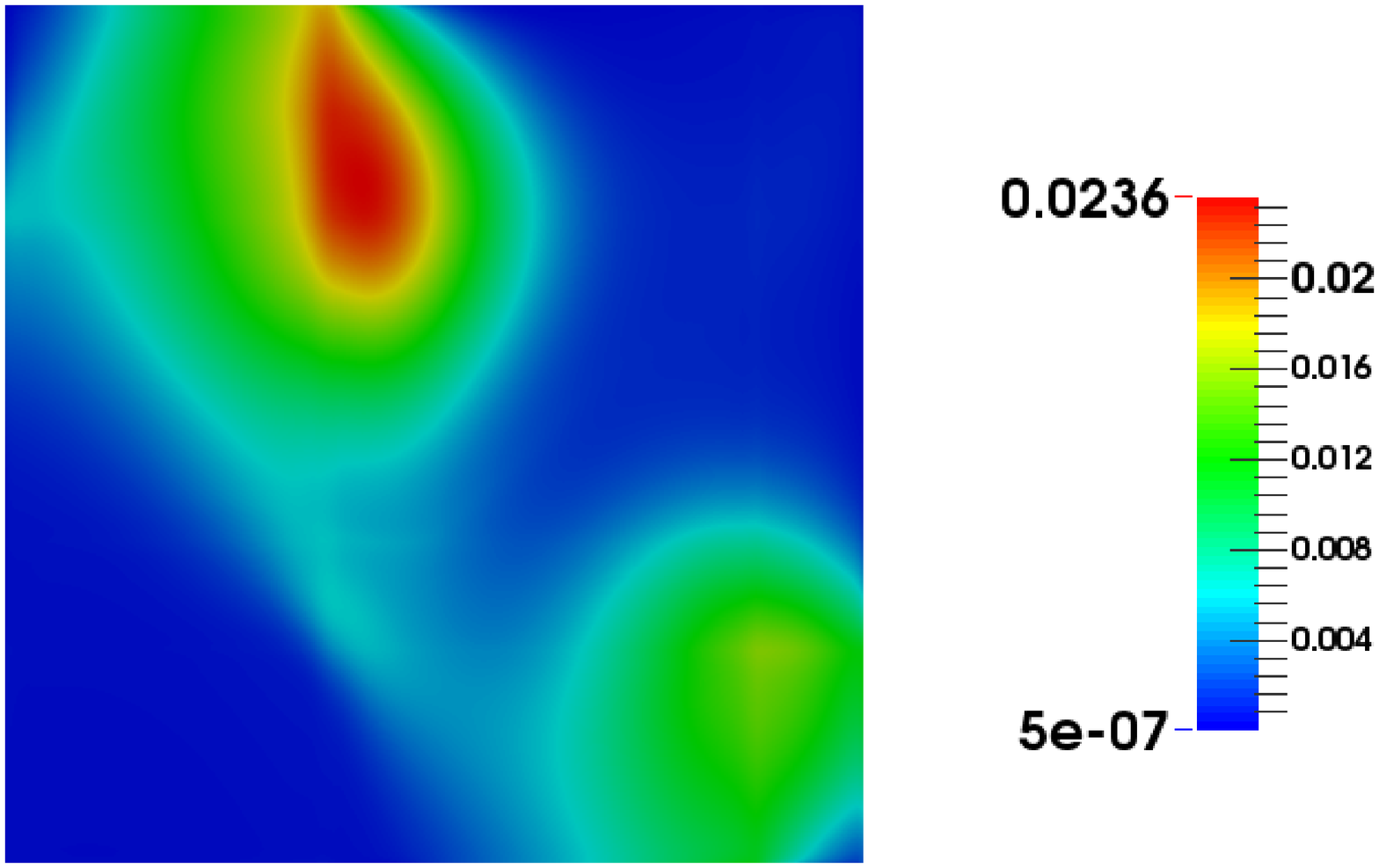}}
  \hspace{0.15in}
  \vspace{0.05in}
  \subfigure[Product $C$ at $t = 5.0$]{\includegraphics[scale=0.27,clip]
    {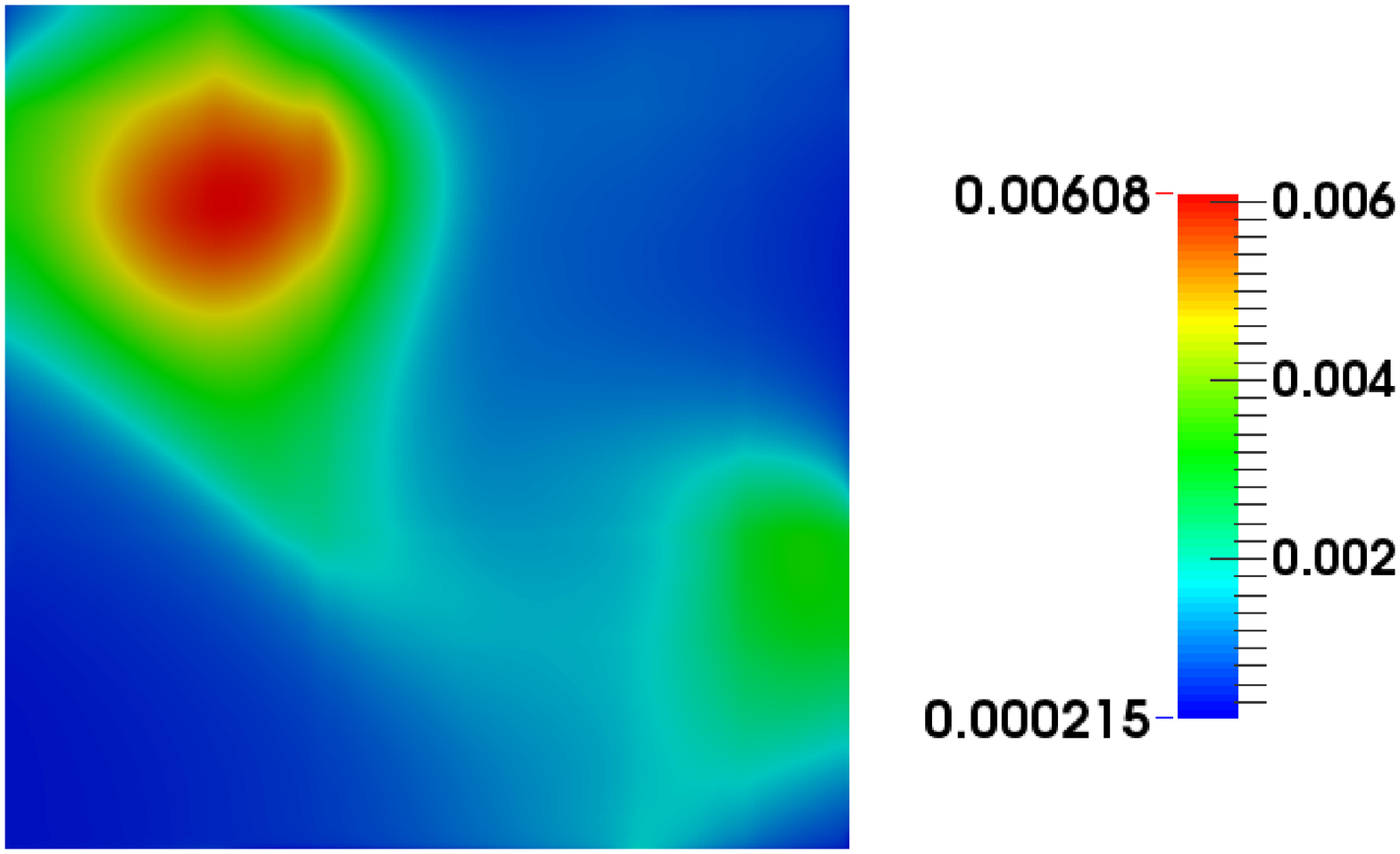}}
  \caption{\textsf{Chaotic vortex-stirred mixing in 
    a reaction tank:}~This figure shows the concentration
    profiles of the product $C$ at various time levels
    using the proposed method with non-negative constraints.
    The time-step is taken as $\Delta t = 0.1$. Herein, 
    XSeed = YSeed = 121. An interesting feature observed 
    is that the mixing of $c_C$ is enhanced in the entire 
    domain when compared to non-chaotic advection (as the 
    minimum value for $c_C$ is greater than zero).
    \label{Fig:2D_Bimolecular_ScalDiff_Mixing_ConcC1}}
\end{figure}


\begin{figure}
  \centering
  \psfrag{O}{$(0,0)$}
  \psfrag{A}{$A$}
  \psfrag{B}{$B$}
  \psfrag{A1}{$c^{0}_A(\mathbf{x}) = 1$}
  \psfrag{A2}{$c^{0}_B(\mathbf{x}) = 0$}
  \psfrag{A3}{$c^{0}_C(\mathbf{x}) = 0$}
  \psfrag{B1}{$c^{0}_A(\mathbf{x}) = 0$}
  \psfrag{B2}{$c^{0}_B(\mathbf{x}) = 1$}
  \psfrag{B3}{$c^{0}_C(\mathbf{x}) = 0$}
  \psfrag{BC1}{$h^{\mathrm{p}}_i(\mathbf{x},t) = 0$}
  \psfrag{BC2}{\rotatebox{-90}{$h^{\mathrm{p}}_i(\mathbf{x},t) = 0$}}
  \psfrag{BC3}{$h^{\mathrm{p}}_i(\mathbf{x},t) = 0$}
  \psfrag{BC4}{\rotatebox{90}{$h^{\mathrm{p}}_i(\mathbf{x},t) = 0$}}
  \psfrag{f1}{$f_i(\mathbf{x},t) = 0$}
  \psfrag{f2}{$f_i(\mathbf{x},t) = 0$}
  \psfrag{Lx}{$L_x$}
  \psfrag{Ly1}{\rotatebox{-90}{$L_y/2$}}
  \psfrag{Ly2}{\rotatebox{-90}{$L_y/2$}}
  \subfigure[Problem description]{\includegraphics[scale=0.8]
    {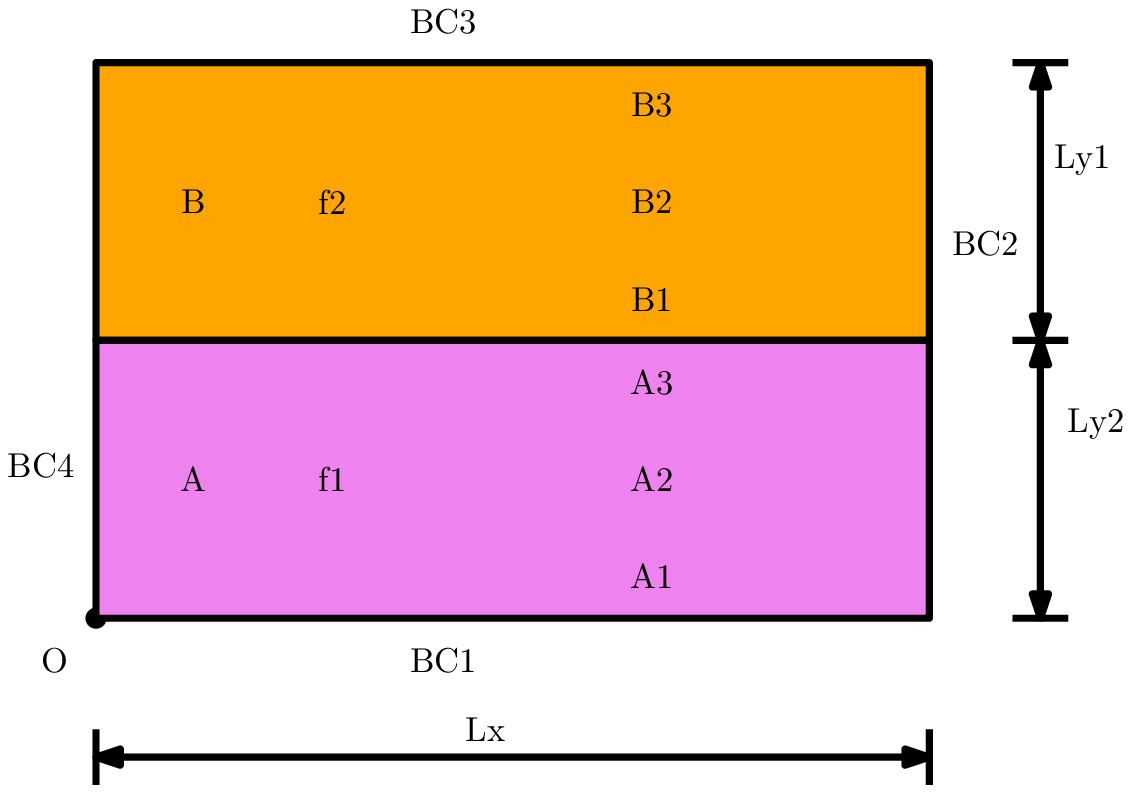}}
  \subfigure[Stream function and advection velocity]{\includegraphics[scale=0.42]
    {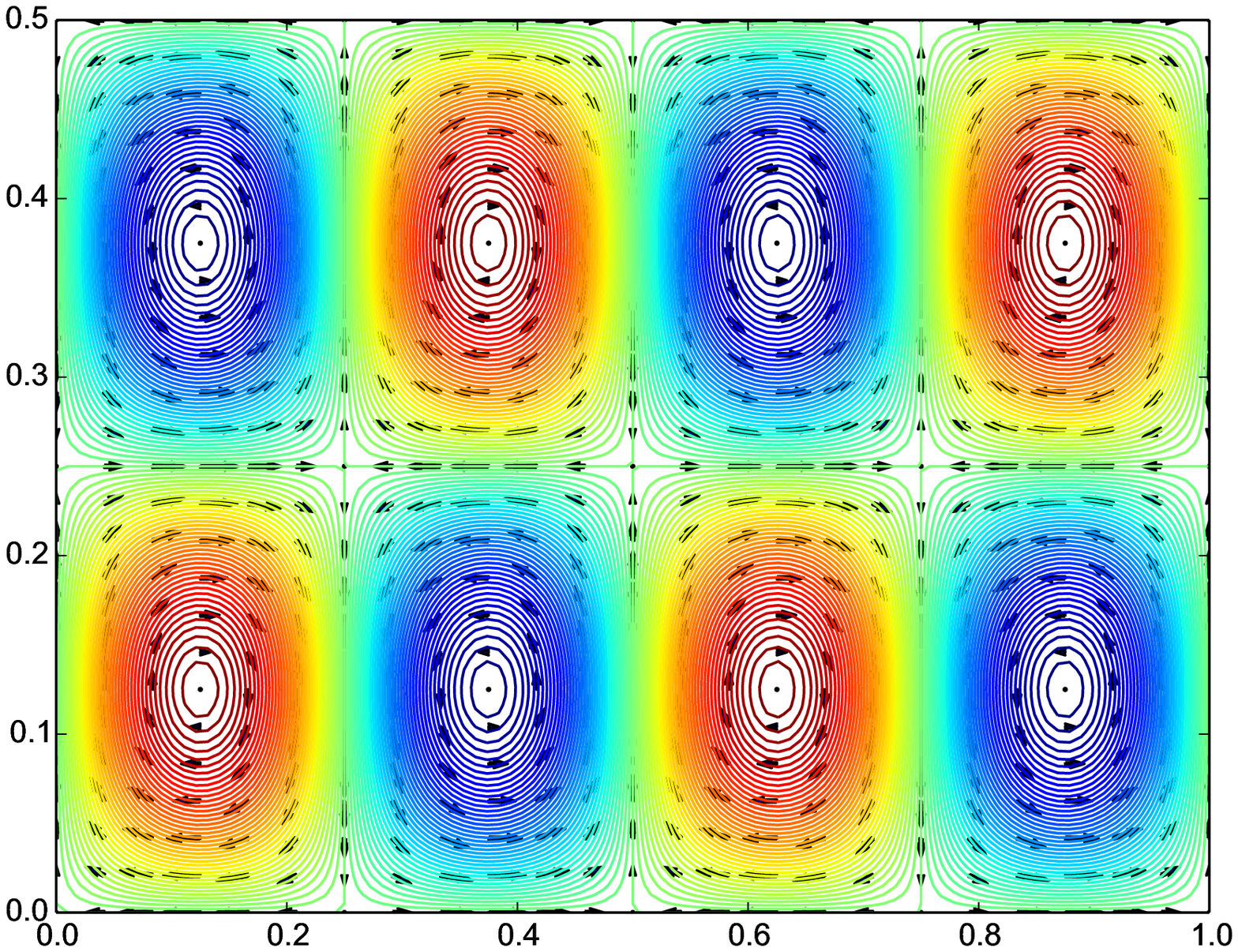}}
  \caption{\textsf{Transport-controlled mixing in cellular 
    flows:}~A pictorial description of the initial boundary 
    value problem and associated advection velocity field for 
    the cellular flow when $L_{\text{\tiny {Cell}}} = 0.5$.
  \label{Fig:2D_CellFlow_ScalDiff_Mixing}}
\end{figure}

\begin{figure}
  \centering
  \subfigure[Product $C$ at $t = 0.1$ (No constraints)]{\includegraphics[scale=0.30,clip]
    {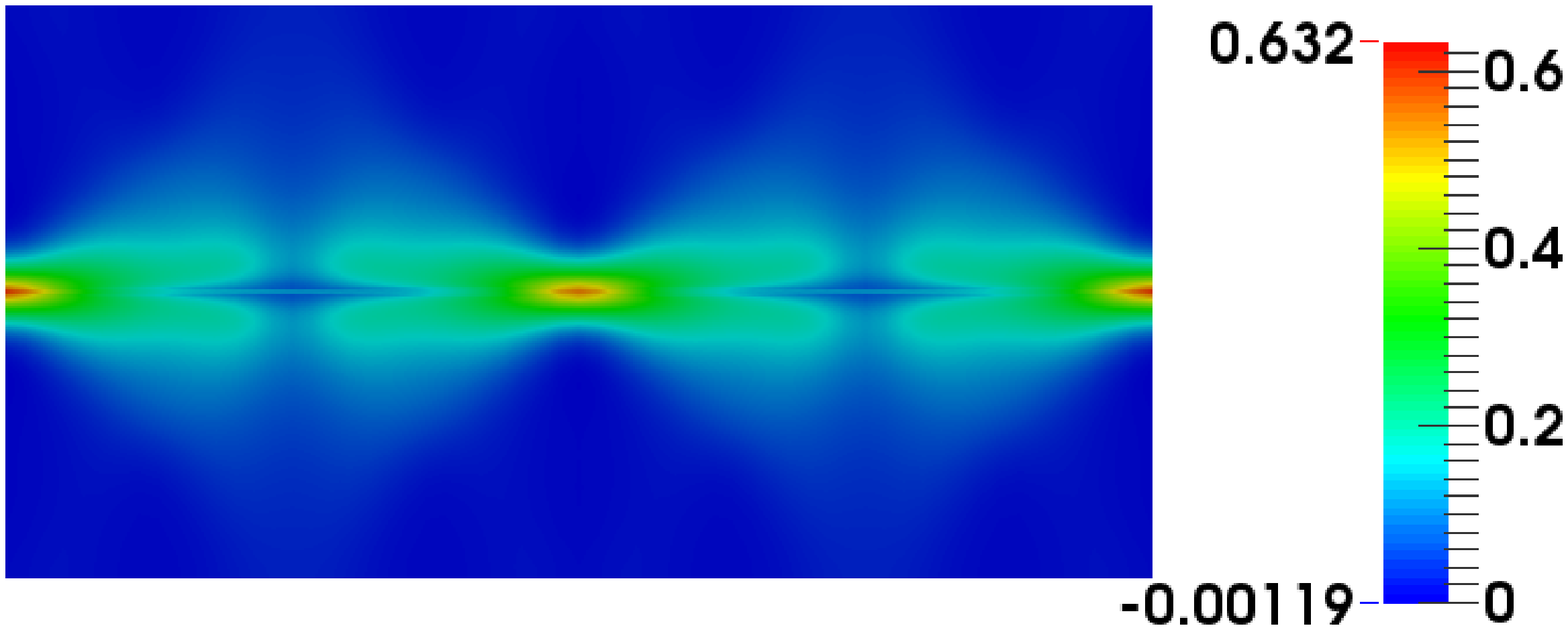}}
  \hspace{0.15in}
  \vspace{0.1in}
  \subfigure[Product $C$ at $t = 0.1$ (LSB and DMP constraints)]{\includegraphics[scale=0.30,clip]
    {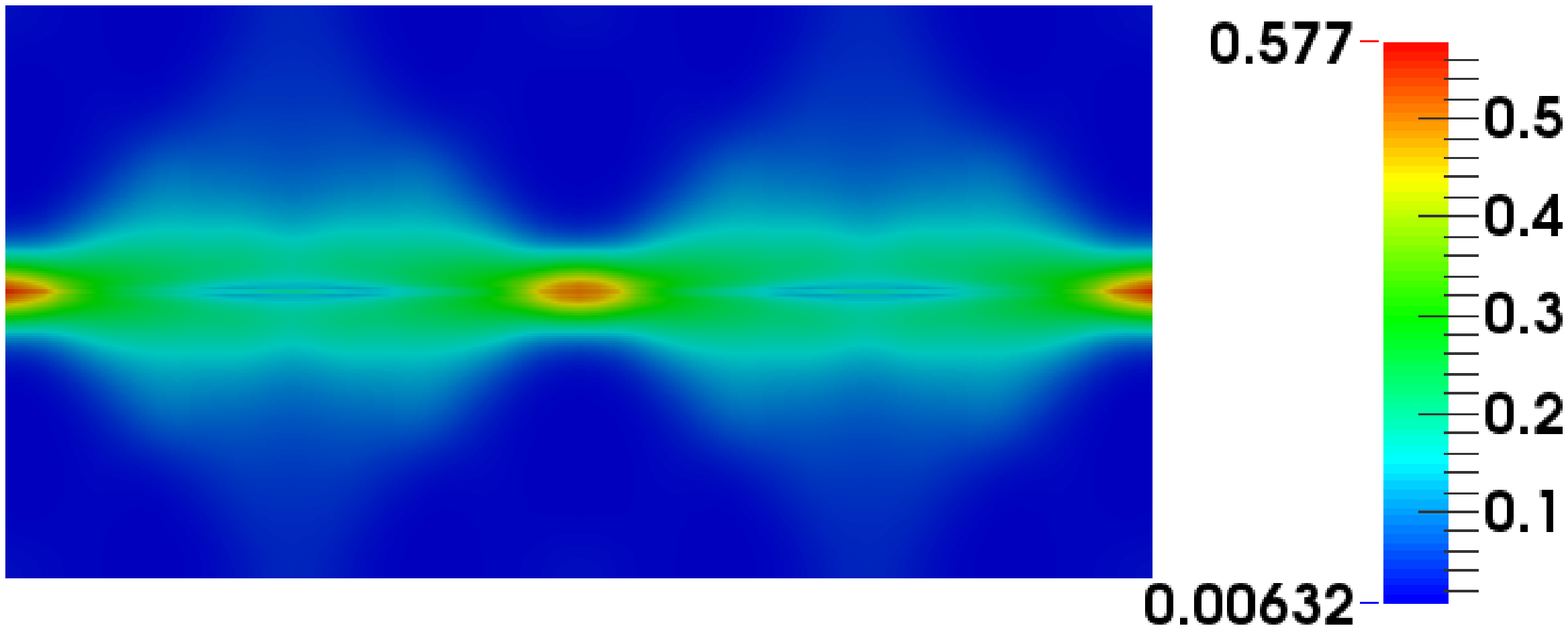}}
  \vspace{0.1in}
  \subfigure[Product $C$ at $t = 0.5$ (No constraints)]{\includegraphics[scale=0.30,clip]
    {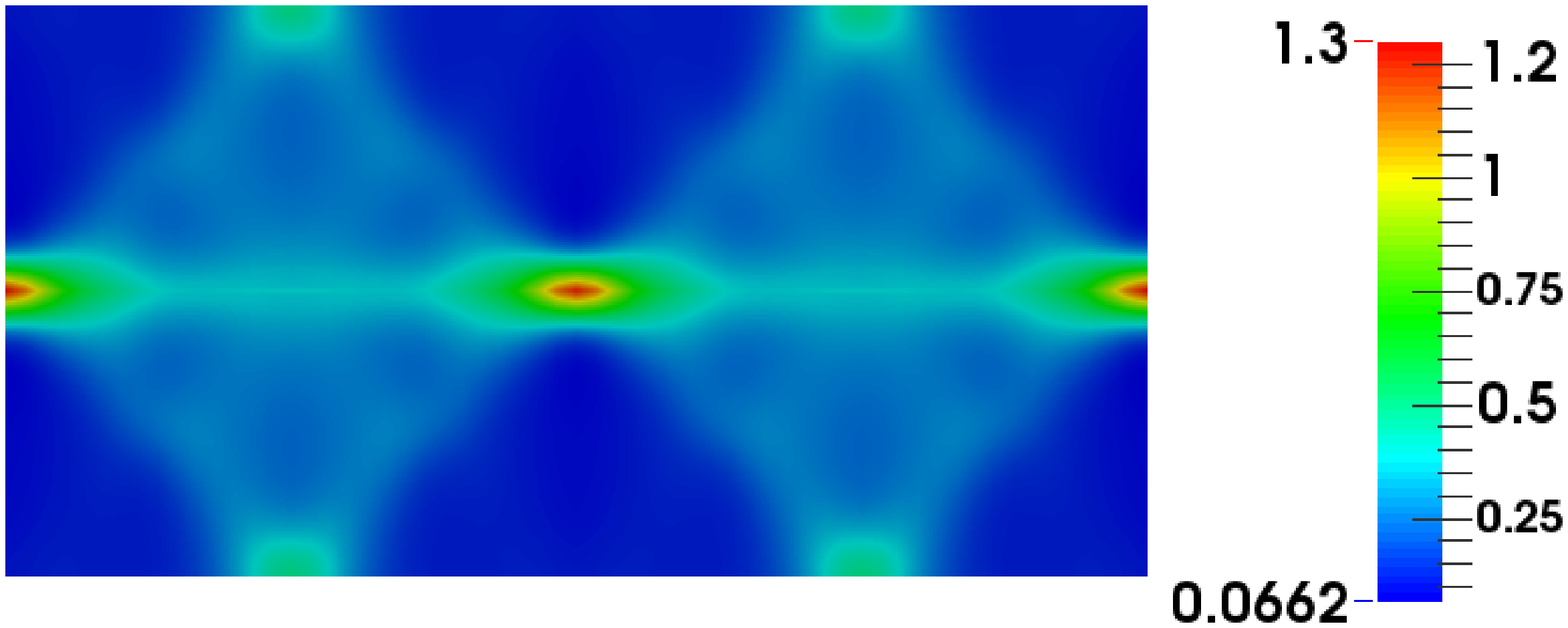}}
  \hspace{0.15in}
  \vspace{0.1in}
  \subfigure[Product $C$ at $t = 0.5$ (LSB and DMP constraints)]{\includegraphics[scale=0.30,clip]
    {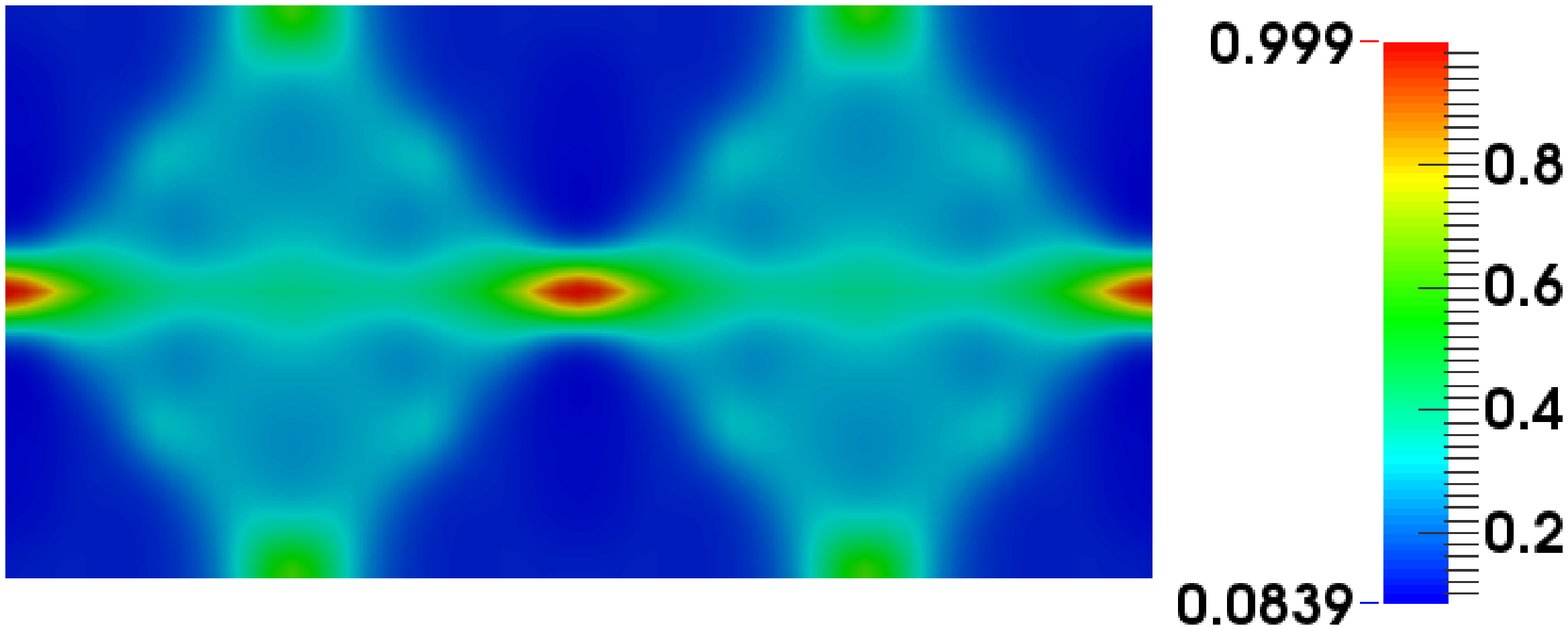}}
  \vspace{0.1in}
  \subfigure[Product $C$ at $t = 1.0$ (No constraints)]{\includegraphics[scale=0.30,clip]
    {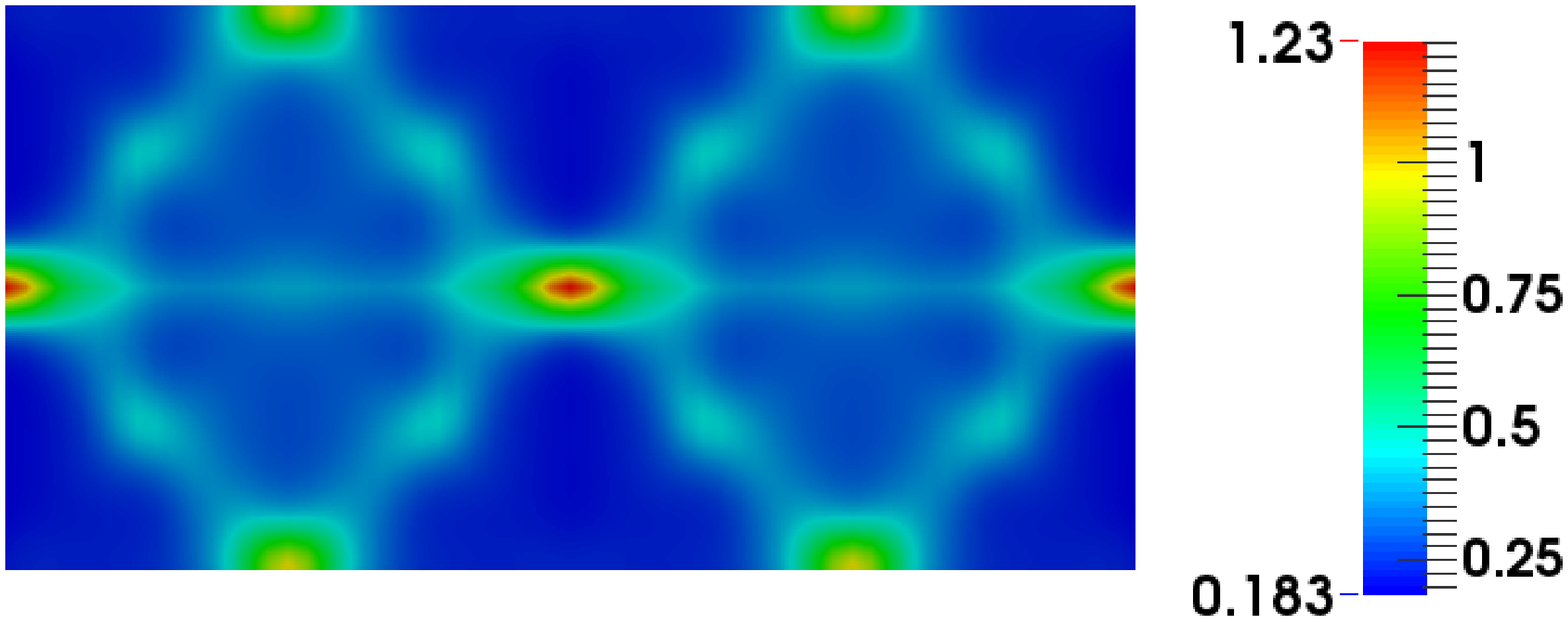}}
  \hspace{0.15in}
  \vspace{0.1in}
  \subfigure[Product $C$ at $t = 1.0$ (LSB and DMP constraints)]{\includegraphics[scale=0.30,clip]
    {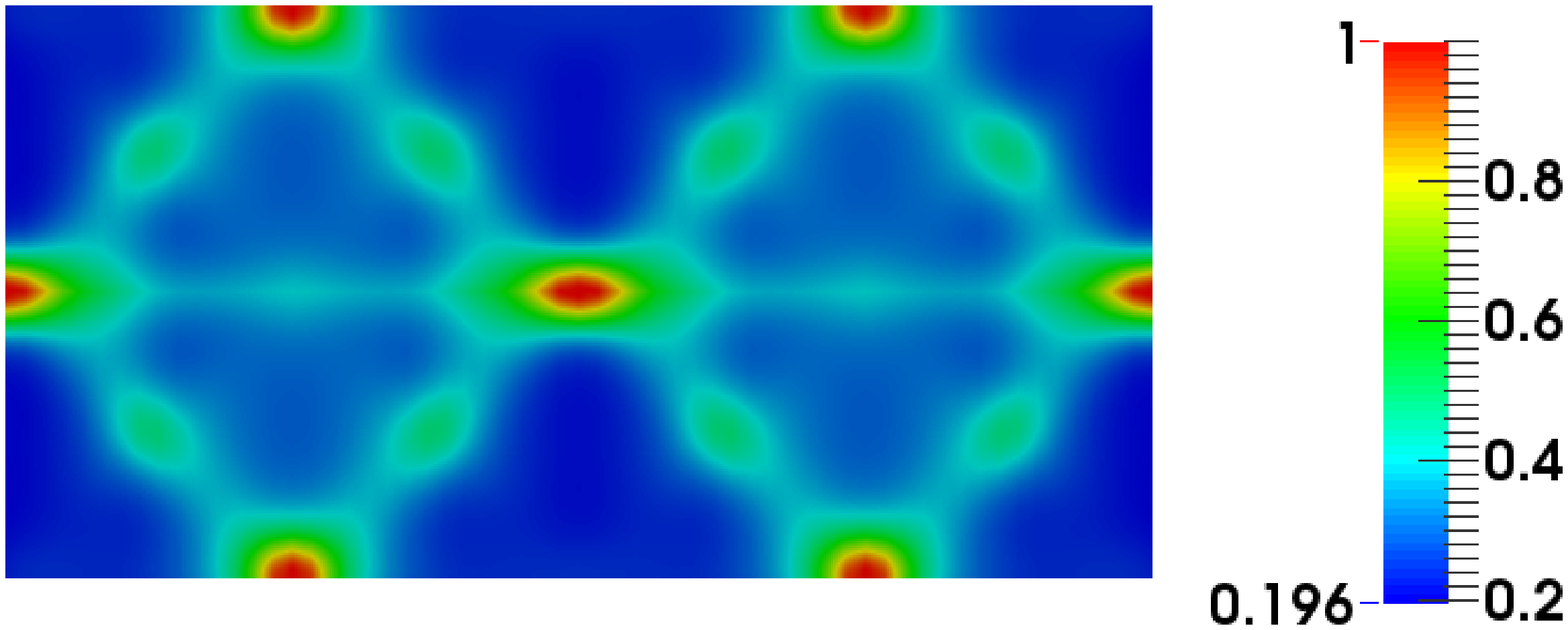}}
  \vspace{0.1in}
  \subfigure[Product $C$ at $t = 5.0$ (No constraints)]{\includegraphics[scale=0.30,clip]
    {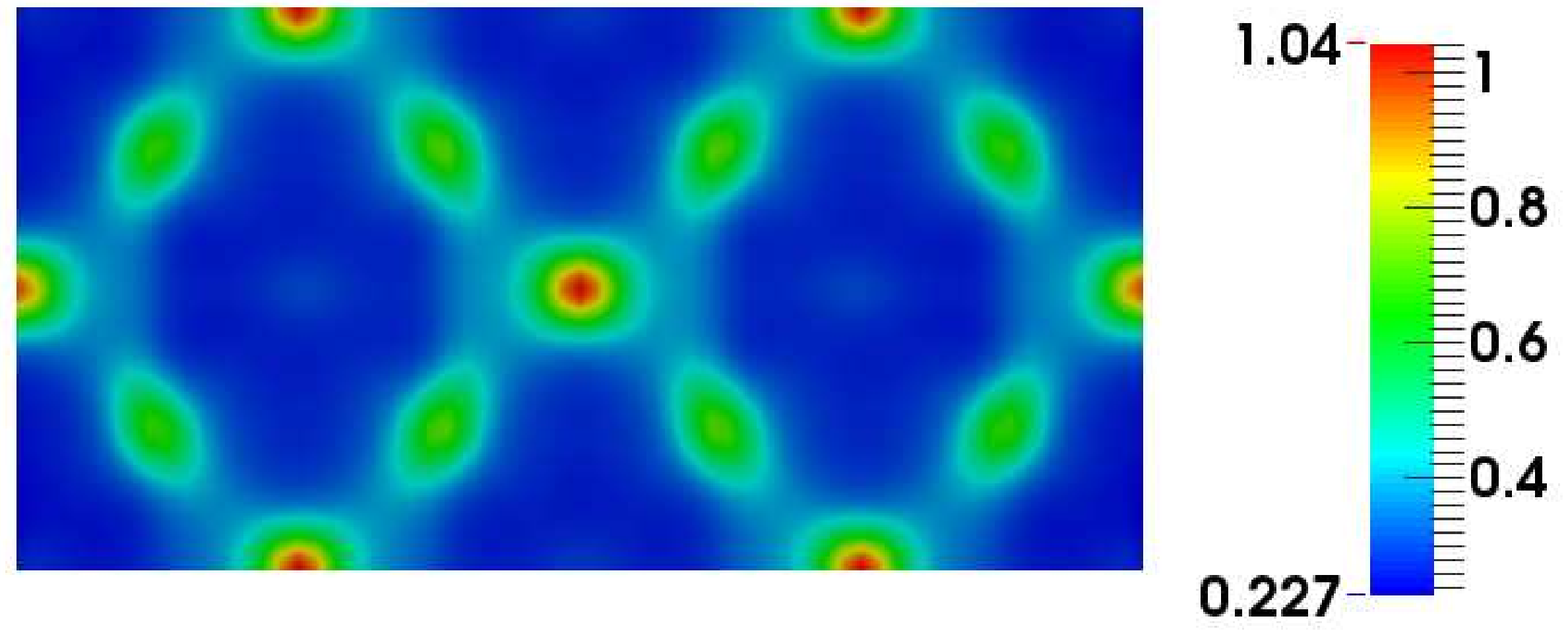}}
  \hspace{0.15in}
  \vspace{0.1in}
  \subfigure[Product $C$ at $t = 5.0$ (LSB and DMP constraints)]{\includegraphics[scale=0.30,clip]
    {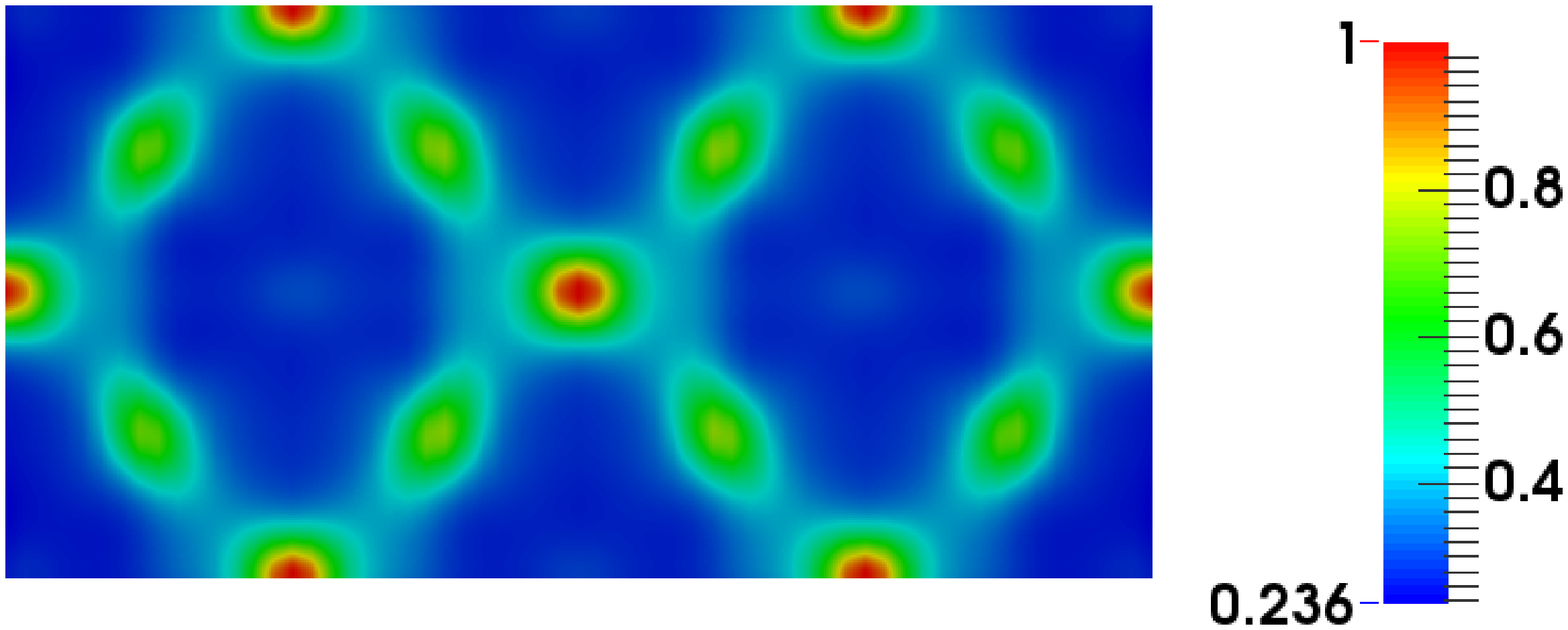}}
  \caption{\textsf{Transport-controlled mixing in cellular 
    flows:}~This figure shows the concentration profiles 
    of the product $C$ at various time levels using the 
    unconstrained and constrained weighted negatively 
    stabilized streamline diffusion LSFEM when $L_{\text{\tiny {Cell}}} 
    = 0.5$. The proposed computational framework is able 
    to produce physically meaningful solution (i.e., 
    satisfies the non-negative constraint, maximum 
    principle, and local species balance) for the 
    product $C$ in a transient cellular flow. The 
    time-step $\Delta t$ taken for the numerical simulation 
    is equal to 0.1.
    \label{Fig:2D_CellFlow_ScalDiffMixing_cCNSSDNoCons}}
\end{figure}

\begin{figure}
  \centering
  \subfigure[Product $C$ for $L_{\text{\tiny {Cell}}} = 0.25$]
    {\includegraphics[scale=0.30,clip]
    {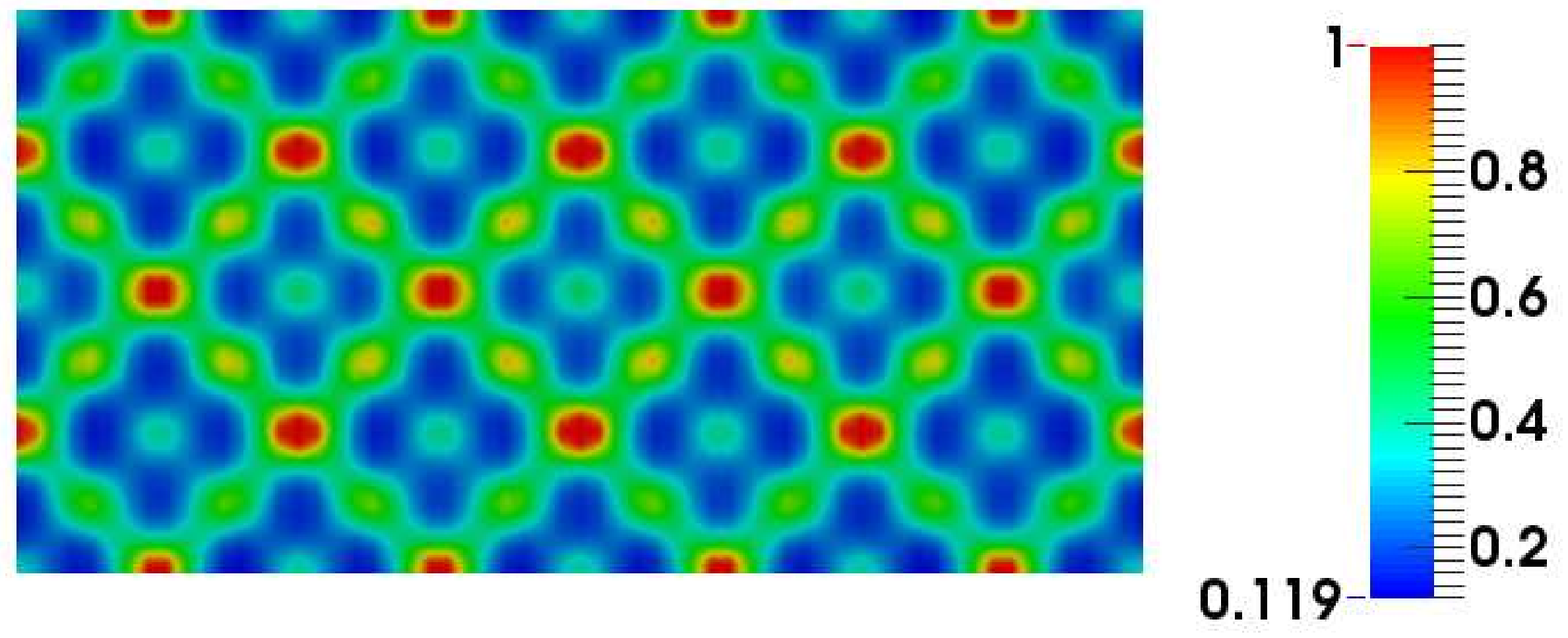}}
  \hspace{0.15in}
  \vspace{0.1in}
  \subfigure[Product $C$ for $L_{\text{\tiny {Cell}}} = 0.125$]
    {\includegraphics[scale=0.30,clip]
    {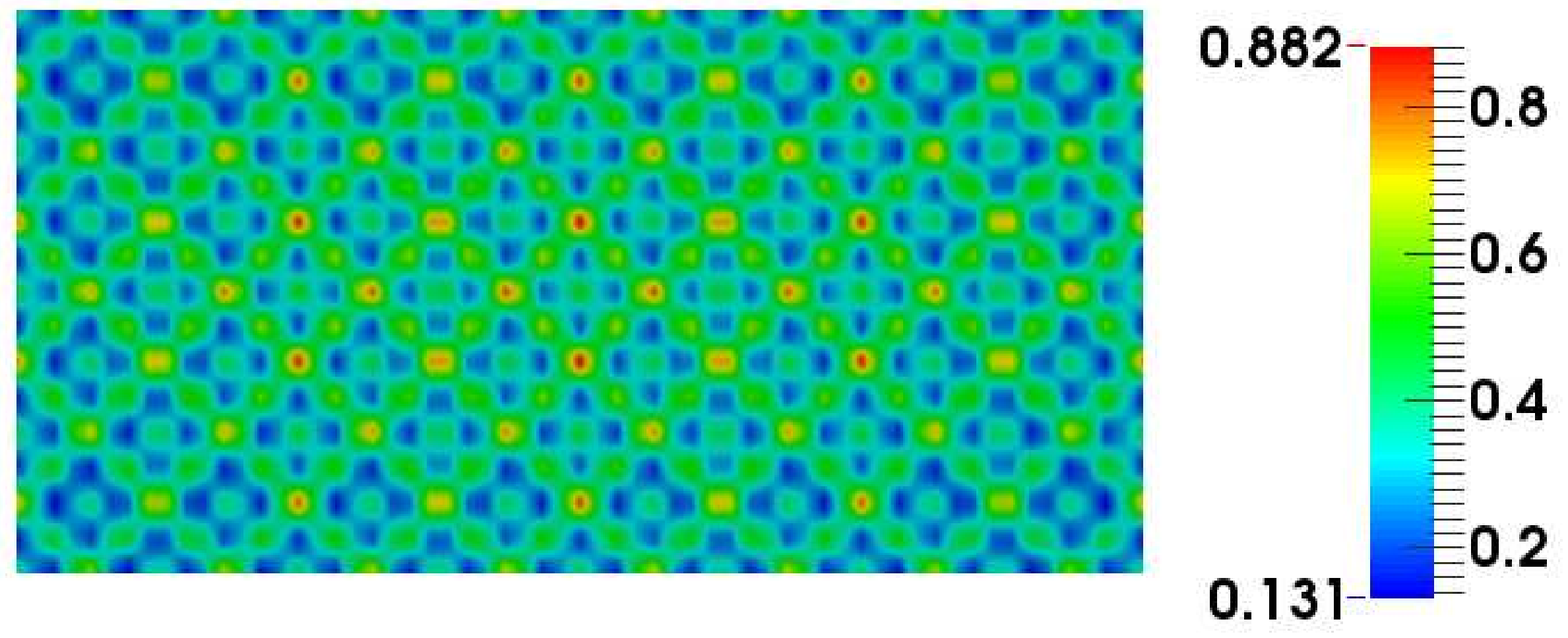}}
  \caption{\textsf{Transport-controlled mixing in cellular 
    flows (with hierarchical cell lengths):}~This 
    figure shows the concentration profiles of the product $C$ 
    at $t = 1$ using the proposed formulation with LSB and DMP 
    constraints. Analysis is performed for a series of hierarchical 
    $L_{\text{\tiny {Cell}}}$. The time-step $\Delta t$ is taken 
    to be equal to 0.1.
    \label{Fig:2D_CellFlow_ScalDiffMixing_cCNSSDConsLCells}}
\end{figure}

\begin{figure}
  \centering
  \subfigure[$\langle c_C \rangle$ vs $t$]
    {\includegraphics[scale=0.12,clip]{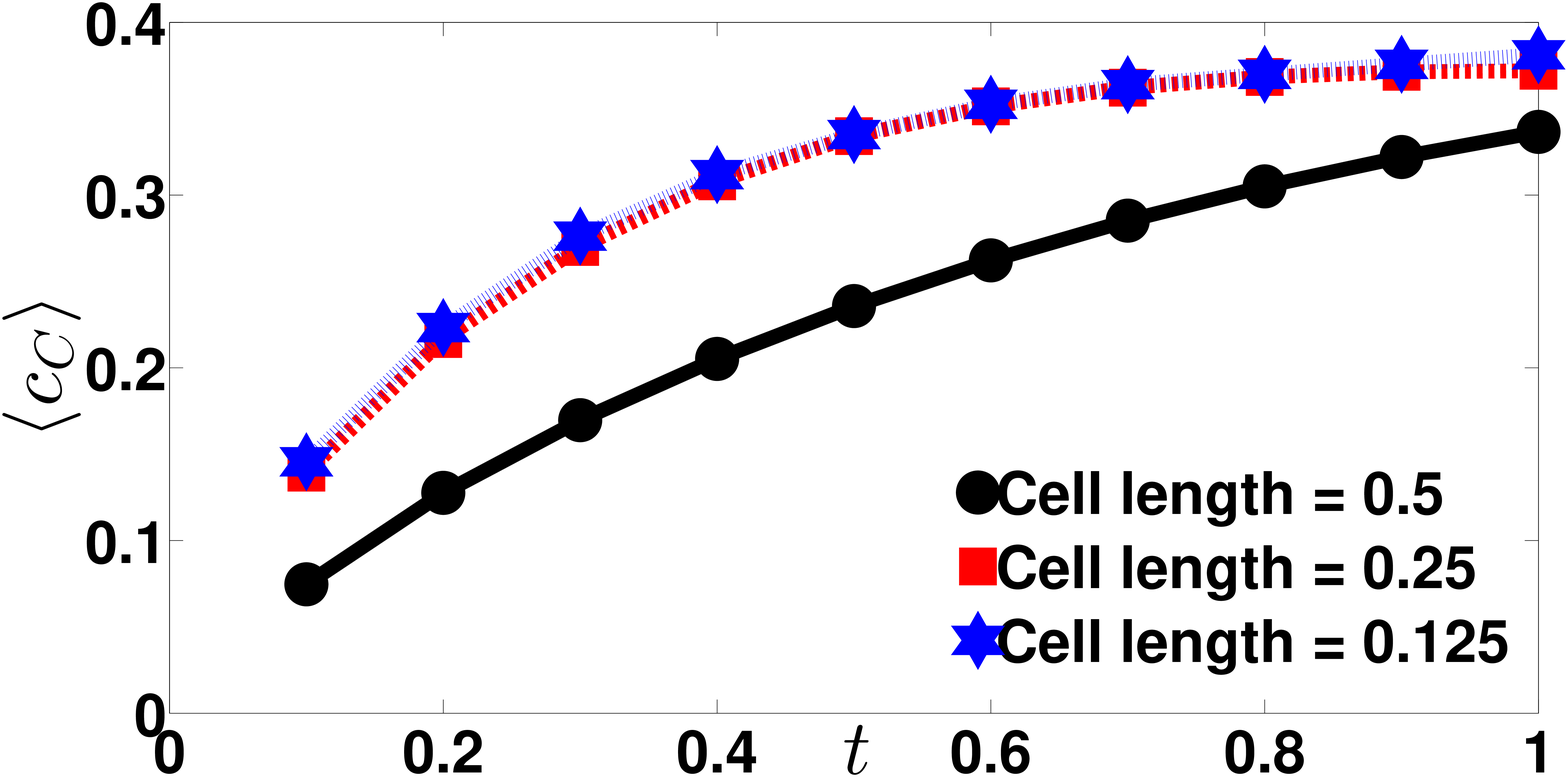}}
  \subfigure[$\langle c_C \rangle$ vs $L_{\text{\tiny {Cell}}}$]
    {\includegraphics[scale=0.12,clip]{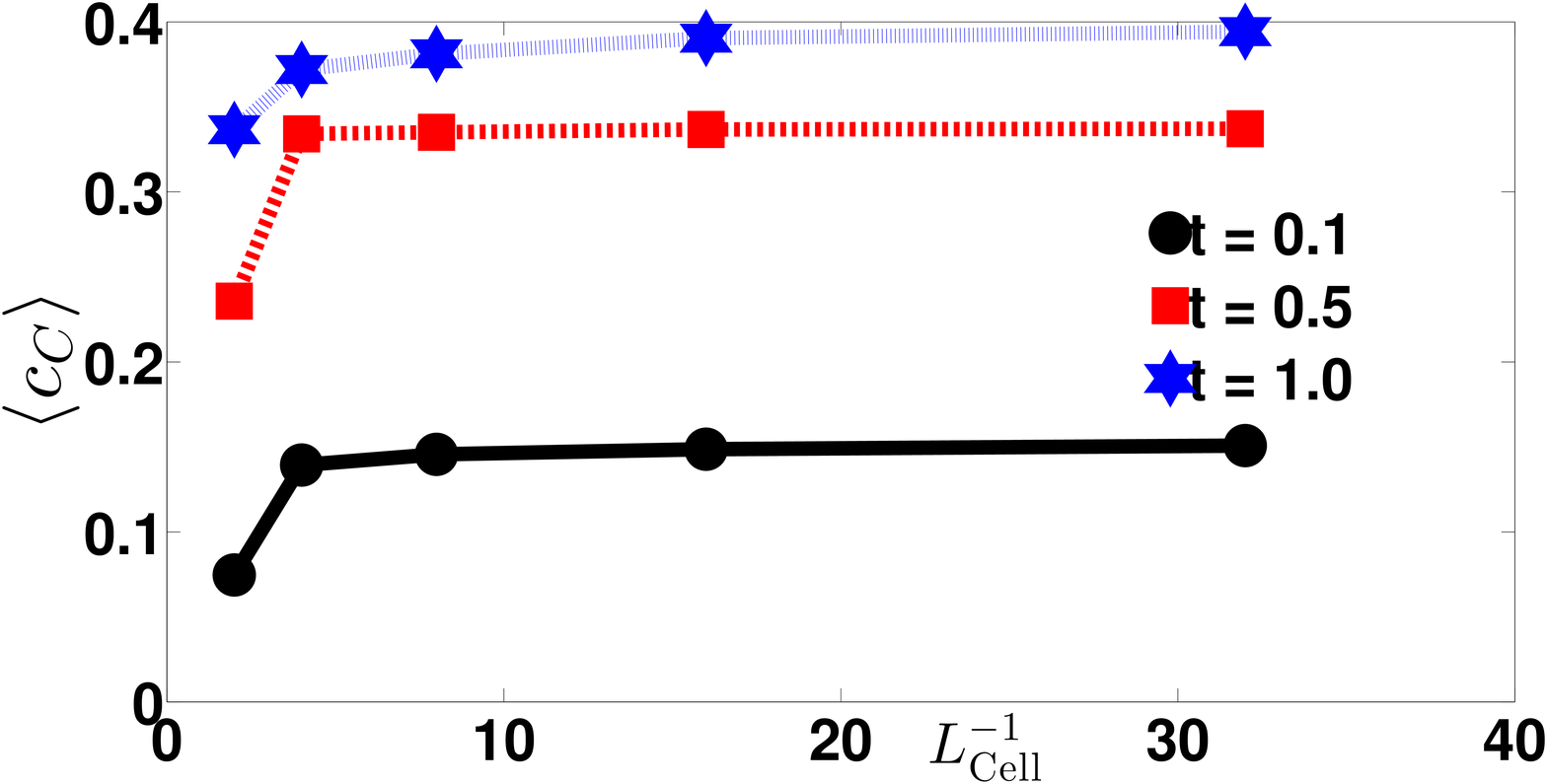}}
  \caption{\textsf{Transport-controlled mixing in cellular 
    flows (with hierarchical cell lengths):}~This 
    figure shows the average concentration of the product $C$ 
    at various times and different cell lengths. Analysis is 
    performed using the proposed formulation with LSB and DMP 
    constraints. The time-step $\Delta t$ is taken to be equal 
    to 0.1. The main inference from this numerical simulation 
    is that species mixing happens faster as $L_{\text{\tiny {Cell}}}$ 
    decreases.
    \label{Fig:2D_CellFlow_ScalDiffMixing_cCNSSDAvgLCells}}
\end{figure}

\end{document}